\documentclass{amsart} 
\usepackage{tikz-cd}
\usepackage{tikz}
\usepackage[utf8]{inputenc}
\usepackage{calculator}
\usepackage{comment}
\usepackage{xcolor}
\usepackage{graphicx}
\usepackage{adjustbox}

\usepackage{graphicx}
\usepackage{amsfonts}
\usepackage{bbm}
\usepackage{amsmath}
\usepackage{amssymb}
\usepackage{amsthm}
\usepackage{hyperref}

\tikzcdset{scale cd/.style={every label/.append style={scale=#1},
    cells={nodes={scale=#1}}}}
\numberwithin{equation}{section}
\newcommand{\nn}{\nonumber}

\newcommand{\Gr}{{\rm Gr}}

\newcommand{\Ind}{{\rm Ind}}

\newcommand{\Spec}{{\rm Spec}}
\newcommand{\QCoh}{{\rm QCoh}}
\newcommand{\Hom}{{\rm Hom}}
\newcommand{\Map}{{\rm Map}}
\newcommand{\Coh}{{\rm Coh}}
\newcommand{\Cofib}{{\rm Cofib}}
\newcommand{\Crit}{{\rm Crit}}
\newcommand{\Sym}{{\rm Sym}}

\newcommand{\Ext}{{\rm Ext}}
\newcommand{\Stab}{{\rm Stab}}

\newcommand{\cofib}{{\rm cofib}}

\newcommand{\IndCoh}{{\rm IndCoh}}
\newcommand{\colim}{{\rm colim}}
\newcommand{\Grad}{{\rm Grad}}
\newcommand{\Perv}{{\rm Perv}}
\newcommand{\Perf}{{\rm Perf}}
\newcommand{\Filt}{{\rm Filt}}

\newcommand{\sps}{{\rm sp}}
\newcommand{\mA}{\mathcal{A}}
\newcommand{\mPP}{\mathcal{P}}
\newcommand{\mC}{\mathcal{C}}
\newcommand{\mM}{\mathcal{M}}
\newcommand{\mD}{\mathcal{D}}
\newcommand{\mE}{\mathcal{E}}

\newcommand{\mS}{\mathcal{S}}
\newcommand{\bA}{\mathbb{A}}
\newcommand{\un}{\mathbbm{1}}
\newcommand{\bG}{\mathbb{G}}
\newcommand{\bZ}{\mathbb{Z}}

\newcommand{\bC}{\mathbb{C}}
\newcommand{\bD}{\mathbb{D}}
\newcommand{\bQ}{\mathbb{Q}}
\newcommand{\bN}{\mathbb{N}}
\newcommand{\bH}{\mathbb{H}}
\newcommand{\bV}{\mathbb{V}}
\newcommand{\bP}{\mathbb{P}}
\newcommand{\bR}{\mathbb{R}}
\newcommand{\bL}{\mathbb{L}}
\newcommand{\bT}{\mathbb{T}}
\newcommand{\mR}{\mathcal{R}}

\newcommand{\mU}{\mathcal{U}}
\newcommand{\mX}{\mathcal{X}}
\newcommand{\mY}{\mathcal{Y}}
\newcommand{\mZ}{\mathcal{Z}}
\newcommand{\ie}{i.e. }
\newcommand{\ud}{\underline}
\ifx\notheorems\undefined
\theoremstyle{plain}
\newtheorem{theorem}{Theorem}[section]
\newtheorem{proposition}[theorem]{Proposition}
\newtheorem{lemma}[theorem]{Lemma}
\newtheorem{corollary}[theorem]{Corollary}

\theoremstyle{definition}
\newtheorem{remark}[theorem]{Remark}
\newtheorem{definition}[theorem]{Definition}

\theoremstyle{remark}

\title{Hyperbolic localization in Donalson-Thomas theory}

\author{Pierre Descombes}

\begin{document}

\begin{abstract}
    In this paper we prove a toric localization formula in the cohomological Donaldson-Thomas theory. Consider a $-1$-shifted symplectic algebraic space with a $\bG_m$-action leaving the $-1$-shifted symplectic form invariant (typical examples are the moduli space of stable sheaves or complexes of sheaves on a Calabi-Yau threefold with a $\bG_m$-invariant Calabi-Yau form or the intersection of two $\bG_m$-invariant Lagrangians in a symplectic space with a $\bG_m$-invariant symplectic form). In this case we express the restriction of the Donaldson-Thomas perverse sheaf (or monodromic mixed Hodge module) defined by Joyce et al. to the attracting variety as a sum of cohomological shifts of the DT perverse sheaves on the $\bG_m$-fixed components. This result can be seen as a $-1$-shifted version of the Białynicki-Birula decomposition for smooth schemes. We obtain our result from a similar formula for stacks and Halpern-Leistner's $\Theta$-correspondence, at the level of perverse Nori motives, which we use also to derive foundational constructions in DT theory, in particular the Kontsevich-Soibelman wall crossing formula and the construction of the Cohomological Hall Algebra for smooth projective Calabi-Yau threefolds (a similar construction of the CoHA was also done independently by Kinjo, Park, and Safronov in a recent work). This paper subsumes the previous paper "Hyperbolic localization of the Donaldson-Thomas sheaf" from the same author.
\end{abstract}

\maketitle

\tableofcontents\newpage

\section{Introduction}

\subsection*{The Donaldson-Thomas perverse sheaf}

We work over an algebraically closed field $k$ of characteristic $0$. All our algebraic spaces are assumed to be quasi-separated and locally of finite type over $k$, and all our stacks are assumed to be quasi-separated Artin $1$-stacks, locally of finite type over $k$, with affine stabilizers. Donaldson-Thomas theory was first developed to count sheaves on Calabi-Yau threefolds. In \cite{thomas1998holomorphic}, Thomas defined the numerical Donaldson-Thomas invariants, giving the virtual Euler number of the moduli space of stable coherent sheaves on a Calabi-Yau threefold, using the perfect obstruction theory given by the Serre duality on the Ext spaces of the sheaves. In \cite{behrend_donaldson-thomas}, Behrend gave a new interpretation of these invariants: expressing the moduli space locally as the critical locus of a potential, the numerical Donaldson-Thomas invariant is given by an Euler characteristic weighted at each point by the Milnor number. In \cite{ks} and \cite{KonSol10}, Kontsevich and Soibelman sketched the definition of a cohomological refinement of this counting, with value in the abelian category of monodromic mixed Hodge modules (MMHM), using the functor of vanishing cycles of this potential, in a partially conjectural framework.\medskip

In papers \cite{Joyce2013ACM}, \cite{Joycesymstab}, and \cite{darbstack}, Joyce and collaborators have developed a rigorous cohomological Donaldson-Thomas theory, using the language of $-1$-shifted symplectic structures in derived geometry introduced in \cite{shifsymp}. One considers a derived algebraic space (or stack) $\ud{X}$ enhancing the classical one $X$: In particular, $\ud{X}$ carries a tangent-obstruction complex $\mathbb{T}_{\ud{X}}$ giving a deformation theory for $X$, \ie $H^0(\mathbb{T}_{\ud{X}})=T_X$ gives the tangent directions, $H^1(\mathbb{T}_{\ud{X}})$ gives the obstructions, $H^2(\mathbb{T}_{\ud{X}})$, $H^3(\mathbb{T}_X)_{\ud{X}}$, ... give the higher obstructions, and $H^{-1}(\mathbb{T}_{\ud{X}})$ gives the infinitesimal automorphisms (in the stack case). A $-1$-shifted symplectic structure is then a closed $2$ form of degree $-1$ on $\ud{X}$, such that the underlying $2$-form gives a nondegenerate pairing between $\mathbb{T}_{\ud{X}}$ and $\mathbb{T}_{\ud{X}}[-1]$: in particular, it pairs tangent directions with obstructions. In \cite{shifsymp}, \cite{Brav2018RelativeCS}, it was shown that moduli stacks of complexes on a Calabi-Yau 3-fold and intersections of two Lagrangians in a symplectic space have natural enhancement to $-1$-shifted symplectic stacks.\medskip

In \cite{darbscheme}, it was shown that a $-1$-shifted derived algebraic space $\ud{X}$ can be written étale locally as the critical locus of a function on a smooth scheme. More precisely, its classical truncation $X$ has a d-critical structure $s$, \ie $X$ can be covered by critical charts of the form $(R,U,f,i)$, where $R\to X$ is étale, $U$ is a smooth scheme, $f: U\to \mathbb{C}$ is a regular map, and $i:R\to U$ is the closed embedding of the critical locus of $f$, and two critical charts are related by adding quadratic terms. For any $d$-critical algebraic space $(X,s)$, there is a natural line bundle $K_{X,s}$ on $X^{red}$, which is equal in the above case to $\det(\bL_{\ud{X}})|_{X^{red}}$. The vanishing cycle of a quadratic form is trivial, up to a sign ambiguity, which is resolved by fixing an orientation of the quadratic form. In \cite{Joycesymstab}, Joyce and collaborators constructed the Donaldson-Thomas perverse sheaf $P_{X,s,K_{X,s}^{1/2}}$ carrying a monodromic mixed Hodge module for a d-critical algebraic space $(X,s)$, with additional data called the orientation $K_{X,s}^{1/2}$, \ie a square root of $K_{X,s}$, which is used to fix this sign ambiguity. By definition, the restriction of $P_{X,s,K_{X,s}^{1/2}}$ to a critical chart $(R,U,f,i)$ is given by $P_{U,f}\otimes_{\bZ/2\bZ}Q_{R,U,f,i}$, a twist of the perverse sheaf of vanishing cycles $P_{U,f}$ by a $\bZ/2\bZ$-bundle $Q_{R,U,f,i}$ depending on the orientation $K_X^{1/2}$, which serves to cancel the sign ambiguity. Here, $P_{U,f}$ is defined by applying the monodromic vanishing cycles functor $\phi^{mon,tot}_f$ of $f$ to $\bQ_U\{\dim(U)/2\}$, the shifted constant sheaf of $U$, and then restricting to the critical locus $R$. The numerical Donaldson-Thomas invariant defined in \cite{thomas1998holomorphic} is then, as expected, the Euler number of the cohomology of this perverse sheaf.\medskip

We show here that this construction upgrades to the level of monodromic perverse Nori motives by extending in Section \ref{sectmonothomseb} the formalism of monodromic objects, the Thom-Sebastiani theorem, and the description of the square root of the Tate twist. The results of \cite{darbscheme}, \cite{Joycesymstab} have been extended to stacks in \cite{darbstack}. Notice that, at that time, the formalism of mixed Hodge modules on stacks was not developed; hence, for an oriented $d$-critical stack, the object $P_{\mX,s,K_{\mX,s}^{1/2}}$ was only built as a perverse sheaf. Thanks to the development of mixed Hodge modules and perverse Nori motives on stacks in \cite{tubachMHM}, one obtains an enhancement of $P_{\mX,s,K_{\mX,s}^{1/2}}$ at the level of mixed Hodge modules and perverse Nori motives.\medskip

We will mainly work at the level of d-critical structures in this article, as this is sufficient in the formalism of \cite{Joycesymstab}. Shifted symplectic geometry will only be used to check the compatibility between various $d$-critical structures, as the classical construction of $d$-critical structures in enumerative geometry always comes from constructions at the level of $-1$-shifted symplectic structures.

\subsection*{Hyperbolic localization}

The aim of this paper is to provide a way to compute the cohomological Donaldson-Thomas invariants by localization. Namely, given an algebraic space $X$ with a $\bG_m$-action, a $\bG_m$-invariant d-critical structure (\ie, of weight $0$) and a $\bG_m$-equivariant orientation, we want to express the DT invariants of $X$ in terms of the DT invariants of the $\bG_m$-fixed algebraic space $X^0$. Graber and Pandharipande proved a torus localization formula for numerical Donaldson-Thomas invariants in \cite{Graber1997LocalizationOV}. A similar formula was derived in \cite{Behrend08symmetricobstruction} using the alternative definition with weighted Euler characteristic: the numerical Donaldson-Thomas invariant of $X$ is the sum of those of each component of $X^0$ weighted by a sign given by the parity of the dimension of the normal space to the component. Thus, numerical Donaldson-Thomas invariants localize under torus action exactly like the Euler numbers of smooth spaces.\medskip

The fixed locus can be described in terms of functor of points as $X^0:=Map^{\bG_m}(\Spec(k),X)$. Consider $\bA^1$ with its canonical $\bG_m$-action. We consider the attracting variety $X^+:=Map^{\bG_m}(\bA^1,X)$: informally, it classifies points $x\in X$ with a limit $\lim_{t\to 0}t.x$. Such a limit is unique when $X$ is separated and always exists and is unique when $X$ is proper. From \cite{Braden2002HyperbolicLO} and \cite{drinfeld2013algebraic}, those functors are representable by algebraic spaces, and one obtains so-called hyperbolic localization correspondence:
\[\begin{tikzcd}
X^0\arrow[rr,bend right,"\iota"] & X^+\arrow[r,"\eta^+"]\arrow[l,swap,"p^+"] & X
\end{tikzcd}\]
where $\eta^+$ forgets $\lim_{t\to 0}t.x$, and $p^+$ sends $x$ to $\lim_{t\to 0}x$. Notice that $X^+$ is the disjoint union of strata flowing to different connected components of $X^0$, so $\eta^+$ can be thought of as the immersion of strata from a stratification. One obtains a true stratification when $X$ is projective with linear action; otherwise, the situation can be more pathological.\medskip

In \cite{BiaynickiBirula1973SomeTO}, Białynicki-Birula proved that, when $X$ is smooth, each component of $X^+$ is an affine fiber bundle over a component of the fixed variety $X^0$, whose dimension is given by $d^+_X$, the number of positive weights in $T_X|_{X^0}$ (this is a locally constant integer-valued function on $X^0$). Thus, one can compute the cohomology of the attracting variety of $X^+$ in terms of the cohomology of $X^0$. Denote by $X^0=\bigsqcup_{\mathbf{c}}X^0_{\mathbf{c}}$ the decomposition of $X^0$ into connected components. In \cite{Braden2002HyperbolicLO}, Braden introduced the hyperbolic localization functors $(p^+)_!(\eta^+)^\ast$, and reinterpreted Białynicki-Birula's result as a decomposition theorem:
\begin{align}
    (p^+)_!(\eta^+)^\ast\mathbb{Q}_X&\simeq \bQ_{X^0}\{-d^+_X\}\simeq\bigoplus_{\mathbf{c}}\mathbb{Q}_{X^0_{\mathbf{c}}}\{-d^+_{\mathbf{c}}\}\label{locintcoh}
\end{align}
where $\mathbb{Q}_Y$ is the constant sheaf of $Y$, and $\{1\}:=[2](1)$ denotes the combination of a shift and a Tate twist. More precisely, Braden, in \cite{Braden2002HyperbolicLO}, and furthermore Drinfeld and Gaitsgory, in \cite{Drinfeld2013ONAT}, showed that, for any algebraic space $X$, $(p^+_X)_!(\eta^+_X)^\ast$ preserve the weight of $\bG_m$-equivariant constructible complexes: in particular, one obtains an analogue of the celebrated decomposition theorem for proper maps of \cite{BBDGfaisceauxpervers}, and, when $X$ is smooth, \ie $\mathbb{Q}_X$ is pure, \eqref{locintcoh} can be interpreted as a particularly simple example of this. An other consequence of Braden's theorem is that hyperbolic localization commutes with vanishing cycles, as proven in \cite{Ric16}, which will be at the heart of our work.\medskip

\subsection*{The main result: toric localization}

Consider an algebraic space $X$ with a $\bG_m$-action, a $\bG_m$-invariant d-critical structure (\ie, of weight $0$) and a $\bG_m$-equivariant orientation. We prove in Proposition \ref{cordcritgraded} that $s^0:=\iota^\star(s)$ is a $d$-critical structure on $X^0$, and that given an orientation $K_{X,s}^{1/2}$ of $(X,s)$, there is a canonical orientation $K_{X^0,s^0}^{1/2}$ on $(X^0,s^0)$. We consider the integer-valued (which is shown a posteriori to be locally constant) function on $X^0$:
    \begin{align}
        \Ind_X(x):=\dim(T_{X,x}^{>0})-\dim(T_{X,x}^{<0})
    \end{align}
In particular, when $(X,s)$ comes from a $\bG_m$-invariant $-1$-shifted symplectic algebraic space $(\ud{X},\omega)$, we prove in Lemma \ref{lemorientgm}, \ref{lemshifsymp} that $(\ud{X}^0,\omega^0:=\iota^\ast(\omega))$ is $-1$-shifted symplectic, with classical truncation $(X^0,s^0)$, that $K_{X^0,s^0}^{1/2}$ is given by the formula:
\begin{align}
K_{X^0,s^0}^{1/2}:=\iota^\ast(K_{X,s}^{1/2})\otimes\det(\bL_{\ud{X}}|_{\ud{X}^0}^{<0})|_{(X^0)^{red}}^{\otimes -1}
\end{align}
and that $\Ind_X$ coincide with the signed count of positive weights in the tangent complex $\bT_{\ud{X}}|_{\ud{X}^0}^{>0}$. The aim of this article is to prove that an analogue of \eqref{locintcoh} holds for the Donaldson-Thomas perverse sheaf (resp. monodromic mixed Hodge module or perverse Nori motive) on a d-critical oriented algebraic space:\medskip

\begin{theorem}(Theorem \ref{theoBB})
    For $X$ an algebraic space with a $\bG_m$-action and a $\bG_m$-invariant d-critical structure $s$ and a $\bG_m$-equivariant orientation $K_{X,s}^{1/2}$, there is a natural isomorphism of perverse sheaves (resp. monodromic mixed Hodge modules, resp. monodromic perverse Nori motives):
    \begin{align}\label{isomspace}
       (p^+)_!(\eta^+)^\ast P_{X,s,K_{X,s}^{1/2}}\simeq P_{X^0,s^0,K_{X^0,s^0}^{1/2}}\{-\Ind_X/2\}
    \end{align}
    Denote by $X^0=\bigsqcup_{\mathbf{c}}X^0_{\mathbf{c}}$ the decomposition of $X^0$ into connected components, and by $\Ind_{\mathbf{c}}$ the constant value of $\Ind_X$ on $X^0_{\mathbf{c}}$. Suppose, moreover, that $X$ is separated of finite type; then $\eta$ is injective on $k$-points, and we have the following equality for the class of the cohomology with compact support in the Grothendieck ring of monodromic Nori motives and monodromic mixed Hodge structures (where we have not written the dependency on the $d$-critical structure and orientation in $P$ for readablility):
    \begin{align}\label{nonproj}
[\bH_c(X,P_X)]=\sum_{\mathbf{c}}\mathbb{L}^{\Ind_{\mathbf{c}}/2}[\bH _c(X^0_{\mathbf{c}},P_{X^0_{\mathbf{c}}})]+[\bH_c(X-\eta(X^+),P_X|_{X-\eta(X^+)})]
    \end{align}
    where $[\bH_c(Y,F)]$ denotes the class in the Grothendieck group of the hypercohomology with compact support of $F$, and $\bL^{1/2}$ the square root of the Tate motive/Hodge structure, which exists at the monodromic level. In particular, if $X$ is proper, $\eta$ is bijective on $k$-points, and we obtain the simpler formula:
    \begin{align}\label{circomploc}
[\bH ^T_c(X,P_X)]=\sum_{\mathbf{c}}\mathbb{L}^{\Ind_{\mathbf{c}}/2}[\bH_c(X^0_{\mathbf{c}},P_{X^0_{\mathbf{c}}})]
    \end{align}
\end{theorem}

\subsection*{The main result: stacky version}

We derive our result from a more general stacky version that is also useful to obtain foundational results in DT theory. As said above, all our stacks are assumed to be quasi-separated Artin 1-stacks, locally of finite type over an algebraically closed field $k$ of characteristic $0$, with affine stabilizers. For such a stack $\mX$, Halpern-Leistner introduced in \cite{HalpernLeistner2014OnTS} the stack of graded and filtered points:
\begin{align}
    \Grad(\mX):=Map(B\bG_{m,k},\mX)\nn\\
    \Filt(\mX):=Map([\bA^1_k/\bG_{m,k}],\mX)
\end{align}
Considering the maps $\Spec(k)\to B\bG_{m,k}$, $1:\Spec(k)\to[\bA^1_k/\bG_{m,k}]$ and $0:B\bG_{m,k}\to [\bA^1_k/\bG_{m,k}]$, one obtains by functoriality a stacky version of the hyperbolic localization correspondence:
\[\begin{tikzcd}
    \Grad(\mX)\arrow[rr,bend right,"\iota"] & \Filt(\mX)\arrow[l,swap,"p"]\arrow[r,"\eta"] & \mX
\end{tikzcd}\]
When $\mX$ is the stack of objects in an Abelian category, according to \cite[Lemma 6.3.1]{HalpernLeistner2014OnTS}, $\Filt(\mX)$ classifies the filtered objects, $\Grad(\mX)$ classifies the graded objects, $\iota$ forgets the gradation, $\eta$ forgets the filtration, and $p$ takes the associated graded. This explains the name: in this situation, we obtain the usual correspondence used to study Harder-Narasimhan stratification and cohomological Hall algebras. When $X$ is an algebraic space with an action of a reductive group $G$, and we consider $\mX=[X/G]$, we obtain:
\[\begin{tikzcd}
    \bigsqcup_\lambda {[X^0_\lambda/L_\lambda]}\arrow[rr,bend right,"\bigsqcup_\lambda\iota_\lambda"]& \bigsqcup_\lambda {[X^+_\lambda/P_\lambda]}\arrow[l,swap,"\bigsqcup_\lambda p^+_\lambda"]\arrow[r,"\bigsqcup_\lambda\eta^+_\lambda"] & {[X/G]}
\end{tikzcd}\]
where the sum is over conjugacy classes of cocharacter $\lambda:\bG_m\to G$, $P_\lambda$ and $L_\lambda$ denote the associated parabolic and Levi, and $X^+_\lambda,X^0_\lambda$ the associated attracting and fixed varieties. In particular, for $G=\bG_m$, a connected component of the above correspondence is:
\[\begin{tikzcd}
{[X^0/\bG_m]}\arrow[rr,bend right,"\iota"] & {[X^+/\bG_m]}\arrow[r,"\eta^+_X"]\arrow[l,swap,"p^+_X"] & {[X/\bG_m]}
\end{tikzcd}\]
and the usual hyperbolic localization correspondence is a smooth cover of this.\medskip

For a stack $\mX$, the restriction of the tangent complex of $\mX$ to $\Grad(\mX)$ inherits naturally a $\bZ$ grading from the mapping construction (as a quasi-coherent complex on $B\bG_m$ is equivalent to a $\bZ$-graded complex). For a $d$-critical stack $(\mX,s)$, we consider the integer-valued (which is shown a posteriori to be locally constant in Proposition \ref{cordcritgraded} $ii)$) function on $\Grad(\mX)$:
\begin{align}
    \Ind_\mX(x):=-\dim(\mathfrak{Iso}_{\iota(x)}(\mX)^{>0})+\dim(T_{X,\iota(x)}^{>0})-\dim(T_{X,\iota(x)}^{<0})+\dim(\mathfrak{Iso}_{\iota(x)}(\mX)^{<0})
\end{align}
where $T_{\mX,y}$ denotes the Zariski tangent space, and $\mathfrak{Iso}_{y}(\mX)$ the Lie algebra of the isotropy group. We show in Proposition \ref{cordcritgraded} that $(\Grad(\mX),\Grad(s):=\iota^\ast(s))$ is a natural $d$-critical stack, and that an orientation on $(\mX,s)$ gives a canonical orientation on $(\Grad(\mX),\Grad(s))$. When $(\mX,s)$ comes from a $-1$-shifted symplectic stack $(\ud{\mX},\omega)$, we prove in Lemma \ref{lemorientgm}, \ref{lemshifsymp} that $(\ud{\Grad}(\ud{\mX}),\Grad(\omega):=\ud{\iota}^\ast(\omega))$ is a $-1$-shifted symplectic stack with classical  truncation $(\Grad(\mX),\Grad(s))$, that $\Ind_\mX$ coincides with the signed count of positive weights in the tangent complex $(\ud{\iota}^\ast\bT_{\ud{\mX}})^{>0}$, and that $K_{\Grad(\mX),\Grad(s)}^{1/2}$ is given by the formula:
\begin{align}
K_{\Grad(\mX),\Grad(s)}^{1/2}:=\iota^\ast(K_{\mX,s}^{1/2})\otimes\det((\ud{\iota}^\ast\bL_{\ud{\mX}})^{<0})|_{(\Grad(\mX))^{red}}^{\otimes -1}
\end{align} 
The main result is then a stacky version of the above:

\begin{theorem}(Theorem \ref{theohyploc})
    For an oriented d-critical stack $(\mX,s,K_{\mX,s}^{1/2})$, there are natural isomorphisms of perverse sheaves (resp. monodromic mixed Hodge modules, resp. monodromic perverse Nori motives):
    \begin{align}\label{isomstack}
       p_!\eta^\ast P_{\mX,s,K_{\mX,s}^{1/2}}\simeq P_{\Grad(\mX),\Grad(s),K_{\Grad(\mX),\Grad(s)}^{1/2}}\{-\Ind_\mX/2\}
    \end{align}
\end{theorem}

$\Theta$-stratifications are generalizations of Harder-Narasimhan stratifications which are introduced in \cite{HalpernLeistner2014OnTS}, see Section \ref{sectthetastrat}. We obtain for them:

\begin{theorem}(Theorem \ref{theothetastrat})
    Consider $(\mX,s,K_{\mX,s}^{1/2})$ an oriented d-critical stack of finite type. Consider a $\Theta$-stratification on $\mX$ (in particular, because $\mX$ is of finite type, it has a finite number of nonempty strata). The centers $\mathcal{Z}_{\mathbf{c}}$ of the $\Theta$-strata naturally enhance to oriented d-critical stacks $(\mathcal{Z}_{\mathbf{c}},s_{\mathbf{c}},K_{\mathcal{Z}_{\mathbf{c}},s_{\mathbf{c}}}^{1/2})$, and we have the following equality in the Grothendieck ring of monodromic mixed Hodge structures (resp. monodromic Nori motives) completed at $\bL^{-1/2}$ (see Section \ref{SectKgroup}):
    \begin{align}
        [\bH c(\mX,P_{\mX,s,K_{\mX,s}^{1/2}})]=\sum_{\mathbf{c}}\mathbb{L}^{\Ind_{\mathbf{c}}/2}[\bH c(\mathcal{Z}_{\mathbf{c}},P_{\mathcal{Z}_{\mathbf{c}},s_{\mathbf{c}},K_{\mathcal{Z}_{\mathbf{c}},s_{\mathbf{c}}}^{1/2}})]
    \end{align}
\end{theorem}

\subsection*{Application: Kontsevich-Soibelman wall crossing formula for smooth projective CY3}

We consider a smooth projective Calabi-Yau threefold $X$ (\ie, we fix a trivialization $\omega_X\simeq\mathcal{O}_X$ of the canonical bundle). For $E,F\in D^b\Coh(X)$, we consider the Euler pairing:
\begin{align}
    \langle E,F\rangle:=\sum_{i\in\bZ}(-1)^{i+1}\dim(\Ext^i(E,F))
\end{align}
which is antisymmetric as $X$ is CY3, by Serre duality, and pass to the Grothendieck group. We consider $K^{num}(X)$, the quotient of the Grothendieck group by the kernel of this pairing, which is a finite-dimensional lattice by Grothendieck-Hirzebruch-Riemann-Roch. From \cite{Toen2005ModuliOO}, \cite{Brav2018RelativeCS}, there is a higher derived stack $\ud{\mM}$ of objects in $D^b\Coh(X)$, which is locally of finite presentation, with a canonical $-1$-shifted symplectic structure (as $X$ is smooth and proper, this can also be obtained by the mapping construction from \cite{shifsymp}, but the description of \cite{Toen2005ModuliOO}, \cite{Brav2018RelativeCS} is more useful to study substacks of objects lying in subcategories). Its tangent-obstruction complex at a $k$-point $E\in D^b\Coh(X)$ is given by
\begin{align}\label{eqTanDG}
    \bT_{\ud{\mM},E}\simeq \bR \Hom(E,E)[1]
\end{align}
Moreover, its classical truncation $\mM$ carries an orientation data compatible with direct sum from \cite{JOYCEorient} (see Definition \ref{defOrientdg} here). \medskip

We consider a connected component $\Stab^\ast(X)$ of the space of Bridgeland stability conditions $\Stab(X)$ on $D^b\Coh(X)$ (where we consider numerical stability conditions satisfying the support property; see Definition \ref{defstab}). We need some technical conditions, studied in \cite{AbramovichPolishchuk}, \cite{Toda2007ModuliSA}, \cite{PiyaratneToda+2019+175+219}, to be able to obtain that the substacks of semistable objects are open and of finite type: for this, we assume that $\Stab^\ast(X)$ is good, \ie that it contains an algebraic stability condition satisfying boundedness and generic flatness; see Definition \ref{defgood}. From \cite[Corollary 4.21]{PiyaratneToda+2019+175+219}, this is satisfied for any connected component containing a stability condition built in \cite{Bayer2011BridgelandSC} (informally, for a connected component containing a large volume limit).\medskip

We follow then the discussion from \cite{ks}. We consider the ring $N_{mon}$ of monodromic Nori motives completed at $\bL^{-1/2}$ from Section \ref{SectKgroup}, and the motivic quantum torus:
\begin{align}
    G:=N_{mon}\langle (x^\gamma)_{\gamma\in K^{num}(X)}\rangle/((x^\gamma\cdot x^{\gamma'}-\bL^{\langle \gamma,\gamma'\rangle/2}x^{\gamma+\gamma'})_{\gamma,\gamma'\in K^{num}(X)},x^0-1)
\end{align}
For an interval $I$ of $\bR$ and a numerical class $\gamma\in K^{num}(X)$, we consider the stack $\mM_{I,\gamma}$ of objects of $D^b\Coh(X)$ of class $\gamma$, with their Harder-Narasimhan factors having their phase in $I$. From \cite[Theorem 6.5.3]{HalpernLeistner2014OnTS}, when $I$ has length $<1$, the Harder-Narasimhan decomposition gives a $\Theta$-stratification on $\mM_{I,\gamma}$: we apply then our Theorem \ref{theothetastrat} to it. For a strata $(\gamma_1,\cdots,\gamma_n)$ corresponding to a decomposition $\gamma=\gamma_1+\cdots+\gamma_n$, we obtain from \eqref{eqTanDG}:
\begin{align}
\Ind_{(\gamma_1,\cdots,\gamma_n)}=\sum_{i<j}\langle \gamma_i,\gamma_j\rangle
\end{align}
We obtain then the Kontsevich-Soibelman wall crossing formula (following the discussion of \cite[Section 2]{ks}, it is an equality in a particular completion of the motivic quantum torus $G$ defined using the support property; see Section \ref{sectKSWCF} for the precise statement):

\begin{theorem}(Theorem \ref{theoKSwcf})
    Consider a smooth and projective Calabi-Yau threefold $X$, and a good connected component $\Stab^\ast(X)$ of the space of stability conditions.
    \begin{enumerate}
        \item[$i)$] Consider $\sigma\in \Stab^\ast(X)$ and an interval $I$ of length $<1$ with associated strict sector $V:=\{me^{i\pi\phi}|m>0,\phi\in I\}\subset \bC$. Then $\mM_{I,\gamma}$, $\mM_{\phi,\gamma}$ are oriented $d$-critical Artin $1$-stacks of finite type with affine diagonal, and the Kontsevich-Soibelman wall crossing formula holds:
    \begin{align}
A^\sigma_V:=\sum_\gamma[\bH_c(\mM_{I,\gamma},P_{\mM_{I,\gamma}})]x^\gamma=\prod_{\phi\in I}^\to\sum_\gamma[\bH_c(\mM_{\phi,\gamma},P_{\mM_{\phi,\gamma}})]x^\gamma=:\prod_{l\subset V}^\to A^\sigma_l
    \end{align}
    where the symbol $\prod_{\phi\in I}^\to$ (resp. $\prod_{l\in V}^\to$) denotes an oriented product on decreasing $\phi\in I$ (resp. on the half-lines $l\subset V$ in the clockwise order).

    \item[$ii)$] The family of stability data $(Z,(\log(A^\sigma_{l_\gamma})_\gamma)_{\gamma\in K^{num}_Z}))$ (where $l_\gamma$ denotes the half line of $\bC$ containing $Z(\gamma)$) defines a continuous family of stability data on $\Stab^\ast(X)$ in the sense of \cite[Definition 3]{ks}. In particular, it defines a wall crossing structure on $\Stab^\ast(X)$ in the sense of \cite[Definition 2.2.1]{kontsevich_wall}; hence, one obtains a scattering diagram on $\Stab^\ast(X)$ encoding the cohomological DT invariants.
    \end{enumerate}

\end{theorem}

\subsection*{Application: cohomological Hall algebra for smooth projective CY3}

Until now, the isomorphism of Theorem \ref{theohyploc} was used in Theorems \ref{theoBB} and \ref{theothetastrat} to deduce equalities in the Grothendieck group. Notice that, to obtain such results, it would be sufficient to establish an equality in the Grothendieck group:
\begin{align}
    p_!\eta^\ast[P_\mX]=\bL^{\Ind_\mX/2}[P_{\Grad(\mX)}]
\end{align}
instead of Theorem \ref{theohyploc} (a similar result, in the Grothendieck ring of monodromic motives, is the main result of \cite{Bu2024AMI}). This problem is situated one categorical degree lower; namely, it would be sufficient to check the equality locally on a critical chart, and one would not have to check compatibility with smooth restrictions nor stabilization by quadratic bundle stacks. \medskip

However, having an isomorphism instead of an equality in the Grothendieck group also has applications, the main one being the construction of cohomological Hall algebras predicted in \cite{KonSol10}, \cite{ks}. Consider the Abelian heart $\mA_c$ of a $t$-structure on $D^b\Coh(X)$ that is Noetherian and satisfies generic flatness (see Definition \ref{defgenflat}), and a Serre subcategory $\mS$ of $\mA_c$ (\ie, a full subcategory of $\mA_c$ such that an extension of two objects of $\mA_c$ is in $\mS$ iff both objects are in $\mS$). Consider $\mM_\mS\subset\mM_\mA\subset\mM$ the substacks of objects of $\mS$ and $\mA_c$, and assume that the $\mM_\mS\subset\mM_\mA$ is a locally closed immersion (we say that $\mS$ is locally closed). The $\Theta$-correspondence has an open and closed subset given by

\[\begin{tikzcd}
\mM_{\mA,\gamma_1}\times\mM_{\mA,\gamma_1}\arrow[rr,bend right,"\oplus_{\gamma_1,\gamma_2}"]& \Filt_{\mA,\gamma_1,\gamma_2}(\mA)\arrow[l,swap,"p_{\gamma_1,\gamma_2}"]\arrow[r,"\eta_{\gamma_1,\gamma_2}"] & \mM_{\mA,\gamma_1+\gamma_2}
\end{tikzcd}\]
where $\Filt_{\mA,\gamma_1,\gamma_2}$ is the stack of short exact sequences $0\to E\to F\to G\to 0$ of objects of $\mA_c$ such that $[E]=\gamma_1$, $[G]=\gamma_2$, and then $[F]=\gamma_1+\gamma_2$. As $\mS$ is a Serre subcategory, one obtains a similar diagram for $\mM_\mS$, Cartesian over this one. If the $\mM_{\mS,\gamma}$ are bounded, adapting slightly \cite[Theorem 7.23]{Alper2018ExistenceOM}, we show in Proposition \ref{propModuli} that they admit a good moduli space $JH_\gamma:\mM_{\mS,\gamma}\to M_{\mS,\gamma}$, and that there are, from the universal property of good moduli space, natural maps $\oplus_{\gamma_1,\gamma_2}:M_{\mS,\gamma_1}\times M_{\mS,\gamma_2}\to M_{\mS,\gamma_1+\gamma_2}$ giving to $M_\mS$ a monoidal structure.\medskip

We obtain then (notice that we consider the cohomology with compact support here; taking the dual, i.e., Borel-Moore homology, would give an associative product):

\begin{theorem}(Theorem \ref{theoCoHA})
    Consider a smooth and projective Calabi-Yau threefold $X$, with a strong orientation data on $D^b\Coh(X)$, and a $t$-structure with Noetherian Abelian heart $\mA_c$ on $D^b\Coh(X)$ satisfying generic flatness (see Definition \ref{defgenflat}). In particular, $\mM_\mA$ is an oriented $d$-critical Artin $1$-stack, locally of finite presentation, with affine diagonal. Consider a locally closed Serre subcategory $\mathcal{S}$ of $\mA_c$, and suppose that, for any $\gamma_1,\gamma_2\in K^{num}(X)$, the map $\eta_{\gamma_1,\gamma_2}:\Filt_{\mS,\gamma_1,\gamma_2}\to \mM_{\mS,\gamma_1+\gamma_2}$ is of finite type.

    \begin{enumerate}
        \item[$i)$] 
    Then there is a natural coassociative coproduct (the absolute CoHA):
    \begin{align}
\bH_c(\mM_{\mS,\gamma_1+\gamma_2},P_\mA|_{\mM_{\mS,\gamma_1+\gamma_2}}) \to \bH_c(\mM_{\mS,\gamma_1},P_\mA|_{\mM_{\mS,\gamma_1}})\otimes_k\bH_c(\mM_{\mS,\gamma_2},P_\mA|_{\mM_{\mS,\gamma_2}})\{-\langle\gamma_1,\gamma_2\rangle/2\}
    \end{align}
     in $\mD_{mon}(\Spec(k))$, the triangulated category of monodromic Nori motives (resp. the Ind-category of the triangulated category of monodromic mixed Hodge structures), defined from the extension correspondence.
    \item[$ii)$] If, moreover, the $\mM_{\mS,\gamma}$ are of finite type, there is then a natural coassociative coproduct (the relative CoHA): 
    \begin{align}
(JH_{\gamma_1+\gamma_2})_!(P_\mA|_{\mM_{\mS,\gamma_1+\gamma_2}}) \to (\oplus_{\gamma_1,\gamma_2})_!((JH_{\gamma_1})_!(P_\mA|_{\mM_{\mS,\gamma_1}})\boxtimes (JH_{\gamma_2})_!(P_\mA|_{\mM_{\mS,\gamma_2}}))\{-\langle\gamma_1,\gamma_2\rangle/2\}
    \end{align}
    in $\mD^-_{mon,-}(\Spec(k))$, the triangulated category of bounded above complexes of constructible monodromic Nori motives (resp. of monodromic mixed Hodge structures), with bounded above weights, defined from the extension correspondence, whose hypercohomology with compact support is the absolute CoHA.
    \end{enumerate}
\end{theorem}

Consider a smooth projective CY3 $X$ such that $D^b\Coh(X)$ admits a strong orientation data compatible with direct sum (the canonical orientation data of \cite{JOYCEorient} is not known to upgrade to a strong orientation data, but it is plausible that it is the case). Following Remark \ref{remAssumCoHA}, the assumptions of the theorem are verified in the following situations:

\begin{itemize}
    \item When $\mA_c=Coh(X)$ is the heart of the classical $t$-structure on $D^b\Coh(X)$, one can then take various Serre subcategories, giving locally closed substacks: torsion or torsion-free sheaves, semistable sheaves for a Gieseker stability condition, sheaves with support on a locally closed subvariety, sheaves with support of dimension $\leq i$,... There will be an absolute CoHA in any of those situations. In general the $\mM_{\mA,\gamma}$ will not be bounded, so there will be no relative CoHA for the whole $Coh(X)$, but there is a relative CoHA for Gieseker-semistable sheaves.

    \item When $\mA_c=\mPP((0,1])$ is the heart of a $t$-structure coming from an algebraic stability condition in a good connected component $\Stab^\ast(X)$ (e.g., a component containing a stability condition from \cite{Bayer2011BridgelandSC}), one obtains an absolute CoHA for $\mM_\mA$, and a restriction of it for any sub-interval $I\subset (0,1]$ . When $I$ is of length $<1$, there is moreover a relative CoHA.
    \end{itemize}

We sketch here the construction of the absolute CoHA in the case $\mS=\mA_c$. The fact that the orientation data is compatible with direct sums gives exactly that the orientation on $\mM_{\mA,\gamma_1}\times \mM_{\mA,\gamma_2}$ and the orientation obtained by the product are isomorphic. Using again \eqref{eqTanDG}, the isomorphism of Theorem \ref{theohyploc} restricts to an isomorphism:
\begin{align}
    (p_{\gamma_1,\gamma_2})_!(\eta_{\gamma_1,\gamma_2})^\ast P_{\mM_{\mA,\gamma_1+\gamma_2}}\simeq P_{\mM_{\mA,\gamma_1}\times \mM_{\mA,\gamma_2}}\{-\langle \gamma_1,\gamma_2\rangle/2\}\simeq P_{\mM_{\mA,\gamma_1}}\boxtimes P_{\mM_{\mA,\gamma_2}}\{-\langle \gamma_1,\gamma_2\rangle/2\}
\end{align}
which satisfies a kind of associativity property with respect to triple filtrations $\Filt_{\mA,\gamma_1,\gamma_2,\gamma_3}$ given by the diagram \ref{diagassoc1} (here one uses the fact that the orientation is a strong orientation). This is an isomorphism in an Abelian category, and this is the only point where we have to use the gluing technology of \cite{Joycesymstab}. This is the trick to obtain the CoHA, which is a (co)algebra in a triangulated category, without having to do homotopy coherent gluing. Now, as $\eta_{\gamma_1,\gamma_2}$ is assumed to be of finite type, it is proper by \cite[Lemma 7.17]{Alper2018ExistenceOM}. We consider then the sequence of morphisms:
 \begin{align}
\bH_c(\mM_{\mA,\gamma_1+\gamma_2},P_{\mA,\gamma_1+\gamma_2})&\to \bH_c(\mM_{\mA,\gamma_1+\gamma_2},(\eta_{\gamma_1,\gamma_2})_!(\eta_{\gamma_1,\gamma_2})^\ast P_{\mA,\gamma_1+\gamma_2})\nn\\
&\simeq \bH_c(\mM_{\mA,\gamma_1}\times \mM_{\mA,\gamma_2},(p_{\gamma_1,\gamma_2})_!(\eta_{\gamma_1,\gamma_2})^\ast P_{\mA,\gamma_1+\gamma_2})\nn\\
&\simeq \bH_c(\mM_{\mA,\gamma_1}\times \mM_{\mA,\gamma_2},P_{\mM_{\mA,\gamma_1}}\boxtimes P_{\mM_{\mA,\gamma_2}}\{-\langle \gamma_1,\gamma_2\rangle\})\nn\\
&\simeq \bH_c(\mM_{\mA,\gamma_1},P_{\mM_{\mA,\gamma_2}})\otimes_k\bH_c(\mM_{\mA,\gamma_2},P_{\mM_{\mA,\gamma_2}})\{-\langle\gamma_1,\gamma_2\rangle/2\}
\end{align}
and one deduces the coassociativity of the coproduct from the commutativity of diagram \ref{diagassoc1}.

\subsection*{Sketch of the proof of the main theorem}

We begin by establishing a few results about the hyperbolic localization functor $(p_\mX)_!(\eta_\mX)^\ast$ for a stack $\mX$. The first is Theorem \ref{theobradenstacks}, which is a stacky version of Braden's theorem of \cite{Braden2002HyperbolicLO}, or more precisely of the adjunction of \cite{Drinfeld2013ONAT}, obtained by adapting Drinfeld and Gaitsgory's construction. In Theorem \ref{commutspechyploc}, we deduce from this adjunction that the hyperbolic localization functor commutes with any specialization system; in particular, it commutes with base change and vanishing cycles (this is a stacky version of the main results of \cite{Ric16}). In proposition \ref{propBB}, we establish a relative and stacky version of Bialynicki-Birula's main result from \cite{BiaynickiBirula1973SomeTO}; namely that, for $\phi:\mX\to\mY$ smooth, $\Filt(\mX)\to\Filt(\mY)\times_{\Grad(\mY)}\Grad(\mX)$ is an affine bundle stack modelled on $\bT_\phi|_{\Grad(\mX)}^{>0}$. We deduce in \ref{propsmoothBB} that hyperbolic localization commutes with smooth pullback up to a Thom twist. We also prove in Lemma \ref{lemhyploclosed} that hyperbolic localization commutes with restriction to the support.\medskip

Given a smooth stack $\mU$ with a function $f:\mU\to\bA^1_k$, considering the closed immersion of the critical locus $i:\mR\to\mU$, we call $(\mR,\mU,f,i)$ a critical chart, extending slightly the terminology of \cite{Joyce2013ACM} from scheme to stacks. We consider then, as in \cite{Joycesymstab}, the perverse sheaf, monodromic mixed Hodge module, or monodromic perverse Nori motive on $\mR$:
\begin{align}
    P_{\mU,f}:=i^\ast\phi_f^{mon,tot}\bQ_\mU\{\dim(\mU)/2\}
\end{align}
where $\phi_f^{mon,tot}$ denotes the total monodromic vanishing cycle functor, obtained by summing on the critical values and keeping the monodromy. The different isomorphisms that we have built are then combined to give an isomorphism:
\begin{align}\label{isomloc}
    (p_\mR)_!(\eta_\mR)^\ast P_{\mU,f} \simeq P_{\Grad(\mU),\Grad(f)}\{-\Ind_\mR/2\}
\end{align}
These isomorphisms form the local version of the isomorphism of Theorem \ref{theohyploc}.\medskip

In section \ref{sectjoyce}, we reformulate slightly the formalism of \cite{Joyce2013ACM}, \cite{Joycesymstab}, and \cite{darbstack} to allow us to glue the DT perverse sheaf using stacky critical charts as above. This is necessary, as the hyperbolic localization correspondence is trivial for schemes. For a $d$-critical stack $(\mX,s)$, we consider critical charts $(\mR,\mU,f,i)$ with smooth maps $\mR\to\mX$ and two operations on them. The first is called smooth restriction: given $\phi:\mU'\to\mU$, we consider the critical chart $(\mR'=\mR\times_\mU\mU',\mU',f':=f\circ\phi,i')$ with the smooth map $\mR'\to\mR\to\mX$. The second, corresponding to the fact of adding quadratic terms to the function, is called stabilization by quadratic bundle: given a quadratic bundle $(\mE,q)$ on $\mU$, with total space $\pi:\bV_\mU(\mE)\to \mX$ and zero section $s:\mU\to \bV_\mU(\mE)$, we consider the critical chart $(\mR,\bV_\mU(\mE),f\circ\pi+q,s\circ i)$ with the smooth map $\mR\to\mX$. In Proposition \ref{proplocdefperv}, we reformulate \cite[Theorem 4.8]{darbstack} by saying that the DT sheaf $P_{\mX,s,K_{\mX,s}^{1/2}}$ on an oriented d-critical stack is obtained by gluing a $\bZ/2\bZ$-twisted version of $P_{\mU,f}$ for critical charts as above, where the comparison isomorphisms are built from smooth comparison of critical charts and stabilization by quadratic bundle stacks. To show that the isomorphisms \eqref{isomloc} glue, we must show:

\begin{itemize}
    \item That the isomorphisms \eqref{isomloc} are compatible with the isomorphisms comparing $P_{\mU,f}$ for smooth restriction of critical charts and stabilization by quadratic bundle, which is the content of Lemmas \ref{lemsmoothresthyploc} and \ref{lemcompathyplocstab}.
    \item That $\Grad(\mX)$ is covered by critical charts of the form $(\Grad(\mR),\Grad(\mU),\Grad(f),\Grad(i))$ which can be compared using $\Grad$ of smooth restrictions and stabilization by quadratic bundle, which is the content of Proposition \ref{propcomparchart}. This is done by building in Proposition \eqref{propcompareqchart} (which is a direct adaptation of the arguments of \cite{Joyce2013ACM}) $\bG_m$-equivariant critical charts and comparing them in a $\bG_m$-equivariant way.
\end{itemize}

\subsection*{Content}

In Section \ref{sect6fun}, we present six-functor formalisms (or, more precisely, motivic coefficient systems) and their extensions to Artin stacks. We prove some compatibility results on the operations (in particular concerning the purity isomorphism and specialization systems) that can be difficult to keep track of from the literature, in particular in the context of stacks, adapting constructions of Ayoub \cite{Ayoub2018LesSOI}, \cite{Ayoub2018LesSOII}.\medskip

In Section \ref{secthyploc}, we study the hyperbolic localization functor for stacks. After a brief recalling of the formalism of graded and filtered stacks from \cite{HalpernLeistner2014OnTS}, we adapt Drinfeld and Gaitsgory's proof of Braden's adjunction from \cite{Drinfeld2013ONAT}. We finally build the isomorphisms of commutation of hyperbolic localization with various functors as said above and check useful compatibilities between them. This section is written for general coefficient systems, with a view toward a possible generalization of DT theory to motivic stable homotopy. \medskip

In section \ref{sectmono}, we specialize to the coefficient systems that we will use in the remaining part of this work, namely perverse sheaves, monodromic mixed Hodge modules, and monodromic perverse Nori motives. The main reason for this restriction is that those have a perverse $t$-structure, which allows gluing objects in a $1$-category as in \cite{Joycesymstab} (look at \cite{Hennion2024GluingIO} to see a homotopy coherent version of this formalism, which could be used to avoid this restriction). We use in particular the extension of perverse sheaves (resp. mixed Hodge modules and perverse Nori motives) to stacks given by \cite{Liu2012EnhancedSO} (resp. \cite{tubachMHM}). We develop the formalism of monodromic objects, monodromic vanishing cycles, and the Thom-Sebastiani isomorphism, adapted from the classical story for perverse sheaves (resp. mixed Hodge modules) on schemes from \cite{Verdiespecial}, \cite{Sebastiani1971UnRS}, \cite{MasThomSeb} (resp. \cite{Saitothomseb}) (a version for general coefficient systems is presented in \cite{descombesThomSeb}). We introduce the square root of Tate twists that exists at the monodromic level, generalizing this well-known story to the level of perverse Nori motives. Finally, we introduce the Grothendieck ring of monodromic mixed Hodge modules and monodromic perverse Nori motives on stacks, with the subtlety given by the necessity to complete this ring with respect to Tate twists. \medskip

In Section \ref{sectjoyce}, we recall the formalism of d-critical stacks of \cite{Joyce2013ACM} and the gluing of the DT sheaf from \cite{Joycesymstab} and \cite{darbstack}. As said above, we slightly reformulate those results by allowing us to work with stacky critical charts. In particular, using the results of the last section, we check that this construction upgrades naturally to the level of monodromic perverse Nori motives. \medskip

In Section \ref{sectproof}, we prove our main result, Theorem \ref{theohyploc}. We begin by defining the isomorphism of commutation with hyperbolic localization on a critical chart and check its compatibility with smooth restriction and stabilization by quadratic bundle. We show then that there are enough critical charts, as said above. We then define a natural structure of oriented d-critical stack (resp. $-1$-shifted symplectic stack) on the $\Grad$ of an oriented d-critical stack (resp. $-1$-shifted symplectic stack) and prove our main result and useful compatibilities. \medskip

In Section \ref{sectapp}, we discuss several applications. We present Theorem \ref{theoBB}, our virtual version of Bialynicki-Birula decomposition, which was our initial motivation. We then present in Theorem \ref{theothetastrat} the analogue of this formula for $\Theta$-stratifications. We present the construction of stacks of objects in Abelian hearts in the bounded derived category of coherent sheaves on a smooth and projective Calabi-Yau threefold. We then applied Theorem \ref{theothetastrat} to prove the Kontsevich-Soibelman wall crossing formula in Theorem \ref{theoKSwcf}. Finally, we applied the main theorem \ref{theohyploc} to build the Cohomological Hall algebra.

\subsection*{Relations to other works}\label{relwork}

The idea of using hyperbolic localization to obtain a localization formula in cohomological DT theory was first formulated by Balazs Szendroi in \cite[Section 8.4]{Szendroi:2015xna}. It was applied in \cite[Section 6]{Nakajima2016LecturesOP} and in \cite[Section 8.3]{RSYZ19}, where it was used in a specific example of framed representations of quivers with potential. In \cite{Ric16}, Timo Richarz proved the commutativity of hyperbolic localization with the vanishing cycles functor in a more general way, using Braden's adjunction. We used this result to establish the formula \eqref{nonproj} for any critical locus of a potential in \cite{Descombes2021MotivicDI}, and the extension of this result to $-1$-shifted symplectic spaces was suggested to us by Richard Thomas. Since then, upgrades of Richarz's result to the level of motivic stable homotopy have been obtained in \cite{Ivorra2024APO} and \cite{Pham2024TheII}.

In \cite{Bussi2019OnMV}, the authors defined, for oriented $d$-critical schemes, a Donaldson-Thomas motive that glues the motive of vanishing cycles defined by Denef-Loeser \cite{Denef1998MotivicEI} and Looijenga \cite{Looijenga1999-2000}. This construction has been extended to oriented d-critical stacks $(\mX,s,K_{\mX,s}^{1/2})$ in \cite[Theorem 5.14]{darbstack}. Such motives have a realization in the Grothendieck ring of monodromic perverse Nori motives or monodromic mixed Hodge modules, which coincide by \cite{Ivorra2021QuasiunipotentMA} with the class in the Grothendieck group of $P_\mX$. Notice that, on one hand, motives glue only in the Nisnevich topology and not in the étale topology; then the local definition of the motivic DT invariants is slightly more involved than the local definition of $P_X$. But, on the other hand, motives form a set and not a category; hence, the constructions in motivic DT theory are one categorical degree lower. In particular, to prove a formula for motivic DT invariants, it suffices to prove an equality on a critical chart, and one does not have to check compatibility with a comparison of critical charts. Davesh Maulik has proved a formula similar to \eqref{circomploc} for motivic DT invariants on d-critical normal schemes with a good circle-compact $\bG_m$-action, as explained in \cite[section 5.3]{darbstack}, in an unpublished preprint (private communication). A generalization of this result for non-Archimedean geometry was subsequently proved in \cite[Theo 7.17]{jiang2017moduli}. For non-circle compact actions, one can compute the DT motive of the attracting variety, which is by definition circle compact, in this way. In \cite{Bu2024AMI}, the author has proven a localization formula similar to Theorem \ref{theohyploc} for motivic DT invariants on stacks, from which one can derive the Kontsevich-Soibelman wall crossing formula from \cite{ks}.\medskip

A toric localization similar to \eqref{circomploc} also exists for K-theoretic DT invariants as defined in \cite{NekOk}. The K-theoretic DT invariants are a refinement of the numerical DT invariants defined for projective moduli spaces with a $\bG_m$-action and a $\bG_m$-equivariant symmetric obstruction theory, developed in parallel with the motivic and cohomological refinement of Kontsevich-Soibelman and Joyce and collaborators. In general, they are expected to correspond to the $\chi_y$ genus of the Hodge polynomial of cohomological DT theory, so in particular one replaces $\mathbb{L}^{1/2}$ by $-y$ in K-theoretic formulae. If the moduli space $X$ is not projective but has a $\bG_m$-action with projective fixed components such that it is $\bG_m$-invariant, it was suggested in \cite{NekOk} to use the equation \eqref{circomploc} to define the $K$-theoretic invariants of $X$. However, this definition depends on the choice of this $\bG_m$-action (this choice is called the choice of the slope). The equation \eqref{nonproj} in the non-projective case explains the origin of this ambiguity: one computes by toric localization only the virtual cohomology of the attracting variety, which is not the whole moduli space and depends on the chosen $\bG_m$-action.\medskip

This dependence on the slope was studied explicitly in \cite{Arb} for the moduli space of framed representations of a toric quiver, and this was related to the ambiguity in the refined topological vertex of \cite{Iqbal:2007ii}. In this case there is a two-dimensional torus invariant acting on the moduli space of framed representations, scaling the arrows of the quiver by leaving the potential invariant, so the space of slopes is $\mathbb{P}^1_\mathbb{R}$. The fixed points can be described as molten crystals from \cite{Mozgovoy:2008fd}. In \cite[Prop 3.3]{Arb}, it was established that there is a wall and chamber structure on the space of slopes. Namely, the generating functions of framed invariants are constant in a chamber and jump at a wall, the walls corresponding to slopes where the weight of an elementary cycle of the quiver becomes attracting or repelling. This is rather strange at first sight, because inside a given chamber the cohomological weight of a given molten crystal changes at many walls, but the final result does not change: these walls are 'invisible.' In \cite{Descombes2021MotivicDI}, we have established that the attracting variety is the subspace of representations where the cycles with repelling weights are nilpotent, so the attracting variety changes exactly on the walls defined in \cite{Arb}, \ie \eqref{nonproj} give an explanation of this wall and chamber structure. Moreover, using a nilpotent/invertible decomposition for the unframed representation and a wall-crossing relation between framed and unframed invariants, we have obtained in \cite{Descombes2021MotivicDI} the full framed generating series by multiplying the one obtained by localization by a generating series of framed invariants where some cycles are imposed to be invertible. The latter is easy to compute and has a universal closed formula for all toric quivers. Note that in this case the moduli space is the critical locus of the potential of the quiver, so one does not need all the subtleties of gluing.\medskip

In \cite{Leigh2024InstantonsAM}, Leigh use Theorem \ref{theoBB} in an example where there is no $\bG_m$-action scaling the symmetric obstruction theory; hence, there is no K-theoretic refinement, but the cohomological DT invariants are still defined, and there is a $\bG_m$-action leaving the symmetric obstruction theory invariant in order to apply Theorem \ref{theoBB}.\medskip

In the preprint \cite{Descombes:2022cpc}, which is subsumed by this paper, we built the isomorphism \eqref{isomspace} of Theorem \ref{theoBB} for algebraic space with $\bG_m$-action. A similar isomorphism for stacks with $\bG_m$-action, obtained using smooth descent, was also presented, but this construction was not satisfactory because the formalism of mixed Hodge modules on stacks was not developed at the time, and the Grothendieck group of perverse sheaves on stacks is famously ill-behaved, as pointed out to us by Dominic Joyce. This is now solved by Tubach's work \cite{tubachMHM}. Thanks to this work, for this new version, we wanted to extend the construction of this isomorphism to the isomorphism \eqref{isomstack} of Theorem \ref{theohyploc} for the $\Theta$-correspondence for stacks, with applications to the proof of the Kontsevich-Soibelman wall crossing formula and construction of the cohomological Hall algebra in mind. In the meantime, we heard about the work \cite{Kinjo2024CohomologicalHA}, where the authors were also extending our original construction in order to build the cohomological Hall algebra. We have then decided to place the emphasis here on the construction of the isomorphism \eqref{isomstack} of Theorem \ref{theohyploc} and its enhancement to mixed Hodge modules and perverse Nori motives (with a view toward a possible extension to motivic stable homotopy in the future), presenting quickly the proof of the Kontsevich-Soibelman wall crossing formula and the construction of the CoHA as an application. We apologize for any duplication of results in the literature and refer the reader to the excellent paper \cite{Kinjo2024CohomologicalHA} for an alternative construction of the CoHA and to \cite{Bu2025CohomologyOS} for further developments on the structure of the CoHA in the spirit of \cite{DavMein}, in particular the proof of integrality and the construction of the BPS Lie algebra.

\subsection*{Acknowledgements}
I thank Chenjing Bu, Ben Davison, Benjamin Hennion, Masoud Kamgarpour, Šarūnas Kaubrys, Alyosha Latyntsev, Oliver Leigh, Andrés Ibáñez Núñez, Davesh Maulik, Tudor Padurariu, Khoa Bang Pham, Boris Pioline, Marco Robalo, Timo Richarz, Matthieu Romagny, David Rydh, Travis Schedler, Olivier Schiffmann, Sebastian Schlegel Mejia, Balazs Szendroï and Swann Tubach for interesting discussions and comments. I thank particularly Dominic Joyce and the anonymous referee for their remarks on a preliminary version of this work, and Richard Thomas, who motivated me for this work and for his constant support throughout the years.

\section{Six functor formalisms toolbox}\label{sect6fun}

\subsection{Construction of six functor formalisms}

\subsubsection{Six functor on schemes}

The six functor formalism is a very powerful tool, first envisioned by Grothendieck and build in the context of $\ell$-adic étale sheaves in \cite{Verdier1972ThorieDT}. It has further been developed for analytic sheaves and $\mathcal{D}$-modules, and for mixed Hodge modules by Saito in \cite{Saito1988ModulesDH}, \cite{SaitoMHM}. After the invention of motivic stable homotopy theory, Voevodsky has discovered an axiomatic framework to build six functor formalisms in motivic contexts. The construction was developed by Ayoub in his thesis \cite{Ayoub2018LesSOI}, \cite{Ayoub2018LesSOII}, with further details provided by Cisinki and Déglise in \cite{Cisinski2009TriangulatedCO}. Noetherianity hypotheses were removed by Hoyois and Khan. Furthermore, these constructions have been upgraded to the $\infty$-categorical context in Khan's thesis \cite{Khan2016MotivicHT} by applying the work of Gaistgory and Rozenblyum \cite{Gaitsgory_Rozenblyum_2017} Liu-Zengh \cite{Liu2012EnhancedSO}, and Mann \cite{Mann2022A6}. We follow here the presentation of \cite[Section 2]{Khan2021VOEVODSKYSCF}, see \cite{Gallauer2021AnIT} and \cite{Scholze2023SixFunctorF} for an other introduction to six functor formalisms. Notice that we will work mainly in a $2$-categorical setting in this paper, but the $\infty$-categorical enhancement is important to extends the six functor formalisms to stacks, as we will see later.\medskip

We begin with a full subcategory, denoted $Sch$, of the category of quasi-compact quasi-separated (qcqs) schemes over a qcqs base scheme $B$, which is stable by fiber products and coproducts, containing any closed subscheme of a scheme in $Sch$, any quasi-compact open subscheme of a scheme in $Sch$, and any proj of a finite locally free sheaf over a scheme in $Sch$. In general, one take the category of quasiprojective schemes over a base $B$, or the category of schemes which are separated of finite type over a base $B$, or the category of quasi-compact and quasi-separated schemes over a base $B$. Notice that Khan allows to work with derived algebraic spaces in \cite{Khan2021VOEVODSKYSCF}, but we work only with classical schemes. This is not a serious restriction, because, from the nil-invariance \cite[Lemma 2.13]{Khan2016MotivicHT}, the six functor is insensitive to derived thickening, as it is insensitive to nilpotent thikening, \ie it depends only on the reduced scheme of the classical truncation.\medskip

Denote by $Pr^L_{\otimes,St}$ is $(\infty,1)$-category of symmetric monoidal stable presentable $(\infty,1)$-categories (an enhancement of triangulated category with a symmetric tensor product). A motivic coefficient system is a functor of $(\infty,1)$ categories:
\begin{align}
    \mD^\ast:Sch^{op}\to Pr^L_{\otimes,St}
\end{align}
satisfying few properties, that can be checked at the level of the homotopy category. For $f:X\to Y$ in $Sch$, we denote the induced functor by $f^\ast:\mD(Y)\to\mD(X)$, and we denote by $-\otimes_X-$ the symmetric tensor product on $\mD(X)$. Notice that, from the presentability assumption, $f^\ast$ admits a right adjoint $f_\ast:\mD(X)\to\mD(Y)$, and $-\otimes_X-$ admits a right adjoint $\Hom(-,-)$. The notion of coefficient system gives simple conditions to ensure that one can built an adjoint couple of exceptional functors $(f_!,f^!)$.\medskip

One ask that, for each $f:X\to Y$ smooth, $f^\ast$ admits a left adjoint $f_\#:\mD(X)\to\mD(Y)$. Consider a Cartesian diagram:
\[\begin{tikzcd}
    X'\arrow[r,"g"]\arrow[d,"q"] & X\arrow[d,"p"]\\
    Y'\arrow[r,"f"] & Y
\end{tikzcd}\]
Using adjunction and the isomorphism of composition for $\mD^\ast$, if $f$ (and then also $g$) is smooth, there is a natural base change morphism:
\begin{align}
    g_\# q^\ast\to p^\ast f_\#
\end{align}
One ask for the smooth base change property, saying that that any such morphism is an isomorphism. Using the adjunction, there is also, for $f:X\to Y$ smooth, a natural morphism:
\begin{align}
    f_\#(f^\ast-\otimes-)\to -\otimes f_\#
\end{align}
One ask for the projection formula, saying that any such morphism is an isomorphism. Moreover, one ask that $\mD_\ast$ is additive, \ie the natural map $\mD(\bigsqcup_\alpha S_\alpha)\to\prod_\alpha \mD(S_\alpha)$ is an equivalence. Finally, one ask for the three Voevodsky's conditions:
\begin{itemize}
    \item Homotopy invariance: for any $X\in Sch$, denoting $p:\bA^1_X\to X$, the unit map $Id\to p_\ast p^\ast$. With the other axioms, it induces the same result for any vector bundle.
    \item Localization: consider an closed/open pair:
    \[\begin{tikzcd}
        Z\arrow[r,"i"] & X & U\arrow[l,swap,"j"]
    \end{tikzcd}\]
    in $Sch$, the functor $i_\ast$ is fully faithful with essential image spanned by objects in the kernel of $j_\ast$. As noticed in \cite[Remark 2.9]{Khan2021VOEVODSKYSCF}, it implies that there is a canonical exact triangle:
    \begin{align}
        j_\#j^\ast\to Id\to i_\ast i^\ast
    \end{align}
    \item Stability: For $X\in Sch$, denote by $s:X\to \bA^1_X$ the zero section. Then $p_\#s_\ast:\mD(X)\to\mD(X)$ is an equivalence called the Tate twist, denoted by $\{1\}=[2](1)$. With the other axioms, it induces the same result for any vector bundle. Given a locally free sheaf $\mE$, we denote by $\Sigma^\mE:\mD(X)\to\mD(X)$ the equivalence, called Thom twist, associated to $\Spec_X(\Sym(\mE))$, \ie the total space of $\mE^\vee$.
\end{itemize}

The construction of the exceptional functoriality follows from the strategy used by Deligne in \cite[Exposé XVII]{Verdier1972ThorieDT}, and rely on deep properties of proper morphisms in motivic coefficient systems. These properties were shown in \cite{Ayoub2018LesSOI} to holds for projective morphisms, and were extended to proper morphisms in \cite{Cisinski2009TriangulatedCO} using Chow lemma. As in \cite[Exposé XVII]{Verdier1972ThorieDT}, for any $f:X\to Y$ which is separated of finite type, one want to define the exceptional functor $f_!:\mD(X)\to\mD(Y)$ by using Nagata compactification. One write $f=p\circ j$, with $j$ an open immersion, and $f$ proper, and define $f_!:=p_\ast j_\#$. However, one has to do a construction which is independent of the Nagata compactification, and turn this into a functor. One of the most difficult result of \cite{Ayoub2018LesSOI} is that, given a Cartesian diagram as above when $f$ is smooth and $p$ is proper, the natural base change morphism:
\begin{align}
    f_\# q_\ast\to p_\ast g_\#
\end{align}
is an isomorphism. One can then use this to relate the definition of $f_!$ for different Nagata compactifications by an isomorphism, and to construct a functor $\mD_!$ (at the $2$-categorical level). Moreover, an other difficult result of \cite{Ayoub2018LesSOI} is that, for $p$ proper, $p_\ast$ has a right adjoint. Because, for $j$ open, $j^\ast$ is right adjoint to $j_\#$, it means that, for any $f$ separated of finite type, $f_!$ has an adjoint, denoted $f^!$. An other result of \cite{Ayoub2018LesSOI}, the proper base change, gives that, given a Cartesian diagram where $p$ is proper, the natural morphism:
\begin{align}
    f^\ast p_\ast\to q_\ast g^\ast
\end{align}
is an isomorphism. Together with the smooth base change, this gives a natural base change isomorphism for any $p$ which is separated of finite type:
\begin{align}
    Ex^\ast_!:f^\ast p_!\simeq q_! g^\ast
\end{align}
and similarly, by passing to the adjoint:
\begin{align}
    Ex^!_\ast:f_\ast p^!\simeq q^! g_\ast
\end{align}
An other result of \cite{Ayoub2018LesSOI} is the proper projection formula, saying that, for $p$ proper, the natural morphism:
\begin{align}
    p\ast\otimes-\to p_\ast(-\otimes p^\ast-)
\end{align}
is an isomorphism. Together with the smooth projection formula, this gives a natural projection formula for any $p$ which is separated of finite type:
\begin{align}
    p_!\otimes-\to p_!(-\otimes p^\ast-)
\end{align}
The properties of the tensor product are better expressed by using the exterior tensor product:
\begin{align}
    -\boxtimes_B-:=(p_1)^\ast\otimes (p_2)^\ast:\mD(X)\times\mD(Y)\to\mD(X\times_B Y)
\end{align}
Then the fact that $f^\ast:\mD(X)\to\mD(Y)$ is a monoidal functor gives a natural commutation isomorphism:
\begin{align}
    (f_1\times f_2)^\ast(-\boxtimes_B-)\simeq ((f_1)^\ast\boxtimes_B(f_2)^\ast) 
\end{align}
and a similar isomorphism for $f_\ast$ by passing to the adjoint. Similarly, the projection formula gives a natural commutation isomorphism:
\begin{align}
    (f_1\times f_2)_!(-\boxtimes_B-)\simeq ((f_1)_!\boxtimes_B(f_2)_!) 
\end{align}
and a similar isomorphism for $f_!$ by passing to the adjoint. Most of the time, the base scheme $B$ will be implicit, and we will use the notation $\boxtimes:=\boxtimes_B$.\medskip

This construction admits now an $(\infty,1)$-categorical enhancement, thanks to the work of \cite{Gaitsgory_Rozenblyum_2017} and \cite{Liu2012EnhancedSO}, giving all the higher coherence of the base change isomorphism. One consider the symmetric monoidal $(\infty,1)$-category of correspondences $Corr_{sep-ft,all}(Sch)$, whose objects are the objects of $Sch$, 1-morphism $X\to Y$ are roofs:
\[\begin{tikzcd}[sep=tiny]
    & W\arrow[dl,swap,"f"]\arrow[dr,"g"] &\\
    X &  & Y
\end{tikzcd}\]
where $g$ is separated of finite type. The higher morphisms are automorphisms of roofs, and a composition of two roofs is given by a big roof:
\[\begin{tikzcd}[sep=tiny]
    & & W\arrow[dl]\arrow[dr] & & \\
    & W_1\arrow[dl]\arrow[dr] & & W_2\arrow[dl]\arrow[dr]&\\
    X && Y && Z
\end{tikzcd}\]
where the upper square is Cartesian. The symmetric monooidal sstructure comes from the Cartesian product $\times_B$ in $Sch$. Then the output of the construction is a lax symmetric monoidal functor:
\begin{align}
    \mD^\ast_!:Corr_{sep-ft,all}(Sch)\to Pr^L_{St}
\end{align}
where $Pr^L_{St}$ is provided with the symmetric monoidal structure coming from Lurie's tensor product. Informally, it sends an object $X$ to $\mD(X)$, a morphism $(f,g)$ to $g_! f^\ast$, and the composition isomorphism gives both the composition isomorphism for $g_!$, for $f^\ast$, and the base change isomorphism $Ex^\ast_!$. The lax monoidality gives the exterior tensor product, and its compatibility with $f^\ast$ and $g_!$. All the coherences for $f_\ast$ and $g^!$ are obtained by passing to the adjoint.\medskip

Now, as we will review in Section \ref{sectionpropsmooth} for any separated morphism of finite type $f$, using the fact that its diagonal $\Delta_f$ is a closed immersion, it is easy to build a natural morphism $f_!\to f_\ast$. It is an isomorphism when $f$ is proper.\medskip

For $f$ smooth, as we will review in Section \ref{sectionpropsmooth}, by denoting by $\mathcal{L}_f$ the cotangent sheaf of $f$, there is similarly a natural purity isomorphism $\Sigma^{\mathcal{L}_f}f^\ast\simeq f^!$. Its construction, using the deformation to the normal cone, is slightly involved, and is the main technical ingredient of \cite{Ayoub2018LesSOI}.\medskip

\subsubsection{Six functors on algebraic stacks}\label{sectionschemestacks}

In this work, by stack we mean non-derived algebraic $1$-stacks.\medskip

Suppose that $\mD^\ast$ is valued in $1$-categories. For a morphism $f:X\to Y$, consider the first levels of the Čech diagram of $f$:
\[\begin{tikzcd}
    X\times_Y X\times_Y X\arrow[r,shift left=4,"p_{12}"]\arrow[r,"p_{13}"]\arrow[r,shift left=-4,"p_{23}"] & X\times_Y X\arrow[r,shift left=2,"p_1"]\arrow[r,shift left=-2,"p_2"] & X
\end{tikzcd}\]
Now, $\mD^\ast$ is says to satisfies descent along $f$ if objects and morphisms can be built uniquely by descent data. Namely, an object $f\in \mD(Y)$ is equivalent to the data $(F,\alpha)$ of an object $F\in \mD(X)$, and an isomorphism $\alpha:(p_1)^\ast F\simeq (p_2)^\ast F$ satisfying the cocycle condition for the $p_{ij}$. A morphism between the perverse sheaves corresponding to $(F,\alpha_F)$ and $(G,\alpha_G)$ is given by the data of a morphism $\gamma:F\to G$ such that the following square commutes:
\[\begin{tikzcd}
    (p_1)^\ast F\arrow[r,"\alpha_F"]\arrow[d,"(p_1)^\ast\gamma",swap] & (p_2)^\ast\arrow[d,"(p_2)^\ast\gamma"]\\
    (p_1)^\ast G\arrow[r,swap,"\alpha_G"] & (p_2)^\ast G
\end{tikzcd}\]
If $\mD^\ast$ satisfies descent with respect to smooth morphism, one can extend it to stacks by defining objects and morphisms as being descent data on a smooth presentation. In particular, one find that $G$-equivariant objects on $X$ are identified naturally with objects on the stack $[X/G]$.\medskip

In general, objects in Abelian categories satisfies descent with respect to interesting cohomology. As an example, perverse sheaves satisfies descent with respect to smooth morphisms (up to a dimension shift), hence one can define by descent perverse sheaves on stacks as in \cite{Laszlo2005TheSO}, \cite{Laszlo2006PerverseSO}. However, triangulated categories almost never satisfies descent, hence one cannot hope to extend a six functor formalism to stacks by using triangulated categories and descent. The solution is to work at the level of stable $(\infty,1)$-categories, a homotopy coherent enhancement of triangulated categories. There is in this setting a homotopy coherent version of descent. Namely, one consider the whole Chech diagram of $f$ as a diagram $X_\bullet:\Delta^{op}\to Sch$: $\mD^\ast$ is says to statisfies descent along $f$ if $\mD^\ast\circ X_\bullet^{op}:\Delta\to Pr^{L,\otimes}_{St}$ is a limit diagram, which means that any objects and (higher) morphism in $\mD(Y)$ is equivalent to the data of a coherent system of objects or (higher) morphism on $\mD(X_\bullet)$.\medskip

Suppose that our motivic coefficient system satisfies étale descent, then it automatically satisfies descent along smooth cover. Liu and Zengh have used this property in the context of torsion étale sheaves in \cite{Liu2012EnhancedSO} to extend the six operations for them from schemes to stacks. See Mann's thesis \cite[Appendix A.5]{Mann2022A6} for a reformulation of Liu-Zengh's construction in the language of the category of correspondences, explaining how it allows to extend any six functor formalism to bigger category of correspondences using descent along a class of morphisms. Namely, denote by $St'$ the $(2,1)$ full subcategories of the category of stacks, whose objects are the stacks admitting smooth cover by schemes in $Sch$. Denote by $Corr_{St',lft}$ the $(\infty,1)$-category of correspondences, where morphisms are roofs $(f,g)$ where $g$ is locally of finite type. one obtains a lax monoidal functor of symmetric monoidal $(\infty,1)$-categories:
\begin{align}
    \mD^\ast_!:Corr^\times_{lft,all}(St')\to Pr^{L,\otimes}_{St}
\end{align}
encoding the functor $f^\ast$ for any morphism, $g_!$ for $g$ locally of finite type, and the base change isomorphism. The right adjoint, which exists automatically, gives the functors $f_\ast$ and $g^!$. In \cite[Appendix A]{Khan2019VirtualFC}, Khan has shown that the six functor formalism obtained this way satisfies homotopy for arbitrary vector bundle stacks, and that one can build by smooth descent a morphism of purity even for non representable smooth morphisms.\medskip

General motivic coefficient system do not satisfies étale descent. However, from the localization property, they satisfies descent along Nisnevich cover (it is in particular the case for algebraic K theory). Recall that the Nisnevich topology is an intermediate topology between the étale and the Zariski one, where covering are étale covering which are surjective on $k$-valued points for any field $k$. It implies in particular that any coefficient system satisfies descent along cover by smooth morphisms with Nisnevich local sections.\medskip

Now, from \cite[Corollary 2.9]{Chowdhury2024NonrepresentableSF}, any quasi-separated stack admits a cover by smooth morphisms with Nisnevich-local sections from a quasi-compact and quasi-separated schemes.  Denote by $St$ the $(2,1)$ full subcategories of the category of quasi-separated stacks over $B$, whose objects are the stacks admitting smooth cover with Nisnevich local sections by schemes in $Sch$. Namely, if $Sch$ is the category of quasi-compact and quasi-separated schemes (resp separated of finite type, resp quasi-projective) over $B$, then $St$ is the $(2,1)$-category of quasiseparated stacks (resp quasi-separated and locally of finite type) over this base. Then, as done in \cite[Corollary 7.2.4]{Khan2025LisseEO} using Liu-Zengh and Mann' arguments from \cite{Liu2012EnhancedSO}, \cite{Mann2022A6} (see also \cite{Khan2021GeneralizedCT}, \cite{Khan2022EquivariantGC}), $\mD^\ast_!$ extends to a lax monoidal functor of symmetric monoidal $(\infty,1)$-categories:
\begin{align}
    \mD^\ast_!:Corr_{lft,all}(St)\to Pr^L_{St}
\end{align}
From \cite[Axiom Sm1, Sm3, Sm4]{Khan2025LisseEO}, one obtains from descent from the scheme case that, for $f$ smooth, $f^\ast$ admits a left adjoint $f_\#$ such that smooth base change is satisfied, \ie for any Cartesian square, the morphisms:
\begin{align}\label{eqexsmoothast}
    Ex_\#^\ast&:g_\# q^\ast\to p^\ast f_\#\nn\\
    Ex_{\#,!}&: f_\# q_!\to p_! g_\#\nn\\
    Ex^{\ast,!}&:g^\ast p^!\to q^! f^\ast
\end{align}
are isomorphisms.

\subsubsection{Specialization systems}

We consider specialization systems as introduced in \cite[Section 3.1]{Ayoub2018LesSOII} (with the definition extended from schemes to stacks). Given two six functor formalisms $\mD_1,\mD_2$, a (monoidal) specialization system $\sps:\mD_1\to \mD_2$ over $\eta\rightarrow B\leftarrow\sigma$ gives the data:
\begin{itemize}
    \item For each $B$-stack $f:\mX\to B$, a functor $\sps_f:\mD_1(\mX_\eta)\to \mD_2(\mX_\sigma)$.
    \item For each morphism $g:\mY\to\mX$ over $B$, exchange morphisms:
\begin{align}
    (g_\sigma)^\ast \sps_f&\to \sps_{f\circ g} (g_\eta)^\ast\label{specpullback}\\
    (g_\sigma)_! \sps_{f\circ g}&\to \sps_f (g_\eta)_!\label{specpushforward}
\end{align}
compatible with composition and base change. Moreover, if $g$ is smooth (resp $g$ proper and representable), \eqref{specpullback} (resp \eqref{specpushforward}) is assumed to be an isomorphism.
\item In the monoidal case, for $f_i:\mX_i\to B$ $i=1,2$, an exchange isomorphism:
\begin{align}
    \sps_{f_1}\boxtimes_\sigma\sps_{f_2}\simeq \sps_{f_1\times_B f_2}(-\boxtimes_\eta-)
\end{align}
compatible with the above exchange morphisms, and satisfying the obvious commutativity and associativity relations.
\end{itemize}

A morphism of specialization systems $\chi:\sps\to \sps'$ gives the data, for each $f:\mX\to B$, of a natural transformation $\chi_f:\sps_f\to \sps'_f$, commuting with the exchange morphisms, it is said to be an isomorphism if each $\chi_f$ is an isomorphism.\medskip

Passing to the right adjoint, one obtains exchange morphisms:
\begin{align}
    \sps_{f\circ g} (g_\eta)^!&\to(g_\sigma)^! \sps_f\nn\\
    \sps_f (g_\eta)_\ast&\to(g_\sigma)_\ast \sps_{f\circ g}
\end{align}
compatible with composition and $Ex^!_\ast$. If $g$ is smooth, passing to the left adjoint from the inverse isomorphism $\sps g^\ast\to g^\ast\sps$ of \eqref{specpullback}, one obtain:
\begin{align}
    (g_\sigma)_\# \sps_{f\circ g}f\leftarrow \sps_f (g_\eta)_\#
\end{align}
compatible with composition and $Ex^\ast_\#$. In the monoidal case, using the formula $-\otimes_\mX-=(p_1)^\ast-\boxtimes_B(p_2)^\ast$, we obtain a natural morphism:
\begin{align}
    \sps_f\otimes_{\mX_\sigma}\sps_f\mapsto \sps_f(-\otimes_\eta-)
\end{align}
which is compatible with the projection formula. We will show below that these exchange transformations satisfies all the desired compatibility with the constructions of the six functor formalism.\medskip

The notion of (monoidal) specialization system is particularly interesting because, if $j:\eta\to B$ (resp $i:\sigma\to B$) is any morphism, using base change and adjunctions, there is naturally monoidal specialization systems $j_!,j_\ast:\mD\to\mD$ over $\eta\to B\leftarrow B$ (resp $i^\ast,i^\ast:\mD\to\mD$ over $B\to B\leftarrow \sigma$. This implies (informally) that any functor that one can write from the six functor formalism will induce a monoidal specialization system by base change. In particular, (monodromic, unipotent) nearby and vanishing cycles have also a natural structure of specialization system. Then, by showing that the constructions of the six functor formalism commutes with specialization systems, we will obtain plenty of useful compatibility results, that are often difficult to find in the literature (in particular in the world of stacks).

\subsection{Useful results about six functors on stacks}
\subsubsection{Affine bundle stacks and homotopy invariance}

\begin{definition}
    Consider an algebraic stack $\mX$, and a quasi-coherent complex $\mE$ on $\mX$ which is perfect with amplitude in $[0,1]$ (our main example will be $\mE=\bL_\phi$, the cotangent complex of a smooth morphism $\phi:\mX\to\mY$). Notice that we use the cohomological grading convention, \ie work with cochain complexes $d:\mE^i\to \mE^{i-1}$, such that $H^{-1}(\bT_\phi)$ gives the stacky part and $H^{>0}(\bT_\mX)$ gives the obstructions, as usual. Hence, dually, $H^1(\bL_\phi)$ gives the stacky part and $H^{<0}(\bL_\phi)$ gives the obstructions, which vanishes when $\phi$ is smooth. We will follow the references \cite[Theorem 3.10]{intrnormcone} and \cite{Khan2019VirtualFC}, which use homological conventions, hence we will have to invert the signs from their results.
    \begin{itemize}
        \item We consider the smooth $\mX$-stack $\bV_\mX(\mE)\to\mX$, defined in \cite[Definition 3.1, Theorem 3.10]{intrnormcone} as a functor in groupoids by:
\begin{align}
    \bV_\mX(\mE):(\Spec(A)\overset{f}{\to}\mX)\mapsto map_{D(A)}(f^\ast(\mE),A)
\end{align}
where the latter is the mapping space in $\QCoh(\Spec(A))=D(A)$. As noticed above \cite[Lemma 3.4]{intrnormcone}, the Abelian group structure on the mapping space gives to $\bV_\mX(\mE)$ a structure of Abelian group object in the monoidal 2-category of stacks over $\mX$. Any Abelian $\mX$-stack of this form is called a vector bundle stack.
\item An affine bundle stack (modeled on a vector bundle stack $\bV_\mX(\mE)$) is the data of a torsor over the Abelian $\mX$-stack $\bV_\mX(\mE)$, \ie a surjective morphism $\mY\to\mX$, with an action $\mu:\bV_\mX(\mE)\times_\mX\mY\to\mY$ such that the morphism $\sigma:\bV_\mX(\mE)\times_\mX\mY\to\mY\times_\mX\mY$, defined by:
\[\begin{tikzcd}[column sep=huge]
\bV_\mX(\mE)\times_\mX\mY\arrow[r,"Id_{\bV_\mX(\mE)}\times_\mX\Delta_\pi"] & \bV_\mX(\mE)\times_\mX\mY\times_\mX\mY\arrow[r,"\mu\times_\mX Id_\mY"] & \mY\times_\mX\mY
\end{tikzcd}\]
is an isomorphism.
    \end{itemize}
\end{definition}

\begin{remark}
    According to our cohomological grading convention, $H^0(\mE)$ gives the classical part of $\bV_\mX(\mE)$, and $H^1(\mE)$ gives the stacky part. In particular, according to \cite[Example 3.2]{intrnormcone}, if $\mE$ is a vector bundle (\ie, is perfect in degree $0$), $\bV_\mX(\mE)\to\mX$ is the vector bundle $\Spec(\Sym(\mE))$ given by the total space of $\mE^\vee$.
\end{remark}

\begin{remark}
    Notice that, as proven in \cite[Lemma 3.6, Theorem 3.10]{intrnormcone} (changing the convention), if $\mE$ is of perfect amplitude $(-\infty,n]$, $\Map_{D(A)}(f^\ast(\mE),A)$ is a $n$-groupoid, and then $\bV_\mX(\mE)\simeq \bV_\mX(\tau^{\geq 0}(\mE))\to\mX$ is a $n$-stacks, but we consider only $1$-stacks here. Notice also that $\bV_\mX(\mE)=\bV_\mX(\tau^{\geq 0}(\mE))$ holds because wee work with non-derived stacks, but these would be false with derived stacks. See \cite[Section 1.2]{Khan2019VirtualFC} (also with the homological conventions) for a derived version.\medskip
\end{remark}

\begin{lemma}\label{lemmaffinesmooth}
    Consider a quasi-coherent complex $\mE$ on $\mX$ which is perfect with amplitude in $[0,1]$ and an affine bundle stack $\pi:\mY\to\mX$ modeled on $\bV_\mX(\mE)$. Then $\pi$ is smooth, of cotangent complex $\bL_\pi=\pi^\ast\mE$.
\end{lemma}

\begin{proof}
    Consider a ring $A$, $A$-module $M$ and the square-zero extension $A\oplus M$. Consider a commutative diagram:
    \[\begin{tikzcd}
         \Spec(A)\arrow[r,"g"]\arrow[d,"i"] & \mY\arrow[d,"\pi"]\\
        \Spec(A\oplus M)\arrow[r,"f"]\arrow[ur,dotted] & \mX
    \end{tikzcd}\]
    Then considering the morphism of groupoids:
    \begin{align}
        \Map_{D(A\oplus M)}(f^\ast\mE,A\oplus M)\to \Map_{D(A)}((i\circ f)^\ast\mE,A)
    \end{align}
    Using that $A\oplus M=i_\ast (A\oplus M)$ (where the first member of the equality is a $A\oplus M$ module, and the second is a $A$-module), this map is identified with:
    \begin{align}
        \Map_{D(A\oplus M)}(f^\ast\mE,i_\ast(A\oplus M))&\simeq \Map_{D(A)}((i\circ f)^\ast\mE,A\oplus M)\simeq \Map_{D(A)}((i\circ f)^\ast\mE,A)\times \Map_{D(A)}(g^\ast \pi^\ast\mE,M)\nn\\
        &\to \Map_{D(A)}((i\circ f)^\ast\mE,A)
    \end{align}
    By the definition of $\bV_\mX(\mE)$, the groupoid of arrows $g$ is naturally a torsor over the right hand side, and the groupoid of dotted arrow is naturally a torsor over the left hand side. Hence the groupoid of dotted arrow and 2-isomorphisms such that the diagram is commutative is naturally (in a way which is functorial in $A,M$) a torsor over $\Map_{D(A)}(g^\ast \pi^\ast\mE,M)$. By definition of the cotangent complex, we have then $\bL_\pi=\pi^\ast\mE$, and $\pi$ is smooth.
\end{proof}

The $\bA^1$-homotopy property claims that, for $\pi:\bA^1_\mX\to\mX$, the counit $Id\to\pi_\ast\pi^\ast$ (and similarly the unit $\pi_!\pi^!\to Id$) are isomorphisms. In the scheme case, using conservativity of Zariski pullbacks, this allows to show that it is still the case for any affine bundle. This property was shown for vector bundle stacks using étale descent in \cite[Proposition A.10]{Khan2019VirtualFC}, and we extends it to affine bundle stacks:

\begin{lemma}(Khan, Rydth)\label{lemhomotop}
    Consider an affine bundle stack $\pi:\mY\to\mX$, if the coefficient system satisfies étale descent, or if $\pi$ admits Nisnevich-local sections, then the counit $Id\to\pi_\ast\pi^\ast$ and unit $\pi_!\pi^!\to Id$ are isomorphisms.
\end{lemma}

\begin{remark}\label{remNLaffine}
    As the affine group is special from \cite{Grothendieck1958}, any affine bundle on a scheme is Zariski-locally trivial, in particular it admits Nisnevich-local sections. We think that it must be the case also for any affine bundle stack, but we don't know how to prove it.
\end{remark}

\begin{proof}
    Considering first the vector bundle stack $p:\bV_\mX(\mE)\to\mX$: $p$ is smooth from Lemma \ref{lemmaffinesmooth} (see also \cite[Theorem 3.10]{intrnormcone}). We have to check that the arguments of Khan \cite[Proposition A.10]{Khan2019VirtualFC} (inspired by David Rydh) proving that $Id\to p_\ast p^\ast$ is an isomorphism holds without assuming étale descent. Suppose that $\mE$ is a locally free sheaf, and consider a cover of $\mX$ by smooth morphisms with Nisnevich-local sections $f_i:X_i\to \mX$ from schemes. Using some Zariski refinement, one can suppose that $f_i^\ast(\mE)$ is free. By conservativity, it suffices to prove the claim for $\bV_{X_i}(f_i^\ast(\mE))$, but this follows from the homotopy axiom. Hence, the claim is true if $\mE$ is a locally free sheaf. The rest of the proof of \cite[Proposition A.10]{Khan2019VirtualFC} use only conservativity along the smooth surjective morphism: $\sigma:\mX\to\bV_\mX(\mE[-1])$, which admit a section, hence the arguments extends to our case. The same arguments shows that the unit $p_!p^!\to Id$ is an isomorphism. Notice that, in \cite{Khan2019VirtualFC}, Khan works with six functors extended to derived stacks, but coefficient systems depends only on the classical truncation as said above, hence the same arguments applies in our non-derived setting.\medskip

    Consider now an affine bundle $\pi:\mY\to\mX$ modeled on $\bV_\mX(\mE)$. Consider the following commutative diagram:
    \[\begin{tikzcd}
\bV_\mX(\mE)\arrow[d,"p"] & \bV_\mX(\mE)\times_\mX\mY \arrow[r,"p_2"]\arrow[d,"\mu"]\arrow[l,swap,"p_1"] & \mY\arrow[d,"\pi"]\\
       \mX & \mY\arrow[r,"\pi"]\arrow[l,swap,"\pi"] & \mX
    \end{tikzcd}\]
    The left square is Cartesian because is is obtained by applying the isomorphism:
    \[\begin{tikzcd}[column sep=huge]
\bV_\mX(\mE)\times_\mX\mY\arrow[r,"\Delta_p\times_\mX Id_\mY"] & \bV_\mX(\mE)\times_\mX\bV_\mX(\mE)\times_\mX\mY\arrow[r,"Id_{\bV_\mX(\mE)}\times\mu"] & \bV_\mX(\mE)\times_\mX\mY
\end{tikzcd}\]
    to the Cartesian diagram for $\bV_\mX(\mE)\times_\mX\mY$. The right square is Cartesian because is is obtained by applying $\sigma:\bV_\mX(\mE)\times_\mX\mY\to\mY\times_\mX\mY$ (which is assumed to be an isomorphism) to the Cartesian diagram for $\mY\times_\mX\mY$. From \cite[Lemma 3.5]{intrnormcone}, $\mu$ is canonically equivalent with $\bV_\mY(\pi^\ast\mE)\to\mY$ and then $Id\to\mu_\ast\mu^\ast$ is an isomorphism. By smooth base change, we have:
    \begin{align}
        \pi^\ast(Id\to\pi_\ast\pi^\ast)\simeq (Id\to\mu_\ast\mu^\ast)\pi^\ast
    \end{align}
    which is an isomorphism, and $\pi$ is smooth from Lemma \ref{lemmaffinesmooth}. Then, if the coefficient system satisfies étale descent, or if $\pi$ admits Nisnevich-local sections, $\pi^\ast$ is conservative, and then $Id\to\pi_\ast\pi^\ast$ is an isomorphism. The proof that $\pi_!\pi^!\to Id$ is an isomorphism is similar.\medskip
\end{proof}

\subsubsection{Thom twists and stability}

This is a version for stacks of well known results of Ayoub \cite[Section 1.5]{Ayoub2018LesSOI}:

\begin{lemma}\label{lemthomtwist}
    Consider a quasi-coherent complex $\mE$ on $\mX$ which is perfect with amplitude in $[0,1]$. Consider the smooth morphism $p:\bV_\mX(\mE)\to\mX$, and its zero section $s:\mX\to\bV_\mX(\mE)$. The adjoint pair of functors:
    \begin{align}\label{adjunthom}
    \Sigma^\mE:=p_\#s_!:\mD(\mX)\rightleftharpoons\mD(\mX):s^!p^\ast=:\Sigma^{-\mE}
    \end{align}
    forms inverse equivalence, called Thom equivalences. Given an exact triangle of quasi-coherent complex $\mE\to\mathcal{F}\to\mathcal{G}$ on $\mX$ which are perfect with amplitude in $[0,1]$, there are a natural isomorphism $\Sigma^{\mathcal{F}}\simeq \Sigma^\mE\circ \Sigma^{\mathcal{G}}$ and $\Sigma^{-\mathcal{F}}\simeq \Sigma^{-\mathcal{G}}\circ \Sigma^{-\mE}$.
\end{lemma}

\begin{proof}
    Given an exact triangle of quasi-coherent complex $\mE\to\mathcal{F}\to\mathcal{G}$ on $\mX$ which are perfect with amplitude in $[0,1]$, consider the following commutative diagram of stacks:
    \[\begin{tikzcd}
        \mX\arrow[d,"s_{\mathcal{G}}"]\arrow[dr,"s_{\mathcal{F}}"] & & \\
        \bV_\mX(\mathcal{G})\arrow[r,"p"]\arrow[d,"p_{\mathcal{G}}"] & \bV_\mX(\mathcal{F})\arrow[d,"q"]\arrow[dr,"p_{\mathcal{F}}"] &\\
        \mX\arrow[r,"s_\mE"] & \bV_\mX(\mE)\arrow[r,"p_\mE"] & \mX
    \end{tikzcd}\]
    the square is Cartesian because the corresponding square is a pushout in $\QCoh(\mX)$. Notice that the cotangent complex of $q$ is $(p_{\mathcal{F}})^\ast(\mathcal{G})$, hence $q$ is smooth. The exchange isomorphisms \eqref{eqexsmoothast} gives natural isomorphisms:
    \begin{align}
        \Sigma^{\mathcal{F}}&:=(p_{\mathcal{F}})_\#(s_{\mathcal{F}})_!\simeq (p_\mE)_\# q_\# p_! (s_{\mathcal{G}})_!\simeq (p_\mE)_\#(s_\mE)_! (p_{\mathcal{G}})_\#(s_{\mathcal{G}})_!=:\Sigma^\mE\circ \Sigma^{\mathcal{G}}\nn\\
        \Sigma^{-\mathcal{F}}&:=(s_{\mathcal{F}})^!(p_{\mathcal{F}})^\ast\simeq   (s_{\mathcal{G}})^!p^! q^\ast(p_\mE)^\ast\simeq  (s_{\mathcal{G}})^!(p_{\mathcal{G}})^\ast(s_\mE)^!(p_\mE)^\ast=:\Sigma^{-\mathcal{G}}\circ\Sigma^{-\mE}
    \end{align}
    where the second isomorphism of both lines is the isomorphism from \eqref{eqexsmoothast}. These isomorphisms are compatible with the unit and counit of the adjunction \eqref{adjunthom}, hence $\Sigma^{\pm\mathcal{F}}$ are inverse isomorphisms if an only if $\Sigma^{\pm\mE}$ and $\Sigma^{\pm\mathcal{G}}$ are. Hence it suffices to prove that for $\Sigma^{\pm\mE}$ and $\Sigma^{\pm\mE[-1]}$, for $\mE$ a locally free sheaf.\medskip
    
    Consider a locally free sheaf $\mE$. As in \cite[Lemma 4.1]{Chowdhury2024NonrepresentableSF}, using a cover of $\mX$ by smooth morphisms with Nisnevich-local sections $q_i:X_i\to \mX$ by qcqs schemes where $(q_i)^\ast(\mE)$ is trivial, we obtain from the stability axiom that $\Sigma^{\pm\mE}$ are inverse isomorphisms. Now, from the exact triangle $\mE[-1]\to 0\to \mE$, one get isomorphisms $Id\simeq\Sigma^0\simeq\Sigma^{\mE[-1]}\circ\Sigma^{\mE}$ and $Id\simeq\Sigma^0\simeq\Sigma^{-\mE}\circ \Sigma^{-\mE[-1]}$. These isomorphisms are compatible with the adjunctions, in the sense that:
    \begin{align}
        Id&\to\Sigma^{-\mE}\circ\Sigma^{\mE}\to\Sigma^{-\mE}\circ\Sigma^{-\mE[-1]}\circ\Sigma^{\mE[-1]}\circ\Sigma^{\mE}\simeq Id\nn\\
        Id&\simeq \Sigma^{\mE[-1]}\circ\Sigma^{\mE}\circ\Sigma^{-\mE}\circ\Sigma^{-\mE[-1]}\to\Sigma^{\mE[-1]}\circ\Sigma^{-\mE[-1]}\to Id
    \end{align}
    are the identity. But the morphisms $Id\to\Sigma^{-\mE}\circ\Sigma^{\mE}$ and $\Sigma^{\mE}\to\Sigma^{-\mE}\to Id$ are isomorphisms as proven above, hence the two morphisms above are isomorphisms, and $\Sigma^{\pm\mE}$ are isomorphisms, hence $Id\to\Sigma^{-\mE[-1]}\circ\Sigma^{\mE[-1]}$ and $\Sigma^{\mE[-1]}\circ\Sigma^{-\mE[-1]}\to Id$ are isomorphisms.
\end{proof}

This is an analogue of \cite[Proposition 3.1.7]{Ayoub2018LesSOII}:

\begin{lemma}\label{lemmaspecthom}
    Consider a specialization system $\sps:\mD_1\to\mD_2$ over $\eta\to B\leftarrow\sigma$. Given a stack $f:\mX\to B$ and a perfect complex of amplitude $[0,1]$, the natural morphisms:
\begin{align}
    \Sigma^{\mE_\sigma}\sps_f:=(p_\sigma)_\#(s_\sigma)_!\sps_f&\to\sps_f(p_\eta)_\#(s_\eta)_!=:\sps_f\Sigma^{\mE_\eta}\nn\\
    \sps_f\Sigma^{-\mE_\eta}:=\sps_f(s_\eta)^!(p_\eta)^\ast&\to(s_\sigma)^!(p_\sigma)^\ast\sps_f=:\Sigma^{-\mE_\sigma}\sps_f
\end{align}
are inverse isomorphisms, compatible with the isomorphisms coming from short exact sequences.
\end{lemma}

\begin{proof}
    From the compatibility between specialization and adjunctions, these morphisms are compatible with the isomorphisms $Id\simeq\Sigma^{-\mE}\circ\Sigma^{\mE}$ and $\Sigma^{\mE}\simeq\Sigma^{-\mE}\to Id$, hence it suffices to prove that the morphism for $\Sigma^{-\mE}$ is an isomorphism. Moreover, from the compatibility between specialization systems and composition and base change, these morphisms are compatible with the isomorphisms coming from short exact sequences. Using the same trick that in the proof of Lemma \ref{lemthomtwist}, it suffices to show that for $\mE$ a locally free sheaf. Consider a cover of $\mX$ by smooth morphisms with Nisnevich-local sections $q_i:X_i\to \mX$ by qcqs schemes where $(q_i)^\ast(\mE)$ is trivial. Using commutation between specialization systems and smooth pullbacks, with compositions of $\ast$-pullbacks, and with the exchange morphism $Ex^{!\ast}$ (obtained by adjointness from the compatibility with $Ex^\ast_!$), we obtain that:
    \begin{align}
        (q_{i,\sigma})^\ast(\sps_f\Sigma^{-\mE_\eta}\to\Sigma^{-\mE_\sigma}\sps_f)
    \end{align}
    is isomorphic with:
    \begin{align}
        (\sps_{f\circ q_i}\Sigma^{-((q_i)^\ast(\mE))_\eta}\to\Sigma^{-((q_i)^\ast(\mE))_\sigma}{f\circ q_i})(q_{i,\eta})^\ast
    \end{align}
    hence, by conservativity of $(q_i)^\ast$, it suffices to prove it for $\mE$ trivial on a qcqs scheme $X$. Using the compatibility with exact sequence, it suffices to prove it for $\mE=\mathcal{O}_X$, \ie for $p:\bA^1_X\to X$ and $s:X\to \bA^1_X$. This case is proven in \cite[Proposition 3.1.7]{Ayoub2018LesSOII}.
\end{proof}

In particular, Thom twists commutes with base change, hence for any morphism $g$, we have natural isomorphisms:
\begin{align}
    g_{!,\ast}\Sigma^{\pm g^\ast(\mE)}&\simeq \Sigma^{\pm\mE}g_{!,\ast}\nn\\
    \Sigma^{\pm g^\ast(\mE)}g^{!,\ast}&\simeq g^{!,\ast}\Sigma^{\pm\mE}
\end{align}
compatible with composition of $g$, and composition of Thom twists. This is contained in \cite[Scholie 1.4.2 2)]{Ayoub2018LesSOI} in the scheme case.

\subsubsection{Proper and smooth morphism}\label{sectionpropsmooth}

We recall here the construction of the natural morphism $f_!\to f_\ast$, for $f$ separated, and the purity isomorphism. The main point of giving explicitly this construction is to prove the compatibility between these morphisms and specialization systems. The main idea of the construction is the diagonal trick. For $f:\mX\to\mY$, consider the following diagram with Cartesian square:
\[\begin{tikzcd}
    \mX\arrow[dr,"\Delta_f"]\arrow[drr,bend left,"Id"]\arrow[ddr,bend right,swap,"Id"] & &\\
    & \mX\times_\mY \mX\arrow[r,"p_1"]\arrow[d,"p_2"] & \mX\arrow[d,"f"]\\
    & \mX\arrow[r,"f"] & \mY
\end{tikzcd}\]
where $\Delta_f$ is the diagonal of $f$. Suppose we have an isomorphism $(\Delta_f)_!\simeq (\Delta_f)_\ast$ (resp. $(\Delta_f)^!\simeq (\Delta_f)^\ast$). Then one obtains a natural morphism:
\begin{align}
    f_!&\simeq f_!(p_1)_\ast(\Delta_f)_\ast\overset{Ex_{!\ast}}{\to} f_\ast(p_2)_!(\Delta_f)_\ast\simeq f_\ast(p_2)_!(\Delta_f)_!\simeq f_\ast\nn\\
    f^!&\simeq (\Delta_f)^\ast(p_2)^\ast f^!\simeq(\Delta_f)^!(p_2)^\ast f^! \overset{Ex^{\ast!}}{\to} (\Delta_f)^!(p_1)^!f^\ast\simeq f^\ast
\end{align}

This procedure can be applied recursively:
\begin{itemize}
    \item If $f$ is a monomorphism, $\Delta_f=Id$, hence one obtains $f_!\to f_\ast$, which is an isomorphism if $f$ is a closed immersion (this can be checked at the scheme level).
    \item If $f$ is separated and representable by algebraic spaces, $\Delta_f$ is an closed immersion, hence one obtains a morphism $f_!\to f_\ast$, which is an isomorphism if $f$ is proper (and representable by algebraic spaces) (this can be checked at the level of algebraic spaces, where it follows from \cite[Theorem 2.34 ii)]{Khan2021VOEVODSKYSCF}).
    \item If $f$ is separated (possibly not representable), $\Delta_f$ is proper and representable by algebraic spaces, hence one obtains a morphism $f_!\to f_\ast$.
\end{itemize}

Similarly:
\begin{itemize}
    \item If $f$ is a monomorphism, $\Delta_f=Id$, hence one obtains $f^!\to f^\ast$, which is an isomorphism if $f$ is an open immersion (this can be checked at the scheme level).
    \item If $f$ is étale, $\Delta_f$ is an open immersion, hence one obtains a morphism $f^!\to f^\ast$, which is an isomorphism (this can be checked at the scheme level).
\end{itemize}

The construction of the purity isomorphism in the stack case was described in \cite[Section 4]{Chowdhury2024NonrepresentableSF}, following Ayoub \cite[Section 1.6]{Ayoub2018LesSOI} in the scheme case. Notice that the authors of \cite{Chowdhury2024NonrepresentableSF} works with motivic stable homotopy, but one checks directly that the arguments of \cite[Section 4]{Chowdhury2024NonrepresentableSF} works for any coefficient system extended to stacks using descent along smooth morphisms with Nisnevich-local sections as in \cite[Corollary 7.2.4]{Khan2025LisseEO}. As seen above, for $f$ smooth, one obtains:
\begin{align}
    (\Delta_f)^!(p_2)^\ast f^!\to f^\ast
\end{align}
From \cite[Theorem 4.24]{Chowdhury2024NonrepresentableSF} (by passing to the adjoint), this morphism is an isomorphism. From \cite[Theorem 1.5]{intrnormcone}, $\Delta_f$, which is representable by algebraic spaces, admits a deformation to the normal cone, which is a $1$-stack. Namely, there is a commutative diagram with Cartesian squares:
\[\begin{tikzcd}
    \mX\arrow[d,"\Delta_f"]\arrow[r,shift left=2,"1"] & \bA^1_\mX\arrow[d,"\tilde{s}"]\arrow[l,"\pi"]\arrow[r,swap,"\pi"] & \mX\arrow[l,shift right=2,swap,"0"]\arrow[d,"s"]\\
    \mX\times_\mY\mX\arrow[d,"p_2"]\arrow[r,"\tilde{1}"] & D_{\Delta_f}\arrow[d,"\tilde{p}"] & \bV_\mX(\bL_f)\arrow[d,"p"]\arrow[l,swap,"\tilde{0}"]\\
    \mX\arrow[r,"1"] & \bA^1_\mX\arrow[l,shift left=2,"\pi"]\arrow[r,shift right=2,swap,"\pi"] & \mX\arrow[l,swap,"0"]
\end{tikzcd}\]
where the vertical arrows from the second to the third line are smooth. Using base change, we obtain natural morphisms:
\begin{align}\label{roofpurity}
    (\Delta_f)^!(p_2)^\ast\leftarrow \pi_\ast\tilde{s}^!\tilde{p}^\ast\pi^\ast\rightarrow s^!p^\ast=:\Sigma^{-\bL_f}
\end{align}
More precisely, they are given by:
\begin{align}
    &\pi_\ast\tilde{s}^!\tilde{p}^\ast\pi^\ast\to \pi_\ast\tilde{s}^!\tilde{1}_\ast \tilde{1}^\ast\tilde{p}^\ast\pi^\ast\simeq \pi_\ast\tilde{s}^!\tilde{1}_\ast (p_2)^\ast\simeq \pi_\ast 1_\ast (\Delta_f)^!(p_2)^\ast\simeq (\Delta_f)^!(p_2)^\ast\nn\\
    &\pi_\ast\tilde{s}^!\tilde{p}^\ast\pi^\ast\to \pi_\ast\tilde{s}^!\tilde{0}_\ast \tilde{0}^\ast\tilde{p}^\ast\pi^\ast\simeq \pi_\ast\tilde{s}^!\tilde{0}_\ast p^\ast\simeq \pi_\ast 0_\ast s^!p^\ast\simeq s^!p^\ast
\end{align}
From \cite[Corollary 4.25]{Chowdhury2024NonrepresentableSF} (by passing to the adjoint), those two morphisms are isomorphisms, which gives an isomorphism, called the isomorphism of relative purity:
\begin{align}
    \Sigma^{-\bL_f}\circ f^!\simeq f^\ast
\end{align}

We now prove compatibility between the morphisms introduced above and specialization systems.

\begin{lemma}\label{lemmaspecproper}
    Given any specialization system $\sps:\mD_1\to\mD_2$ over $\eta\to B\leftarrow\sigma$, and a separated map $g:\mX\to\mY$, the following square is commutative:
    \[\begin{tikzcd}
         (g_\sigma)_!\sps_f\arrow[r]\arrow[d] & \sps_{f\circ g}(g_\eta)_!\arrow[d]\\
         (g_\sigma)_\ast\sps_f & \sps_{f\circ g}(g_\eta)_\ast\arrow[l]
    \end{tikzcd}\]
    if $f$ is proper and representable by algebraic spaces, it is a commutative square of isomorphisms. If $g$ is a monomorphism, the following square is commutative:
    \[\begin{tikzcd}
        \sps_{f\circ g}(g_\eta)^!\arrow[r]\arrow[d] & (g_\sigma)^!\sps_f\arrow[d]\\
        \sps_{f\circ g}(g_\eta)^\ast & (g_\sigma)^\ast\sps_f\arrow[l]
    \end{tikzcd}\]
    and it is a square of isomorphisms if $g$ is open.
\end{lemma}

\begin{proof}
    We show the first part of the Lemma, the second part being formally similar. We show this inductively. If $f$ is the identity, all the morphisms are the identity, hence the square is trivially a square of isomorphisms. Suppose we have proven that the square of the Lemma for $\Delta_f$ is a square of isomorphisms. We can divide the square of the Lemma for $f$ into the diagram of morphisms (where we have removed the subscripts for readability):
    \[\begin{tikzcd}
        f_!\arrow[d]\sps\arrow[rrr]\arrow[drr] &&& \sps f_!\arrow[d]\\
        f_!(p_1)_\ast(\Delta_f)_\ast\sps\arrow[d]& f_!(p_1)_\ast\sps(\Delta_f)_\ast\arrow[d]\arrow[l]& f_!\sps(p_1)_\ast(\Delta_f)_\ast\arrow[l]\arrow[r] & \sps f_!(p_1)_\ast(\Delta_f)_\ast\arrow[d]\\
        f_\ast(p_2)_!(\Delta_f)_!\sps\arrow[r]\arrow[d] & f_\ast(p_2)_!\sps(\Delta_f)_!\arrow[r]\arrow[drr] & f_\ast\sps(p_2)_!(\Delta_f)_! & \sps f_\ast(p_2)_!(\Delta_f)_!\arrow[l]\arrow[d]\\
        f_\ast\sps & & & \sps f_\ast\arrow[lll]
    \end{tikzcd}\]
    the upper left and lower right triangle are commutative from the compatibility of specialization systems with composition. The central left square is commutative because the square of the Lemma for $\Delta_f$ is a square of isomorphisms. The central right rectangle is commutative because it expresses the compatibility between specialization systems and the exchange morphism $Ex_{!\ast}$, which is obtained by adjunction from the compatibility between specialization systems and $Ex^\ast_!$. The upper right and lower left trapezoids commutes because they are formed by applying independent functors. Hence the outer square commutes. If $g$ is proper and representable, by assumption $g_!\sps\to\sps g_!$ is an isomorphism, hence the square is a square of isomorphism. By recursion, one obtains that the square of the lemma commutes for $g$ separated, and is a square of isomorphisms for $g$ proper and representable.
\end{proof}

\begin{lemma}\label{lemmaspecpurity}
    Given any specialization system $\sps:\mD_1\to\mD_2$ over $\eta\to B\leftarrow\sigma$, and a smooth map $g:\mX\to\mY$, the following square is a commutative square of isomorphisms:
    \[\begin{tikzcd}
         \sps_{f\circ g}\Sigma^{-\bL_{g_\sigma}}(g_\sigma)^!\arrow[d,"\simeq"]\arrow[r] & \Sigma^{-\bL_{g_\eta}}(g_\eta)^!\sps_f\arrow[d,"\simeq"]\\
         \sps_{f\circ g}(g_\eta)^\ast & (g_\sigma)^\ast\sps_f\arrow[l]
    \end{tikzcd}\]

\end{lemma}
\begin{proof}
    By assumption, the lower horizontal morphism is an isomorphism, hence the commutativity will imply that the upper horizontal morphism is an isomorphism. Consider the diagram:
    \[\begin{tikzcd}
        \pi_\ast\tilde{s}^!\tilde{p}^\ast\pi^\ast\sps\arrow[rr,"\simeq"]\arrow[d,"\simeq"] && \pi_\ast\tilde{s}^!\sps\tilde{p}^\ast\pi^\ast\arrow[d] & \sps\pi_\ast\tilde{s}^!\tilde{p}^\ast\pi^\ast\arrow[l]\arrow[d,"\simeq"]\\
        \pi_\ast\tilde{s}^!0_\ast p^\ast\sps\arrow[r,"\simeq"]\arrow[d,"\simeq"]& \pi_\ast\tilde{s}^!0_\ast \sps p^\ast\arrow[d,"\simeq"]& \pi_\ast\tilde{s}^!\sps 0_\ast p^\ast\arrow[l,swap,"\simeq"] & \sps \pi_\ast\tilde{s}^!0_\ast p^\ast\arrow[l]\arrow[d,"\simeq"]\\
        s^! p^\ast\sps\arrow[r,"\simeq"]& s^!\sps p^\ast && \sps s^! p^\ast\arrow[ll,swap,"\simeq"]
    \end{tikzcd}\]
    where the upper left horizontal arrow is an isomorphism because $\tilde{p},\pi$ are smooth, the left and right vertical long arrows are isomorphisms from the proof of relative purity, the central horizontal arrow is an isomorphism because $0$ is a closed immersion, using Lemma \ref{lemmaspecproper}, and the lower right horizontal arrow is an isomorphism from Lemma \ref{lemmaspecthom}. The upper right and lower left squares are commutative because they comes from applying independent functors, the upper left (resp lower right) rectangle is commutative from the compatibility of specialization systems with the exchange morphism $Ex^\ast_\ast$ (resp $Ex^!_\ast$). Hence, by diagram chase, the whole diagram is a commutative diagram of isomorphism, in particular, the upper right horizontal arrow is an isomorphism. We can now write the same diagram, replacing $0$ by $1$, $s$ by $\Delta_f$ and $p$ by $p_2$, to obtain a commutative diagram of isomorphisms:
    \[\begin{tikzcd}
\pi_\ast\tilde{s}^!\tilde{p}^\ast\pi^\ast\sps\arrow[r,"\simeq"]\arrow[d,"\simeq"] & \sps\pi_\ast\tilde{s}^!\tilde{p}^\ast\pi^\ast\arrow[d,"\simeq"]\\
        (\Delta_f)^!(p_2)^\ast\sps\arrow[r,"\simeq"] & \sps(\Delta_f)^!(p_2)^\ast
    \end{tikzcd}\]

    Now, consider the diagram:
    \[\begin{tikzcd}
        (\Delta_f)^!(p_2)^\ast f^!\sps\arrow[d,"\simeq"] & (\Delta_f)^!(p_2)^\ast \sps f^!\arrow[l,swap,"\simeq"]\arrow[r,"\simeq"] & (\Delta_f)^!\sps (p_2)^\ast f^!\arrow[d,"\simeq"] & \sps (\Delta_f)^!(p_2)^\ast f^!\arrow[l]\arrow[d,"\simeq"]\\
        (\Delta_f)^!(p_1)^! f^\ast\sps\arrow[r,"\simeq"]\arrow[d,"\simeq"] & (\Delta_f)^!(p_1)^! \sps f^\ast\arrow[drr,"\simeq"] & (\Delta_f)^!\sps(p_1)^! f^\ast\arrow[l] & \sps(\Delta_f)^!(p_1)^! f^\ast\arrow[l]\arrow[d,"\simeq"]\\
        f^\ast\sps\arrow[rrr,"\simeq"] &&& \sps f^\ast
    \end{tikzcd}\]
    where the upper left horizontal arrow is an isomorphism as proven above, and the arrows from the first to the second line, obtained from the exchange morphism $Ex^{\ast!}$ are isomorphisms because $f$ is smooth. The upper right and lower left squares trivially commutes, the lower right triangle commutes from the compatibility of specialization systems with compositions, and the upper left rectangles commutes from the compatibility of specialization systems with $Ex^{\ast!}$. A simple diagram chase show then that this diagram is a commutative diagram of isomorphisms. By the construction of the purity isomorphism, the diagram of the lemma is obtained by gluing vertically the three above diagram, hence is a commutative diagram of isomorphisms.
\end{proof}

We prove compatibility with composition:

\begin{lemma}\label{lemmapurcomp}
    The morphism $f_!\to f_\ast$ for separated $f$ is compatible with composition. The relative purity isomorphism is compatible with composition, \ie for $f:\mX\to \mY$ and $g:\mY\to\mZ$ two smooth morphisms, the following square of isomorphisms is commutative:
    \[\begin{tikzcd}
        \Sigma^{-\bL_f}f^!\circ \Sigma^{-\bL_g}g^!\arrow[r,"\simeq"]\arrow[d,"\simeq"] & f^\ast\circ g^\ast\arrow[d,"\simeq"]\\
        \Sigma^{-\bL_{g\circ f}}(g\circ f)^!\arrow[r,"\simeq"] & (g\circ f)^\ast
    \end{tikzcd}\]
    where the left vertical arrow comes from the cofiber sequence $f^\ast(\bL_g)\to \bL_{g\circ f}\to\bL_f$.
\end{lemma}

\begin{proof}
   Consider the following commutative diagram with Cartesian square:
   \[\begin{tikzcd}
        \mX\arrow[d,"\Delta_f"]\arrow[dr,"\Delta_{g\circ f}"] & &\\
        \mX\times_\mY\mX\arrow[r]\arrow[d,"p_2"] & \mX\times_\mZ\mX\arrow[d]\arrow[dr,"p_2"] &\\
       \mX\arrow[d,"f"]\arrow[r,"\Delta_{f^\ast g}"] & \mY\times_\mZ\mX\arrow[r,"p_2"]\arrow[d] & \mX\arrow[d,"f"]\\
       \mY\arrow[r,"\Delta_g"] & \mY\times_\mZ\mY\arrow[r,"p_2"] & \mY
   \end{tikzcd}\]

    To prove the claim that the morphism $f_!\to f_\ast$ for separated $f$ is compatible with composition, one can proceed recursively. We show how to adapt the proof of \cite[Proposition 1.7.3]{Ayoub2018LesSOI}. If $f$ and $g$ are isomorphisms, this is obvious. Suppose that we have proven compatibility with composition for a class $E$ of proper and representable morphisms that is stable by base change and composition (e.g., isomorphisms, or monomorphisms, or closed immersions, or proper and representable morphisms). Consider now the class $E'\supset E$ of morphisms whose diagonal is in $E$. This class is stable by base change, and it is also compatible with composition, from the above diagram. From Lemma \ref{lemmaspecproper}, the morphism $f_!\to f_\ast$ for $f\in E'$ is compatible with base change. Then the arguments of \cite[Proposition 1.7.3]{Ayoub2018LesSOI} carry on directly to show compatibility with composition for morphisms of $E'$.\medskip

    we consider now the compatibility of the purity isomorphism with composition, for smooth $f,g$. Consider the following square of isomorphisms:
    \[\begin{tikzcd}
        (\Delta_f)^!(p_2)^\ast f^!\circ (\Delta_g)^!(p_2)^\ast g^!\arrow[r,"\simeq"]\arrow[d,"\simeq"] & f^\ast\circ g^\ast\arrow[d,"\simeq"]\\
        (\Delta_{g\circ f})^!(p_2)^\ast(g\circ f)^!\arrow[r,"\simeq"] & (g\circ f)^\ast
    \end{tikzcd}\]
    where the left vertical arrows comes from smooth base change \eqref{eqexsmoothast} in the above diagram, and the horizontal arrows from the purity isomorphisms for $f,g$ and $g\circ f$. This diagram is the dual of those of \cite[Proposition 3.23]{Chowdhury2024NonrepresentableSF}, which is proven here (adapting the arguments of \cite[Proposition 1.7.3]{Ayoub2018LesSOI}) to be commutative for $f$ and $gf$ representable. This restriction arise there because \cite[Proposition 4.23]{Chowdhury2024NonrepresentableSF} is a preliminary result to show the relative purity, which is used afterward to prove smooth base change \eqref{eqexsmoothast}. Once one knows that smooth base change holds, the arguments of \cite[Proposition 4.23]{Chowdhury2024NonrepresentableSF} applies directly to our case, giving that the above diagram is commutative.\medskip

    Using deformation to the normal cone, we obtain a commutative diagram with Cartesian squares over $\bA^1$:
    \[\begin{tikzcd}
        \bA^1_\mX\arrow[d,"\tilde{s}_f"]\arrow[dr,"\tilde{s}_{g\circ f}"] & &\\
        D_{\Delta_f}\arrow[r]\arrow[d,"\tilde{p}_f"] & D_{\Delta_{g\circ f}}\arrow[d]\arrow[dr,"\tilde{p}_{g\circ f}"] &\\
       \bA^1_\mX\arrow[d,"\bA^1_f"]\arrow[r,"\tilde{s}_{f^\ast g}"] & D_{\Delta_{f^\ast g}}\arrow[r,"\tilde{p}_{f^\ast g}"]\arrow[d] & \bA^1_\mX\arrow[d,"\bA^1_f"]\\
       \bA^1_\mY\arrow[r,"\tilde{s}_g"] & D_{\Delta_g}\arrow[r,"\tilde{p}_g"] & \bA^1_\mY
   \end{tikzcd}\]
   whose restriction over $1$ is the above diagram, and whose restriction over $0$ is:
   \[\begin{tikzcd}
        \mX\arrow[d,"s_f"]\arrow[dr,"s_{g\circ f}"] & &\\
        \bV_\mX(\bL_f)\arrow[r]\arrow[d,"p_f"] & \bV_\mX(\bL_{g\circ f})\arrow[d]\arrow[dr,"p_{g\circ f}"] &\\
       \mX\arrow[d,"f"]\arrow[r,"s_{f^\ast g}"] & \bV_\mX(f^\ast(\bL_g))\arrow[r,"p_{f^\ast g}"]\arrow[d] & \mX\arrow[d,"f"]\\
       \mY\arrow[r,"s_g"] & \bV_\mY(\bL_g)\arrow[r,"p_g"] & \bA^1_\mY
   \end{tikzcd}\]

   Using base change, we obtain the following commutative diagram of isomorphisms:
   \[\begin{tikzcd}
        (\Delta_f)^!(p_2)^\ast\circ f^!\circ (\Delta_g)^!(p_2)^\ast\arrow[d,"\simeq"] & \pi_\ast (\tilde{s}_f)^!(\tilde{p}_f)^\ast\pi^\ast \circ f^!\circ\pi_\ast(\tilde{s}_g)^!(\tilde{p}_g)^\ast  \pi^\ast\arrow[r,"\simeq"]\arrow[l,swap,"\simeq"]\arrow[d,"\simeq"]& (s_f)^!(p_f)^\ast \circ f^!\circ (s_g)^!(p_g)^\ast\arrow[d,"\simeq"] \\
        (\Delta_{g\circ f})^!(p_2)^\ast & \pi_\ast (\tilde{s}_{g\circ f})^!(\tilde{p}_{g\circ f})^\ast\pi^\ast \arrow[r,"\simeq"]\arrow[l,swap,"\simeq"] & (s_{g\circ f})^!(p_{g\circ f})^\ast
    \end{tikzcd}\]
    where the right vertical arrow is by definition $\Sigma^{-\bL_f}f^!\circ\Sigma^{-\bL_g}\to \Sigma^{-\bL_{g\circ f}}$, and the horizontal arrows are the roofs \eqref{roofpurity} for $f,g$ and $g\circ f$. By construction of the purity isomorphism, the result follows.
\end{proof}

\subsubsection{A contraction lemma}

Consider a stack $\mX$, a closed immersion $i:\mZ\to\mX$, and its open complement $j:\mU\to\mX$. Using the unit $j_!j^\ast\to Id$, $Id\to i_\ast i^\ast$ and the isomorphisms $j^!\simeq j^\ast$, $i_!\simeq i_\ast$, one obtains:
\begin{align}
    j_!j^\ast\to Id\to i_!i^\ast
\end{align}
one can check using a cover by smooth morphisms with Nisnevich-local sections that $(j^\ast,i^\ast)$ are jointly conservative. By base change, $j^\ast i_!\simeq 0$, hence the arguments of \cite[Lemma 1.1]{deligne2001voevodsky} shows that the above extends uniquely to a distinguished triangle:
\begin{align}
    j_!j^\ast\to Id\to i_!i^\ast\to j_!j^\ast[1]
\end{align}

\begin{lemma}\label{lemmaspecloc}
    Given a specialization system  $\sps:\mD_1\to\mD_2$ over $\eta\to B\leftarrow\sigma$, and an open-closed decomposition $(j,i)$, the following diagram is commutative:
    \[\begin{tikzcd}
        (j_\sigma)_!(j_\sigma)^\ast\sps\arrow[r]\arrow[d] & \sps\arrow[r]\arrow[d,"="] & (i_\sigma)_!(i_\sigma)^\ast\sps\arrow[d]\arrow[r] &  {(j_\sigma)_!(j_\sigma)^\ast[1]\sps}\arrow[d]\\
        \sps (j_\eta)_!(j_\eta)^\ast\arrow[r] & \sps\arrow[r] & \sps (i_\eta)_!(i_\eta)^\ast\arrow[r] &  {\sps (j_\eta)_!(j_\eta)^\ast[1]}
    \end{tikzcd}\]
\end{lemma}

\begin{proof}
    The leftmost part can be decomposed into:
    \[\begin{tikzcd}
        j_!j^\ast\sps\arrow[r,"\simeq"]\arrow[d,"\simeq"] & j_!j^!\sps\arrow[dr] &  & i_\ast i^\ast\sps\arrow[d]\arrow[r,"\simeq"] & i_!i^\ast\sps\arrow[d]\\
        j_!\sps j^\ast\arrow[r,"\simeq"]\arrow[d,"\simeq"] & j_!\sps j^!\arrow[d,"\simeq"]\arrow[u,"\simeq"] & \sps\arrow[ur]\arrow[dr] & i_\ast\sps i^\ast\arrow[r,"\simeq"] & i_!\sps i^\ast\arrow[d,"\simeq"]\\
        \sps j_!j^\ast\arrow[r,"\simeq"] & \sps j_!j^!\arrow[ur] &  & \sps i_\ast i^\ast\sps\arrow[r,"\simeq"]\arrow[u,"\simeq"] & \sps i_!i^\ast\\
    \end{tikzcd}\]
    the upper left and lower right squares commutes respectively from Lemma \ref{lemmaspecproper}, the upper right and lower left squares trivially commutes, and the two central triangles commutes from adjunctions. To show that the right square of the diagram of the Lemma commutes, we can argue as in \cite[Lemma 1.1]{deligne2001voevodsky}. Namely, consider the morphism $(i_\sigma)_!(i_\sigma)^\ast\sps\to \sps (i_\eta)_!(i_\eta)^\ast$ obtained from the cone of the two leftmost vertical arrows of the diagram of the Lemma, and the difference $w$ with the third vertical arrow of the diagram of the Lemma: we have to show $w=0$. The composite arrow:
    \begin{align}
        \sps \to (i_\sigma)_!(i_\sigma)^\ast\sps\overset{w}{\to} \sps (i_\eta)_!(i_\eta)^\ast
    \end{align}
    has to vanish, as the second square commutes. Then $w$ factorize through a map $(j_\sigma)_!(j_\sigma)^\ast[1]\sps\to \sps (i_\eta)_!(i_\eta)^\ast$, which corresponds by adjunction to a map:
    \begin{align}
        (j_\sigma)^\ast[1]\sps\to (j_\sigma)^\ast\sps (i_\eta)_!(i_\eta)^\ast\simeq \sps (j_\eta)^\ast(i_\eta)_!(i_\eta)^\ast\simeq 0
    \end{align}
    hence $w=0$.
\end{proof}

Consider now the closed immersion $0:\mX\to\bA^1_\mX$, and the projection $\pi:\bA^1_\mX\to\mX$. There is a natural morphism $\pi_\ast\to \pi_\ast 0_\ast 0^\ast\simeq 0^\ast$. The first part of the following lemma is a well known consequence of the localization and homotopy property. We  are particularly interested in the second part, expressing the compatibility with specialization systems.

\begin{lemma}\label{lemmaspeccontr}
    Consider the morphism $\bar{q}:\bA^1_\mX\to[\bA^1_\mX/\bG_m]$. On the image of $\bar{q}^\ast$, the morphism $\pi_\ast\to 0^\ast$ is an isomorphism. For any specialization system, on the image of $(\bar{q}_\eta)^\ast$, the following:
    \[\begin{tikzcd}
        (\pi_\sigma)_\ast\sps\arrow[d] & \sps(\pi_\eta)_\ast\arrow[l]\arrow[d]\\
        (0_\sigma)^\ast\sps\arrow[r] & \sps(0_\eta)^\ast
    \end{tikzcd}\]
    is a square of isomorphisms.
\end{lemma}

\begin{proof}
    Consider the commutative and Cartesian square:
    \[\begin{tikzcd}
        \bG_{m,\mX}\arrow[d,"q"]\arrow[r,"j"] &\bA^1_\mX\arrow[d,"\bar{q}"]\\
        \mX\arrow[r,"\tilde{j}"] & {[\bA^1_\mX/\bG_m]}
    \end{tikzcd}\]
    where $\tilde{j}$ is an open immersion, and $\bar{q}\circ j=\tilde{j}\circ \pi\circ j$. The composition of morphisms:
    \begin{align}
        Id\to\pi_\ast\pi^\ast\to\pi_\ast0_\ast 0^\ast\pi^\ast\simeq (\pi\circ 0)_\ast(0\circ \pi)^\ast
    \end{align}
    is identified with the unit of the adjunction for $0\circ \pi\simeq Id$, hence is an isomorphism. But $Id\to\pi_\ast\pi^\ast$ is an isomorphism by the homotopy property, hence the second morphism is an isomorphism. In the localization exact triangle:
    \begin{align}
        \pi_\ast j_!j^\ast \pi^\ast\to & \pi_\ast \pi^\ast\to \pi_\ast 0_!0^\ast \pi^\ast
    \end{align}
    the second morphism is then an isomorphism, hence $\pi_\ast j_!j^\ast \pi^\ast\simeq 0$. In the localization exact triangle:
    \begin{align}
        \pi_\ast j_!j^\ast\bar{q}^\ast\to&\pi_\ast\bar{q}^\ast\to 0^\ast\bar{q}^\ast\nn
    \end{align}
    the first term is isomorphic with $\pi_\ast j_!j^\ast\pi^\ast\tilde{j}^\ast$, hence vanishes, and then the second arrow is an isomorphism, which proves the first part.\medskip
    
    We check now the compatibility with specialization systems. The compositions:
    \begin{align}
        &\sps\simeq 0^\ast\pi^\ast\sps\to 0^\ast\sps\pi^\ast\to \sps 0^\ast\pi^\ast\simeq \sps\nn\\
        &\sps \simeq\sps \pi_\ast 0_\ast\to \pi_\ast\sps 0_\ast\to \pi_\ast 0_\ast\sps\simeq \sps
    \end{align}
    are the identity, and the first (resp second) arrow is an isomorphism because $\pi$ is smooth (resp $0$ is a closed immersion), hence the second (resp first) arrow is an isomorphism too. Hence $0_!0^\ast\sps\pi^\ast\to \sps 0_!0^\ast\pi^\ast $ and $\sps\pi_\ast 0_!0^\ast\to \pi_\ast\sps 0_!0^\ast$ are isomorphisms. Consider the following diagram:
    \[\begin{tikzcd}
        \pi_\ast j_!j^\ast\pi^\ast\sps\arrow[r]\arrow[d,"\simeq"] & \pi_\ast\pi^\ast\sps\arrow[r,"\simeq"]\arrow[d,"\simeq"] & \pi_\ast 0_!0^\ast\pi^\ast\sps\arrow[d,"\simeq"]\\
        \pi_\ast j_!j^\ast\sps\pi^\ast\arrow[r]\arrow[d] & \pi_\ast\sps\pi^\ast\arrow[r]\arrow[d,"="] & \pi_\ast 0_!0^\ast\sps\pi^\ast\arrow[d,"\simeq"]\\
        \pi_\ast \sps j_!j^\ast\pi^\ast\arrow[r] & \pi_\ast\sps\pi^\ast\arrow[r] & \pi_\ast \sps0_!0^\ast\pi^\ast\\
        \sps \pi_\ast j_!j^\ast\pi^\ast\arrow[r]\arrow[u] & \sps\pi_\ast\pi^\ast\arrow[r,"\simeq"]\arrow[u] & \sps\pi_\ast 0_!0^\ast\pi^\ast\arrow[u]
    \end{tikzcd}\]
    where the morphisms from the second to third line forms morphisms of exact triangle from Lemma \ref{lemmaspecloc}. By the property of morphisms of exact triangles, the central left vertical arrow is an isomorphism, hence all the objects of the left column are trivial. Consider now the following diagram:
    \[\begin{tikzcd}
        \pi_\ast j_!j^\ast\bar{q}^\ast\sps\arrow[r]\arrow[d,"\simeq"] & \pi_\ast\bar{q}^\ast\sps\arrow[r]\arrow[d,"\simeq"] & \pi_\ast 0_!0^\ast\bar{q}^\ast\sps\arrow[d,"\simeq"]\\
        \pi_\ast j_!j^\ast\sps\bar{q}^\ast\arrow[r]\arrow[d] & \pi_\ast\sps\bar{q}^\ast\arrow[r]\arrow[d,"="] & \pi_\ast 0_!0^\ast\sps\bar{q}^\ast\arrow[d]\\
        \pi_\ast \sps j_!j^\ast\bar{q}^\ast\arrow[r] & \pi_\ast\sps\bar{q}^\ast\arrow[r] & \pi_\ast \sps0_!0^\ast\bar{q}^\ast\\
        \sps \pi_\ast j_!j^\ast\bar{q}^\ast\arrow[r]\arrow[u] & \sps\pi_\ast\bar{q}^\ast\arrow[r]\arrow[u] & \sps\pi_\ast 0_!0^\ast\bar{q}^\ast\arrow[u,swap,"\simeq"]
    \end{tikzcd}\]
    where the morphisms from the second to third line forms morphisms of exact triangle from Lemma \ref{lemmaspecloc}. As proven above, all the objects of the left column are trivial, hence all the arrows from the central column to the right column are isomorphisms. By a diagram chase, the lower central and central right vertical arrows are isomorphisms. We obtain the following diagram:
    \[\begin{tikzcd}
    \pi_\ast\sps\bar{q}^\ast\arrow[r,"\simeq"]\arrow[dd,"="] & \pi_\ast 0_!0^\ast\sps\bar{q}^\ast\arrow[d,"\simeq"]\arrow[r,"\simeq"] & \pi_\ast 0_\ast0^\ast\sps\bar{q}^\ast\arrow[r,"\simeq"]\arrow[d] & 0^\ast\sps\bar{q}^\ast\arrow[d]\\
    & \pi_\ast 0_!\sps0^\ast\bar{q}^\ast\arrow[d,"\simeq"]\arrow[r,"\simeq"] & \pi_\ast 0_\ast\sps0^\ast\bar{q}^\ast\arrow[r,"\simeq"] & \sps0^\ast\bar{q}^\ast\arrow[dd,"="]\\
        \pi_\ast\sps\bar{q}^\ast\arrow[r,"\simeq"] & \pi_\ast \sps0_!0^\ast\bar{q}^\ast\arrow[r,"\simeq"] & \pi_\ast \sps0_\ast 0^\ast\bar{q}^\ast\arrow[u,swap,"\simeq"] & \\
\sps\pi_\ast\bar{q}^\ast\arrow[r,"\simeq"]\arrow[u,swap,"\simeq"] & \sps\pi_\ast 0_!0^\ast\bar{q}^\ast\arrow[u,swap,"\simeq"]\arrow[r,"\simeq"] & \sps\pi_\ast 0_\ast 0^\ast\bar{q}^\ast\arrow[u]\arrow[r,"\simeq"] & \sps 0^\ast\bar{q}^\ast
    \end{tikzcd}\]
    where the left squares commute by the above, the central square commute by Lemma \ref{lemmaspecproper}, the lower right rectangle commutes by functoriality of the exchange morphism, and the remaining ones trivially commutes. By diagram chase, we obtain that the outer square, which is the square of the Lemma, is a commutative square of isomorphisms.
\end{proof}

\begin{corollary}\label{corsp01}
    On the image of $\bar{q}^\ast$, one obtains a natural morphism $0^\ast\simeq \pi_\ast\to 1^\ast$. For any specialization system, on the image of $(\bar{q}_\eta)^\ast$, the following square is commutative:
\[\begin{tikzcd}
    (0_\sigma)^\ast\sps\arrow[r]\arrow[d] & \sps (0_\eta)^\ast\arrow[d]\\
    (1_\sigma)^\ast\sps\arrow[r] & \sps(1_\eta)^\ast
\end{tikzcd}\]
\end{corollary}

\begin{proof}
    it suffices to check that the square is commutative. It is obtained by gluing vertically the commutative diagram of isomorphisms of Lemma \ref{lemmaspeccontr} and the diagram:
    \[\begin{tikzcd}
        \pi_\ast\sps\arrow[d]\arrow[drr] & & & \sps\pi_\ast\arrow[d]\arrow[lll]\\
        \pi_\ast 1_\ast 1^\ast\sps\arrow[r]\arrow[d,"\simeq"] & \pi_\ast 1_\ast \sps1^\ast\arrow[drr,"\simeq"] & \pi_\ast\sps 1_\ast 1^\ast\arrow[l] & \sps\pi_\ast 1_\ast 1^\ast\arrow[l]\arrow[d,"\simeq"]\\
        1^\ast\sps\arrow[rrr] & & & \sps 1^\ast
    \end{tikzcd}\]
    where the upper right (resp lower left) triangle commutes from the compatibility of specialization systems with adjunctions (resp composition).
    
\end{proof}

\section{Hyperbolic localization for stacks}\label{secthyploc}

In this section, we work over an excellent quasi-separated algebraic space $S$. All our stacks are assumed to be quasi-separated, locally of finite presentation over $S$, with affine stabilizers.

\subsection{The Theta-correspondence}

\subsubsection{The Theta-correspondence and Braden-Drinfeld-Gaitsgory theorem}

Consider the algebraic stack $\Theta:=[\bA^1/\bG_m]$ over $\bZ$, where $\bG_m$ acts on $\bA^1$ with its usual action. We denote by $\Theta_S$ its base change over $S$ There is a natural diagram:

\begin{equation}\label{diagbgm}\begin{tikzcd}
B\bG_{m,S}\arrow[loop left]\arrow[r] & \Theta_S\arrow[l,shift right=2] & S\arrow[l]
\end{tikzcd}\end{equation}

where the arrow $B\bG_{m,S}\to\Theta_S$ (resp. $S\to\Theta_S$) is given by the inclusion of $0$ (resp. $1$), the arrow $S\to B\bG_{m,S}$ is the natural one, and the involution of $B\bG_{m,S}$ comes from the inversion $x\to x^{-1}$ of $\bG_{m,S}$ (one use here the fact that $\bG_{m,S}$ is commutative).\medskip

Consider now a algebraic stack $\mX$ over $S$ (as said above, $S$ is assumed to be quasi-separated and excellent, and $\mX$ is assumed to be quasi-separated, locally of finite presentation over $S$, with affine stabilizers). Denote the mapping stacks:
\begin{align}
    \Filt(\mX)&=\Map_S(\Theta_S,\mX)\nn\\
    \Grad(\mX)&=\Map_S(B\bG_{m,S},\mX)
\end{align}
According to \cite[Theorem 6.22]{Alper2019TheL} (building on the results of \cite{HalpernLeistner2014MappingSA}), $\Grad(\mX)$ and $\Filt(\mX)$ are algebraic stacks, which are quasi-separated and locally of finite presentation over $S$, with affine stabilizers. One obtains morover from the above diagram a natural diagram of algebraic stacks (which are quasi-separated, locally of presentation, with affine stabilizer):
\[\begin{tikzcd}
    \Grad(\mX)\arrow[rr,bend left,"\iota"]\arrow[r,shift left=2,"\tilde{\iota}"]\arrow[loop left,"r"] & \Filt(\mX)\arrow[l,"p"]\arrow[r,"\eta"] & \mX
\end{tikzcd}\]
with $r$ an involution, $\iota=\eta\circ\iota$ and $p\circ\tilde{\iota}=Id$, $p$ is of finite presentation, and $\eta$ is representable by algebraic spaces.\medskip

We can now consider the hyperbolic localization functor:
\begin{align}
    p_!\eta^\ast: \mD(\mX)\to \mD(\Grad(\mX))
\end{align}
The following result is a generalization of the main theorem of Braden \cite{Braden2002HyperbolicLO} and Drinfeld-Gaitsgory \cite{Drinfeld2013ONAT}:

\begin{theorem}\label{theobradenstacks}
    Given a stack $\mX$ over $S$, one has a canonical adjunction:
    \begin{align}
         p_!\eta^\ast: \mD(\mX)\rightleftharpoons \mD(\Grad(\mX)):\eta_!(r\circ p)^\ast
    \end{align}
\end{theorem}

We will prove this in the next section, following the proof of Drinfeld-Gaitsgory \cite{Drinfeld2013ONAT}, \cite{drinfeld2013algebraic}.

\subsubsection{The quotient case}

Consider an algebraic space $X$ (as always, assumed to be quasi-separated and locally of finite presentation over $S$), with an action of $\bG_{m,S}$. We consider the functor of points:
\begin{align}
    X^0&:=Map^{\bG_{m,S}}_S(S,X)\nn\\
    X^\pm&:=Map^{\bG_{m,S}}_S((\bA^1_S)^\pm,X)
\end{align}
where $(\bA^1_S)^+$ (resp. $(\bA^1_S)^-$) is $\bA^1_S$ with its canonical (resp. opposite) $\bG_{m,S}$-action. $X^0$ is the fixed points variety, $X^+$ (resp. $X^-$) is called the attracting (resp. repelling) variety: informally, it classifies points $x\in X$ with a limit $\lim_{t\to 0}t.x$ (resp. $\lim_{t\to \infty}t.x$). Considering the commutative diagram:
\[\begin{tikzcd}[sep=small]
    & (\bA^1_S)^+\arrow[dl,shift left=2] & \\
    S\arrow[ur,"0"]\arrow[dr,swap,"0"] && \bG_{m,S}\arrow[ul]\arrow[dl]\arrow[ll]\\
    & (\bA^1_S)^-\arrow[ul,shift right=2]
\end{tikzcd}\]
One obtains from the mapping construction the hyperbolic localization diagram of \cite{Braden2002HyperbolicLO}:
\[\begin{tikzcd}
    & X^+\arrow[dl,"p^+",swap]\arrow[dr,"\eta^+"] & \\
X^0 \arrow[ur,"\tilde{\iota}^+",swap,shift right=2]\arrow[dr,shift left=2,"\tilde{\iota}^-"]\arrow[rr,"\iota"] & & X\\
& X^-\arrow[ul,"p^-"]\arrow[ur,swap,"\eta^-"] &
\end{tikzcd}\]
Notice that $[X^0/\bG_m]$ (resp. $[X^\pm/\bG_m]$) is a component of $\Grad([X/\bG_m]$ (resp. $\Filt([X/\bG_m])$, hence they are representable quasi-separated algebraic spaces, locally of finite presentation over $S$, and $p^\pm$ is of finite presentation. In particular, we see directly that our Theorem \ref{theobradenstacks} is a direct generalization of the main result of Braden \cite{Braden2002HyperbolicLO} and Drinfeld-Gaitsgory \cite{Drinfeld2013ONAT}, which gives an adjunction $(p^+)_!(\eta^+)^\ast\rightleftharpoons(\eta^-)_!(p^-)^\ast$, where those functors are restricted to the triangulated subcategory of $\mD(X)$ generated by objects pulled back from $\mD([X/\bG_{m,S}])$.\medskip

Consider a smooth affine group $G$ over $\bZ$ or a field, with a split maximal torus $T$. Consider $\mX=[X/G_S]$ with $X$ an algebraic space (as always, assumed to be quasi-separated and locally of finite presentation over $S$), with an action of $G_S$ over $S$. Consider:
\begin{align}
    \Lambda=\Hom(\bG_{m},T)
\end{align}
considered as a set of cocharacters of $G$ (they represent class of cocharacters up to conjugation). For $\lambda\in\Lambda$, denote by $X^0_\lambda,X^+_\lambda$
the fixed locus and attracting locus of $X$ with the induced action of $\bG_{m,S}$. Considering the $\bG_m$-action on $G$ given by conjugation, denote by $L_\lambda:=G^0\subset G$ the corresponding Levi subgroup, and by $P_\lambda:=G^+\subset G$ the corresponding parabolic subgroup. Notice that $L_{\lambda,S}$ (resp $P_{\lambda,S}$) acts naturally on $X^0_\lambda$ (resp $X^+_\lambda$), and that there is a natural isomorphism $r_\lambda:X^0_\lambda\simeq X^0_{-\lambda}$.\medskip

\begin{proposition}(Halpern-Leistner, \cite[Theo 1.4.7]{HalpernLeistner2014OnTS})\label{prophlp}
    For $\mX=[X/G_S]$ with $X$ an algebraic space, the $\Theta$-correspondence is given by:
    \[\begin{tikzcd}
\bigsqcup_{\lambda\in\Lambda}{[X^0_\lambda/L_{\lambda,S}]}\arrow[r,shift left=2,"\bigsqcup_{\lambda\in\Lambda}\tilde{\iota}_\lambda"]\arrow[loop left,"\bigsqcup_{\lambda\in\Lambda}r_\lambda"] & \bigsqcup_{\lambda\in\Lambda}{[X^+_\lambda/P_{\lambda,S}]}\arrow[l,"\bigsqcup_{\lambda\in\Lambda}p_\lambda"]\arrow[r,"\bigsqcup_{\lambda\in\Lambda}\eta_\lambda"] & {[X/G_S]}
\end{tikzcd}\]
    
\end{proposition}

In some sense, in the smooth topology, one can always reduce to the case of schemes with $\bG_m$-action, thanks to the following result of Halpern-Leistner\cite[Proposition 1.7.3]{Ayoub2018LesSOI}:

\begin{lemma}(Halpern-Leistner \cite[Lemma 4.4.6]{HalpernLeistner2014OnTS})\label{covergrad}
    Consider a quasi-separated algebraic stack $\mX$, locally of finite presentation, with affine stabilizers, over a quasi-separated and excellent algebraic space $S$. There is a smooth covering by affine schemes with $\bG_m$-action $f_i:[X_i/\bG_m]\to\mX$ such that $\Filt(f_i):\Filt([X_i/\bG_m])\to\Filt(\mX)$ and $\Grad(f_i):\Grad([X_i/\bG_m])\to\Grad(\mX)$ are a smooth covering.
\end{lemma}

Remark that \cite[Lemma 4.4.6]{HalpernLeistner2014OnTS} gives that, for a quasi-compact and quasi-separated stacks $\mX$, there is a smooth covering $[X/(\bG_m)^n]\to\mX$ for some $n$ which is still surjective after passing to $\Grad$ and $\Filt$. But, as we allow infinite covering, for $\mX$ quasi-separated, we can consider such covering for a Zariski covering by quasi-compact stacks, obtaining $f_i[X_i/(\bG_m)^{n_i}]\to\mX$ which are surjective after passing to $\Grad$ and $\Filt$, and then consider $f_{i,\lambda}:[X_i/\bG_m]\to\mX$ for each cocharacter $\lambda$ of $(\bG_m)^{n_i}$, which are still surjective after passing to $\Grad$ and $\Filt$.

\subsection{Proof of Braden-Drinfeld-Gaitsgory adjunction}

\subsubsection{Construction of the unit}\label{construnit}

Consider the following diagram, with a Cartesian square:
\[\begin{tikzcd}
    \Grad(\mX)\arrow[dr,"j"]\arrow[drr,"\tilde{\iota}",bend left]\arrow[ddr,"\tilde{\iota}\circ r",swap,bend right] &  & & \\
     & \Filt(\mX)\times_{\mX}\Filt(\mX)\arrow[r,"q_1"]\arrow[d,"q_2"] & \Filt(\mX)\arrow[d,"\eta"]\arrow[r,"p"] & \Grad(\mX)\\
     & \Filt(\mX)\arrow[r,"\eta"]\arrow[d,"r\circ p"] & \mX & \\
     & \Grad(\mX) & &
\end{tikzcd}\]
such that $p\circ \tilde{\iota}=r\circ p\circ \tilde{\iota}\circ r=Id$.

The main point is this analogue of \cite[Proposition 1.6.2]{drinfeld2013algebraic}:

\begin{lemma}
    The morphism $j:\Grad(\mX)\to \Filt(\mX)\times_{\mX}\Filt(\mX)$ is an open and closed immersion.
\end{lemma}

\begin{proof}
    The morphism $j$ is a section of the morphism:
    \begin{align}
        \pi:\Filt(\mX)\times_{\mX}\Filt(\mX)\overset{p_1}{\to} \Filt(\mX)\overset{p}{\to}\Grad(\mX)
    \end{align}
    it suffices to show that $\pi$ is étale at the image of $j$. By construction, $\pi$ is locally of finite type, hence it suffices to prove the formal lifting criterion. Consider an algebraic space $T\to S$, and a $T$-point of $\Map_S(B\bG_{m,S},\mX)$ given by $f:B\bG_{m,T}\to \mX$. Now, consider a square-zero extension $T'$ of $T$ defined by a module $M$, and an extension $f':T'\to\mX$. We want to study the groupoid of extension of the commutative diagram:
    \[\begin{tikzcd}
        & T\arrow[d]\arrow[r]& {[\bA^1_T/\bG_{m,T}]}\arrow[d]\\
        \bG_{m,T}\arrow[r] & {[\bA^1_T/\bG_{m,T}]}\arrow[r]  &\bG_{m,T}\arrow[d,"f"]\\
        && \mX
    \end{tikzcd}\]
    to a commutative diagram:
    \[\begin{tikzcd}
        & T'\arrow[d]\arrow[r]& {[\bA^1_{T'}/\bG_{m,T'}]}\arrow[dd,dotted,"h"] \\
        \bG_{m,T'}\arrow[r]\arrow[drr,swap,"f'"] & {[\bA^1_{T'}/\bG_{m,T'}]}\arrow[dr,dotted,"g"] & \\
        && \mX
    \end{tikzcd}\]
    Recall that objects $\QCoh(B\bG_{m,T}))$ are canonically identified with graded objects of $\QCoh(T)$. Extensions of $f$ to $f'$ are classified by the groupoid:
    \begin{align}
        \Map_{B\bG_{m,T}}(f^\ast\bL_\mX,M)\simeq \Map_T((f^\ast\bL_\mX)^0,M)
    \end{align}
    Extensions of $f$ to $g$ are classified by the groupoid:
    \begin{align}
        \Map_{B\bG_{m,T}}(f^\ast\bL_\mX,M[t])\simeq \Map_T((f^\ast\bL_\mX)^{\geq 0},M)
    \end{align}
    Extensions of $f$ to $h$ are classified by the groupoid:
    \begin{align}
        \Map_{B\bG_{m,T}}(f^\ast\bL_\mX,M[t^{-1}])\simeq \Map_T((f^\ast\bL_\mX)^{\leq 0},M)
    \end{align}
    and extensions of $f$ to a map $T'\to\mX$ are classified by the groupoid:
    \begin{align}
        \Map_T(f^\ast\bL_\mX,M)
    \end{align}
    Finally, The groupoid of maps $T'\to\Filt(\mX)\times_{\mX}\Filt(\mX)$ extending $j\circ f:T\to $ is given by the groupoid:
    \begin{align}
        \Map_T((f^\ast\bL_\mX)^{\geq 0},M)\times_{\Map_T(f^\ast\bL_\mX,M)}\Map_T((f^\ast\bL_\mX)^{\leq 0},M)\simeq \Map_T((f^\ast\bL_\mX)^0,M)
    \end{align}
    hence $\pi$ is étale at $f:T\to\mX$.
    \end{proof}

The unit is then given by the formula inpired from \cite{Braden2002HyperbolicLO}:
\begin{align}
    Id\simeq (p\circ \tilde{\iota})_! (r\circ p\circ \tilde{\iota}\circ r)^\ast\simeq p_!(q_1)_!j_!j^\ast(q_2)^\ast(r\circ p)^\ast\to p_!(q_1)_!(q_2)^\ast(r\circ p)^\ast\simeq p_!\eta^\ast\eta_!(r\circ p)^\ast
\end{align}
where the arrow comes from the $(j_!,j^\ast)$ adjunction (because $j$ is open) and the last morphism is obtained by base change.

\subsubsection{Construction of the counit}\label{constrcounit}

The construction of the counit is the main input of Drinfeld-Gaitsgory \cite{Drinfeld2013ONAT}. We adapt their construction to the stack case. Consider the following commutative diagram with Cartesian squares:
\[\begin{tikzcd}
     \Filt(\mX)\times_{\Grad(\mX)}\Filt(\mX)\arrow[r,"q'_2"]\arrow[d,"q'_1"] & \Filt(\mX)\arrow[r,"\eta"]\arrow[d,"r\circ p"] & \mX\\
      \Filt(\mX)\arrow[d,"\eta"]\arrow[r,"p"] & \Grad(\mX) & \\
      \mX & &
\end{tikzcd}\]

Consider $\bA^2$ with its antidiagonal $\bG_m$-action $s.(t,t')\mapsto (st,s^{-1}t')$. Consider the $\bG_m$-invariant map $\bA^2\to\bA^1$ given by $(t,t')\mapsto tt'$. It gives a flat GIT quotient $[\bA^2_S/\bG_{m,S}]\to\bA^1_S$, and $\bA^1_S$ is an excellent quasi-separated algebraic space. Moreover, under our assumption, $\bA^1_\mX$ is quasi-separated, locally of finite presentation over $\bA^1_S$, with affine stabilizers. Hence, from \cite[Theorem 6.22]{Alper2019TheL}, the mapping stack:
\begin{align}
    \tilde{\mX}:=\Map_{\bA^1_S}([\bA^2_S/\bG_{m,S}],\bA^1_\mX)
\end{align}
is an algebraic stack, quasi-separated and locally of finite presentation over $\bA^1_S$, with affine stabilizers. Consider the commutative diagram with Cartesian square:
\[\begin{tikzcd}
    {([\bA^2_S/\bG_{m,S}])_0}\arrow[r]\arrow[d] & {[\bA^2_S/\bG_{m,S}]}\arrow[d] & S\arrow[l]\arrow[d]\\
    S\arrow[r,"0"] & \bA^1_S &S\arrow[l,swap,"1"]
\end{tikzcd}\]
where $([\bA^2/\bG_m])_0$ is the pushout:
\[\begin{tikzcd}
    B\bG_{m}\arrow[r]\arrow[d] & {[\bA^1/\bG_m]}\arrow[d]\\
    {[\bA^1/\bG_m]}\arrow[r] & ([\bA^2/\bG_m])_0
\end{tikzcd}\]
where the upper horizontal (resp left vertical) morphism is the morphism inducing $p$ (resp $r\circ p$). Passing to the mapping stack, this gives commutative diagrams with Cartesian squares:
\[\begin{tikzcd}
    \Filt(\mX)\times_{\Grad(\mX)}\Filt(\mX)\arrow[d]\arrow[r,"i_0"]& \tilde{\mX}\arrow[d]& \mX\arrow[d]\arrow[l,swap,"i_1"]\\
    S\arrow[r,"0"] & \bA^1_S & S\arrow[l,swap,"1"]
\end{tikzcd}\]

Moreover, consider the commutative diagram with Cartesian squares of sections:
\[\begin{tikzcd}
    S\arrow[r,"0"]\arrow[d,shift left=1]\arrow[d,shift right=1] & \bA^1_S\arrow[d,shift left=1]\arrow[d,shift right=1] &S\arrow[l,swap,"1"]\arrow[d,shift left=1]\arrow[d,shift right=1] \\
   {([\bA^2_S/\bG_{m,S}])_0}\arrow[r] & {[\bA^2_S/\bG_{m,S}]}& S\arrow[l]
\end{tikzcd}\]
where the two vertical central sections are given respectively by $\sigma_1:t\mapsto(1,t)$ and $\sigma_2:t\mapsto(t,1)$. These two sections define two morphisms $\tilde{p}_i:\tilde{\mX}\to\bA^1_\mX$ over $\bA^1_S$, whose restriction over $0$ (resp $1$) are $\eta\circ q'_1$ and $\eta\circ q'_2$ (resp $Id$ and $Id$), hence it gives a commutative and Cartesian diagram:
\[\begin{tikzcd}
    \Filt(\mX)\times_{\Grad(\mX)}\Filt(\mX)\arrow[d,"\eta\circ q'_1\times_S\eta\circ q'_2"]\arrow[r]& \tilde{\mX}\arrow[d,"\tilde{p}"]& \mX\arrow[d,"\Delta_{\mX/S}"]\arrow[l]\\
    \mX\times_S\mX\arrow[r,"0"] & \bA^1_{\mX\times_S\mX} & \mX\times_S\mX\arrow[l,swap,"1"]
\end{tikzcd}\]
Notice now that, giving to $\bA^1_S$ and $\bA^1_\mX$ its tautological $\bG_m$-action, and to $\bA^2_S/\bG_m$ the $\bG_m$-action induced by the diagonal $\bG_m$-action on $\bA^2_S$; $[\bA^2_S/\bG_{m,S}]\to\bA^1_S$, $\bA^1_\mX\to\mX$, and the two sections $\bA^1_S\to[\bA^2_S/\bG_m]$ have a natural $\bG_m$-equivariant structure. Then $\tilde{\mX}$ has a natural $\bG_m$-action, and there is a natural $\bG_m$-equivariant structure on $\tilde{p}:\tilde{\mX}\to \bA^1_{\mX\times_S\mX}$. There is then a Cartesian diagram:
\[\begin{tikzcd}
    \tilde{\mX}\arrow[r,"\tilde{p}"]\arrow[d,"\bar{q}'"] & \bA^1_{\mX\times_S\mX}\arrow[d,"\bar{q}"]\\
    {[\tilde{\mX}/\bG_m]}\arrow[r,"\tilde{p}"] & {[\bA^1_{\mX\times_S\mX}/\bG_m]}
\end{tikzcd}\]
In particular, by base change:
\begin{align}
    \tilde{p}_!\un_{\tilde{\mX}}\simeq \tilde{p}_!(\bar{q}')^\ast\un_{[\tilde{\mX}/\bG_m]}\simeq \bar{q}^\ast(\tilde{p}')_!\un_{[\tilde{\mX}/\bG_m]}
\end{align}
Hence, using Corollary \ref{corsp01}, we obtain a natural morphism:
\begin{align}
    (\eta\circ q'_1\times_S\eta\circ q'_2)_!\un_{\Filt(\mX)\times_{\Grad(\mX)}\Filt(\mX)}\simeq 0^\ast\tilde{p}_!\un_{\tilde{\mX}}\to 1^\ast \tilde{p}_!\un_{\tilde{\mX}}\simeq (\Delta_{\mX/S})_!\un_\mX
\end{align}
Using this morphism, the counit is given by the formula:
\begin{align}
    \eta_!(r\circ p)^\ast p_!\eta^\ast&\simeq(\eta\circ q'_2)_!(\eta\circ q'_1)^\ast\nn\\
    &\simeq (pr_2)_!((pr_1^\ast-\otimes_{\mX\times_S\mX}(\eta\circ q'_1\times_S\eta\circ q'_2)_!\un_{\Filt(\mX)\times_{\Grad(\mX)}\Filt(\mX)})\nn\\
    &\to (pr_2)_!((pr_1^\ast-\otimes_{\mX\times_S\mX}(\Delta_{\mX/S})_!\un_\mX)\nn\\
    &\simeq Id_! Id^\ast\simeq Id
\end{align}
where the first isomorphism is base change, and the isomorphisms of the second and last lines are obtained by the projection formula.

\subsubsection{Proof of the adjunction}

We follow and adapt the arguments of \cite[Section 5]{Drinfeld2013ONAT}. We obtain by composition two morphisms:
\begin{align}
    &p_!\eta^\ast\to p_!\eta^\ast\circ\eta_!(r\circ p)^\ast\circ p_!\eta^\ast\to p_!\eta^\ast:\mD(\mX)\to\mD(\Grad(\mX))\label{firstcomp}\\
    &\eta_!(r\circ p)^\ast\to\eta_!(r\circ p)^\ast\circ p_!\eta^\ast\circ\eta_!(r\circ p)^\ast\to \eta_!(r\circ p)^\ast:\mD(\Grad(\mX))\to\mD(\mX)\label{secondcomp}
\end{align}
and we have to show that these are the identity. We will show this for the first morphism. Consider the commutative and Cartesian diagram:
\[\begin{tikzcd}
     \Filt(\mX)\times_{\Grad(\mX)}\Filt(\mX)\times_\mX\Filt(\mX)\arrow[r,"q_{23}"]\arrow[d,"q_{12}"] & \Filt(\mX)\times_\mX\Filt(\mX)\arrow[d,"q_1"]\arrow[r,"q_2"] & \Filt(\mX)\arrow[d,"\eta"]\arrow[r,"p"] & \Grad(\mX)\\
\Filt(\mX)\times_{\Grad(\mX)}\Filt(\mX)\arrow[r,"q'_2"]\arrow[d,"q'_1"] & \Filt(\mX)\arrow[r,"\eta"]\arrow[d,"r\circ p"] & \mX\\
      \Filt(\mX)\arrow[d,"\eta"]\arrow[r,"p"] & \Grad(\mX) & \\
      \mX & &
\end{tikzcd}\]
We consider the following diagram, obtained from the diagram of Section \ref{construnit} by base change along $p:\Filt(\mX)\to \Grad(\mX)$:
\[\begin{tikzcd}
    \Filt(\mX)\arrow[ddr,bend right,swap,"\eta"]\arrow[drr,bend left,"p"]\arrow[dr,"j'"] & \\
    & \Filt(\mX)\times_{\Grad(\mX)}\Filt(\mX)\times_\mX\Filt(\mX)\arrow[d,"\tilde{q}_1"]\arrow[r,"\tilde{q}_2"] & \Grad(\mX)\\
    & \mX &
\end{tikzcd}\]
Using base change, the first arrow is identified with:
\begin{align}
    p_!\eta^\ast\simeq (\tilde{q}_1)_! \tilde{j}_!\tilde{j}^\ast(\tilde{q}_2)^\ast\to (\tilde{q}_1)_! (\tilde{q}_2)^\ast
\end{align}
obtained from the adjunction for the open immersion $\tilde{j}$. We consider the following diagram, obtained from the main diagram of Section \ref{constrcounit} by base change along $\bA^1_\eta:\bA^1_{\Filt(\mX)}\to\bA^1_{\mX}$:
\[\begin{tikzcd}
    \Filt(\mX)\times_{\Grad(\mX)}\Filt(\mX)\times_\mX\Filt(\mX)\arrow[d,"\tilde{q}_1\times_S\tilde{q}_2"]\arrow[r]& \tilde{\mX}\times_{\bA^1_\mX}\bA^1_{\Filt(\mX)}\arrow[d,"\tilde{p}'"]& \Filt(\mX)\arrow[d,"\eta\times_S p"]\arrow[l]\\
    \mX\times_S\Grad(\mX)\arrow[r,"0"] & \bA^1_{\mX\times_S\Grad(\mX)} & \mX\times_S\Grad(\mX)\arrow[l,swap,"1"]
\end{tikzcd}\]
where the arrow $\tilde{p}'$ is obtained by the composition:
\begin{align}
    \tilde{\mX}\times_\mX\Filt(\mX)\to\bA^1_{\mX\times_S\Filt(\mX)}\overset{\bA^1_{Id\times_S p}}{\to}\bA^1_{\mX\times_S\Grad(\mX)}
\end{align}
Using base change, the second arrow is obtained from the morphism:
\begin{align}
    (\tilde{q}_1\times_S\tilde{q}_2)_!\un_{\tilde{q}_1\times_S\tilde{q}_2}\simeq 0^\ast(\tilde{p}')_!\un_{\tilde{\mX}\times_\mX\Filt(\mX)}\to 1^\ast(\tilde{p}')_!\un_{\tilde{\mX}\times_\mX\Filt(\mX)} \simeq (\eta\times_S p)_!\un_{\Filt(\mX)}
\end{align}

\begin{lemma}\label{cruciallemma}
    There is a monomorphism of $\bA^1_S$-stacks $\tilde{j}':\bA^1_{\Filt(\mX)}\to\tilde{\mX}\times_{\bA^1_\mX}\bA^1_{\Filt(\mX)}$ such that the following diagram is commutative with Cartesian square:
\[\begin{tikzcd}
    \Filt(\mX)\arrow[d,"j'"]\arrow[r,"0"] & \bA^1_{\Filt(\mX)}\arrow[d,"\tilde{j}'"] & \Filt(\mX)\arrow[l,swap,"1"]\arrow[d,"Id"]\\
\Filt(\mX)\times_{\Grad(\mX)}\Filt(\mX)\times_\mX\Filt(\mX)\arrow[d,"\tilde{q}_1\times_S\tilde{q}_2"]\arrow[r]& \tilde{\mX}\times_{\bA^1_\mX}\bA^1_{\Filt(\mX)}\arrow[d,"\tilde{p}'"]& \Filt(\mX)\arrow[d,"\eta\times_S p"]\arrow[l]\\
    \mX\times_S\Grad(\mX)\arrow[r,"0"] & \bA^1_{\mX\times_S\Grad(\mX)} & \mX\times_S\Grad(\mX)\arrow[l,swap,"1"]
\end{tikzcd}\]
\end{lemma}

\begin{proof}
    Consider the commutative square of $\bA^1_S$-stacks
    \[\begin{tikzcd}[column sep=huge]
        {\bA^1_S\times_S[\bA^1_S/\bG_{m,S}]} & {\bA^1_S\times_S[\bA^1_S/\bG_{m,S}]}\arrow[l,swap,"{(s,t)\mapsto(s,st)}"]\\
        {[\bA^2_S/\bG_{m,S}]}\arrow[u,swap,"{(s,t)\mapsto(st,s)}"] & {\bA^1_S\times_S[\bG_{m,S}/\bG_{m,S}]}\arrow[u,swap,"{(s,t) \mapsto(s,t)}"]\arrow[l,swap,"{(s,t)\mapsto(st,t^{-1})}"]
    \end{tikzcd}\]
    It defines a fpqc covering of $\bA^1_S\times_S[\bA^1_S/\bG_{m,S}]$, which gives in particular a universal epimorphic family. Applying $\Map_{\bA^1_S}(-,\bA^1_\mX)$, we obtain a commutative square of $\bA^1_S$-stacks:
    \[\begin{tikzcd}
        \bA^1_{\Filt(\mX)}\arrow[r]\arrow[d] & \bA^1_{\Filt(\mX)}\arrow[d,"\bA^1_\eta"]\\
        \tilde{\mX}\arrow[r,"\tilde{p}_2"] & \bA^1_\mX
    \end{tikzcd}\]
    which gives a morphism of $\bA^1_S$-stacks:
\begin{align}
    \tilde{j}':\bA^1_{\Filt(\mX)}\to \tilde{\mX}\times_{\bA^1_\mX}\bA^1_{\Filt(\mX)}
\end{align}
which is a monomorphism, from universal epimorphicity.\medskip

Over $1_S$, we obtain directly that $\tilde{j}'_1$ is the identity of $\Filt(\mX)$. Over $0_S$, the above diagram is the outer rectangle of the commutative diagram:
\[\begin{tikzcd}
        {[\bA^1_S/\bG_{m,S}]} & B\bG_{m,S}\arrow[l,swap,"0"] & {[\bA^1_S/\bG_{m,S}]}\arrow[l]\\
        {([\bA^2_S/\bG_m])_0}\arrow[u,swap,"{(s,t)\mapsto s}"] & {[(\bA^1_S)^-/\bG_{m,S}]}\arrow[l,swap,"{t\mapsto (0,t)}"]\arrow[u] & {[\bG_{m,S}/\bG_{m,S}]}\arrow[u,swap,"{t\mapsto t}"]\arrow[l,swap,"{t\mapsto t^{-1}}"]
    \end{tikzcd}\]
where the left square is a pushout. It gives the commutative diagram diagram:
\[\begin{tikzcd}
        \Filt(\mX)\arrow[r,"p"]\arrow[d] & \Grad(\mX)\arrow[d,"\tilde{\iota}\circ r"]\arrow[r,"\tilde{\iota}"]& \Filt(\mX)\arrow[d,"\eta"]\\
        \Filt(\mX)\times_{\Grad(\mX)}\Filt(\mX)\arrow[r,"q'_2"] & \Filt(\mX)\arrow[r,"\eta"] & \mX
    \end{tikzcd}\]
where the right square is Cartesian. The map $j:\Grad(\mX)\to \Filt(\mX)\times_\mX\Filt(\mX)$ is induced by the right square, hence the map $j':\Filt(\mX)\to\Filt(\mX)\times_{\Grad(\mX)}\Filt(\mX)\times_\mX\Filt(\mX)$ obtained by base change is induced by the outer rectangle, \ie coincide with the restriction of $\tilde{j}'_0$ of $\tilde{j}'$ over $0_S$.
\end{proof}

As in \cite[Section 5]{Drinfeld2013ONAT}, the above construction is the crucial step in the proof. In \cite[Section 3]{drinfeld2013algebraic}, Drinfeld has that the analogue of $\tilde{j}'$ is an open immersion, which is a crucial point in their proof. We do not see how to generalize Drinfeld's agument to our setting, fortunately we can use only the fact that $\tilde{j}'$ is a monomorphism, at the price of a subtle diagram chase.\medskip 

Consider the following diagram of morphisms:
\[\begin{tikzcd}
    (\eta\times_S p)_!\un\arrow[r,"\simeq"]\arrow[d,"\simeq"] & (\eta\times_S p)_!0^\ast\un\arrow[r,"\simeq"]\arrow[d,"\simeq"] & (\eta\times_S p)_!1^\ast\un\arrow[r,"\simeq"]\arrow[d,"\simeq"] & (\eta\times_S p)_!\un\arrow[d,"\simeq"]\\
    (\tilde{q}_1\times_S\tilde{q}_2)_!(j')_!\un\arrow[d]\arrow[r,"\simeq"] & 0^\ast(\tilde{p}')_!(\tilde{j}')_!\un\arrow[r] & 1^\ast(\tilde{p}')_!(\tilde{j}')_!\un\arrow[r,"\simeq"] & (\eta\times_S p)_!(Id)_!\un\arrow[d,"\simeq"]\\
    (\tilde{q}_1\times_S\tilde{q}_2)_!\un\arrow[r,"\simeq"] & 0^\ast(\tilde{p}')_!\un\arrow[r] & 1^\ast(\tilde{p}')_!\un\arrow[r,"\simeq"] & (\eta\times_S p)_!\un
\end{tikzcd}\]
where the left and right upper squares commutes by base change, and the upper central square commutes because the morphism of Corollary \ref{corsp01} commutes with specialization system, in particular with base change. Notice that the central horizontal arrow is then an isomorphism.  As said above, under the equivalence $g_!f^\ast\simeq (p_2)_!(p_1^\ast-\otimes (f\times g)_!\un)$, the morphism \eqref{firstcomp} is identified with the composition of the left vertical arrow with the lower horizontal arrow, and the upper horizontal and right vertical arrows are the identity, hence it suffices to prove that the lower rectangle is commutative. If we knew that $\tilde{j}'$ is an open immersion, one would to complete the lower rectangle, but we only know that $\tilde{j}'$ is a monomorphism. Recall that there is from the discussion of Section \ref{sectionpropsmooth} a natural morphism $(\tilde{j}')^!\to(\tilde{j}')^\ast$, and that there is an adjunction morphism $(\tilde{j}')_!(\tilde{j}')^!\to Id$. The lower rectangle is identified with:

\[\begin{tikzcd}
    (\tilde{q}_1\times_S\tilde{q}_2)_!(j')_!(j')^\ast 0^\ast\un & 0^\ast(\tilde{p}')_!(\tilde{j}')_!(\tilde{j}')^\ast\un\arrow[l,swap,"\simeq"]\arrow[r,"\simeq"] & 1^\ast(\tilde{p}')_!(\tilde{j}')_!(\tilde{j}')^\ast\un\arrow[r,"\simeq"] & (\eta\times_S p)_!(Id)_!(Id)^\ast 1^\ast\un\\
    (\tilde{q}_1\times_S\tilde{q}_2)_!(j')_!(j')^! 0^\ast\un\arrow[u,"\simeq"]\arrow[d] & 0^\ast(\tilde{p}')_!(\tilde{j}')_!(\tilde{j}')^!\un\arrow[l,swap,"\simeq"]\arrow[r]\arrow[u]\arrow[d] & 1^\ast(\tilde{p}')_!(\tilde{j}')_!(\tilde{j}')^!\un\arrow[r]\arrow[u]\arrow[d] & (\eta\times_S p)_!(Id)_!(Id)^! 1^\ast\un\arrow[u,"\simeq"]\arrow[d,"\simeq"]\\
    (\tilde{q}_1\times_S\tilde{q}_2)_!0^\ast\un & 0^\ast(\tilde{p}')_!\un\arrow[l,swap,"\simeq"]\arrow[r] & 1^\ast(\tilde{p}')_!\un\arrow[r,"\simeq"] & (\eta\times_S p)_!1^\ast\un
\end{tikzcd}\]
where the central left horizontal arrow comes from the exchange transformation $Ex^{!\ast}$: it is an isomorphism because, from Lemma \ref{lemmaspeccontr}, $0^\ast$ commutes with specialization systems, in particular with $(j')^!$. The upper left and right square commutes because the transformation $(\tilde{j}')^!\to(\tilde{j}')^\ast$ commutes with base change from Lemma \ref{lemmaspecproper}. The central squares trivially commutes, and the lower left and right square commutes from the compatibility of adjunction with base change. We obtain then that the upper rectangle and the lower rectangle commutes, hence the outer rectangle commutes, which finishes the proof that \eqref{firstcomp} is the identity.\medskip

To prove that \eqref{secondcomp} is the identity, the proof is formally similar. One consider instead the commutative diagram with Cartesian squares:
\[\begin{tikzcd}
     \Filt(\mX)\times_\mX\Filt(\mX)\times_{\Grad(\mX)}\Filt(\mX)\arrow[r,"q_{23}"]\arrow[d,"q_{12}"] & \Filt(\mX)\times_{\Grad(\mX)}\Filt(\mX)\arrow[d,"q'_1"]\arrow[r,"q'_2"] & \Filt(\mX)\arrow[d,"r\circ p"]\arrow[r,"\eta"] & \mX\\
\Filt(\mX)\times_\mX\Filt(\mX)\arrow[r,"q_2"]\arrow[d,"q_1"] & \Filt(\mX)\arrow[r,"p"]\arrow[d,"\eta"] & \Grad(\mX)\\
      \Filt(\mX)\arrow[d,"r\circ p"]\arrow[r,"\eta"] & \mX & \\
      \Grad(\mX) & &
\end{tikzcd}\]
and use the square:
\[\begin{tikzcd}[column sep=huge]
        {\bA^1_S\times_S[\bA^1_S/\bG_{m,S}]} & {\bA^1_S\times_S[\bA^1_S/\bG_{m,S}]}\arrow[l,swap,"{(s,t)\mapsto(s,s^{-1}t)}"]\\
        {[\bA^2_S/\bG_{m,S}]}\arrow[u,swap,"{(s,t)\mapsto(st,t^{-1})}"] & {\bA^1_S\times_S[\bG_{m,S}/\bG_{m,S}]}\arrow[u,swap,"{(s,t) \mapsto(s,t)}"]\arrow[l,swap,"{(s,t)\mapsto(t,st^{-1})}"]
    \end{tikzcd}\]
to build an open immersion $\tilde{j}'':\bA^1_{\Filt(\mX)}\to \bA^1_{\Filt(\mX)}\times_{\bA^1_\mX}\tilde{\mX}$ fitting in the commutative diagram with Cartesian squares:
\[\begin{tikzcd}
    \Filt(\mX)\arrow[d,"j''"]\arrow[r,"0"] & \bA^1_{\Filt(\mX)}\arrow[d,"\tilde{j}''"] & \Filt(\mX)\arrow[l,swap,"1"]\arrow[d,"Id"]\\
\Filt(\mX)\times_\mX\Filt(\mX)\times_{\Grad(\mX)}\Filt(\mX)\arrow[d,"\tilde{q}_1\times_S\tilde{q}_2"]\arrow[r]& \bA^1_{\Filt(\mX)}\times_{\bA^1_\mX}\tilde{\mX}\arrow[d,"\tilde{p}''"]& \Filt(\mX)\arrow[d,"(r\circ p)\times_S\eta"]\arrow[l]\\
    \Grad(\mX)\times_S\mX\arrow[r,"0"] & \bA^1_{\Grad(\mX)\times_S\mX} & \Grad(\mX)\times_S\mX\arrow[l,swap,"1"]
\end{tikzcd}\]
look at \cite[Section 5]{Drinfeld2013ONAT}, \cite[Section 3]{drinfeld2013algebraic} for more details.

\subsection{Functoriality of hyperbolic localization}

\subsubsection{Functoriality with specialization systems}

\begin{lemma}\label{basechangehyploc}
    If $\phi:\mX'\to \mX$ is obtained by based change along a morphism of algebraic spaces $B'\to B$, then $\eta_{\mX'}$, $p_{\mX'}$, $\tilde{\iota}_{\mX'}$ and $r_{\mX'}$ are obtained by base change from  $\eta_\mX$, $p_\mX$, $\tilde{\iota}_\mX$ and $r_\mX$. This is also the case for all the functors build in the proof of Theorem \ref{theobradenstacks}.
\end{lemma}

\begin{proof}
    All the morphisms in the proof of Theorem \ref{theobradenstacks} are build using mapping spaces. Namely, for the quasi-separated algebraic space $T=S$ (resp $T=\bA^1_S$), one consider a morphism of $T$-stacks $\mZ'\to\mZ$ having $T$ as a good moduli space, and the morphism of mapping spaces $\Map_T(\mZ,\mY_T)\to Map(\mZ',\mY_T)$. The morphism $\Map_T(\mZ,\mX'_T)\to \Map_T(\mZ,\mX_T)\times_{\Map_T(\mZ',\mX_T)}Ma_Tp(\mZ',\mX'_T)$ is obtained by base change from $\Map_T(\mZ,B'_T)\to \Map_T(\mZ,B_T)\times_{\Map_T(\mZ',B_T)}\Map_T(\mZ',B'_T)$, hence it suffices to show the result of the proposition of $B'\to B$. But, from the universal property of good moduli spaces, $B^{(')}_T\to \Map_T(\mZ^{(')},B^{(')}_T))$ is an isomorphism, hence this is trivial.
\end{proof}

\begin{theorem}\label{commutspechyploc}
    Given any specialization system $\sps:\mD_1\to\mD_1$ over algebraic spaces $\eta\to B\leftarrow\sigma$, given any stack $f:\mX\to B$, the natural morphism:
    \begin{align}
        (p_\sigma)_!(\eta_\sigma)^\ast\sps_f\to\sps_{\Grad(f)}(p_\eta)_!(\eta_\eta)^\ast\nn\\
        (\eta_\sigma)_!(p_\sigma)^\ast\sps_{\Grad(f)}\to\sps_f(\eta_\eta)_!(p_\eta)^\ast
    \end{align}
    are isomorphisms.
\end{theorem}

\begin{proof}
    We begin to show the the exchange morphisms of the specialization systems are compatible with the adjunctions of Braden-Drinfeld-Gaitsgory theorem. We use for that the fact that all the diagrams used to define the unit and counit for $\mX_\sigma$ and $\mX_\eta$ are Cartesian over those used to define the unit and counit for $\mX$ from Lemma \ref{basechangehyploc}, which allows to use the compatibility of specialization systems with base change. Consider the following diagram:
    \[\begin{tikzcd}
        & \sps_{\Grad(f)}\arrow[dl]\arrow[dr] &\\
        (p_\sigma)_!(\eta_\sigma)^\ast(\eta_\sigma)_!((r\circ p)_\sigma)^\ast\sps_{\Grad(f)}\arrow[r] & (p_\sigma)_!(\eta_\sigma)^\ast\sps_f(\eta_\eta)_!((r\circ p)_\eta)^\ast\arrow[r] & \sps_{\Grad(f)}(p_\sigma)_!(\eta_\sigma)^\ast(\eta_\eta)_!((r\circ p)_\eta)^\ast
    \end{tikzcd}\]
    where the two vertical arrows come from the unit build in Section \ref{construnit}. This construction use base change, which commutes with specialization systems, and the purity isomorphism for $j$, which commutes with specialization system from Lemma \ref{lemmaspecpurity}, hence the above diagram is commutative. Consider the following diagram:
    \[\begin{tikzcd}
        (\eta_\sigma)_!((r\circ p)_\sigma)^\ast(p_\sigma)_!(\eta_\sigma)^\ast\sps_f\arrow[r]\arrow[dr] & (\eta_\sigma)_!((r\circ p)_\sigma)^\ast\sps_{\Grad(f)}(p_\eta)_!(\eta_\eta)^\ast\arrow[r] & \sps_f(\eta_\sigma)_!((r\circ p)_\sigma)^\ast(p_\eta)_!(\eta_\eta)^\ast\arrow[dl]\\
        &\sps_f &
    \end{tikzcd}\]
    where the two vertical arrows come from the counit build in Section \ref{constrcounit}. Notice that, denoting by $\tilde{p}_1,\tilde{p}_2:\tilde{\mX}\to\bA^1_\mX$ the two projections of $\tilde{p}$, this morphism can be rewritten as:
    \begin{align}
    \eta_!(r\circ p)^\ast p_!\eta^\ast&\simeq(\eta\circ q'_2)_!(\eta\circ q'_1)^\ast\nn\\
    &\simeq 0^\ast (\tilde{p}_2)_!(\tilde{p}_1)^\ast\pi^\ast\nn\\
    &\to 1^\ast (\tilde{p}_2)_!(\tilde{p}_1)^\ast\pi^\ast\nn\\
    &\simeq Id_! Id^\ast\simeq Id
\end{align}
the second morphism commutes with specialization systems from Corollary \ref{corsp01}, and the other isomorphisms commutes with specialization systems as they are obtained by functoriality and base change, hence the above diagram commutes. The end of the proof is a formal game of adjunction. We have the commutative diagram (where we have removed the subscripts for readability):
    \[\begin{tikzcd}
        p_!\eta^\ast\sps\arrow[r]\arrow[d]& \sps p_!\eta^\ast\arrow[d]\arrow[drr] & &\\
        p_!\eta^\ast\eta_!(r\circ p)^\ast p_!\eta^\ast\sps\arrow[r]\arrow[drr] &  p_!\eta^\ast\eta_!(r\circ p)^\ast\sps p_!\eta^\ast\arrow[r] & p_!\eta^\ast\sps \eta_!(r\circ p)^\ast p_!\eta^\ast\arrow[r]\arrow[d] & \sps p_!\eta^\ast\eta_!(r\circ p)^\ast p_!\eta^\ast\arrow[d]\\
        & & p_!\eta^\ast\sps\arrow[r] & \sps p_!\eta^\ast
    \end{tikzcd}\]
    by the adjunction property, the two outer diagonal paths are the identity, hence the morphism $\sps p_!\eta^\ast\to p_!\eta^\ast\sps$ obtained by the staircase central composition is the inverse of $p_!\eta^\ast\sps\to\sps p_!\eta^\ast$. A dual reasoning gives an inverse of $\eta_!(r\circ p)^\ast\sps\to\sps\eta_!(r\circ p)^\ast$, but $r$ is an involution, hence $r^\ast$ trivially commutes with specialization systems, then $\eta_!p^\ast\sps\to\sps\eta_!p^\ast$ is an isomorphism.
\end{proof}

In particular, this implies the commutation of hyperbolic localization with $g^\ast,g^!,g_!,g_\ast$, when $g$ is obtained by base change from a morphism of algebraic spaces, as obtained in \cite[Proposition 3.1]{Ric16}, and with (monodromic) nearby and vanishing cycles, as in \cite[Theorem 3.3]{Ric16}. We put also this simple result for reference:

\begin{lemma}\label{lemhyplocprod}
    Hyperbolic localization commutes with exterior product, namely, given $\mX_1,\mX_2$, one has a natural isomorphisms:
    \begin{align}
        (p_{\mX_1\times\mX_2})_!(\eta_{\mX_1\times\mX_2})^\ast(-\boxtimes_S-)\simeq ((p_{\mX_1})_!(\eta_{\mX_1})^!)\boxtimes_S((p_{\mX_2})_!(\eta_{\mX_2})^!)
    \end{align} satisfying associativity and commutativity
\end{lemma}

\begin{proof}
    From the monoidality of the mapping stack construction, $p_{\mX_1\times\mX_2}=p_{\mX_1}\times_S p_{\mX_2}$ and $\eta_{\mX_1\times\mX_2}=\eta_{\mX_1}\times_S \eta_{\mX_2}$, hence this follows from the monoidality of the four functors.
\end{proof}

\subsubsection{Functoriality with smooth morphisms}\label{sectsmoothhyploc}

We want to talk about the behavior cotangent complex of stacks, or morphisms of stacks, under the construction $\Grad(-)$ and $\Filt(-)$. Those behaves better in the world of derived stacks, but at the end we will be interested in cotangent complexes of smooth morphisms, such that these will agree with their classical truncations. We follow here the presentation in \cite[Section 1.2]{HalpernLeistner2014OnTS} and \cite[Section 1]{HalpernLeistner2020DerivedA}, adapting it to our non-derived setting.\medskip

For a stack with affine stabilizers $\mX$, $\Grad(\mX)$ has a natural weak action action of the monoid $B\bG_m$, which gives from \cite[Lemma 1.5.3]{HalpernLeistner2020DerivedA} (when we consider them as derived stacks) a natural $\bZ$-grading on any quasi-coherent complex $\mE=\bigoplus_{n\in\bZ}\mE^n$. Similarly, $\Filt(\mX)$ has a natural weak action of the monoid $\Theta$, such that, from \cite[Proposition 1.1.2]{HalpernLeistner2020DerivedA} (when we consider them as derived stacks), there is a baric structure on $\QCoh(\Grad(\mX))$. It gives roughly a version of truncation at the level of derived category, in particular, for $w\in\bZ$, each object lies in an exact triangle:
\begin{align}
    \mathcal{F}^{\geq w}\to \mathcal{F}\to \mathcal{F}^{<w}
\end{align}
We use freely the notation $\mE^I,\mathcal{F}^I$ for any interval $I$, defined in the obvious way. From \cite[Lemma 1.1.5, Lemma 1.5.3]{HalpernLeistner2020DerivedA}, these structures are compatible in the sense that:
\begin{align}\label{compattrunc}
    &\tilde{\iota}^\ast(\mathcal{F}^{\geq w})=\bigoplus_{n\geq w} (\tilde{\iota}^\ast \mathcal{F})^n\quad &\tilde{\iota}^\ast(\mathcal{F}^{< w})=\bigoplus_{n< w} (\tilde{\iota}^\ast \mathcal{F})^n\nn\\
    &(p^\ast \mE)^{\geq w}=\bigoplus_{n\geq w}p^\ast(\mE^n)\quad &(p^\ast \mE)^{< w}=\bigoplus_{n< w} p^\ast (\mE^n)
\end{align}
In particular, when $\mathcal{F}$ is a locally free sheaf, $(p\circ \tilde{i})^\ast\mathcal{F}$ is the graded objects associated to the filtration. This is a transcription for non-derived stacks of the results \cite[Lemma 1.3.2, Lemma 1.5.5]{HalpernLeistner2020DerivedA}:

\begin{lemma}\label{lemcotanbun} (Halpern-Leistner, \cite{HalpernLeistner2020DerivedA})
    For $\phi:\mX\to\mY$ a smooth morphism of stacks, $\Grad(\phi)$ and $\Filt(\phi)$ are smooth too, and one has canonical isomorphisms:
    \begin{align}
        \bL_{\Filt(\phi)}\simeq\eta^\ast(\bL_\mX)^{\leq 0}\nn\\
        \bL_{\Grad(\phi)}\simeq \iota^\ast(\bL_\mX)^0
    \end{align}
\end{lemma}

\begin{proof}
    As explained in \cite[Section 1.2]{HalpernLeistner2020DerivedA}, there is a derived version of mapping stack for derived stacks (as functors of groupoids on the $(\infty,1)$ categories of connective simplicial rings), such that, extending trivially $B\bG_m$ and $\Theta$, one obtains derived versions $\ud{\Grad}(-)$ and $\ud{\Filt}(-)$ of the $\Grad(-)$ and $\Filt(-)$ functors. Denoting by $\ud{\mX}$ the trivial derived extension of a classical stack, we have:
\begin{align}
    \ud{\Grad}(\ud{\mX})^{cl}=\Grad(\mX)\nn\\
    \ud{\Filt}(\ud{\mX})^{cl}=\Filt(\mX)
\end{align}
and similarly for morphisms (where the superscript $cl$ denotes the classical truncation). By \cite[Lemma 1.3.2, Lemma 1.5.5]{HalpernLeistner2020DerivedA}, one has canonical identifications:
\begin{align}
    \bL_{\ud{\Filt}(\ud{\mX})}\simeq \ud{\eta}^\ast(\bL_{\ud{\mX}})^{\leq 0}\nn\\
    \bL_{\ud{\Grad}(\ud{\mX})}\simeq \ud{\iota}^\ast(\bL_{\ud{\mX}})^0
\end{align}
Now, consider a morphism $\phi:\mX\to\mY$. Using the exact triangle of cotangent complexes for $\ud{\phi}$, and applying $\ud{\eta}^\ast$ to those of $\ud{\Filt}(\ud{\phi})$ (resp. $\ud{\iota}^\ast$ to those of $\ud{\Grad}(\ud{\phi})$), using the fact that the truncations commutes with pullbacks, we obtain canonical isomorphisms:
\begin{align}
    \bL_{\ud{\Filt}(\ud{\phi})}\simeq \ud{\eta}^\ast(\bL_{\ud{\phi}})^{\leq 0}\nn\\
    \bL_{\ud{\Grad}(\ud{\phi})}\simeq \ud{\iota}^\ast(\bL_{\ud{\phi}})^0
\end{align}
(this was noticed in the proof of \cite[Corollary 1.3.2.1]{HalpernLeistner2020DerivedA} for $\ud{\Filt}(\ud{\phi}$). In particular, if $\phi$ is smooth, hence $\bL_{\ud{\phi}}$ is perfect in amplitude $[-1,0]$, because pullbacks and truncation functors are exact, $\bL_{\ud{\Filt}(\ud{\phi})}$ and $\bL_{\ud{\Grad}(\ud{\phi})}$ are also perfect in amplitude $[-1,0]$. One obtains from \cite[Lemma 1.2.4]{HalpernLeistner2014OnTS} that their pullbacks to the classical truncations $\Grad(\mX)$ and $\Filt(\mX)$ agree with $\bL_{\Grad(\phi)}$ and $\bL_{\Filt(\phi)}$. Using the commutations of the truncation functors with pullbacks, one obtains the isomorphisms of the proposition from the derived version above. In particular, $\Grad(\phi)$ and $\Filt(\phi)$ are smooth.
\end{proof}

The following is a relative and stacky version of the main result of Bialynicki-Birula \cite{BiaynickiBirula1973SomeTO}:

\begin{proposition}\label{propBB}
    Consider a smooth morphism of stacks $\phi:\mX\to\mY$. Then $\pi:\Filt(\mX)\to\Grad(\mX)\times_{\Grad(\mY)}\Filt(\mY)$ is canonically an affine bundle stack modeled on $(pr_1)^\ast((\iota^\ast\bL_\phi)^{<0}))$.
\end{proposition}

\begin{proof}
    We will prove that at the level of the functors of groupoids, using deformation theory. The groupoid of $\Spec(A)$-point of $\Grad(\mX)\times_{\Grad(\mY)}\Filt(\mY)$ is the groupoid of commutative diagrams:
    \[\begin{tikzcd}[column sep=4cm]
        B\bG_m\times\Spec(A)\arrow[r,"f"]\arrow[d,"i_n"] & \mX\arrow[ddd,"\phi"]\\
        \Spec(A[t]/t^n)/\bG_m\arrow[d]\arrow[ur,dotted,"f_n"] &\\
        \Spec(A[t]/t^{n+1})/\bG_m\arrow[d]\arrow[uur,dotted,"f_{n+1}"] &\\
        \Spec(A[t])/\bG_m\arrow[r]\arrow[uuur,dotted,"\bar{f}"] & \mY
    \end{tikzcd}\]
    consider the point $g:\Spec(A)\to \Grad(\mX)\times_{\Grad(\mY)}\Filt(\mY)$ corresponding to this diagram. In particular, $p_2\circ g:\Spec(A)\to\mX$ is the $\Spec(A)$-point defined by $f$. We want to study the deformation-theoretic problem of lifting $f_n$ to $f_{n+1}$, for $n\geq 1$. The square-zero ideal $I$ of the closed immersion $\Spec(A[t]/t^n)/\bG_m\to \Spec(A[t]/t^{n+1})/\bG_m$ is the $\bG_m$-equivariant ideal $(t^n)/(t^{n+1})$, which is, as a $\bG_m$-equivariant $A[t]/t^n$-module, identified with $(i_n)_\ast(\mathcal{O}_{\bG_m}\langle n\rangle\boxtimes A)$. We have then:
    \begin{align}
        \bR \Hom_{\Spec(A[t]/t^n)/\bG_m}((f_n)^\ast\bL_\phi,I)&=\bR \Hom_{\Spec(A[t]/t^n)/\bG_m}((f_n)^\ast\bL_\phi,(i_n)_\ast(\mathcal{O}_{\bG_m}\langle n\rangle\boxtimes A))\nn\\
        &\simeq \bR \Hom_{\Spec(A)\times B\bG_m}(f^\ast\bL_\phi,\mathcal{O}_{\bG_m}\langle n\rangle\boxtimes A)\nn\\
        &\simeq \bR \Hom_{\Spec(A)}((f^\ast\bL_\phi)^{-n},A)\nn\\
        &\simeq \bR \Hom_{\Spec(A)}(g^\ast(pr_1)^\ast(\iota^\ast\bL_\phi)^{-n},A)
    \end{align}
    here the second line holds by adjunction, the third by the definition of the truncation functor given below \cite[Definition 1.1.1]{HalpernLeistner2020DerivedA}, and the last line from the definition of the truncation functor on $\Grad(\mX)$. Given a lift $f_n$, according to \cite[Theorem 8.5]{Pridham2013PresentingHS}, the obstruction to extends $f_n$ to $f_{n+1}$ lies in:
    \begin{align}
        \Ext^1_{\Spec(A[t]/t^n)/\bG_m}((f_n)^\ast\bL_\phi,I)\simeq \Ext^1_{\Spec(A)}(g^\ast(pr_1)^\ast(\iota^\ast\bL_\phi)^{-n},A)
    \end{align}
    which vanishes because the right hand term is perfect in degree $[0,1]$. Then still according to \cite[Theorem 8.5]{Pridham2013PresentingHS}, the groupoid of extension is naturally a torsor over:
    \begin{align}
        \Map_{\Spec(A[t]/t^n)/\bG_m}((f_n)^\ast\bL_\phi,I)\simeq \Map_{\Spec(A)}(g^\ast(pr_1)^\ast(\iota^\ast\bL_\phi)^{-n},A)
    \end{align}
     Using the coherent completeness of $\Theta_{\Spec(A)}$ along $B\bG_m\times\Spec(A)$ \cite[Proposition 5.18]{Alper2015AL} and the Tannaka duality result \cite[Corollary 2.8]{Alper2015AL}, the groupoid of lifts $\bar{f}$ is equivalent to the groupoid of coherent systems of lifts $f_n$. Hence, the groupoid of lifts $\bar{f}$, \ie of $\Spec(A)$-points of $\Filt(\mX)$ over $g:\Spec(A)\to \Grad(\mX)\times_{\Grad(\mY)}\Filt(\mY)$, is naturally a torsor over the Abelian groupoid:
     \begin{align}
       \Map_{\Spec(A)}(g^\ast(pr_1)^\ast(\iota^\ast\bL_\phi)^{<0},A)
    \end{align}
    which is by definition the groupoid of $\Spec(A)$-points of $\bV_{\Grad(\mX)\times_{\Grad(\mY)}\Filt(\mY)}((pr_1)^\ast(\iota^\ast\bL_\phi)^{<0})$. This action is obviously functorial in $A$, hence we have built a structure of $\bV_{\Grad(\mX)\times_{\Grad(\mY)}\Filt(\mY)}((pr_1)^\ast(\bL_\phi^{<0}))$-torsor on $\Filt(\mX)\to\Grad(\mX)\times_{\Grad(\mY)}\Filt(\mY)$.
\end{proof}

\begin{remark}
    In particular, if $\mX$ is smooth over $\mY=S$, from Lemma \ref{basechangehyploc}, $\Grad(\mX)\times_{\Grad(S)}\Filt(S)\simeq\Grad(\mX)$. One obtains that $\Filt(\mX)\to\Grad(\mX)$ is an affine bundle stack over $(\iota^\ast\bL_{\mX})^{<0}$. In this case, $\tilde{\iota}:\Grad(\mX)\to\Filt(\mX)$ gives a canonical section, hence $\Filt(\mX)\to\Grad(\mX)$ is identified with the vector bundle stack $\bV_{\Grad(\mX)}((\iota^\ast\bL_\mX)^{<0})$. In particular, if $\mX=[X/\bG_m]$, $[X^+/\bG_m]\to[X^0/\bG_m]$ is a component of $\Filt(\mX)\to\Grad(\mX)$, and, passing to the cover $X^0\to[X^0/\bG_m]$,
    one find the classical result of  Bialynicki-Birula \cite{BiaynickiBirula1973SomeTO}, saying that $X^+\to X^0$ is identified with the vector bundle $\bV_X((\bL_X|_{X^0})^{<0})$. In this sense, our result is a relative and stacky version of \cite{BiaynickiBirula1973SomeTO}.
\end{remark}

From now on, we will slightly abuse the notations by denoting $\bL_\phi^I$ for $(\iota^\ast\bL_\phi)^I$.

\begin{proposition}\label{propsmoothBB}
    Let $\phi:\mX\to\mY$ be a smooth morphism of stacks, and suppose that the coefficient system satisfies étale descent, or that $\pi:\Filt(\mX)\to\Grad(\mX)\times_{\Grad(\mY)}\Filt(\mY)$ has Nisnevich-local sections (it is plausible that this condition is always satisfied, see Remark \ref{remNLaffine}). There is a natural isomorphism:
    \begin{align}
        (p_\mX)_!(\eta_\mX)^\ast\phi^\ast\simeq \Sigma^{-\bL_\phi^{<0}}\Grad(\phi)^\ast(p_\mY)_!(\eta_\mY)^\ast
    \end{align}
    It is compatible with composition, namely for $\phi':\mY\to\mZ$ smooth, the following square of isomorphisms is commutative:
    \[\begin{tikzcd}
        (p_\mX)_!(\eta_\mX)^\ast\phi^\ast(\phi')^\ast\arrow[r,"\simeq"]\arrow[d,"\simeq"] & \Sigma^{-\bL_\phi^{<0}}\Grad(\phi)^\ast(p_\mY)_!(\eta_\mY)^\ast(\phi')^\ast\arrow[r,"\simeq"] & \Sigma^{-\bL_\phi^{<0}}\Grad(\phi)^\ast\Sigma^{-\bL_{\phi'}^{<0}}\Grad(\phi')^\ast(p_\mZ)_!(\eta_\mZ)^\ast\arrow[d,"\simeq"]\\
        (p_\mX)_!(\eta_\mX)^\ast(\phi'\circ\phi)^\ast\arrow[rr,"\simeq"] & & \Sigma^{-\bL_{\phi'\circ\phi}^{<0}}(p_\mZ)_!(\eta_\mZ)^\ast
    \end{tikzcd}\]
    where the right vertical arrow comes from the exact triangle $\Grad(\phi)^\ast(\bL_{\phi'}^{<0})\to\bL_{\phi'\circ\phi}^{<0}\to \bL_\phi^{<0}$. It is also compatible with exterior tensor product, using the isomorphism of Lemma \ref{lemhyplocprod}.
\end{proposition}

\begin{proof}
    We consider the following commutative diagram:
    \[\begin{tikzcd}
    & \Filt(\mX)\arrow[dl,swap,"p_\mX"]\arrow[r,"\eta_\mX"]\arrow[d,"\pi"] & \mX\arrow[dd,"\phi"]\\
    \Grad(\mX)\arrow[d,"\Grad(\phi)"] & \Grad(\mX)\times_{\Grad(\mY)}\Filt(\mY)\arrow[l,swap,"pr_1"]\arrow[d,"pr_2"] & \\
    \Grad(\mY)&\Filt(\mY)\arrow[l,swap,"p_\mY"]\arrow[r,"\eta_\mY"] & \mY
\end{tikzcd}\]
    We consider first the isomorphism:
    \begin{align}\label{eqsimple}
        (\eta_\mX)^\ast\phi^\ast\simeq \Filt(\phi)^\ast(\eta_\mY)^\ast
    \end{align}
    Proposition \ref{propBB} gives that $\pi$ is an affine bundle stack, and Lemma \ref{lemmaffinesmooth} gives that it is smooth of cotangent bundle $\pi^\ast(pr_1)^\ast(\bL_\phi^{<0})$. Using Lemma \ref{lemhomotop}, $\pi_!\pi^!\to Id$ is an isomorphisms. We obtain then:
    \begin{align}\label{eqpipi}
        \pi_!\pi^\ast\simeq \pi_!\Sigma^{-\pi^\ast(pr_1)^\ast(\bL_\phi^{<0})}\pi^!\simeq\Sigma^{-(pr_1)^\ast(\bL_\phi^{<0}))} \pi_!\pi^!\simeq \Sigma^{-(pr_1)^\ast(\bL_\phi^{<0})}
    \end{align}
    where the first isomorphism is the purity isomorphism. With that, we obtain a 'base change isomorphisms':
    \begin{align}\label{exbasechange}
        (p_\mX)_!\Filt(\phi)^\ast\simeq (pr_1)_! \pi_!\pi^\ast(pr_2)^\ast\simeq  (pr_1)_!\Sigma^{-(pr_1)^\ast(\bL_\phi^{<0})}(pr_2)^\ast\simeq \Sigma^{-\bL_\phi^{<0}}(pr_1)_!(pr_2)^\ast\simeq \Sigma^{-\bL_\phi^{<0}}\Grad(\phi)^\ast (p_\mY)_!
    \end{align}
    which gives the claimed isomorphism.\medskip

     The isomorphism \eqref{eqsimple} is obviously compatible with composition, so we want to show that the isomorphism \eqref{exbasechange} is compatible with composition. We have the following commutative diagram with Cartesian square:
     \[\begin{tikzcd}
         & \Filt(\mX)\arrow[dl,swap,"\pi_{\phi'\circ\phi}"]\arrow[d,"\pi_\phi"]\\
         \Grad(\mX)\times_{\Grad(\mZ)}\Filt(\mZ)\arrow[d,"\Grad(\phi)\times Id"] & \Grad(\mX)\times_{\Grad(\mY)}\Filt(\mY)\arrow[l,swap,"\tilde{\pi}_{\phi'}"]\arrow[d,"pr_2"]\\
         \Grad(\mY)\times_{\Grad(\mZ)}\Filt(\mZ) & \Filt(\mY)\arrow[l,swap,"\pi_{\phi'}"]
     \end{tikzcd}\]
     From the compatibility of the counit with composition, and the compatibility of the purity isomorphism with composition (Lemma \ref{lemmapurcomp}), one has a commutative square of isomorphisms:
     \[\begin{tikzcd}
        (\tilde{\pi}_{\phi'})_!(\pi_\phi)_!(\pi_\phi)^\ast(\tilde{\pi}_{\phi'})^\ast\arrow[r,"\simeq"]\arrow[d,"\simeq"] & (\tilde{\pi}_{\phi'})_!\Sigma^{-(pr_1)^\ast(\bL_\phi^{<0})}(\tilde{\pi}_{\phi'})^\ast\arrow[r,"\simeq"] & \Sigma^{-(pr_1)^\ast(\bL_\phi^{<0})}\Sigma^{-(\Grad(\phi)\times Id)^\ast(pr_1)^\ast(\bL_{\phi'}^{<0})}\arrow[d,"\simeq"]\\
        (\pi_{\phi'\circ\phi})_!(\pi_{\phi'\circ\phi})^\ast\arrow[rr,"\simeq"] & & \Sigma^{-(pr_1)^\ast(\bL_{\phi'\circ\phi}^{<0})}
    \end{tikzcd}\]
    from the compatibility of the counit with base change, and the compatibility of the purity isomorphism with base change (Lemma \ref{lemmaspecpurity}), the upper right vertical arrow is obtained by base change from $(\pi_{\phi'})_!(\pi_{\phi'})^\ast\simeq \Sigma^{-(pr_1)^\ast(\bL_{\phi'}^{<0})}$. It follows then that \eqref{exbasechange} is compatible with composition.\medskip

    The statement about products follows directly, as all our constructions are monoidal.
\end{proof}

\begin{remark}
    In particular, when $\mX\to S$ is smooth (in this case, $\pi=p$ has always a section $\tilde{\iota}$), one obtains the " Bialynicki-Birula decomposition":
    \begin{align}\label{exdecompbraden}
        p_!\eta^\ast\un_\mX\simeq\Sigma^{-\bL_\mX^{<0}}\un_{\Grad(\mX)}
    \end{align}
    and, if $\mX=[X/\bG_m]$:
    \begin{align}
        (p^+)_!(\eta^+)^\ast\un_X\simeq\Sigma^{-(\bL_X|_{X^0})^{<0}}\un_{X^0}
    \end{align}
\end{remark}

\begin{lemma}\label{lemspecsmoothhyploc}
    Consider specialization system $\sps:\mD_1\to\mD_1$ over algebraic spaces $\eta\to B\leftarrow\sigma$. Consider a smooth morphism of stacks $\phi:\mX\to\mY$. Then the following square of isomorphisms is commutative:
    \[\begin{tikzcd}
        (p_\mX)_!(\eta_\mX)^\ast\phi^\ast\sps\arrow[r,"\simeq"]\arrow[d,"\simeq"] & \sps(p_\mX)_!(\eta_\mX)^\ast\phi^\ast\arrow[d,"\simeq"]\\
        \Sigma^{-\bL_\phi^{<0}}\Grad(\phi)^\ast(p_\mY)_!(\eta_\mY)^\ast\sps\arrow[r,"\simeq"] & \sps\Sigma^{-\bL_\phi^{<0}}\Grad(\phi)^\ast(p_\mY)_!(\eta_\mY)^\ast
    \end{tikzcd}\]
\end{lemma}

\begin{proof}
    Notice that, from Lemma \ref{basechangehyploc}, all the diagram considered over $\eta$ and $\sigma$ are obtain by base change from those considered over $B$. The morphism \eqref{eqsimple} is obviously compatible with specialization systems. Specialization systems are compatible with counit, and with the purity isomorphism (Lemma \ref{lemmaspecpurity}), hence the isomorphism \eqref{eqpipi} commutes with specialization systems. Using compatibility of specialization systems with base change, the isomorphism \eqref{exbasechange} commutes then also with specialization systems, hence we are done.
\end{proof}

\subsubsection{Functoriality with closed pullbacks}

This the following (iso)morphism will be useful to restrict vanishing cycles to the critical locus:

\begin{lemma}\label{lemhyploclosed}
    Given a closed immersion $i:\mZ\to\mX$, there is a natural morphism:
    \begin{align}
        \Grad(i)^\ast (p_\mX)_!(\eta_\mX)^\ast\to (p_\mZ)_!(\eta_\mZ)^\ast i^\ast
    \end{align}
    compatible with exterior products, composition of closed immersions, and with the isomorphism of Proposition \ref{propsmoothBB} under base change. It is an isomorphism when applied on objects of $\mD(\mX)$ supported on $\mZ$.
\end{lemma}

\begin{proof}
    Consider the following commutative diagram:
    \[\begin{tikzcd}
        & \Filt(\mZ)\arrow[dl,swap,"p_\mZ"]\arrow[d,"\hat{i}"]\arrow[r,"\eta_\mZ"] & \mZ\arrow[dd,"i"]\\
        \Grad(\mZ)\arrow[d,"\Grad(i)"] & \Grad(\mZ)\times_{\Grad(\mX)}\Filt(\mZ)\arrow[l,swap,"pr_1"]\arrow[d,"pr_2"] & \\
        \Grad(\mX) & \Filt(\mX)\arrow[l,swap,"p_\mU"]\arrow[r,"\eta_\mU"] & \mX
    \end{tikzcd}\]
    By \cite[Corollary 1.1.7, Proposition 1.3.1 2)]{HalpernLeistner2014OnTS}, because $i$ is closed we have $\Grad(\mR)=\Grad(\mU)\times_\mU\mR$ and $\Filt(\mR)=\Filt(\mU)\times_\mU\mR$. In particular the right square diagram is Cartesian, $\Filt(i)=pr_2\circ\hat{i}$ and $pr_2$ are closed immersions, hence $\hat{i}$ is a closed immersion. We consider then the following sequence of (iso)morphisms:
    \begin{align}
    \Grad(i)^\ast(p_\mX)_!(\eta_\mX)^\ast\simeq (pr_1)_!(pr_2)^\ast(\eta_\mX)^\ast \to(pr_1)_!\hat{i}_!\hat{i}^\ast(pr_2)^\ast(\eta_\mX)^\ast
    \simeq  (p_\mZ)_!(\eta_\mZ)^\ast i^\ast
    \end{align}
    where the second morphism comes from $Id\to\hat{i}_!\hat{i}^\ast$, as $\hat{i}$ is a closed immersion. If an object $F\in\mD(\mX)$ is supported on $\mZ$, by base change $(\eta_\mX)^\ast F$ is supported on $\Filt(\mZ)$, hence this morphism is an isomorphism.\medskip
    
    If we consider a second closed immersion $i':\mZ'\to\mZ$, we want to show that the following diagram is commutative:
    \[\begin{tikzcd}
        \Grad(i')^\ast\Grad(i)^\ast (p_\mX)_!(\eta_\mX)^\ast\arrow[r]\arrow[d,"\simeq"] & \Grad(i')^\ast(p_\mZ)_!(\eta_\mZ)^\ast i^\ast\arrow[r] & (p_{\mZ'})_!(\eta_{\mZ'})^\ast (i')^\ast i^\ast\arrow[d,"\simeq"]\\
        \Grad(i\circ i')^\ast(p_\mX)_!(\eta_\mX)^\ast\arrow[rr] && (p_{\mZ'})_!(\eta_{\mZ'})^\ast (i\circ i')^\ast
    \end{tikzcd}\]
    The proof is similar to the proof of the similar statement in Proposition \ref{propsmoothBB}, and we have to prove that $Id\to \hat{i}_!\hat{i}^\ast$ is compatible with composition of closed immersions and base change. From Lemma \ref{lemmaspecproper} and \ref{lemmapurcomp}, $\hat{i}_!\simeq \hat{i}_\ast$ is compatible with base change and composition, and $Id\to \hat{i}_\ast \hat{i}^\ast$ is also by properties of the adjunctions, hence $Id\to \hat{i}_!\hat{i}^\ast$ is compatible with base change and composition.\medskip

    Consider now $\phi:\mX'\to\mX$ smooth, $i:\mZ\to\mX$ a closed immersion, and $\phi':\mZ'\to\mZ$, $i':\mZ'\to\mX$ obtained by base change. We want to show that the following diagram is commutative:
    \[\begin{tikzcd}
    \Sigma^{-\bL_{\phi'}^{<0}}\Grad(\phi')^\ast\Grad(i)^\ast(p_\mX)_!(\eta_\mX)^\ast\arrow[d,"\simeq"]\arrow[r] & \Sigma^{-\bL_{\phi'}^{<0}}\Grad(\phi')^\ast(p_\mZ)_!(\eta_\mZ)^\ast i^\ast\arrow[r,"\simeq"] & (p_{\mZ'})_!(\eta_{\mZ'})^\ast(\phi')^\ast i^\ast\arrow[d,"\simeq"]\\
    \Grad(i')^\ast\Sigma^{-\bL_{\phi}^{<0}}\Grad(\phi)^\ast(p_\mX)_!(\eta_\mX)^\ast \arrow[r,"\simeq"] & \Grad(i')^\ast(p_{\mX'})_!(\eta_{\mX'})^\ast\phi^\ast\arrow[r] & (p_{\mZ'})_!(\eta_{\mZ'})^\ast(i')^\ast\phi^\ast
    \end{tikzcd}\]
    As said above, the morphism $Id\to\hat{i}_!\hat{i}^\ast$ and $(\pi_{\phi})_!(\pi_{\phi})^\ast\simeq \Sigma^{-(pr_1)^\ast(\bL_\phi^{<0})}$ are compatible with base change. From \ref{lemmaspecproper}, the isomorphism of specialization systems $(\pi_{\phi})_!(\pi_{\phi})^\ast\simeq \Sigma^{-(pr_1)^\ast(\bL_\phi^{<0})}$ commutes with $\hat{i}_!\simeq \hat{i}_\ast$, hence also with $Id\to \hat{i}_! \hat{i}^\ast$. We obtain then the desired result by base change in the following commutative diagram with Cartesian squares:
    \[\begin{tikzcd}[column sep=large]
        \Filt(\mZ)\arrow[d,"\hat{i}_i"] & \Grad(\mZ')\times_{\Grad(\mZ)}\Filt(\mZ)\arrow[d,"\tilde{\hat{i}}_i"]\arrow[l,swap,"pr_2"] & \Filt(\mZ')\arrow[l,swap,"\pi_{\phi'}"]\arrow[d,"\hat{i}_{i'}"]\\
        \Grad(\mZ)\times_{\Grad(\mX)}\Filt(\mX) & \Grad(\mZ')\times_{\Grad(\mX)}\Filt(\mX)\arrow[l,swap,"\Grad(\phi')\times Id"]\arrow[d,"\Grad(i')\times Id"] & \Grad(\mZ')\times_{\Grad(\mX')}\Filt(\mX')\arrow[l,swap,"\tilde{\pi}_\phi"]\arrow[d,"pr_2"]\\
        & \Grad(\mX')\times_{\Grad(\mX)}\Filt(\mX) & \Filt(\mX')\arrow[l,swap,"\pi_\phi"]
    \end{tikzcd}\] 
    \medskip
    
    The statement about products follows directly, as all our constructions are monoidal.  
\end{proof}

\subsubsection{Functoriality with $\bZ/2\bZ$-bundles}

Consider a principal $\bZ/2\bZ$-bundle $\pi:P\to\mX$ (\ie, a finite map of degree $2$). Consider the locally constant object:
\begin{align}
    \mathcal{L}_P:=\cofib(\un_\mX\to \pi_\ast\pi^\ast\un_\mX) \in\mD(\mX)
\end{align}
Following the notation of \cite[Definition 2.9, remark 2.20]{Joycesymstab}, we denote:
\begin{align}\label{eqtwistbun}
    -\otimes_{\bZ/2\bZ}P:=\cofib(Id\to \pi_\ast\pi^\ast)\simeq -\otimes\mathcal{L}_P:\mD(\mX)\to\mD(\mX)
\end{align}
As $\pi$ is étale and proper, by smooth and proper base change, this defines a morphism of coefficient systems $\mD|_\mX\to\mD|_\mX$, \ie there are natural isomorphisms of commutation with the six functors.\medskip

Given $\pi_i:P_i\to \mX$ (resp. $\pi_i:P_i\to \mX_i$), $i=1,2$, we denote by $P_1\otimes_{\bZ/2\bZ}P_2$ (resp. $P_1\boxtimes_{\bZ/2\bZ}P_2$) the (resp. exterior) product of $P,P'$ as $\bZ/2\bZ$-bundles. 

\begin{lemma}\label{lemtwistbun}
    The operation $P\to -\otimes_{\bZ/2\bZ}P$ enhance to a symmetric monoidal functor from $\bZ/2\bZ$-bundles over $\mX$ to involutive isomorphisms of coefficient systems $\mD|_\mX\to\mD|_\mX$. It is compatible with exterior tensor product, namely for $\pi_i:P_i\to \mX_i$ for $i=1,2$, there is a natural isomorphism of involution of coefficient systems:
    \begin{align}
        (-\otimes_{\bZ/2\bZ}P_1)\boxtimes (-\otimes_{\bZ/2\bZ}P_2)\simeq (-\boxtimes-)\otimes_{\bZ/2\bZ}(P_1\boxtimes_{\bZ/2\bZ}P_2)
    \end{align}
    which is symmetric monoidal.
\end{lemma}

\begin{proof}
   For the trivial $\bZ/2\bZ$-bundle $P_{triv}=\mX\times\bZ/2\bZ$, one has the exact triangle:
   \begin{align}
       Id\overset{(Id,Id)}{\to}\pi_\ast\pi^\ast=Id\oplus Id\overset{pr_1}{\to}Id
   \end{align}
   where one projects to the component corresponding to $1\in \bZ/2\bZ$, which gives a canonical identification $-\otimes_{\bZ/2\bZ}P_{triv}\simeq Id$. For $P,P'$ two $\bZ/2\bZ$-bundles, consider the sequence of $\bZ/2\bZ$-principal bundles:
   \begin{align}
       \tilde{\pi}:P\times_\mX P'\to P\times_{\bZ/2\bZ}P'\to \mX
   \end{align}
   where the first map is the trivial $\bZ/2\bZ$-bundle. Now, we have:
   \begin{align}
       -\otimes_{\bZ/2\bZ}P\otimes_{\bZ/2\bZ}P'&:=\cofib(Id\to (\pi')_\ast(\pi')^\ast)\circ\cofib(Id\to \pi_\ast\pi^\ast)\nn\\
       &\simeq \cofib(Id\to (\pi')_\ast(\pi')^\ast\pi_\ast\pi^\ast)\nn\\
       &\simeq \cofib(Id\to \tilde{\pi}_\ast\tilde{\pi}^\ast)
   \end{align}
   where the third line follows from base change (recalling that $\pi$ is étale and proper). Now, notice that $P\times_\mX P'=(P\times_{\bZ/2\bZ}P')\times_\mX P_{triv}$, hence we have a canonical isomorphism:
   \begin{align}
      - \otimes_{\bZ/2\bZ}P\otimes_{\bZ/2\bZ}P'\simeq -\otimes_{\bZ/2\bZ}(P\times_{\bZ/2\bZ}P')\otimes_{\bZ/2\bZ}P_{triv}\simeq -\otimes_{\bZ/2\bZ}(P\times_{\bZ/2\bZ}P')
   \end{align}
   which is symmetric monoidal, and commutes with the six operations by smooth and proper base change. The isomorphism of commutation with the exterior tensor product can be defined formally from the isomorphism of commutation with tensor product and pullbacks.
\end{proof}

\begin{lemma}\label{lemtwisthyploc}
    Consider a $\bZ/2\bZ$-principal bundle $P\to\mX$: $\Grad(P)\to\mX$ and $\Filt(P)\to\mX$ are $\bZ/2\bZ$-principal bundles, and there is a canonical isomorphism:
    \begin{align}
        p_!\eta^\ast(-\boxtimes_{\bZ/2\bZ}P)\simeq p_!\eta^\ast\boxtimes_{\bZ/2\bZ}\Grad(P)
    \end{align}
    commuting with the isomorphisms of commutation with exterior tensor product of Lemma \ref{lemhyplocprod} and the isomorphism of commutation with smooth pullbacks of Proposition \ref{propsmoothBB}.
\end{lemma}

\begin{proof}
As $P\to\mX$ is étale and representable, we have $\Grad(P)\simeq \Grad(\mX)\times_\mX P$ by \cite[Corollary 1.1.7]{HalpernLeistner2014OnTS} and $\Filt(P)\simeq \Filt(\mX)\times_{\Grad(\mX)}\Grad(P)$ by Proposition \ref{propBB}. Then $\Grad(P)\to\mX$ and $\Filt(P)\to\mX$ are $\bZ/2\bZ$-principal bundles obtained from $P\to\mX$ by base change. By Lemma \ref{lemtwistbun}, $-\boxtimes_{\bZ/2\bZ}P$ enhance to an involution of coefficient systems over $\mX$. We can applies this involution to the functor $p_!\eta^\ast$, which gives the claimed isomorphism, which is compatible with exterior tensor product from Lemma \ref{lemtwistbun}. Applying the involution of coefficient systems $-\boxtimes_{\bZ/2\bZ}P$ to the isomorphism of Proposition \ref{propsmoothBB}, we obtained the claimed compatibility.
\end{proof}

\section{Monodromic mixed Hodge modules and perverse Nori motives}\label{sectmono}

\subsection{Six functor formalisms for mixed Hodge modules and Perverse Nori motives}

In this section, our base will be an algebraically closed field $k$ of characteristic $0$, and all stacks will be algebraic $1$-stacks locally of finite type over $k$ (in particular, they are quasi-separated and locally of finite presentation over $k$).

\subsubsection{Perverse sheaves, mixed Hodge modules and motives}

In this section, we consider the category of separated schemes of finite type over $k$, an algebraically closed field of characteristic $0$. Given a triangulated category $\mD$ with a $t$-structure with heart $\mA$, we denote by $\mD^b$, $\mD^+$ and $\mD^-$ the triangulated subcategory of objects which are bounded, resp left bounded, resp right bounded, with respect to this $t$-structure.\medskip

On a scheme $X$, we consider the derived category $D(X,\bQ_\ell)$ of complexes of étale sheaves of $\bQ_\ell$-modules over $X$. Over $k$, a field of characteristic zero with a fixed embedding $\sigma:k\hookrightarrow\bC$, denoting by $X^{an}$ the analytification of its base change to $\bC$, we consider $D(X,\bQ)$ the derived category of sheaves of $\bQ$-modules for the analytic topology on $X^{an}$. We denote by $D_c(X,\bQ_\ell)$ (resp. $D_c(X,\bQ)$), the triangulated subcategory of $D(X,\bQ_\ell)$ (resp $D(X^{an},\bQ)$) of complexes with constructible homology, with constructibility considered with respect to an algebraic stratification. We consider the perverse $t$-structure (with middle perversity) built in \cite{BBDGfaisceauxpervers}, on $D_c(X,\bQ_\ell)$ (resp. $D_c(X,\bQ)$), with heart denoted by $\Perv(X,\bQ_\ell)$ (resp. $\Perv(X,\bQ)$), called the Abelian category of perverse sheaves. By the comparison theorem for étale cohomology (see \cite[Section 6.1.2]{BBDGfaisceauxpervers}), one has that the étale and analytic constructions are equivalent (and that this equivalence commutes with the six operations introduced below) when one take the coefficient ring to be a torsion ring on the two sides. These carry a six functor formalism (valued in small categories, instead of presentable one), with all the expected properties, built in \cite{Verdier1972ThorieDT} (this is in fact the first construction of six functor formalism, which inspired the general formalism presented above). \medskip

There are various refinements of constructible complexes, giving for $X$ for any sufficiently nice scheme over a certain base, a triangulated category $\mD_c(X)$ with six operations, a perverse $t$-structure with Abelian heart $\mA$, and a faithful exact (and then conservative) functor called the Betti realization $rat_B:\mD_c(X)\to D_c(X,\bQ)$ (in characteristic $0$) or the $\ell$-adic realization $rat_\ell:\mD_c(X)\to D_c(X,\bQ_\ell)$, commuting with the six functors. The idea is then that $\mD_c(Spec(k))$ contains more information on the cohomology on varieties than $D_c(\Spec(k),\Lambda)$, which is the category of complexes of $\Lambda$-modules, and that under $rat$ one keep only the information of the cohomology as a complex of vector space. We will denote by $^p\mathcal{H}^i$ the functors giving the homology with respect to the perverse $t$-structure.\medskip

Over $\bC$, a refinement of usual cohomology, using analytic geometry, is given by Deligne's theory of mixed Hodge structure. A relative version of that theory, with a full six functor formalism, is given by Saito's Abelian category of mixed Hodge module $MHM(X)$, built in \cite{Saito1988ModulesDH} and \cite{SaitoMHM}. One has a forgetful functor $rat:DMHM(X)\to D_c(X,\bQ)$, which is exact for the perverse $t$-structure and commutes with the six functors.\medskip

According to Grothendieck's philosophy of mixed motives, the universal version of cohomology theory must be given by Voedvosky's triangulated category of mixed motives. A relative version of that, with a six functors formalism, called the theory of constructible rational étale motives, denoted by $DM_c(X,\bQ)$, is constructed in Ayoub's thesis \cite{Ayoub2018LesSOI}, \cite{Ayoub2018LesSOII} and by Cisinki and Déglise in \cite{Cisinski2009TriangulatedCO}. It has a natural forgetful functor to $D_c(X,\bQ_\ell)$ (resp. $D_c(X,\bQ)$ over a field $k$ of characteristic zero with a fixed embedding $\sigma:k\hookrightarrow\bC$) commuting with the six functors called the $\ell$-adic (resp. Betti) realization, build in \cite{Ayoub2014LaR} (resp. \cite{Ayoub2009NoteSL}). The existence of a perverse $t$-structure on this triangulated category compatible with the perverse $t$-structure under the forgetful functor is a really difficult conjecture, equivalent to Grothendieck's standard conjectures (and in particular implying the Hodge conjecture).\medskip

A universal approximation of the perverse heart of such a perverse t-structure over a field over a field $k$ of characteristic zero with a fixed embedding $\sigma:k\hookrightarrow\bC$ is given by Nori's motives, and its relative version is provided by the Abelian category of perverse Nori motives $\mM_{perv}(X)$. A six functor formalism on $D\mM_{perv}$ has been built in \cite{ivorra:hal-03058402} and \cite{Terenzi2024TensorSO}. It inherits a Betti and $\ell$-adic realization commuting with the six operations from \cite{ivorra:hal-03058402}. For $X$ a scheme over $\bC$, one has from \cite[Theorem 0.1]{Tubach2023OnTN} a sequence of exact, faithful, perverse exact and conservative functors commuting with the six functors:
\[\begin{tikzcd}
    D\mM_{perv}(X)\arrow[r]\arrow[d] & DMHM(X)\arrow[r] & D_c(X,\bQ)\\
    D_c(X,\bQ_\ell) & &
\end{tikzcd}\] 

Notice that from \cite{Terenzi2024TensorSO}, the perverse Nori motives a posteriori don't depends of the embedding $\sigma:k\hookrightarrow\bC$ (but the Betti realization depends on it), hence one can, by taking limits over subfields that are finite extensions of $\bQ$, and then admits an embedding in $\bC$, define perverse Nori motives for arbitrary large fields $k$ of characteristic $0$. In this case, there will be no Betti realization, but we can still check that a morphism of constructible objects is an isomorphism at the Betti level using a limit argument. According to \cite[Theorem 5.7]{Tubach2023OnTN}, if there exists a perverse t-structure on rational étale motivic sheaves over a field $k$ of characteristic $0$, then the triangulated category of mixed motivic sheaves must correspond with the derived category of perverse Nori motives.\medskip

In this section, we use then the notation $\mD_c,\mA_c$ to denote constructible complexes, mixed Hodge modules and Perverse Nori motives, with their respective heart. In addition to having a perverse $t$-structure, these six-functor formalisms satisfies the following additional properties:
\begin{itemize}
    \item General six functor formalisms satisfies Nisnevich descent, and then also descent along smooth morphism with Nisnevich-local sections. $\mD_c$ satisfies furthermore étale descent, and then also descent along any smooth morphism. 
    \item Consider the Tate twists $\{d\}=[2d](d):\mD_c(X)\to\mD_c(X)$
\begin{align}
    \{d\}=[2d](d):=\left\{
    \begin{array}{ll}
        \Sigma^{\mathcal{O}^d_X}: & \mbox{if } d\geq 0 \in E \\
        \Sigma^{-\mathcal{O}^d_X} & \mbox{if } d<0
    \end{array}
    \right.
\end{align}
one has that $(d)$ is perverse exact. $\mD_c$ is oriented, which means that the Thom twist $\Sigma^{\mE}$ of a locally free sheaf of rank $d$ is canonically isomorphic with $\{d\}$. We slightly abuse the notation, by using also the notation $\{d\}$ for $d$ a locally constant function on $X$. The purity isomorphism gives then, for $f$ smooth of relative dimension $d_f$, a natural isomorphism $f^!\simeq f^\ast\{d_f\}$. Moreover, $f^\ast[d_f]$ is perverse-exact. Notice that we slightly abuse the notation here: indeed, $d_f$ is defined on the domain of $f$, so we use the shorthand notation $f^\ast\{d_f\}$ (resp. $f^\ast[d_f]$) for $\{d_f\}\circ f^\ast$ (resp $[d_f]\circ f^\ast$).
\item Moreover, these affords a formalism of Verdier duality (such a formalism was developed for arbitrary motivic coefficient system in \cite[Section 2.3.10]{Ayoub2018LesSOI}, using constructibility assumptions). Namely, there is a contravariant involution $\bD_X:\bD(X)^{op}\to\bD(X)$, commuting with the exterior tensor product, and exchanging $!$ and $\ast$ functors. Namely, for $f:X\to Y$, there are natural equivalences $\bD_X f^!\simeq f^\ast\bD_Y$ and $\bD_X f_\ast\simeq f_!\bD_Y$, compatible with composition and base change.
\end{itemize}

\subsubsection{The formalism of weights}

One of the main advantages of the refinements of the derived category of constructible complexes considered here is the formalism of weights, inspired by the formalism of weights for mixed $\ell$-adic sheaves. Namely, for mixed Hodge modules and perverse Nori motives, one has for each $w\in \bZ$ a full semisimple subcategory of the perverse heart, called the category of pure objects of weight $w$, and general objects $F$ of the perverse heart has a canonical increasing weight filtration $W_\bullet F$ such that the $Gr^W_wF$ is pure of weight $w$, and the morphisms are strict for the weight filtration.\medskip

For mixed $\ell$-adic perverse sheaves, purity is defined in \cite{BBDGfaisceauxpervers} by considering eigenvalues of the Frobenius, and the weight filtration is built in the definition. For mixed Hodge modules, pure objects are the polarizable Hodge modules defined in \cite{Saito1988ModulesDH}, and the weight filtration is built in the definition in \cite{SaitoMHM}. A weight structure on perverse Nori motives is built by Ivorra and Morel in \cite{ivorra:hal-03058402} using the realization to $\ell$-adic sheaves and the Bondarko weight structure on étale motivic sheaves built by Bondarko \cite{Bondarko2010WeightsFR} and Hébert \cite{Hebert2010StructureDP}. By construction, the Bondarko weight structure on étale motivic sheaves is compatible with the Saito's weight structure on MHM under the Hodge realization, hence Ivorra-Morel weight structure on perverse Nori motives is also compatible with Saito's one under the Hodge realization.\medskip

These weight structure extends to the derived category $\mD_c(X)$. Namely, one consider $\mD_{c,\leq w}(X)$ (resp $\mD_{c,\geq w}(X)$, $\mD_{c,w}(X)$), the subcategories of objects $F\in\mD_c(X)$ such that, for each $i\in\bZ$, $^p\mathcal{H}^i(F)$ has weights less than or equal to $w+i$ (resp more than or equal to $w+i$, resp are pure of weight $w+i$). In all of these cases the category $\mD_{c,w}(X)$ of pure objects are semi-simple, and these define weight structure in the sense of Bondarko \cite{Bondarko2007WeightSV}.\medskip

The six functor formalism interacts well with weights. Namely, for $f:X\to Y$ a morphism, one obtains for each $w\in\bZ$:
\begin{align}
    f^\ast&:\mD_{c,\leq w}(Y)\to \mD_{c,\leq w}(X)\\
    f^!&:\mD_{c,\geq w}(Y)\to \mD_{c,\geq w}(X)\\
    f_\ast&:\mD_{c,\geq w}(X)\to \mD_{c,\geq w}(Y)\\
    f_!&:\mD_{c,\leq w}(X)\to \mD_{c,\leq w}(Y)
\end{align}
moreover, weights are additive under the exterior tensor product, and $\bD $ exchange $\mD_{c,\leq w}$ and $\mD_{c,\geq - w}$. Moreover, the Tate twist $(1)$ shifts the weight by $-2$ (such that $\{1\}=[2](1)$ preserves $\mD_{c,w}(X)$).\medskip

In particular, for $f:X\to Y$ proper, one obtains an analogue of Beilinson-Bernstein-Deligne-Gabber decomposition theorem from \cite{BBDGfaisceauxpervers}. Namely, for $F$ pure of weight $w$ in $\mA(X)$, $f_\ast F$ is pure of weight $w$, and then one has a direct sum decomposition 
\begin{align}
    f_\ast F=\bigoplus_{i\in\bZ} ^p\mathcal{H}^if_\ast F[-i]=\bigoplus_j G_j[-n_j]
\end{align}
with $G_j$ being simple objects, pure of weight $w+n_j$ in the perverse heart.\medskip

Denoting $p:X\to pt$ the projection to a point, we consider the refining of the cohomology with compact support of $X$:
\begin{align}
    H^i_c(X)=^p\mathcal{H}^i p_!p^\ast \un_{pt}\in \mA
\end{align}
one find that if $X$ is smooth and proper, then $H^i_c(X)$ has weights $i$, which corresponds to Grothendieck's idea that smooth and proper varieties corresponds to pure motives.\medskip

\subsubsection{Extension to stacks}\label{sectextstack}

In this section, all our stacks are assumed to be locally of finite type, over an algebraically closed field $k$ of characteristic $0$.\medskip

The extension of six functor formalism for constructible complexes to stacks was done in \cite{Liu2012EnhancedSO}, \cite{Liu2012EnhancedSO}, and for mixed Hodge modules and perverse Nori motives in \cite{tubachMHM}. In \cite{tubachMHM}, Tubach gives moreover the compatibility of this construction to other approaches to defining mixed Hodge modules on stacks, mainly in the quotient case. Two main technical issues must be taken into account to applies the general descent technique as discussed in Section \ref{sectionschemestacks}:

\begin{itemize}
    \item $\mD_c$ is traditionally provided with a $2$-categorical version of a six functor formalism, instead of an $\infty$-categorical one. But, as discussed in Section \ref{sectionschemestacks}, one must really have an $\infty$-categorical enhancement to do descent at the level of the derived category. Constructible complexes are defined as étale (or analytic) sheaves, hence the operations on them admits naturally an $\infty$-categorical enhancement, constructed in \cite{Liu2012EnhancedSO}. However, the construction of mixed Hodge modules and their six functors by Saito in \cite{Saito1988ModulesDH} and \cite{SaitoMHM} is really intricate, hence defining them in a homotopy coherent way could seem hopeless. Tubach has shown in \cite{Tubach2023OnTN} that the six functors on mixed Hodge modules and perverse Nori motives admits a natural $\infty$-categorical enhancement, which was the main missing piece. The main idea was to consider an other $t$-structure on $D^bMHM$ and $D^b\mM_{perv}$, send to the classical (not perverse) $t$-structure of $D^b_c(-,\bQ)$, and to show that $D^bMHM$ and $D^b\mM_{perv}$ are the derived category of this heart. The functors are then half exact with respect to this new heart, hence where the derived functor of functors of half-exact functors of Abelian categories, which admits naturally an $\infty$-categorical enhancement.
    \item $\mD_c$ is valued in small categories, instead of presentable categories, hence arbitrary colimits do not exists, which prevents the use of descent. One need then to embed $\mD_c$ into a presentable categories: one consider the category of Ind-objects of $\mD^b_c$, and applies Liu-Zengh descent to it. One define then $\mD^b_c(\mX)$ to be the full subcategory of objects whose pullbacks to covering scheme $X$ are in $\mD^b_c(X)$. Recall that, for $f$ smooth, $f^\ast[d_f]$ is perverse exact, which allows to define by smooth descent a perverse $t$-structure on the big category. Now, $\mD_c(\mX)$ is defined to be the full subcategory whose objects have cohomology in $\mD^b_c(X)$ with respect to this $t$-structure. One obtains then a stable $(\infty,1)$-category $\mD_c(\mX)$, with a perverse $t$ srtructure with heart $\mA_c(\mX)$, and subcategories of (left-, right-) bounded objects $\mD^b_c(\mX)$ ($\mD^+_c(\mX)$, $\mD^-_c(\mX)$). 
\end{itemize}

The main result of the construction is given in \cite[Proposition 6.4.4, 6.4.5]{Liu2012EnhancedSO} and \cite[Theorem 3.8]{tubachMHM}. One has the following subtleties:

\begin{itemize}
    \item $\mD_c(X)$ is a triangulated category with a perverse $t$-structure which is in general not the derived category of its Abelian perverse heart (for schemes, it was the case by definition for mixed Hodge modules and perverse Nori motives, and this was proven for constructible complexes in \cite{Beilinson1987}).
    \item By descent, for any $f:\mX\to\mY$, $f^\ast, f^!$ preserves $\mD^?_c$. As usual when one deals with constructible objects, only pushforward by morphisms of finite type can preserve them. If $f$ is representable by algebraic spaces and of finite type, $f_\ast$ and $f_!$ preserves $\mD^?_c$, by descent from the scheme case. If $f$ is of finite type (possibly not representable), one need to use a spectral sequence argument, and one obtains only:
    \begin{align}
        f_\ast:\mD^+_c(\mX)\to \mD^+_c(\mY )\nn\\
        f_!:\mD^-_c(\mX)\to \mD^-_c(\mY )
    \end{align}
    for example, equivariant (resp. compact equivariant) cohomology can be unbounded in the positive (resp negative) direction. 
    \item The orientability gives that Thom twists by any perfect complexes in amplitude $[0,1]$ are naturally equivalent with Tate  twists by their virtual dimension. In particular, for $f$ smooth of virtual relative dimension $d_f$, the purity gives a natural isomorphism $f^!\simeq f^\ast\{d_f\}$.
    \item From \cite[Section 3.2]{tubachMHM}, the Verdier duality operation extends to stacks, fix $\mD^b_c$ and exchanges $\mD^+_c$ with $\mD^-_c$, commutes with $\boxtimes$, and exchanges $!$ and $\ast$ functors, as expected.
    \item In the mixed Hodge module and perverse Nori motives case, from \cite[Section 3.3]{tubachMHM}, there is a weight structure in the sense of Bondarko which is defined on $\mD_c(\mX)$ for $\mX$ an algebraic stack with affine stabilizers. In this case, the functors $f^\ast,f^!,f_\ast,f_!,\bD,\boxtimes$ have the expected behavior with respect to weights. Notice that one can define naturally a notion of weights for any algebraic stacks, but this will in general not provide a weight structure in the sense of Bondarko (namely, pure objects will not be always semisimple), and pushforward from stacks with stabilizers which are not affine will not always have the good functoriality with weights. See the example of $BE$, with $E$ an elliptic curve group, studied in the mixed $\ell$-adic context in \cite{sunDecomp}. In this paper, we will consider only stacks with affine stabilizers (as only those appear in generalized GIT problem), hence we will not have to deal with those subtleties.
\end{itemize}

Braden-Drinfeld-Gaitsgory theorem on stacks \ref{theobradenstacks} has then the following consequence, in the presence of a weight structure:

\begin{corollary}
    Let $\mX$ be an algebraic stack, locally of finite type over $k$, with affine stabilizers. If $\mD_c(-)$ denotes the extension of mixed Hodge modules or perverse Nori motives to stacks, then the hyperbolic localization functor:
    \begin{align}
        p_!\eta^\ast\mD^b_c(\mX)\to\mD^b_c(\Grad(\mX))
    \end{align}
    is weight exact, \ie preserves pure objects.
\end{corollary}

\begin{proof}
    Theorem \ref{theobradenstacks} gives an isomorphism between $p_!\eta^\ast$, which preserves $\mD^+_{c,\geq w}$ and $(r\circ p)_\ast\eta^!$, which preserves $\mD^-_{c,\leq w}$. 
\end{proof}

This corollary was the main motivation of Braden in \cite{Braden2002HyperbolicLO} to establish his main result, in the case of schemes with $\bG_m$-action. It allows to interpret \eqref{exdecompbraden} as a decomposition theorem.

\subsection{Vanishing cycles and monodromic objects}\label{sectmonothomseb}

\subsubsection{Vanishing cycles}

Consider $\bC^\ast=\bC-\{0\}$, its universal cover $\tilde{\bC^\ast}$ (notice that this universal cover is the exponential map, which is not algebraic, hence one really use analytic geometry here!). For any $\bC$-algebraic space $X$, and regular map $X\to\bC$, considering its analytification $X^{an}$, one consider the Cartesian diagram obtained by base change:
\[\begin{tikzcd}
\tilde{X^{an}_\eta}\arrow[r,"\tilde{\theta}"]\arrow[d] & X^{an}_\eta\arrow[r,"j"]\arrow[d] & X^{an}\arrow[d,"f"] & X^{an}_0\arrow[l,"i"]\arrow[d]\\
    \tilde{\bC^\ast}\arrow[r]& \bC^\ast\arrow[r] & \bC & \{0\}\arrow[l]
\end{tikzcd}
\]
A deep result of \cite{BBDGfaisceauxpervers} is that the functor:
\begin{align}
i^\ast j_\ast\tilde{\theta}_\ast\tilde{\theta}^\ast:D(X^{an}_\eta,\bQ)\to D(X_0^{an},\bQ)
\end{align}
restricts to a functor $\psi_f^{an}:D^b_c(X_\eta,\bQ)\to D^b_c(X_0,\bQ)$, exact for the perverse $t$-structure up to a shift by $[1]$, called the nearby cycle functors. The problem in this definition is that the universal cover of $\tilde{\bC^\ast}$ is given by the exponential map, which is not algebraic, hence one has to go into the analytic world to do this definition. In the étale world, \ie for $D(X,\bQ_\ell)$, nearby cycles are defined in \cite[Exposé XIII]{SGAVII}: namely, one consider the spectrum $S$ of the strict Henselianization of $\bA^1$ at $0$ (a kind of local ring for the étale topology), with one generic point $\eta$ (the étale local version of $\bG_m$) and one point $s$. $\eta$ is now the spectrum of a field, and has then an universal cover $\bar{\eta}$, given by the spectrum of a separable closure: one can then copy the above definition, using the base change $X_{\bar{\eta}}$ instead of $\tilde{X^{an}_\eta}$. The comparison theorem of \cite[Exposé XIV]{SGAVII} shows that the étale and analytic definitions agrees for a torsion coefficient ring. However, this definition does not work using the coefficient ring $\bQ$, as explained in \cite[Page 12]{AyoubTHEMN}.\medskip

A definition of nearby cycles, which works for any coefficient system, and admits a direct extension to algebraic stacks, is provided by the definition of Ayoub given in \cite[Section 3.5]{Ayoub2018LesSOII}, and was used in \cite{ivorra:hal-03058402} to define nearby cycles for perverse Nori motives on schemes. An equivalent definition, more conceptual and easier to use in $(\infty,1)$ categorical setting, was introduced in \cite{Cass2024CentralMO}, which was shown to be equivalent to the topological definition for schemes, is used by Tubach in \cite[Section 3.4]{tubachMHM} to define the nearby cycles for mixed Hodge modules on stacks, and Tubach shows in \cite[Proposition 3.33]{tubachMHM} that this definition is equivalent to Saito's definition from \cite{SaitoMHM} in the scheme case (up to a shift, because Saito's functor is perverse exact). Notice that these definitions gives a specialization system $\psi_f:\mD_c(\mX_\eta)\to\mD_c(\mX_0)$ exact for the perverse heart up to a shift by $[1]$ by reducing to the scheme case using a smooth presentation and by using the Betti realization .\medskip

There is a natural morphism of specialization system $i^\ast\to \psi_f j^\ast$: one want to define $\phi_f$ as the cone of this functor, but the cone construction is not canonical in a triangulated category. Fortunately, using the fact that these constructions have an $(\infty,1)$-categorical enhancement, one can work at the level of stable $(\infty,1)$ categories, where one has a functor $\Cofib()$, sending a morphism to a cone of it in a functorial way so we can define canonically $\phi_f:=\Cofib(i^\ast\to \psi_f)$. They fits then in a natural cofiber sequence of specialization systems:
\begin{align}\label{trianglenearby}
    i^\ast\to \psi_f\to\phi_f\to i^\ast[1]
\end{align}
one obtains by reducing to the scheme case and the Betti or $\ell$-adic realization that $\phi_f:\mD^b_c(\mX)\to\mD^b_c(\mX_0)$ is perverse exact up to a shift by $[1]$.

\subsubsection{Six functors for monodromic objects}

In \cite[Section 2.9]{Joycesymstab}, \cite[Section 2.1]{DavMein}, the authors works with monodromic mixed Hodge modules over schemes, as suggested by \cite[Section 7.4]{KonSol10}. The main reason for that is that the Thom-Sebastiani theorem from \cite{Sebastiani1971UnRS}, \cite{MasThomSeb}, giving the compatibility of vanishing cycles with product in the setting of analytic sheaves, does not extends directly to the Hodge, $\ell$-adic or motivic level, as was observed by Deligne in \cite{Deligneconf}. The problem arise in the computation of vanishing cycles of quadratic functions in \cite[Exposé XV]{SGAVII}: the vanishing cycle of $x^2+y^2:\bA^2\to\bA^1$ have weight $2$, which would suggest that the vanishing cycles of $x^2$ would be one dimensional with weight $1$, which would is impossible! Deligne has suggested to avoid this by considering the vanishing cycles with the data of the monodromy as objects on $X_0\times \bA^1_k$, extended by $0$ from $X_0\times \bG_{m,k}$, where the $\bG_{m,k}$ factor encodes the monodromy. Then Deligne's idea was that Thom-Sebastiani theorem must holds by replacing the product by the convolution product:
\begin{align}
    \boxtimes_{mon}:=\mu_\ast(-\boxtimes-)
\end{align}
where $\mu:\bA^2_k\to\bA^1_k$ denotes the sum. The proof in the $\ell$-adic setting was written down (30 years later!) by Illusie in \cite{Illusie2017AroundTT}, and was given in the Hodge setting by Saito in the preprint \cite{Saitothomseb}.\medskip

In this paper, we need an extension of this theory to Artin stacks, and also for perverse Nori motives. Moreover, we need several compatibility results that are difficult to find in the literature. For this reason, we will detail this construction. In \cite{descombesThomSeb}, we have developed a formalism of monodromic objects on stacks, with six operations, and proven a Thom-Sebastiani theorem, for more general coefficient system. Here, things will be easier, because we can use the conservativity of the Betti realization to show that the morphisms that we define are isomorphisms, so we will be more sketchy.\medskip

We first explain how to define the cofficient system $\mD_{mon}$. We begin with the coefficient system $\mD$ (the Ind-completion of the coefficient system $\mD^b_c$ we are interested in). We apply the dual of the construction of \cite{Gallauer2022ExponentiationOC} to obtain a coefficient system $\mD_{\exp,\ast}$, which give using \ref{sectextstack} a six functor formalism of stacks, such that:
\begin{itemize}
    \item $\mD_{\exp,\ast}(\mX)$ is the full $(\infty,1)$-stable subcategory of $\mD(\mX\times\bA^1)$ whose objects $F$ satisfies $\pi_\ast F\simeq0$, where $\pi:\mX\times\bA^1\to\mX$ is the first projection. The inclusion functor $\mD_{\exp,\ast}(\mX)\to \mD(\mX\times\bA^1)$ admits a left adjoint $\Pi_\ast:\mD(\mX\times\bA^1)\to \mD_{\exp,\ast}(\mX)$, given by the formula $\Pi_\ast:=\Cofib(\pi^\ast\pi_\ast\to Id)$. Notice that, by base change, the morphism $\pi_\ast\pi^\ast\to Id$ is identified with $\mu_\ast(-\boxtimes \un_{\bA^1})\to\mu_\ast(-\boxtimes 0_!\un_{0})$, \ie, considering the open immersion $j:\mX\times\bG_{m,k}\to\mX\times\bA^1_k$ and the localization triangle $j_!j^\ast\to Id\to i_!i^\ast$, we can define as usual $\Pi_\ast:=\mu_\ast(-\boxtimes j_!\un_{\bG_m}[1])$.
    \item Given $f:\mX\to\mY$, $f^!:\mD_{\exp,\ast}(\mY)\to \mD_{\exp,\ast}(\mX)$ (resp. $f_\ast:\mD_{\exp,\ast}(\mX)\to \mD_{\exp,\ast}(\mY)$) is obtained by restricting $(f\times Id_{\bA^1})^!:\mD(\mY\times\bA^1)\to \mD(\mX\times\bA^1)$ (resp. $(f\times Id_{\bA^1})_\ast:\mD(\mX\times\bA^1)\to \mD(\mY\times\bA^1)$). $f^\ast:\mD_{\exp,\ast}(\mY)\to \mD_{\exp,\ast}(\mX)$ (resp. $f_!:\mD_{\exp,\ast}(\mX)\to \mD_{\exp,\ast}(\mY)$) is obtained by restricting $\Pi_\ast (f\times Id_{\bA^1})^\ast:\mD(\mY\times\bA^1)\to \mD(\mX\times\bA^1)$ (resp. $\Pi_\ast (f\times Id_{\bA^1})_!:\mD(\mX\times\bA^1)\to \mD(\mY\times\bA^1)$).
    \item The exterior tensor product $\boxtimes_{exp}:\mD_{\exp,\ast}(\mX)\times\mD_{\exp,\ast}(\mY)\mD_{\exp,\ast}(\mY)\to \mD_{\exp,\ast}(\mX\times\mY)$ is obtained by restriction from:
    \begin{align}
        \mu_\ast(-\boxtimes-):\mD(\mX\times\bA^1)\times \mD(\mY\times\bA^1)\to \mD(\mX\times\mY\times\bA^1)
    \end{align}
    (one uses here $\pi_\ast\mu_\ast(-\boxtimes-)\simeq \pi_\ast\boxtimes\pi_\ast$).
\end{itemize}

More precisely, in \cite{Gallauer2022ExponentiationOC}, the authors define an operation called exponentiation, sending a coefficient system $\mD$ to a coefficient system $\mD_{exp,!}$ with the above properties, but where $!$ and $\ast$ functors, and the direction of the arrows are switched. We use here a trick from \cite[Page 211]{Ayoub2018LesSOI}: from the six functor formalism $\mD^\ast_!$, we obtain by passing to the adjoints a functor $(\mD^!)^{op}$, which satisfies all the axioms of a coefficient system. Then, we apply the following sequence of constructions:
\begin{align}
    \mD\to (\mD^!)^{op}\to ((\mD^!)^{op})_{exp,!}\to  ((((\mD^!)^{op})_{exp,!})^!)^{op}=:\mD_{exp,\ast}
\end{align}
where the middle row is the construction of \cite{Gallauer2022ExponentiationOC}.\medskip

We consider then the full sub $(\infty,1)$-categories $\mD_{mon}(\mX)$ of $\mD_{exp,!}(\mX)$ whose objects are $F\in \mD(\mX\times\bA^1)$ such that $\pi_\ast F=0$ which are moreover monodromic, \ie such that for each $x\in\mX(k)$, $F|_{\{x\}\times\bG_{m,k}}$ has Betti realization whose cohomology is locally constant, and we denote by $\mD^?_{mon,c}(\mX):=\mD_{mon}(\mX)\cap \mD^?_c(\mX\times\bA^1)$ the full subcategory of constructible objects. From the definition, using base change, for any $f$, $f^\ast,f_!$ preserve monodromic objects. It is standard that $\bD$ and $\mu_\ast(-\boxtimes-)$ preserve monodromic objects in the scheme case, hence the same follows in the stack case by smooth descent: in particular, $f^!,f_\ast$ also preserve monodromic objects. As the image of $\pi^\ast$ is obviously made of monodromic objects, $\Pi_\ast:=\Cofib(\pi^\ast\pi_\ast\to Id)$ obviously preserves monodromic objects. Hence the six operations on $\mD_{exp,\ast}$ preserves $\mD_{mon}$, such that we obtain by restriction a six functor formalism $\mD_{mon}$: we denote the restriction of $\boxtimes_{exp}$ by $\boxtimes_{mon}$, or simply $\boxtimes$ when the fact that we work with monodromic objects is clear from the context.\medskip

From \cite[Lemme 6.1, Section 9]{Verdiespecial}, for schemes in the Betti case, the natural morphism $\pi_\ast\to \pi_\ast0_\ast 0^\ast\simeq 0^\ast$ is an isomorphism when applied to monodromic objects, hence the same applies for stacks and $\mD$ by conservativity of the Betti realization and smooth pullbacks. Then $\mD_{mon}(\mX)$ is also the full subcategory of objects $F\in \mD(\mX\times\bA^1)$ who are monodromic and such that $0^\ast F\simeq 0$. In particular, the natural morphisms $f^\ast\to \Pi_\ast f^\ast$ and $f_!\to \Pi_\ast f_!$ are isomorphisms when applied to objects of $\mD_{mon}$, hence the four functors $f^\ast,f^!,f_\ast,f_!$ on $\mD_{mon}$ are obtained by restricting $(f\times Id_{\bA^1})^\ast,(f\times Id_{\bA^1})^!,(f\times Id_{\bA^1})_\ast,(f\times Id_{\bA^1})_!)$ on $\mD(-\times\bA^1)$.\medskip

By definition, the perverse $t$-structure on $\mD(\mX\times\bA^1)$ restricts to the full subcategory of monodromic objects, and, on this $\Pi_\ast$ preserve the perverse heart (it is standard for schemes in the Betti case, hence it follows for stacks and $\mD$ using smooth pullbacks and the Betti realization), hence this perverse $t$-structure restricts to a perverse $t$-structure on $\mD_{mon}$ (notice that this is false for $\mD_{exp,\ast}$). Then $\boxtimes_{mon}$ is perverse exact (it is standard for schemes in the Betti case, hence it follows for stacks and $\mD$ using smooth pullbacks and the Betti realization), and the four functors have the usual properties with respect to the perverse heart (in particular, for $f$ smooth of relative dimension $d_f$, $f^\ast[d_f]$ is perverse exact).\medskip

Denoting $q:\mX\times\bG_m\to\mX$, the functor $\Pi_\ast 0_\ast\simeq j_!q^\ast[1]:\mD\to\mD_{mon}$ is fully faithful, perverse-exact and commutes with the six functors: it associates an object to the corresponding monodromic object with trivial monodromy. Its quasi-inverse $1^\ast[-1]:\mD_{mon}\to\mD$, the operation of forgetting the monodromy, is perverse exact and faithful (indeed, it is known to be perverse exact and faithful for schemes in the Betti case, hence the same follows for stacks and $\mD$ by perverse-exactness faithfullness of smooth pullback and the Betti realization), it commutes with the operations $f^\ast,f^!,f_\ast,f_!$, but not with $\boxtimes$ (otherwise, the non-monodromic version of Thom-Sebastiani would be true).\medskip

\subsubsection{Thom-Sebastiani theorem}

Following \cite{KonSol10}, for $f:\mX\to\bA^1_k$, we consider $f/t:\mX\times\bG_{m,k}\to\bA^1_k$, with $t$ the coordinate of $\bG_{m,k}$, and define the monodromic vanishing cycles and total monodromic vanishing cycles to be:
\begin{align}
\phi_f^{mon}&:=(j_0)_!\phi_{f/t}q^\ast:\mD^b_c(\mX)\to \mD^b_c(\mX_0\times\bA^1)\nn\\
\phi_f^{mon,tot}&:=\bigoplus_{c\in k}(i_c)_\ast\phi_{f-c}^{mon}:\mD^b_c(\mX)\to \mD^b_c(\mX\times\bA^1)
\end{align}
These are perverse exact, because $\phi_f[-1]$, $q^\ast[1]$, $j_!$ and $(i_c)_!$ are perverse exact. As $\phi_f$ enhance to a specialization system $\mD^b_c\to\mD^b_c$ over $\bA^1\to \bA^1\leftarrow\{0\}$, $\phi_f^{mon}$ (resp. $\phi_f^{mon,tot}$) enhance to a specialization $\mD^b_c\to\mD^b_c$ over $\bA^1\to \bA^1\overset{0}{\leftarrow}\bA^1$ (resp. $\bA^1\to \bA^1\overset{pr_1}{\leftarrow}\bA^1\times\bA^1$. The Betti realization is monodromic in the scheme case from \cite[Proposition 7.1, Section 9]{Verdiespecial}, hence the same follows for stacks by smooth pullbacks. Then, as the four operations $f^\ast,f^!,f_\ast,f_!$ on $\mD_{mon}$ are induced by the four operations for $(f\times Id_{\bA^1})^\ast,(f\times Id_{\bA^1})^!,(f\times Id_{\bA^1})_\ast,(f\times Id_{\bA^1})_!)$ on $\mD(-\times\bA^1)$, those restricts to perverse-exact specialization systems:
\begin{align}
\phi_f^{mon}&:\mD^b_c(\mX)\to \mD^b_{mon,c}(\mX_0)\nn\\
\phi_f^{mon,tot}&:\mD^b_c(\mX)\to \mD^b_{mon,c}(\mX)
\end{align}
over $\bA^1\to \bA^1\leftarrow\{0\}$ (resp. $\bA^1\to \bA^1\leftarrow\bA^1$). We recall that this gives, for $g:\mY\to\mX$, natural morphisms:
\begin{align}\label{specsys1}
    g^\ast \phi_f^{mon,tot}\to \phi_{f\circ g}^{mon,tot}g^\ast\nn\\
    \phi_{f\circ g}^{mon,tot}g^!\to g^! \phi_f^{mon,tot}
\end{align}
which are inverse isomorphisms up to a Tate twist if $g$ is smooth, and:
\begin{align}\label{specsys2}
    g_! \phi_{f\circ g}^{mon,tot}\to \phi_f^{mon,tot}g_!\nn\\
    \phi_f^{mon,tot}g_\ast\to g_\ast \phi_{f\circ g}^{mon,tot}
\end{align}
which are inverse isomorphisms if $g$ is proper, with compatibility with base change and composition (and similarly for $\phi^{mon}$). Notice that there is a natural morphism of specialization system:
\begin{align}
    1^\ast[-1]\phi^{mon}_f:=1^\ast[-1](j_0)_!\phi_{f/t}q^\ast\simeq 1^\ast\phi_{f/t}q^\ast[-1]\to \phi_f 1^\ast q^\ast[-1]\simeq \phi_f [-1]
\end{align}
it is an isomorphism for schemes in the Betti case from \cite[Proposition 7.1, Section 9]{Verdiespecial}, hence it is also an isomorphism for stacks and $\mD$ by conservativity. Hence the monodromic vanishing cycle is send to the classical (perverse) vanishing cycle under the forgetful functor.\medskip

With these definitions in hand, we can extend the construction of Verdier \cite{Verdiespecial}, Deligne and Illusie \cite{Illusie2017AroundTT}, and Saito \cite{Saitothomseb} to stacks and perverse Nori motives:

\begin{theorem}\label{theothomseb}
    For $k$ an algebraically  closed fields of characteristic $0$, $f:\mX\to\bA^1_k$, $g:\mY\to\bA^1_k$ morphisms from locally of finite type Artin $1$-stacks over $k$, there is a natural isomorphism of specialization systems:
    \begin{align}
    &\phi^{mon,tot}_f\boxtimes_{mon}\phi^{mon,tot}_g\simeq \phi^{mon,tot}_{f\boxplus g}(-\boxtimes -):\mD_c(\mX)\times \mD_c(\mY)\to \mD_{mon,c}(\mX\times \mY)
    \end{align}
   (in particular, it is compatible with the morphisms \eqref{specsys1} and \eqref{specsys2}). These isomorphisms satisfies the obvious commutativity and associativity property.
\end{theorem}

\begin{proof}
    We begin by introducing Verdier's specialization functor, following \cite{Verdiespecial}. If $i:\mX'\to \mX$ is a closed immersion, we consider the normal cone $C_{\mX'} \mX$ over $\mX'$ and the deformation to the normal cone $D_{\mX'} \mX$. This construction works also for stacks, see \cite{intrnormcone}. More precisely, if we denote by $\mathcal{I}$ the sheaf of ideals defining $\mX'$, we have by definition that $D_{\mX'}\mX\to \mX$ is the relative spectrum of the $\mathcal{O}_X$-module $\bigoplus_{n\in\bZ}\mathcal{I}^j\boxtimes s^{-j}$ (where we use the convention $\mathcal{I}^j=\mathcal{O}_X$ for $j\leq 0$), such that $s$ gives a map $s:D_{\mX'}\mX\to\bA^1_k$. We have then $s^{-1}(0)=C_{\mX'}\mX$, which is the relative spectrum over $\mX'$ of $\bigoplus_{j\geq 0}\mathcal{I}^j/\mathcal{I}^{j+1}$. One obtains the following commutative diagram with Cartesian squares:
    
\begin{equation}\label{diagdefnormalcone}\begin{tikzcd}
    C_{\mX'} \mX\arrow[d]\arrow[r] & D_{\mX'} \mX\arrow[d,"s"] & \mX\times\bG_{m,k}\arrow[d]\arrow[l]\arrow[r,"q"] & \mX\arrow[d]\\
    \{0\}\arrow[r] & \bA^1_k & \bG_{m,k}\arrow[l]\arrow[r] & \Spec(k)
\end{tikzcd}\end{equation}

We define as in Verdier \cite[Section 8]{Verdiespecial} the specialization functor:
\begin{align}
    Sp_{\mX'}:=\psi_s q^\ast:SH(\mX)\to SH(C_{\mX'} \mX)
\end{align}
Notice that a motivic version of Verdier's specialization functor was also defined by Ayoub in \cite[Section 3.2]{AyoubAnabelian}, under the name of 'monodromic specialization system'.\medskip

Consider the commutative diagram with Cartesian squares:
    \[\begin{tikzcd}
    \mX'\arrow[d,"0"]\arrow[r,"\iota'"] & \mX'\times\bA^1_k\arrow[d,"0'"] & \mX'\times\bG_{m,k}\arrow[d,"i_{\bG_m}"]\arrow[l,swap,"u'"]\arrow[r,"q'"] &\mX'\arrow[d,"i"]\\
    C_{\mX'} \mX\arrow[d]\arrow[r,"\iota"] & D_{\mX'} \mX\arrow[d,"s"] & \mX\times\bG_{m,k}\arrow[d]\arrow[l,swap,"u"]\arrow[r,"q"] & \mX\\
    \{0\}\arrow[r] & \bA^1_k & \bG_{m,k}\arrow[l] &
\end{tikzcd}\]
    where the composite function is $t:=pr_2:\mX\times\bA^1_k\to\bA^1_k$. Consider the morphism:
    \begin{align}\label{isomsommcone}
        0^\ast Sp_{X'}:=0^\ast \psi_s q^\ast\to \psi_t(i_{\bG_m})^\ast q^\ast\simeq \psi_t (q')^\ast i^\ast\simeq i^\ast
    \end{align}
    where the last isomorphism follows from \cite[Lemme 3.5.10]{Ayoub2018LesSOII} (namely, the nearby cycle acts as the identity on constant objects). From \cite[Section 8 SP5, Section 9]{Verdiespecial}, this morphism is an isomorphism for schemes in the Betti case, hence it is also an isomorphism for stacks and $\mD$ by conservativity. 

Consider now the case where $\mX'=\mX_0=f^{-1}(0)$ for $f:\mX\to\bA^1_k$ is a principal divisor: we obtain that $D_{\mX_0}\mX\to \mX$ is the relative spectrum of $\mathcal{O}_\mX[S,T]/(ST-f)$, and $C_{\mX_0} \mX\simeq \mX_0\times\bA^1_k$. We have then that $C_{\mX_0} X\to D_{\mX_0}X$ is given by $(x,t)\mapsto (x,0,t)$, and $\mX\times\bG_{m,k}\to D_{\mX_0}X$ by $(x,s)\mapsto (x,s,f(x)/s)$. In this case, from the functoriality of the deformation to the normal cone and the fact that $\psi_f$ underlies a specialization system, $Sp_{\mX_0}:SH(\mX)\to SH(\mX_0\times\bA^1_k)$ underlies a specialization system over:
\[\begin{tikzcd}
    \bA^1_k\arrow[r,"Id"] & \bA^1_k & \{0\}\times\bA^1_k\arrow[l,swap,"0\circ pr_1"]
\end{tikzcd}\]
we will denote it also by $Sp_f$ when it is convenient. Then the isomorphism $0^\ast Sp_f\simeq (i_0)^\ast$ from \eqref{isomsommcone} is an isomorphism of specialization systems over $\bA^1_k\rightarrow \bA^1_k\leftarrow\{0\}$.\medskip

Consider the open immersion:
    \begin{align}
        v:&\mX\times\bG_{m,k}\to D_{\mX_0}\mX\nn\\
        &(x,t)\mapsto (x,f(x)/t,t)
    \end{align}
    such that $s\circ v=f/t$, and consider the commutative diagram with Cartesian squares:
\[\begin{tikzcd}
    \mX_0\times\bG_{m,k}\arrow[d,"j_0"]\arrow[r,"i_{\bG_m}"] & \mX\times\bG_{m,k}\arrow[d,"v"] & \mX_{\bG_m}\times\bG_{m,k}\arrow[d,"u_{\bG_m}"]\arrow[l,swap,"u_{\bG_m}"]\arrow[r,"q_{\bG_m}"] & \mX_{\bG_m}\arrow[d,"u"]\\
    \mX_0\times\bA^1_k\arrow[d]\arrow[r] & D_{\mX_0} \mX\arrow[d,"s"] & \mX\times\bG_{m,k}\arrow[d]\arrow[l]\arrow[r,"q"] & \mX\arrow[d]\\
    \{0\}\arrow[r] & \bA^1_k & \bG_{m,k}\arrow[l]\arrow[r] & \Spec(k)
\end{tikzcd}\]
We obtain then the following isomorphism of specialization systems:
\begin{align}
    (j_0)^\ast Sp_f:=(j_0)^\ast\psi_\pi q^\ast\simeq \psi_{f/t}(u_{\bG_m})^\ast q^\ast
\end{align}
in particular, $Sp_{\mX_0}$ is monodromic from \cite[Proposition 7.1, Section 9]{Verdiespecial} and conservativity, and we have $\pi_\ast Sp_{\mX_0}\simeq 0^\ast Sp_{\mX_0}$. Using \eqref{isomsommcone}, we find that the morphism of specialization systems:
\begin{align}
    (j_0)_! (j_0)^\ast\pi^\ast0^\ast Sp_{\mX_0}\simeq(j_0)_! (j_0)^\ast\pi^\ast\pi_\ast Sp_{\mX_0}\to (j_0)_! (j_0)^\ast Sp_f
\end{align}
is isomorphic to the morphism of specialization systems:
\begin{align}
    (j_0)_!(i_{\bG_m})^\ast q^\ast \to (j_0)_! \psi_{f/t}(u_{\bG_m})^\ast q^\ast
\end{align}
We find then, using the exact triangle defining $\Pi_\ast$ and $\phi_{f/t}$, an isomorphism of specialization systems:
\begin{align}
    \phi^{mon}_f:=(j_0)_! \phi_{f/t}q^\ast\simeq  (j_0)_! (j_0)^\ast\Pi_\ast Sp_f\simeq \Pi_\ast Sp_f
\end{align}
where the last isomorphism follows from the fact that $Sp_f$ is monodromic. We will now use this representation of the monodromic vanishing cycle functor to build the Thom-Sebastiani isomorphism. Notice that, to obtain an homotopy-coherent formalism, one should adopt this as a definition of $\phi^{mon}$ from scratch, but the definition $(j_0)_! \phi_{f/t}q^\ast$ is the classical one, closer to the usual intuition.\medskip

Now, we notice that the monoidality of the nearby cycle functor induces the monoidality of Verdier's specialization. Given two closed immersions $\mX'\to \mX$, $\mY'\to \mY$, consider the closed immersion $\mX'\times \mY'\to \mX\times \mY$, such that $C_{\mX'}\mX\times C_{\mY'}\mY\simeq C_{\mX'\times \mY'}(\mX\times \mY)$. We denote $s:D_{\mX'\times \mY'} (\mX\times \mY)\to \bA^1_k$, $s_\mX:D_{\mX'}\mX\to \bA^1_k$, $s_\mY:D_{\mY'}\mY\to \bA^1_k$, $q:\mX\times \mY\times \bG_{m,k}\to \mX\times \mY$, $q_\mX:\mX\times\bG_{m,k}\to \mX$, $q_\mY:\mY\times\bG_{m,k}\to \mY$. We have $s=s_\mX\times_{\bA^1_k}s_\mY$, which gives an isomorphism:
\begin{align}\label{spxy}
    Sp_{\mX'}\boxtimes Sp_{\mY'}&:=(\psi_{s_\mX}(q_\mX)^\ast)\boxtimes(\psi_{s_\mY}(q_\mY)^\ast)\simeq \psi_s\delta^\ast((q_\mX)^\ast\boxtimes_{\bG_{m,k}}(q_\mY)^\ast)\simeq \psi_s\Delta^\ast((q_\mX)^\ast\boxtimes(q_\mY)^\ast)\simeq\psi_s q^\ast(-\boxtimes-)\nn\\
    &=:Sp_{\mX'\times \mY'}(-\boxtimes-)
\end{align}
where the first isomorphism comes from the monoidality of nearby cycles, and the third isomorphism comes from $(q_\mX\times q_\mY)\circ\Delta=q$. From the monoidality of the nearby cycle and the deformation to the normal cone, these isomorphisms satisfy the obvious commutativity and associativity condition.\medskip

Consider now the case of two closed immersions $\mX_0=f^{-1}(0)\to \mX$, $\mY_0=g^{-1}(0)\to \mY$ of principal divisors: in this case, \eqref{spxy} underlies an isomorphism of specialization system. We have that $D_{\mX_0\times\mY_0}(\mX\times\mY)\to \mX\times\mY$ is the relative spectrum of the $\mathcal{O}_{\mX\times\mY}$-module $\mathcal{O}_{\mX\times\mY}[S,T_\mX,T_\mY]/(ST_\mX-f,ST_\mY-g)$, and $D_{(\mX\times\mY)_0}(\mX\times\mY)\to \mX\times\mY$ is the relative spectrum of the $\mathcal{O}_{\mX\times\mY}$-module $\mathcal{O}_{\mX\times\mY}[S,T]/(ST-(f+g))$. We obtain then a morphism over $\mX\times\mY$:
    \begin{align}
        p&:D_{\mX_0\times\mY_0}(\mX\times\mY)\to D_{(\mX\times\mY)_0}(\mX\times\mY) \nn\\
        &:(x,y,s,t_\mX,t_\mY)\mapsto (x,y,s,t_\mX+t_\mY)
    \end{align}
    and a commutative diagram with Cartesian squares:
    \[\begin{tikzcd}
    \mX_0\times\mY_0\times\bA^1_k\times\bA^1_k\arrow[r]\arrow[d,"k\times\mu"] & D_{\mX_0\times\mY_0}(\mX\times\mY)\arrow[d,"p"]& \mX\times\mY\times\bG_{m,k}\arrow[l]\arrow[d,"Id"]\\
(\mX\times\mY)_0\times\bA^1_k\arrow[r]\arrow[d] & D_{(\mX\times\mY)_0}(\mX\times\mY)\arrow[d,"s"]& \mX\times\mY\times\bG_{m,k}\arrow[l]\arrow[d]\\
\{0\}\arrow[r] & \bA^1_k &\bG_{m,k} \arrow[l]  
\end{tikzcd}\]

We have then a natural morphism:
\begin{align}\label{morphspec}
      k^\ast Sp_{(\mX\times\mY)_0}:=k^\ast\psi_s q^\ast\to k^\ast(\mu\times k)_\ast\psi_{s\circ p}q^\ast=:k^\ast(\mu\times k)_\ast Sp_{\mX_0\times\mY_0}\simeq \mu_\ast (Sp_{\mX_0}\boxtimes Sp_{\mY_0})
\end{align}
which is a specialization system in each variables, where we have used \eqref{spxy} and $k^\ast k_\ast\simeq Id$ in the last isomorphism. From the commutativity of \eqref{spxy} and the deformation to the normal cone, this morphism is obviously commutative. If we consider furthermore a map $h:\mZ\to\bA^1_k$, have a commutative diagram of stack:
\[\begin{tikzcd}
D_{\mX_0\times\mY_0\times\mZ_0}(\mX\times\mY\times\mZ)\arrow[r]\arrow[d] & D_{(\mX\times\mY)_0\times\mZ_0}(\mX\times\mY\times\mZ)\arrow[d]\\
D_{\mX_0\times(\mY\times\mZ)_0}(\mX\times\mY\times\mZ)\arrow[r] & D_{(\mX\times\mY\times\mZ)_0}(\mX\times\mY\times\mZ)
\end{tikzcd}\]
Using the fact that the exchange morphism with respect to pushforward are compatible with composition for the specialization system $\Psi$, and the associativity property of \eqref{spxy}, we obtain that the morphism \eqref{morphspec} satisfies the obvious associativity property.\medskip

Using the monoidality of the morphism of six functor formalisms $\Pi_\ast$, we obtain a morphism:
\begin{align}\label{morphtomseb}
    k^\ast \phi_{f\boxplus g}^{mon}\simeq k^\ast \Pi_\ast Sp_{(\mX\times\mY)_0}\to\Pi_\ast\mu_\ast(Sp_{\mX_0}\boxtimes Sp_{\mY_0})\simeq (\Pi_\ast Sp_{\mX_0})\boxtimes_{mon}(\Pi_\ast Sp_{\mX_0})\simeq \phi_f^{mon}\boxtimes_{mon}\phi_g^{mon}
\end{align}
which underlies a morphism of specialization system, and satisfies commutativity and associativity. As this is a morphism of specialization system, it commutes with smooth pullbacks, hence, using the conservativity of smooth pullback and Betti realization, it suffices to prove that it is an isomorphism for schemes in the Betti case. This is then the morphism used by Saito in the preprint \cite{Saitothomseb}, which is proven there to be an isomorphism. Alternatively, in \cite[Appendix A]{Joycesymstab}, Shürmann shows that Saito's morphism is isomorphic to the morphism of \cite{MasThomSeb}, which is shown to be an isomorphism. Notice that this proof is topological: for a fully algebraic proof, which do not use conservativity of the Betti realization, but use resolution of singularities, see \cite{descombesThomSeb}.\medskip

The isomorphism of the Theorem is obtained directly by summing \eqref{morphtomseb} for each couple of critical values of $f,g$, and it inherits its functoriality properties.\end{proof}

\subsubsection{Square root of Tate twists and quadratic bundles}

The following is well known for mixed Hodge modules (see \cite[Example 2.23]{Joycesymstab}), we just show that it holds also for perverse Nori motives.

\begin{lemma}\label{lemsqroot}
    Given a choice $\sqrt{-1}$ of square root of $-1$ in $k$ (assumed to be algebraically closed), there is a canonical isomorphism:
    \begin{align}
\phi_{x^2}^{mon}\un_{\bA^1_k}\boxtimes_{mon}\phi_{x^2}^{mon}\un_{\bA^1_k}\simeq \un[-2](-1)\in \mA_{mon}(\Spec(k))[-2]
    \end{align}
    which gives a canonical square root of the Tate twist:
    \begin{align}
    -\{-1/2\}:=-[-1](-1/2):=-\boxtimes_{mon} 0^\ast\phi^{mon,tot}_{x^2}\un_{\bA^1_k}
\end{align}
\end{lemma}

\begin{proof}
    Consider the closed immersion $i:\{(0,0)\}\to (xy)^{-1}(0)$, the projection $\rho:(xy)^{-1}(0)\to \{(0,0)\}$ and the closed immersion $\eta:\bA^1_k\to\bA^2_k$ which acts as $x\mapsto (x,0)$. We can define:
    \begin{align}\label{isomsquareroot}
     i^\ast\phi^{mon}_{xy}\un_{\bA^2_k}\simeq \rho_!\phi^{mon}_{xy}\un_{\bA^2_k}\to \rho_!\phi^{mon}_{xy}\eta_! \eta^\ast\un_{\bA^2_k}\simeq \pi_!\phi^{mon}_0\un_{\bA^1_k}\simeq \pi_!\un_{\bA^1_k}\simeq \un[-2](-1)
\end{align}
the first isomorphism holding because $\phi^{mon}_{xy}\un_{\bA^2_k}\in \mA_{mon}((xy)^{-1}(0))$ is supported on $(0,0)$, the second morphism is obtained by adjunction, the third isomorphism holds because $\eta$ is a closed immersion. From \cite[Appendix A]{Dav13}, the constructible complexes and mixed Hodge module version of this morphism is an isomorphism (it is a simple application of dimensional reduction), and then the monodromic and perverse Nori motive version is an isomorphism too, by conservativity.\medskip

We consider the closed immersion $i':\{(0,0)\}\to (x^2+y^2)^{-1}(0)$. Consider a choice of square root $\sqrt{-1}$ of $-1$, we can build an automorphism (recall that $k$ has characteristic $0\neq 2$):
\begin{align}
    \Phi:&\bA^2_k\to\bA^2_k\nn\\
    &(x,y)\mapsto(x+\sqrt{-1}y,x-\sqrt{-1}y)
\end{align}
Which gives:
\begin{align}
    i^\ast \phi^{mon}_{xy}\un_{\bA^2_k}\simeq (i')^\ast \Phi^\ast\phi^{mon}_{xy}\un_{\bA^2_k}\simeq (i')^\ast\phi^{mon}_{x^2+y^2}\Phi^\ast\un_{\bA^2_k}\simeq (i')^\ast\phi^{mon}_{x^2+y^2}\un_{\bA^2_k}  
\end{align}
Using Thom-Sebastiani, we obtain the desired isomorphism:
\begin{align}
\phi_{x^2}^{mon}\un_{\bA^1_k}\boxtimes_{mon}\phi_{x^2}^{mon}\un_{\bA^1_k}\simeq(i')^\ast \phi^{mon}_{x^2+y^2}\un_{\bA^2_k}\simeq i^\ast \phi^{mon}_{xy}\un_{\bA^2_k}\simeq \un[-2](-1)
\end{align}
Notice that $(-1/2)$ preserves the heart $\mA_{mon}$ from the exactness of $\phi^{mon}_f$.
\end{proof}

From now on, we fix once and for all a choice $\sqrt{-1}$ of square root of $-1$ in $k$, fixing a choice of square root of the Tate twist.\medskip

Consider a quadratic bundle $(\mE,q)$ on a stack $\mX$, by which we means the data of a locally free sheaf $\mE$ on $\mX$, and a non-degenerate quadratic form $q$ on $\mE^\vee$, inducing a regular function $q:\bV_\mX(\mE)\to\bA^1_k$. As before, we denote by $\pi:\bV_\mX(\mE)\to\mX$ the projection, and by $s:\mX\to\bV_\mX(\mE)$ the $0$-section. We consider $P_{(\mE,q)}$, the $\bZ/2\bZ$-principal bundle of orientation of $(\mE,q)$. By definition, its section over $f:\mX'\to\mX$ is the $\bZ/2\bZ$ torsor of orientations of $(\mE,q)$, \ie, trivializations $\det(f^\ast(\mE))\simeq \mathcal{O}_{\mX'}$ which are square root of the trivialization $\det(f^\ast(\mE))^2\simeq \mathcal{O}_{\mX'}$ induced by $\det(f^\ast(q))$.\medskip

We will use the following result, which is inspired by \cite[Theorem 4.4, Theorem 2.20]{Bussi2019OnMV}:

\begin{lemma}\label{lemactquadbun}
    Given a stack $\mX$ with a function $f:\mX\to\bA^1_k$ and a quadratic bundle $(\mE,q)$, there is an isomorphism of specialization system $\mD^b_c\to\mD^b_{mon,c}$ over $\mX\to\mX\leftarrow\mX$:
    \begin{align}
        \phi_f^{mon,tot}\otimes_{\bZ/2\bZ}P_{(\mE,q)}\{-d_\mE/2\}\simeq s^\ast\phi_{f\circ \pi+q}^{mon,tot}\pi^\ast
    \end{align}
    It is compatible with exterior tensor product and with direct sum of quadratic bundles. A trivialization $(\mE,q)\simeq (\mathcal{O}_\mX^n,\sum_{i=1}^n x_i^2)$, giving a section of $P_{(\mE,q)}$, identifies this morphism with the Thom-Sebastiani isomorphism:
    \begin{align}
        \phi_f^{mon,tot}\{-d_\mE/2\}:= \phi_f^{mon,tot}\boxtimes 0^\ast\phi^{mon,tot}_{\sum_{i=1}^n x_i^2}\un_{\bA^n_k}\simeq (Id\times 0)^\ast \phi_{f\boxplus \sum_{i=1}^n x_i^2}(pr_1)^\ast
    \end{align}
\end{lemma}

\begin{proof}
We begin by defining a canonical isomorphism (this isomorphism is standard for mixed Hodge modules on schemes):
    \begin{align}\label{isomorientbun}
        s^\ast\phi^{mon,tot}_q\un_{\bV_\mX(\mE)}\simeq \mathcal{L}_{P_{(\mE,q)}}\{-d_\mE/2\}
    \end{align}
    such that:
    \begin{align}
        -\otimes_{mon}s^\ast\phi^{mon,tot}_q\un_{\bV_\mX(\mE)}\simeq -\otimes_{\bZ/2\bZ}P_{(\mE,q)}\{-d_\mE/2\}
    \end{align}
    underlies an isomorphism of involution of coefficient systems $\mD_{mon,c}\to\mD_{mon,c}$ over $\mX$, and is compatible with exterior products of quadratic bundles.\medskip
    
    First, notice that $\phi^{mon,tot}_q\un_{\bV_\mX(\mE)}$ is supported on the image of $s$, hence $s^\ast\phi^{mon,tot}_q\un_{\bV_\mX(\mE)}\in\mA_{mon,c}(\mX)$, as $\mathcal{L}_{P_{(\mE,q)}}\{-d_\mE/2\}$, hence it suffices to define smooth locally isomorphisms \eqref{isomorientbun}, and to ensure that they glue on overlaps. Consider a smooth morphism from a separated scheme of finite type $g:X\to\mX$ and a trivialization $g^\ast(\mE,q)\simeq (\mathcal{O}_X^d,\sum_{i=1}^d x_i^2)$. By smooth base change, we obtain an isomorphism:
    \begin{align}\label{isom1}
        g^\ast s^\ast\phi^{mon,tot}_q\un_{\bV_\mX(\mE)}\simeq \un_X\{-d/2\}
    \end{align}
    Considering the canonical orientation $\det(\mathcal{O}_X^d)\simeq \mathcal{O}_X$ of $(\mathcal{O}_X^d,\sum_{i=1}^d x_i^2)$, this gives a canonical section of $g^\ast(P_{(\mE,q)})$, hence an isomorphism $g^\ast\mathcal{L}_{P_{(\mE,q)}}\simeq\un_X$, which gives:
    \begin{align}\label{isom2}
        g^\ast s^\ast\phi^{mon,tot}_q\un_{\bV_\mX(\mE)}\simeq\mathcal{L}_{P_{(\mE,q)}}\{-d_\mE/2\} 
    \end{align}
    Given two such choices $g_i:X_i\to\mX$, $i=1,2$, two smooth morphisms $g'_i:X'\to X_i$ such that $g:=g'_1\circ g_1=g'_2\circ g_2:X'\to\mX$, we obtain two trivialization $g^\ast(\mE,q)\simeq (\mathcal{O}_X^d,\sum_{i=1}^d x_i^2)$, and the two isomorphisms:
    \begin{align}
        (g'_i)^\ast (g_i)^\ast s^\ast\phi^{mon,tot}_q\un_{\bV_\mX(\mE)}\simeq (g'_i)^\ast\un_{X_i}\{-d/2\}
    \end{align}
    that we need to compare are obviously the isomorphisms induced by those orientations. On the perverse sheaf realization, the isomorphism \eqref{isom1} only depends on the orientation on the Milnor sphere induced by the trivialization, hence the isomorphism \eqref{isom2} is independent of this choice. By faithfulness of the perverse sheaf realization, those two isomorphisms agree. By descent, one obtains a global isomorphism, that does not depends on any extra choice.\medskip

    If we consider $f:\mX'\to\mX$, and consider a smooth covering with trivialization $g:X\to\mX$, we can built \eqref{isomorientbun} for $\mX'$ by using the covering and trivialization $g':X'\to\mX'$ obtained by base change from $g$. Then \eqref{isomorientbun} induces by base change an isomorphism of involution of coefficient systems over $\mX$. If we consider two stacks $\mX_i$, $i=1,2$ and two quadratic bundles $(\mE_i,q_i)$ on $\mX_i$, we can built \eqref{isomorientbun} for the product $\mX_1\times\mX_2$ by using product of smooth covering and trivializations  $g_i:X_i\to\mX_i$. By Thom-Sebastiani, we obtain that \eqref{isomorientbun} is compatible with exterior tensor product.\medskip

    Consider the following commutative diagram:
    \[\begin{tikzcd}
        \mX\arrow[r,"s"] & \bV_\mX(\mE)\arrow[r,"\pi\times Id"]\arrow[dr,"f\circ \pi+q"] & \mX\times\bV_\mX(\mE)\arrow[r,"pr_1"]\arrow[d,"f\boxplus q"] & \mX\\
        & & \bA^1_k &
    \end{tikzcd}\]
    We have a natural sequence of (iso)morphisms:
    \begin{align}\label{isomlemorient}
        \phi_f^{mon,tot}\otimes_{\bZ/2\bZ}P_{(\mE,q)}\{-d_\mE/2\}&\simeq \phi_f^{mon,tot}\otimes_{mon}(s^\ast\phi_q^{mon,tot}\un_{\bV_{\mX}(\mE)})\nn\\
        &\simeq  \Delta^\ast (\phi_f^{mon,tot}\boxtimes_{mon}(s^\ast\phi_q^{mon,tot}))(pr_1)^\ast\nn\\
        &\simeq s^\ast(\pi\times Id)^\ast (\phi_f^{mon,tot}\boxtimes_{mon}\phi_q^{mon,tot})(pr_1)^\ast\nn\\
        &\simeq s^\ast(\pi\times Id)^\ast \phi_{f\boxplus q}^{mon,tot}(pr_1)^\ast\nn\\
        &\to s^\ast \phi_{f\circ \pi+q}^{mon,tot}(\pi\times Id)^\ast(pr_1)^\ast\nn\\
        &\simeq s^\ast\phi_{f\circ \pi+q}^{mon,tot}\pi^\ast
    \end{align}
    where the first line is \eqref{isomorientbun}, the second line comes from the definition of the internal tensor product, the third from $(Id\times s)\circ\Delta\simeq (p\times Id)\circ s$, the fourth from the Thom-Sebastiani isomorphism, the fifth by the functoriality of the specialization system $\phi^{mon,tot}$ in the above diagram, and the last one from $pr_1\circ (\pi\times Id)\simeq \pi$. As proven above, the first isomorphism underlies an isomorphism of specialization system and is compatible with exterior tensor products, it is also the case for the Thom-Sebastiani of the fourth line from Theorem \ref{theothomseb}, and for the other lines from the classical functoriality of the four functors, hence it is the case for \eqref{isomlemorient}. We mean that, given $g:\mX'\to\mX$, denoting $(\mE',q'):=g^\ast(\mE,q)$ and $f':=f\circ g$, the following squares of morphisms are commutative:
    \[\begin{tikzcd}
        g^\ast(\phi_f^{mon,tot}\otimes_{\bZ/2\bZ}P_{(\mE,q)}\{-d_\mE/2\})\arrow[d]\arrow[r] & g^\ast(s^\ast\phi_{f\circ \pi+q}^{mon,tot}\pi^\ast)\arrow[d]\\
        (\phi_{f'}^{mon,tot}\otimes_{\bZ/2\bZ}P_{(\mE',q')}\{-d_{\mE'}/2\})g^\ast\arrow[r] & ((s')^\ast\phi_{f'\circ \pi'+q'}^{mon,tot}(\pi')^\ast) g^\ast
    \end{tikzcd}\]
    \[\begin{tikzcd}
        g_!(\phi_{f'}^{mon,tot}\otimes_{\bZ/2\bZ}P_{(\mE',q')}\{-d_{\mE'}/2\})\arrow[d]\arrow[r] & g_!((s')^\ast\phi_{f'\circ \pi'+q'}^{mon,tot}(\pi')^\ast)\arrow[d]\\
        (\phi_f^{mon,tot}\otimes_{\bZ/2\bZ}P_{(\mE,q)}\{-d_\mE/2\})g_!\arrow[r] & (s^\ast\phi_{f\circ \pi+q}^{mon,tot}\pi^\ast) g_!
    \end{tikzcd}\]
    and similarly for $g^!,g_\ast$, where the direction of the vertical arrows are reversed. The compatibility with exterior tensor products says that, given two quadratic bundles $(\mE_i,q_i)$ on $\mX_i$ and $f_i:\mX_i\to\bA^1$ for $i=1,2$, considering the quadratic bundle $(\mE,q):=(\mE_1\boxplus\mE_2,q_1\boxplus q_2)$ on $\mX_1\times\mX_2$, the following diagram is commutative:
    \[\begin{tikzcd}[column sep=tiny]
        (\phi_{f_1}^{mon,tot}\otimes_{\bZ/2\bZ}P_{(\mE_1,q_1)}\{-d_\mE/2\})\boxtimes (\phi_{f_2}^{mon,tot}\otimes_{\bZ/2\bZ}P_{(\mE_2,q_2)}\{-d_\mE/2\})\arrow[r]\arrow[d,"\simeq"]  & ((s_1)^\ast\phi_{f_1\circ \pi_1+q_1}^{mon,tot}(p_1)^\ast)\boxtimes ((s_2)^\ast\phi_{f_2\circ \pi_2+q_2}^{mon,tot}(p_2)^\ast)\arrow[d,"\simeq"]\\
        \phi_{f_1\boxplus f_2}^{mon,tot}\otimes_{\bZ/2\bZ}P_{(\mE,q)}\{-d_\mE/2\}\arrow[r] & s^\ast\phi_{(f_1\oplus f_2)\circ\pi+q}^{mon,tot}p^\ast
    \end{tikzcd}\]
    where the vertical arrows are the Thom-Sebastiani isomorphisms.\medskip
    
    Given a trivialization $(\mE,q)\simeq (\mathcal{O}_\mX^n,\sum_{i=1}^n x_i^2)$, the above diagram becomes:
    \[\begin{tikzcd}
        \mX\arrow[r,"Id\times 0"] & \mX\times\bA^n_k\arrow[r,"\Delta\times Id"]\arrow[dr,"f\boxplus q"] & \mX\times\mX\times\bA^n_k\arrow[r,"pr_1"]\arrow[d,"f\boxplus0\boxplus q"] & \mX\\
        & & \bA^1_k &
    \end{tikzcd}\]
    hence the morphism of the Lemma becomes simply the Thom-Sebastiani morphism, hence is an isomorphism. In the general situation, given a smooth cover with trivialization $g:X\to\mX$, we obtain from the functoriality with respect to smooth pullbacks that $g^\ast$ applies to the morphism of the Lemma is a isomorphism, hence by conservativity of smooth pullbacks it is an isomorphism.\medskip

    Consider two quadratic bundles $(\mE_i,q_i)$ on the same stack $\mX$, and denote $(\mE,q):=(\mE_1\oplus\mE_2,q_1\oplus q_2)$. A choice of orientation of $(\mE_i,q_i)$ for $i=1,2$ gives a choice of orientation for $(\mE,q)$, which gives a natural isomorphism:
    \begin{align}
        P_{(\mE,q)}\simeq P_{(\mE_1,q_1)}\otimes_{\bZ/2\bZ}P_{(\mE_2,q_2)}
    \end{align}
    Consider the following square of isomorphisms:
    \[\begin{tikzcd}
        s^\ast\phi^{mon,tot}_{f\circ\pi+q}\pi^\ast\arrow[r,"\simeq"]\arrow[d,"\simeq"] & \phi_f^{mon,tot}\otimes_{\bZ/2\bZ}P_{(\mE,q)}\{-d_\mE/2\}\arrow[d,"\simeq"]\\
((s_1)^\ast\phi^{mon,tot}_{f\circ\pi_1+q_1}\pi^\ast)\otimes_{\bZ/2\bZ}P_{(\mE_2,q_2)}\{-d_{\mE_2}/2\}\arrow[r,"\simeq"] & \phi^{mon,tot}_f\otimes_{\bZ/2\bZ}P_{(\mE_1,q_1)}\{-d_{\mE_1}/2\}\otimes_{\bZ/2\bZ}P_{(\mE_2,q_2)}\{-d_{\mE_2}/2\}
    \end{tikzcd}\]
    where the left arrow is obtained by considering the quadratic bundle $(\pi_1)^\ast (\mE_2,q_2)$ on $\bV_\mX(\mE_1)$. The fact that it commutes can be checked on a smooth cover: using a smooth cover of $\mX$ such that both $(\mE_i,q_i)$ have a trivialization, giving a trivialization of $(\mE,q)$, this follows directly from the associativity of the Thom-Sebastiani isomorphism. Then the above square is commutative, meaning that the isomorphism of the Lemma is compatible with direct sum of quadratic bundles.
\end{proof}

\subsection{Grothendieck groups}\label{SectKgroup}

We want to consider consider $K(\mD^b_c(\mX))$ (resp $K(\mD^b_{mon,c}(\mX))$, the Grothendieck group of the triangulated category $\mD^b_c(\mX)$ (resp $\mD^b_{mon,c}(\mX)$), \ie the free group generated by isomorphism class of objects, divided by the relation $[B]=[A]+[C]$ for each exact triangle $A\to B\to C\to A[1]$. It is identified with the Grotendieck group of the perverse heart $K(\mA_c(\mX))$ (resp. $K(\mA_{mon,c}(\mX))$), and then, in the presence of a weight structures, because each object in the heart has a weight filtration, it is a direct sum over $w\in\bZ$ of the Grothendieck group of pure objects of weight $w$. We denote by $\mathbb{L}$ (resp. $\mathbb{L}^{1/2}$), called the Tate motive, or motive of the affine line $\bA^1$, the class of $\un[-2](-1)$ (resp. $\un[-1](-1/2)$).\medskip

There is a slight subtlety in dealing with Grothendieck group of objects over stacks, because (proper) pushforwards along non-representable maps (and in particular, taking the cohomology (with compact support)) does not preserve $\mD^b_c(-)$ (resp $\mD^b_{mon,c}(-)$). For any triangulated category $\mD_c$ with a Bondarko weight structure, we consider:
\begin{align}\label{defD-}
    &\mD_{c,-}:=\lim_{w\to +\infty}\mD_{c,\leq w}\nn\\
    &\mD_{c,+}:=\lim_{w\to +\infty}\mD_{c,\geq w}
\end{align}
We have then, from the functoriality of the four functors with weights and with objects that $f_!,f^\ast$, $\boxtimes$ preserves $\mD^-_{c,-}$ (and similarly $f_\ast,f^!$, $\boxtimes$ preserves $\mD^+_{c,+}$, and $\bD$ exchanges $\mD^-_{c,-}$ and $\mD^+_{c,+}$). Recall that for an object $F\in\mD_{c,w}$, we have that $^p\mathcal{H}^n(F)$ is pure of weight $w+n$. In particular, the weights of the perverse cohomology of an object of  $\mD^-_{c,-}$ (resp $\mD^+_{c,+}$) are bounded from above (resp bounded from below). Moreover, notice that any pure objects $F\in \mA_c(\mX)$ (resp $F\in \mA_{mon,c}(\mX)$) of weight $w$ can written as $F(\left\lfloor w/2\rfloor\right)(-\left\lfloor w/2\rfloor\right)$ where $F(\left\lfloor w/2\rfloor\right)$ is pure of weight $0$ or $1$ (resp as $F(w/2)(-w/2)$, where $F(-w/2)$ is pure of weight $0$). This means that the Grothendieck group of $\mD^-_{c,-}$ (resp $\mD^-_{mon,c,-}$) can be identified with the completed Grothendieck group $K(\mA_c[[\bL^{-1}]])$ (resp  $K(\mA_{mon,c}[[\bL^{-1/2}]]))$.\medskip

In particular, given a $k$-stack of finite type $p:\mX\to Spec(k)$, we have that $H_c(\mX):=p_!p^\ast\un_{pt}\in\mD^-_{c,-}$, which gives $[H_c(\mX)]\in K(\mA_c(pt))[[\bL^{-1}]])$. For constructible complexes, $D^b_c(pt,\bQ)$ is the bounded derived category of complexes of $\bQ$-vector space, hence $K(\mA_c(pt))=\bZ$ and $\bL=1$, and one find the Euler characteristic in the scheme case, but there is a problem in the stack case. This is a well known problem, see for example \cite{joyce2005MOTIVICIO}. This problem disappear when one uses a refinement, as we obtain then $\bL\neq 1$.\medskip

For mixed Hodge modules, $MHM(pt)$ is the Abelian category of mixed Hodge structure, and its Grothendieck group is isomorphic to $\bZ[q,y]$, with $q$ (resp $y$) the parameter taking into account the weight filtration $W$ (resp the Hodge filtration $F$). This isomorphism is given by the Hodge polynomial, which gives, for $M^\bullet$ a complex of mixed Hodge structure:
\begin{align}
    [M^\bullet]=\sum_{i,j,k}(-1)^k q^{i/2} y^j \dim(\Gr_i^W\Gr^j_F M^k)\in\bZ[q^{\pm 1/2},y^{\pm 1}]
\end{align}
in particular, $\bL$ is identified with $qy$. There are two usual one-parameter specializations the Hodge polynomial, the weight polynomial, for $y=1$, related count of points over finite fields via the Weil conjectures, and the $\chi_y$ genus, for $q=1$, which is topological for smooth and proper varieties from the degeneration of the Hodge to de Rahm spectral sequence. The Euler characteristic (in the scheme case) is then given by $y=q=1$. If the Hodge structures are of Tate type, hence the Hodge polynomial is a polynomial in $\bL=qy$ (as one obtains often in toric localization contexts), the weight polynomial and $\chi_y$ genus coincide, but one must take care that they are in general different.\medskip

The objects of the Abelian category of monodromic mixed Hodge modules $MMHM$ are mixed Hodge structures $M$ with a semisimple action $T_s$ and a nilpotent action $N:M\to M(-1)$. This nilpotent action is not visible in the Grothendieck group, as any object is the extension of objects with $N$ acting trivially. One can split $M\otimes_\bQ \bC=\bigoplus_{0\leq \lambda<1}M_\lambda$, where $T_s$ acts on $M_\lambda$ with eigenvalue $e^{2i\pi\lambda}$. Denoting $\mE_\lambda=0$ if $\lambda=1$ and $\mE_\lambda=1$ if $\lambda\neq 1$, for $M^\bullet$ a complex of mixed Hodge structure, the monodromic Hodge polynomial is defined by:
\begin{align}
    [M^\bullet]=\sum_{i,j,k,\lambda}(-1)^k q^{(i+\epsilon_\lambda)/2} y^{j-\lambda} \dim(\Gr_i^W\Gr^j_F M^k_\lambda)\in\bZ[q^{\pm 1/2},y^\bQ]
\end{align}
and gives a an isomorphism of ring $K(D^b(MMHM(pt))\simeq \bZ[q^{\pm 1/2},y^\bQ]$ from \cite[Definition 2.18]{Saitothomseb}. It is obviously compatible with the inclusion of mixed Hodgge structure into monodromic mixed Hodge structure. The specialization $q=1$ gives the Hodge spectrum $\chi_y^{mon}$.\medskip

$\mM_{perv}(Spec(k))$ is the Abelian category of Nori motives, and its Grothendieck group is by definition the group ofNori motives, and similarly the Grothendieck group of monodromic perverse Nori over $\Spec(k)$ is the group of monodromic Nori motives.\medskip

Consider $\mX$ an algebraic stack, $j:\mU\to \mX$ the inclusion of a Zariski-open subscheme, and $i:\mathcal{Z}\to X$ the inclusion of its closed complements. There is a natural exact triangle in $\mD(\mX)$:
\begin{align}\label{triangmotivic}
    j_! j^\ast \to Id\to i_! i^\ast\to j_! j^\ast[1]
\end{align}
Then, given $p:\mX\to \mS$, we have an equality in the Grothendieck group:
\begin{align}
    [p_!\un_\mX)]=[(p\circ j)_!\un_\mU]+[(p\circ i)_!\un_\mathcal{Z}]\in K(\mD^-_-(\mS))
\end{align}
using also homotopy invariance, we see that there is a natural map from the Grothendieck group of (monodromic) varieties over $\mS$ (with cut and paste relation) to $K(\mA_c(\mS)[[\bL^{-1}]])$ (resp over $K(\mA_{mon,c}(\mS)[[\bL^{-1/2}]])$).\medskip

For $X$ a separated scheme of finite type, there is from the work of Denef-Loeser \cite{Denef1998MotivicEI} and Looijenga \cite{Looijenga1999-2000} a Grothendieck group $\mM^{\bG_m}_{X\times\bG_m}$ of $\bG_m$-equivariant varieties on $X\times\bG_m$, and, for $f:X\to\bA^1_k$, nearby and vanishing cycles $\Psi_f,\Phi_f\in\mM^{\bG_m}_{X_0\times\bG_m}$. From \cite[Proposition 3.5.1]{Ivorra2021QuasiunipotentMA}, there are morphisms $\chi_{X,c}^{\bG_m}:\mM^{\bG_m}_{X\times\bG_m}\to K(\mD^b_{mon,c}(X))$. From \cite[Theorem 5.2.1]{Ivorra2021QuasiunipotentMA}, one has $\chi_{X,c}^{\bG_m}(\Psi_f)=[\psi^{mon}_f\un_{X_\eta}]$, and then directly $\chi_{X,c}^{\bG_m}(\Phi_f)=[\phi^{mon}_f\un_X]$, where one uses Ayoub definition of nearby cycles. This result was well known for the case of $D^b_c(-,\bQ)$ and $D^bMHM(-)$ before \cite{Ivorra2021QuasiunipotentMA}, but this was new for perverse Nori motives. The same follow for stacks by descent: in particular, the class in the Grothendieck group of Nori motives of the cohomology of the DT sheaf that wee will consider below will correspond to the Nori realization of the motivic DT invariants defined in \cite{Bussi2019OnMV}, as was known in the Hodge case and pointed in \cite[Appendix A]{Davison2019RefinedIO}.\medskip

Using the Mayer-Vietoris exact triangle, given $\mX$, a constructible subset $S$ of $\mX(k)$, and an element $F\in\mD(\mX)$, we can define uniquely its motive on $S$, that we denote by a slight abuse of notation $[\bH_c(S,F|_S)]$, and is additive under disjoint union (notice that, in general, $S$ is not a stack, hence $F|_S$ do not make sense, which explains why the notation is abusive). Indeed, writing $S$ as a disjoint union of locally closed substacks $S_i\to\mX$, it suffices to define:
\begin{align}
    [\bH_c(S,F|_S)]:=\sum_i [\bH_c(S_i,F|_{S_i})]
\end{align}
and this does not depends on the choice of the $S_i$ by Mayer-Vietoris. We have the useful "motivic decomposition formula":

\begin{lemma}\label{lemmotdec}
    Let $k$ be an algebraically closed field of characteristic $0$, and $\eta:\mX\to\mY $ be a geometric injection (resp. bijection) of algebraic $k$-stacks of finite type (\ie, such that $\eta(k):\mX(k)\to\mY (k)$ is a injection (resp. bijection) of groupoids). Then, for any $F\in\mD_{(mon)}(\mY)$:
    \begin{align}\label{geominj}
        [\bH_c(\mY,F)]=[\bH_c(\mX,\eta^\ast F)]+[\bH_c(\mY-\eta(\mX),F|_{\mY-\eta(\mX)})]
    \end{align}
    (resp. :
    \begin{align}\label{geombij}
        [\bH_c(\mY,F)]=[\bH_c(\mX,\eta^\ast F)]
    \end{align})
\end{lemma}

\begin{proof}
    We begin by the case when $\eta$ is a geometric bijection. Consider an open immersion $u:\mU\to\mY$, its closed complement $z:\mZ\to\mY$, and the Cartesian diagram:
    \[\begin{tikzcd}
        \mU'\arrow[r,"u'"]\arrow[d,"\eta_U"] & \mX\arrow[d,"\eta"] & \mZ'\arrow[l,"z'"]\arrow[d,"\eta_Z"]\\
        \mU\arrow[r,"u"] & \mY & \mZ\arrow[l,"z"]
    \end{tikzcd}\]
    Applying $\eta_!$ to the right and $\eta^\ast$ to the left to the localization exact triangle of $\mU'\rightarrow \mX\leftarrow\mZ'$, we obtain:
    \begin{align}
        \eta_!(u')_!(u')^\ast\eta^\ast\to \eta_!\eta^\ast\to \eta_!(z')_!(z')^\ast\eta^\ast
    \end{align}
    which gives, using base change:
    \begin{align}
        u_!(\eta_U)_!(\eta_U)^\ast u^\ast\to\eta_!\eta^\ast\to z_!(\eta_Z)_!(\eta_Z)^\ast z^\ast
    \end{align}
    and then passing to the Grothendieck group:
    \begin{align}
        [\eta_!\eta^\ast]=[u_!(\eta_U)_!(\eta_U)^\ast u^\ast]+[z_!(\eta_Z)_!(\eta_Z)^\ast z^\ast]
    \end{align}
    notice that, if $\eta$ is a geometric bijection, then $\eta_U,\eta_Z$ are still geometric bijections. If $[(\eta_U)_!(\eta_U)^\ast]=Id$, $[(\eta_Z)_!(\eta_Z)^\ast]=Id$, then, using $[Id]=[u_!u^\ast]+[z_!z^\ast]$ from the localization exact triangle, one obtains $[\eta_!\eta^\ast]=Id$.\medskip
    
    The main tool used here is \cite[Lemma 3.2]{BRIDGELAND2012102} (notice that the arguments used here are entirely algebraic, and does not use $k=\bC$, they works for any algebraically closed field of characteristic $0$). It gives that, if $\eta$ is a geometric bijection, there is a collection $\mX_i$ (resp $\mY_i$) of locally closed substacks of $\mX$ (resp $\mY$) such that $\eta$ induces isomorphisms $\eta_i:\mX_i\to\mY_i$. Notice that the statement that $[\eta_!\eta^\ast]=Id$ is Zariski-local on $\mY$, hence we can suppose that $\mY$ is of finite type, which implies that the collection is finite. We show that $[\eta_!\eta^\ast]=Id$ by recursion on the number $n$ of strata, the case $n=1$ ($\eta$ is an isomorphism) being trivial. Take $\mY_1=\mU_1\cap\mathcal{Z}_1$. From the above reduction, it suffices to show $[(\eta')_!(\eta')^\ast]=Id$ for $\eta'$ the geometric bijections obtained by restricting to $\mU_1\cap\mathcal{Z}_1$, $(\mY-\mU_1)\cap\mathcal{Z}_1$, $\mU_1\cap(\mY-\mathcal{Z}_1)$ and $(\mY-\mU_1)\cap(\mY-\mathcal{Z}_1)$. But all of those four stacks have $n-1$ strata, hence we obtain $[\eta_!\eta^\ast]=Id$ by recursion, which gives in particular \eqref{geombij}.\medskip

    If $\eta$ is a geometric injection, we can write the subset $\mY(k)-\eta(\mX(k))$, which is constructible by Chevalley's theorem, as a disjoint union of locally closed immersions $\eta_i$ of substacks of $\mY$. We can then form a geometric bijection $\eta'$ from the disjoint union of $\eta$ and the $\eta_i$, and, applying the above result for $\eta'$, we obtain the claimed result for $\eta$.
\end{proof}

\section{Recollection on DT theory on stacks}\label{sectjoyce}

In this section, as in the preceding section, we work over an algebraically closed field $k$ of characteristic $0$. All our stacks are assumed to be quasi-separated Artin $1$-stacks locally of finite type over $k$, with affine stabilizers. We will recall the theory of d-critical stacks, as introduced in \cite{Joyce2013ACM}, and the gluing of the sheaf of vanishing cycles of d-critical stacks, as done in \cite{Joycesymstab} (scheme case) and \cite{darbstack} (stack case). In those articles, the theory is developed in two steps: first on schemes, by considering critical charts given by schemes, open/étale restrictions of them, and stabilization by adding a trivial quadratic factor, and then extending the construction to stacks by smooth descent. To glue our hyperbolic localization formula, we will have to work with more general critical charts, given by stacks, so we express the results of \cite{Joyce2013ACM}, \cite{Joycesymstab} and \cite{darbstack} in this more general setting, using smooth descent.

\subsection{D-critical stacks and stacky critical charts}

In \cite[Section 2.8]{Joyce2013ACM}, Joyce introduce the concept of d-critical stack, which is the main tool to glue vanishing cycles in cohomological DT theory. We will adapt slightly the theory developed in \cite[Section 2.8]{Joyce2013ACM}, in order to consider critical charts defined by functions on smooth stacks instead of smooth schemes. This extension is mostly straightforward, because the whole theory of \cite{Joyce2013ACM} behaves well with respect to smooth maps, hence we will obtains our  results directly by smooth descent from the results in the scheme case.\medskip

For any Artin stack $\mX$ over $k$, Joyce defines in \cite[Corollary 2.52]{Joyce2013ACM} a lisse-étale sheaf $\mS_{\mX}$ of $k$-algebras on $\mX$. We give the main characterization of $\mS_{\mX}$ that we will use:

\begin{lemma}\label{defSx}
\begin{enumerate}
    \item[$i)$] For any smooth map  $\mR\to \mX$ from a $k$-stack, and any closed embedding $i:\mR\hookrightarrow \mU$ into a smooth $k$-stack, denote by $I_{\mR,\mU}$ the sheaf of ideals in $i^{-1}(\mathcal{O}_\mU)$ of functions on $\mU$ near $i(\mR)$ which vanishes on $i(\mR)$. There is a natural exact sequence:
\[\begin{tikzcd}
    0\arrow[r] &\mS_\mX|_\mR\arrow[r,"\iota_{\mR,\mU}"] & \frac{i^{-1}(\mathcal{O}_\mU)}{I^2_{\mR,\mU}}\arrow[r,"d"] & \frac{i^{-1}(T^\ast_\mU)}{I_{\mR,\mU}\cdot i^{-1}(T^\ast_\mU)}
\end{tikzcd}\]
We denote sections of $\mS_X|_\mR$ by $f+I^2_{\mR,\mU}$, with $df|_\mR=0$.

\item[$ii)$] The construction of $\mS_{\mX}$ is functorial, namely, for $\phi:\mX\to\mY$, there is a naturally defined morphism of sheaves of $k$-algebras $\phi^\star:\phi^{-1}(\mS_\mY)\to\mS_\mX$, compatible with compositions, such that $Id^\star=Id$. Given a commutative diagram:
    \[\begin{tikzcd}
    \mX\arrow[d,"\phi"]& \mR\arrow[d,"\tilde{\phi}"]\arrow[l]\arrow[r,"i'"] & \mU\arrow[d]\\
    \mY & \mR'\arrow[r,"i"]\arrow[l] & \mU'
\end{tikzcd}\]
where the left horizontal arrows are smooth and the right horizontal arrows are closed immersions. The natural diagram with exact lines:
\[\begin{tikzcd}
    0\arrow[r] &\tilde{\phi}^{-1}(\mS_\mY|_{\mR'})\arrow[r,"\iota_{\mR',\mU'}"]\arrow[d,"\phi^\star"] & \tilde{\phi}^{-1}(\frac{(i')^{-1}(\mathcal{O}_{\mU'})}{I^2_{\mR',\mU'}})\arrow[r]\arrow[d] & \phi^{-1}(\frac{i^{-1}(T^\ast \mU')}{I_{\mR',\mU'}\cdot (i')^{-1}(T^\ast \mU')})\arrow[d]\\
    0\arrow[r] &\mS_\mX|_\mR\arrow[r,"\iota_{\mR,\mU}"] & \frac{i^{-1}(\mathcal{O}_\mU)}{I^2_{\mR,\mU}}\arrow[r] & \frac{i^{-1}(T^\ast \mU)}{I_{\mR,\mU}\cdot i^{-1}(T^\ast \mU)}
\end{tikzcd}\]
is commutative. In particular, taking $\phi=Id$ (hence $\phi^\star=Id$) but $\tilde{\phi}$ nontrivial, one obtains a stacky analogue of \cite[Theorem 2.1 ii)]{Joyce2013ACM} expressing how $\mS_\mX$ is glued.
\end{enumerate}
\end{lemma}

\begin{proof}
    In \cite[Theorem 2.1, Proposition 2.3]{Joyce2013ACM}, these are the defining property of $\mS_X$ and $\phi^\star$, for $X$ a scheme and $\phi:X\to Y$ a morphism of schemes, where one restricts to open immersions $R\to X$, and closed immersions of schemes $i:R\to U$. In \cite[Corollary 2.52]{Joyce2013ACM}, for $\mX$ a stack, $\mS_\mX$ is defined by smooth descent from the scheme case. Namely, by definition, for any smooth presentation $X\to\mX$ by a scheme, Joyce defines $(\mS_\mX)|_X:=\mS_X$, and for $\phi:X\to X'$ a morphism of smooth presentation by schemes, Joyce use the gluing morphism:
    \begin{align}\label{gluingmor}
        \phi^{-1}((\mS_\mX)|_{X'}):=\phi^{-1}(\mS_{X'})\overset{\phi^\star}{\to} \mS_X=:(\mS_\mX)|_X
    \end{align}
    The result of the lemma follows then formally by smooth descent.\medskip
    
    Consider a smooth map  $\mR\to \mX$ from a $k$-stack, and a closed embedding $i:\mR\hookrightarrow \mU$ into a smooth $k$-stack. Consider a smooth covering $U\to\mU$ by a scheme, and consider the closed immersion $\tilde{i}:R\to U$ obtained by base change. From the definition \cite[Corollary 2.52]{Joyce2013ACM}, $\mS_\mX|_R$ fits into an exact sequence:
    \[\begin{tikzcd}
    0\arrow[r] &\mS_\mX|_R\arrow[r,"\iota_{R,U}"] & \frac{\tilde{i}^{-1}(\mathcal{O}_{U})}{I^2_{R,U}}\arrow[r,"d"] & \frac{\tilde{i}^{-1}(T^\ast U_k)}{I_{R,U}\cdot \tilde{i}^{-1}(T^\ast U)}
\end{tikzcd}\]
    Consider now a smooth covering $U'\to U\times_\mU U$ by a (smooth) scheme, the induced closed immersion $i':R'\to U'$ of schemes, and the smooth morphism covering $R'\to R\times_\mR R$. Still from the definition, the exact sequence associated to $R\times_\mR R\to\mX$ is isomorphic to the pullbacks of the first exact sequence along the two projections $p_i:R\times_\mR R\to R$, hence, by smooth descent, one obtains an exact sequence as claimed. For two smooth cover $U,U'\to\mU$, using a smooth cover $U''\to U\times_\mU U'$ and a similar argument, one obtains that this does not depends on the choice of the smooth presentation.\medskip

    We build $\phi^\star$ by smooth descent: given smooth presentations $X\to\mX$, $Y\to \mY$ by schemes, with a morphism $\tilde{\phi}:X\to Y$ over $\phi$, we define:
\begin{align}
    \phi^\ast|_X:\tilde{\phi}^{-1}((\mS_\mY)|_Y):=\tilde{\phi}^{-1}(\mS_Y)\overset{\tilde{\phi}^\star}{\to} \mS_X=:(\mS_\mX)|_X
\end{align}
from the compatibility of $\phi^\star$ with composition in the scheme case (\cite[Proposition 2.3]{Joyce2013ACM}), one obtains that these are compatible with the gluing morphisms \eqref{gluingmor} hence defines a morphism $\phi^\star:\phi^{-1}(\mS_\mY)\to\mS_\mX$. The corresponding result in the scheme case implies the compatibility with compositions and the fact that $Id^\star=Id$.\medskip

Given a diagram as in the proposition, we consider a smooth morphism of schemes $\tilde{U}\to \tilde{U}'$ inducing smooth coverings of $\mU$ and $\mU'$. We define smooth covering $\tilde{R}'\to\mR'$, $\tilde{R}\to\mR$ by base change. it suffices to check the commutativity of the diagram of the Lemma after a pullback to $\tilde{R}$, but, using the definitions, it is the commutative diagram of introduced in \cite[Proposition 2.3]{Joyce2013ACM}for the schemes $\tilde{R},\tilde{U},\tilde{R}',\tilde{U}'$.
\end{proof}

In particular, there is a natural map $\mS_{\mX}\to\mathcal{O}_\mX$, sending $f$ to $f|_\mR$ for each $f\in \frac{i^{-1}(\mathcal{O}_\mU)}{I^2_{\mR,\mU}}$. Notice that each function $f\in i^{-1}(\mathcal{O}_\mU)$ satisfying $df|_{\mR}=0$ is locally constant on $\mR^{red}$ (but not necessarily on $\mR$). Hence, defining $\mS^0_{\mX}$ to be the kernel of $\mS_{\mX}\to\mathcal{O}_\mX\to\mathcal{O}_{\mX^{red}}$, there is a natural splitting $\mS_{\mX}=\mS^0_{\mX}\oplus k_{\mX}$, see \cite[Theo 2.1 a), Cor 2.52 b)]{Joyce2013ACM}. As in the scheme case (\cite[Proposition 2.3]{Joyce2013ACM}), one obtains directly that $\phi^\star$ maps $\phi^{-1}(\mS_\mY^0)\to \mS_\mX^0$.\medskip

We reformulate then Joyce's definition, which is obtained by combining \cite[Def 2.5, Def 2.53]{Joyce2013ACM}, (notice that we allows us to consider stacky critical charts, when Joyce consider only schematic ones, but our definition of d-critical stack is the same):

\begin{definition}
    \begin{enumerate}
        \item[$i)$] For $\mX$ a stack and a global section $s\in H^0(\mS^0)$, a (stacky) critical chart $(\mR,\mU,f,i)$ is the data of a smooth map $\mR\to \mX$, a closed immersion $i:\mR\to\mU$ into a smooth stack $\mU$, and a function $f:\mU\to\bA^1_k$ such that $f+I^2_{\mR,\mU}=s|_\mR$ and $\Crit(f)=\mR$ (notice that one has a priori by definition of $s$ that $\mR\subset\Crit(f)$).
        \item[$ii)$] Given a critical chart $(\mR,\mU,f,i)$ and a smooth map $\phi:\mU'\to\mU$, one obtains by base change a closed immersion $i':\mR'\to\mU'$ into a smooth stack $\mU'$, a smooth map $\mR'\to\mR\to\mX$, and a function $f':=f\circ \Phi:\mU'\to\bA^1_k$, giving a critical chart $(\mR',\mU',f',i')$, which is called a smooth restriction of the critical chart $(\mR,\mU,f,i)$.
        \item[$iii)$] $s$ is said to be a $d$-critical structure if there is a collection of (stacky) critical charts $(\mR,\mU,f,i)$ such that the smooth maps $\mR\to\mX$ are jointly surjective (one says in this case that these (stacky) critical charts cover $\mX$).
    \end{enumerate}
\end{definition}

\begin{lemma}
    The definition of a d-critical structure on a stack given above coincide with Joyce's definition \cite[Definition 2.53]{Joyce2013ACM}.
\end{lemma}

\begin{proof}
    In \cite[Definition 2.53]{Joyce2013ACM}, a global section $s$ of $\mS^0_\mX$ on a $k$-Artin stack $\mX$ is said to be d-critical structure if, for each smooth map $X\to\mX$ from a $k$-scheme (or equivalently, for a single smooth cover, using \cite[Proposition 2.8]{Joyce2013ACM}), there is near each each point $x\in X(k)$ a Zariski-open subset $R$ containing $x$ and a critical chart $(R,U,f,i)$. Notice that it implies in particular that $\mX$ is covered by d-critical charts in our sense. Now, consider a covering of $\mX$ by d-critical charts $(\mR,\mU,f,i)$: we consider a cover $U'\to\mU$ by a scheme, and consider the smooth restriction $(R',U',f',i')$, such that the $\phi:R'\to\mX$ cover $\mX$. Then $R'$ is a global critical locus, $\phi^\star(s)$ is a d-critical structure on the scheme $R'$ in the sense of \cite[Definition 2.5]{Joyce2013ACM}, hence $s$ is a d-critical structure on $s$ in the sense of \cite[Definition 2.53]{Joyce2013ACM}.
\end{proof}

Moreover, the structure of $d$-critical stack is extremely powerful, and gives a way to build a lot of critical charts, thanks to this proposition:

\begin{proposition} (Joyce, \cite[Proposition 2.7, Proposition 2.8]{Joyce2013ACM})\label{proplocchart}
\begin{itemize}
    \item Consider a d-critical scheme $(R,s)$ with a closed immersion $i:R\to U$ into a smooth $k$-scheme $U$, a point $x\in R(k)$ with $\dim T_x R=\dim(U)$, and a function $f:U\to\bA^1_k$ such that $s=f+I_{R,U}^2$. There is a Zariski-open subset $R'$ containing $x$, a Zariski open subset $U'$ of $U$ containing $R'$, such that denoting $i':R'\to U'$ the closed embedding and $f'=f|_{R'}$, $(R',U',f',i')$ is a critical chart.
    \item Given a smooth morphism $\phi:\mX\to\mY$, if $(\mY,s)$ is a d-critical stack, then $(\mX,\phi^\star(s))$ is a d-critical stack (the converse being true if $\phi$ is moreover surjective). We call $\phi:(\mX,\phi^\star(s))\to (\mY,s)$ a smooth morphism of d-critical stacks, and we obtain a 2-category of d-critical stacks, with smooth morphisms.
\end{itemize}
    
\end{proposition}

\begin{proof}
    The first part is \cite[Proposition 2.7]{Joyce2013ACM}. In \cite[Proposition 2.8]{Joyce2013ACM}, the second result is proven for schemes, but, given two smooth (resp. and surjective) presentation $X\to\mX$, $Y\to\mY$ and a smooth map $\tilde{\phi}:X\to Y$ over $\phi$, one has that $(X,\tilde{\phi}^\star(s|_Y)=\phi^\star(s)|_X)$ is a d-critical scheme if (resp. if and only if) $(Y,s|_Y)$ is a d-critical scheme. But, from the definition, $(\mX,\phi^\star(s))$ (resp $(\mY,s)$) is a d-critical stack if and only if $(X,\phi^\star(s)|_X)$ (resp. $(Y,s|_Y)$) is a d-critical scheme, which gives the claim.
\end{proof}

\subsection{Canonical bundle on d-critical stack}

Consider a smooth stack $\mU$, a function $f:\mU\to\bA^1_k$ with critical locus $i:\mR\to\mU$, and a quadratic bundle $(\mE,q)$ over $\mU$. The critical locus of $f\circ p+q:\bV_\mU(\mE)\to\bA^1_k$ is simply given by the $0$ section of $\mR$. Consider now a smooth morphism $\mR\to\mX$. We have obviously:
\begin{align}
    f\circ p+q\in f\circ p+(\mathcal{I}_{\mR,\bV_\mU(\mE)})^2
\end{align}
hence from Lemma \ref{defSx} $ii)$ applied to $p$, we have that, if $(\mR,\mU,f,i)$ is a critical chart for $(\mX,s)$, then $(\mR,\bV_\mU(\mE),f\circ p+q,s\circ i)$ is also a critical chart for $(\mX,s)$. We will call such operations stabilization by quadratic bundle.\medskip

For $\mX\to\mY$ a smooth morphism of stacks, we consider the canonical line bundle $K_{\mX/\mY}:=\det(\bL_{\mX/\mY})$. In \cite[Theo 2.28, Theorem 2.56]{Joyce2013ACM}, for $(\mX,s)$ a d-critical stack, Joyce builds also a line bundle $K_{\mX,s}$ on $\mX^{red}$, the orientation bundle, and defines an orientation data to be a choice of a square root $K_{\mX,s}^{1/2}$ (with a choice of isomorphism $(K_{\mX,s}^{1/2})^{\otimes 2}\simeq K_{\mX,s}$). We will use the following characterization:

\begin{lemma}\label{defcanbun}
  The line bundle $K_{\mX,s}$ on $\mX^{red}$ build in \cite[Theo 2.28, Theorem 2.56]{Joyce2013ACM} is uniquely characterized by the following properties:
\begin{enumerate}
    \item[$i)$]  For each critical chart $(\mR,\mU,f,i)$, there is a natural isomorphism:
\begin{align}\label{eqcanbun}
    K_{\mX,s}|_{\mR^{red}}\simeq i^\ast(K_\mU^{\otimes 2})|_{\mR^{red}}\otimes (K_{\mR/\mX})|^{\otimes -2}_{\mR^{red}}
\end{align}
Given an orientation data $K_{\mX,s}^{1/2}$, the $\bZ/2\bZ$ bundle $Q_{\mR,\mU,f,i}$ over $\mR^{red}$ is defined to be the bundle of local isomorphisms:
\begin{align}
    K_{\mX,s}^{1/2}|_{\mR^{red}}\simeq i^\ast(K_\mU)|_{\mR^{red}}\otimes (K_{\mR/\mX})|^{\otimes -1}_{\mR^{red}}
\end{align}
which are square roots of the above isomorphism.
\item[$ii)$] Given a smooth restriction of critical charts $(\mR',\mU',f',i')\to(\mR,\mU,f,i)$, the isomorphism \eqref{eqcanbun} for $(\mR',\mU',f,i)$ is obtained by smooth restriction from the isomorphism for $(\mR,\mU,f,i)$. In the presence of an orientation data $K_{\mX,s}^{1/2}$, this gives a canonical isomorphism:
\begin{align}
    Q_{\mR',\mU',f',i'}|_{(\mR')^{red}}\simeq Q_{\mR,\mU,f,i}
\end{align}

\item[$iii)$] Given a critical chart $(\mR,\mU,f,i)$ and a quadratic bundle $(\mE,q)$ on $\mU$, the following isomorphism:
\begin{align}
    i^\ast(K_\mU^{\otimes 2})|_{\mR^{red}}\otimes (K_{\mR/\mX})|^{\otimes -2}_{\mR^{red}}\simeq  K_{\mX,s}|_{\mR^{red}}\simeq (s\circ i)^\ast(K_{\bV_\mU(\mE)}^{\otimes 2})|_{\mR^{red}}\otimes (K_{\mR/\mX})|^{\otimes -2}_{\mR^{red}}
\end{align}
is the isomorphism induced by $\det(q):\det(\mE)^2\simeq \mathcal{O}_\mX$. In the presence of an orientation data $K_{\mX,s}^{1/2}$, this gives a canonical isomorphism:
\begin{align}\label{isomQstab}
    Q_{\mR,\bV_\mU(\mE),f\circ p+q,s\circ i}\otimes_{\bZ/2\bZ}P_{(\mE,q)}\simeq Q_{\mR,\mU,f,i}
\end{align}
\end{enumerate}
\end{lemma}

\begin{proof}
Notice that \cite[Theo 2.28, Theorem 2.56]{Joyce2013ACM} characterize uniquely $K_{\mX,s}$ in terms of the above properties for schematic critical charts, hence those properties characterize $K_{\mX,s}$ too.

\begin{enumerate}
    \item[$i)$] Consider a critical chart $(\mR,\mU,f,i)$, and take a smooth presentation $\tilde{U}\to\mU$ by a smooth scheme, and consider the smooth restriction of critical charts $(\tilde{R},\tilde{U},\tilde{f},\tilde{i})\to(\mR,\mU,f,i)$. From \cite[Theorem 2.56 b), Theorem 2.28 i)]{Joyce2013ACM}, there are natural isomorphisms:
    \begin{align}
    K_{\mX,s}|_{\tilde{R}^{red}}\simeq K_{\tilde{R},s|_{\tilde{R}}}\otimes (K_{\tilde{R}/\mX})|^{\otimes -2}_{\tilde{R}^{red}}\simeq i^\ast(K_{\tilde{U}}^{\otimes 2})|_{\tilde{R}^{red}}\otimes (K_{\tilde{R}/\mX})|^{\otimes -2}_{R^{red}}
    \end{align}
    Consider a smooth covering of the diagonal $\tilde{U}'\to \tilde{U}\times_\mU \tilde{U}$: from \cite[Proposition 2.30]{Joyce2013ACM}, the pullbacks of this isomorphisms along the two projections $\tilde{R}'\to \tilde{R}$ agree with the similar isomorphism for $\tilde{R}'\to\mX$, hence by smooth descent one obtains an isomorphism \eqref{eqcanbun}. Considering two covering $\tilde{U},\tilde{U}'\to\mU$, using the same argument for a smooth covering $\tilde{U}''\to \tilde{U}\times_\mU \tilde{U}'$, one obtains that this isomorphism does not depends on the choice of the smooth presentation.\medskip

    \item[$ii)$] We wean that the isomorphism:
    \begin{align}\label{isomsmoothcanbun}
        K_{\mX,s}|_{(\mR')^{red}}&\simeq K_{\mX,s}|_{\mR^{red}}|_{(\mR')^{red}}\nn\\
        &\simeq i^\ast (K_\mU^{\otimes 2})|_{(\mR')^{red}}\otimes (K_{\mR/\mX})|^{\otimes -2}_{(\mR')^{red}}\nn\\
        &\simeq (i')^\ast (K_{\mU'}^{\otimes 2}\otimes K_{\mU'/\mU}^{\otimes -2})|_{(\mR')^{red}}\otimes K_{\mR/\mX}|_{(\mR')^{red}}\nn\\
        &\simeq (i')^\ast (K_{\mU'}^{\otimes 2})|_{(\mR')^{red}}\otimes K_{\mR'/\mR})|^{\otimes -2}_{(\mR')^{red}}\otimes K_{\mR/\mX}|^{\otimes -2}_{(\mR')^{red}}\nn\\
        &\simeq (i')^\ast (K_{\mU'}^{\otimes 2})|_{(\mR')^{red}}\otimes K_{\mR'/\mX}|^{\otimes -2}_{(\mR')^{red}}
    \end{align}
    (where the second line is the isomorphism of $i)$ for $(\mR,\mU,f,i)$) is the isomorphism of $i)$ for $(\mR',\mU',f',i')$. Consider a smooth cover $\tilde{U}\to \mU'$, which gives by composition a smooth morphism $\tilde{U}\to\mU$: we can build the isomorphism of $i)$ for $(\mR',\mU',f',i')$ and $(\mR,\mU,f,i)$ using those presentations, hence the result follows directly.\medskip

    \item[$iii)$] Consider a smooth cover by schemes $g_i:U_i\to\mU$ such that $(g_i)^\ast(\mE,q)$ is isomorphic to trivial quadratic forms (this can be done étale locally on scheme, and then smooth locally on stacks). The claimed result is obtained for $(g_i)^\ast(\mE,q)$ by \cite[Proposition 2.25 b), Theorem 2.28 $ii)$]{Joyce2013ACM}.  Building the isomorphism of $i)$ with the smooth covers $g_i$ and $g'_i:\bV_{U_i}((g_i)^\ast(\mE))\to\bV_\mU(\mE)$, we obtain the claimed result.
\end{enumerate}
\end{proof}

The above characterization gives directly that, for $f:(\mX,f^\star(s))\to(\mY,s)$ a smooth morphism of $d$-critical stacks, one obtains a functorial isomorphism $K_{\mX,f^\star(s)}\simeq K_{\mY,s}|_{\mX^{red}}\otimes K_{\mX/\mY}|_{\mX^{red}}^{\otimes 2}$. We obtain then a $2$-category of oriented d-critical stacks, whose objects are oriented d-critical stacks, 1-morphisms are smooth morphisms of d-critical stacks $\phi:(\mX,\phi^\star(s))\to (\mY,s)$ with the data of an isomorphism of square roots:
\begin{align}
    K_{\mX,f^\star(s)}^{1/2}\simeq K_{\mY,s}^{red}|_{\mX^{red}}\otimes K_{\mX/\mY}|_{\mX^{red}}
\end{align}
and 2-morphisms are 2-morphisms of the category of stacks, such that the two isomorphisms of square roots agree.

\begin{lemma}\label{lemsymmon}
    There is a symmetric monoidal product on the $2$-category of d-critical stacks with smooth morphisms, given by $(\mX_1,s_1)\times(\mX_2,s_2):=(\mX_1\times\mX_2,s_1\boxplus s_2)$, where $s\boxplus t:=(p_1)^\star s+(p_2)^\star t$. If $(\mR_i,\mU_i,f_i,i_i)$ are critical charts for $\mX_i$, $(\mR_1\times\mR_2,\mU_1\times\mU_2,f_1\boxplus f_2,i_1\times i_2)$ is a critical chart for $(\mX_1\times\mX_2,s_1\boxplus s_2)$. Moreover, there is a symmetric monoidal isomorphism $K_{\mX_1\times\mX_2,s_1\boxplus s_2}\simeq K_{\mX_1,s_1}\boxtimes K_{\mX_2,s_2}$, compatible with the isomorphisms \eqref{eqcanbun}. It gives a symmetric monoidal structure to the 2-category of oriented d-critical stacks with smooth morphisms. In particular, given oriented d-critical stacks $(\mX_i,s_i,K_{\mX_i,s_i}^{1/2})$, taking the product $(\mX_1\times\mX_2,s_1\boxplus s_2,K_{\mX_1,s_1}^{1/2}\boxtimes K_{\mX_2,s_2}^{1/2})$, there is a symmetric monoidal isomorphism:
    \begin{align}
        Q_{(\mR_1\times\mR_2,\mU_1\times\mU_2,f_1\boxplus f_2,i_1\times i_2)}\simeq Q_{\mR_1,\mU_1,f_1,i_1}\boxplus Q_{\mR_2,\mU_2,f_2,i_2}
    \end{align}
\end{lemma}

\begin{proof}
    In the scheme case, look at \cite[Proposition 2.11]{Joyce2013ACM} for a similar statement. Notice that, from the cofiber sequence $(f_1\times f_2)^\ast\bL_\mu\to \bL_{f_1}\boxtimes\bL_{f_2}\to\bL_{f_1\boxplus f_2}$, we have $\Crit(f_1\boxplus f_2)=\Crit(f_1)\times\Crit(f_2)$. Hence, if $(\mR_i,\mU_i,f_i,i_i)$ are critical charts for $(\mX_i,s_i)$, $(\mR_1\times\mR_2,\mU_1\times\mU_2,f_1\boxplus f_2,i_1\times i_2)$ is a critical chart for $(\mX_1\times\mX_2,s_1\boxplus s_2)$. This implies that $(\mX_1\times\mX_2,s_1\boxplus s_2)$ is a d-critical stack if $(\mX_i,s_i)$ are. The product is obviously monoidal.\medskip

    By definition, a $d$-critical stack is covered by schematic critical charts. From the results of \cite[Section 2.3]{Joyce2013ACM}, two schematic critical charts can be related by étale restrictions and stabilization by quadratic form. It means that $(\mX_1\times\mX_2,s_1\boxplus s_2)$ can be covered by products of critical charts on the $(\mX_i,s_i)$, and two such product can be related by products of smooth restrictions and stabilization by quadratic forms. It means the conditions of Lemma \ref{defcanbun} for $(\mX_1\times\mX_2,s_1\boxplus s_2)$ restricted to charts and stabilization by quadratic form being products of charts and stabilization by quadratic form for $K_{\mX_i,s_i}$ determines uniquely $K_{\mX_1\times\mX_2,s_1\boxplus s_2}\simeq K_{\mX_1,s_1}\boxtimes K_{\mX_2,s_2}$. But the conditions of Lemma \ref{defcanbun} are clearly monoidal with respect to product of critical charts, hence this defines canonically a monoidal isomorphism $K_{\mX_1\times\mX_2,s_1\boxplus s_2}\simeq K_{\mX_1,s_1}\boxtimes K_{\mX_2,s_2}$.
\end{proof}

\subsection{Shifted symplectic stacks and Darboux theorem}

The main motivation for the introduction of $d$-critical stacks is that, from the "Darboux theorem" of \cite{darbscheme}, \cite{darbstack}, they forms a classical truncation of $-1$ shifted symplectic stacks. We will recall this now.\medskip

In \cite[Definition 1.12]{shifsymp}, the authors define for each derived Artin $k$-stack $\ud{\mX}$, and integers $p,n$, the space $\mA^p_k(\ud{\mX},n)$ of $p$-forms of degree $n$ on $\ud{\mX}$, and the space $\mA^{p,cl}_k(\ud{\mX},n)$ of closed $p$-forms of degree $n$ on $\ud{\mX}$, and a natural map $\mA^{p,cl}_k(\ud{\mX},n)\to \mA^p_k(\ud{\mX},n)$ sending a closed form to the underlying form. In particular, in derived geometry, the fact of being closed is not a property, but an extra structure. By \cite[Proposition 1.14]{shifsymp}, there is a natural equivalence:
\begin{align}
    \mA^p_k(\ud{\mX},n)\simeq \Map_{L_{qcoh}(\ud{\mX})}(\mathcal{O}_{\ud{\mX}},\wedge^p\bL_{\ud{\mX}}[n])
\end{align}
In particular, a $2$-form $\omega$ of degree $n$ induces a map $\mathcal{O}_{\ud{\mX}}\to \wedge^2\bL_{\ud{\mX}}[n]$, \ie a morphism:
\begin{align}
    \mathbb{T}_{\ud{\mX}}\to \bL_{\ud{\mX}}[n]
\end{align}
and it is said to be nondegenerate if this is an isomorphism. By \cite[Definition 1.18]{shifsymp}, a $n$-shifted symplectic structure $\omega$ is a closed $2$-form of degree $n$ whose underlying $2$-form of degree $n$ is nondegenerate. In particular, a $-1$-shifted symplectic structure on a derived scheme induces a symmetric perfect obstruction theory on the classical truncation.\medskip

Consider a function $f:\mU\to\bA^1$ on a smooth stack $\mU$: one obtains from \cite[Corollary 2.11]{shifsymp} a canonical $-1$-shifted symplectic structure $\omega_{\ud{\Crit}(f)}$ on $\ud{\Crit}(f):=\mU\overset{h}{\times}_{0,T^\ast\mU,df}\mU$, obtained by considering it as an intersection of $0$-Lagrangians in the $0$-symplectic stack $T^\ast\mU$, and applying \cite[Corollary 2.10]{shifsymp}.\medskip

In \cite[Theorem 3.18]{darbstack} (based on \cite[Theorem 6.6]{darbscheme}), the authors build, for any derived Artin stack $\ud{\mX}$ with a $-1$-shifted symplectic structure $\omega$, a natural $d$-critical structure $s\in H^0(\mS^0_\mX)$ on its classical truncation $\mX$. From \cite[Theorem 2.10]{darbstack}, $\ud{\mX}$ is covered by smooth maps $\ud{\phi}:\ud{R}\to \ud{\mX}$, such that there is a smooth scheme $U$, a function $f:U\to\bA^1_k$ and a map $\ud{j}: \ud{R}\to\ud{\Crit(f)}$ inducing an isomorphism of classical schemes such that $\ud{j}^\ast(\omega_{\ud{\Crit(f)}})\sim \ud{\phi}^\ast(\omega)$. According to \cite[Theorem 3.18 a)]{darbstack}, denoting by $\phi:R\to\mX$ the classical truncation of $\ud{\phi}:\ud{R}\to \ud{\mX}$, $s$ is uniquely defined by the condition that $s|_R=f+I_{R,U}^2$ for any such data, \ie $(R,U,f,i)$ is a critical chart for $(\mX,s)$. According to \cite[Theorem 3.18 b)]{darbstack}, $K_{\mX,s}$ is naturally isomorphic to $\det(\bL_{\ud{\mX}})|_{\mX^{red}}$.\medskip

The structure of d-critical stack is sufficient to define the DT sheaf, hence we will mainly work at this level. However, for stacks appearing in moduli problems, in general it is more natural to give them a $-1$-shifted symplectic structure using general results of \cite{shifsymp}, \cite{Brav2018RelativeCS} (look for example at Section \ref{secStackComp}), and then applies the Darboux theorem to obtain a d-critical structure, then we will have to work a bit at the $-1$-shifted symplectic level to obtain some compatibility results.

\subsection{The DT sheaf on critical charts}

\subsubsection{Definition}

Consider a critical chart $(\mR,\mU,f,i)$, \ie a smooth stack $\mU$ with a function $f:\mU\to\bA^1_k$, and the closed immersion of the critical locus $i:\mR\to\mU$. Consider $\phi_f^{mon,tot}\un_\mU\in \mA_{mon,c}(\mU)[-2d_\mU]$: one can check using a smooth presentation and the Betti representation that it is supported on $\mR$, hence its restriction to $\mR$ is still in the shift of the heart of the perverse $t$-structure, we define then:
\begin{align}
    P_{\mU,f}:=i^\ast\{d_\mU/2\}\phi_f^{mon,tot}\un_\mU\in\mA_{mon,c}(\mR)
\end{align}

\begin{remark}\label{remarkstrat}
    We see from the definition that $P_{\mU,f}$ is obtained by applying successively three functors to $\un_{\Spec(k)}$: firstly, the functor $(\mU\to\Spec(k))^\ast$, secondly $\phi_f^{mon,tot}$, and lastly $i^\ast\{d_\mU/2\}$: let's denote these temporarily by $a_3,a_2,a_1$ for simplicity. We will build various compatibility isomorphisms, and check compatibility between these isomorphisms (namely, compatibility with respect to smooth pullbacks, products, stabilization and hyperbolic localization). We denote by $\alpha_{b,c}:bc\simeq cb$ an isomorphism of commutation between the functors $b$ and $c$.
    Each of these isomorphism will be expressed as the commutation between a functor $f$ and $a_1a_2a_3$ obtained by a sequence:
    \[\begin{tikzcd}
        ba_1a_2a_3\arrow[r,"\alpha_{b,a_1}a_2a_3"] & a_1ba_2a_3\arrow[r,"a_1\alpha_{b,a_2}a_3"] & a_1a_2ba_3\arrow[r,"a_1a_2\alpha_{b,a_3}"] & a_1a_2a_3b
    \end{tikzcd}\]
    We will show compatibility, namely that for $b,c$ two functors, the following square is commutative:
    \[\begin{tikzcd}
        & bca_1a_2a_3\arrow[r,"\simeq"] & ba_1a_2a_3c\arrow[d,"\simeq"]\\
        cba_1a_2a_3\arrow[ur,"\alpha_{c,b}a_1a_2a_3"]\arrow[d,"\simeq"] && a_1a_2a_3bc\\
        ca_1a_2a_3b\arrow[r,"\simeq"] & a_1a_2a_3cb\arrow[ur,"a_1a_2a_3\alpha_{c,b}"]
    \end{tikzcd}\]
    To prove that, it will suffice to prove that for each $1\leq i\leq 3$, the following diagram is commutative:
    \[\begin{tikzcd}
        & bca_i\arrow[r,"b\alpha_{c,a_i}"] & ba_ic\arrow[d,"\alpha_{b,a_i}c"]\\
cba_i\arrow[ur,"\alpha_{c,b}a_i"]\arrow[d,"\simeq"] && a_ibc\\
        ca_ib\arrow[r,"\alpha_{c,a_i}"] & a_icb\arrow[ur,"a_i\alpha_{c,b}"]
    \end{tikzcd}\]
    We say that the isomorphism of commutation between $a_i$ and $b$ and between $a_i$ and $c$ are compatible. In general, it will follows from the functoriality of the six functor formalism, or from functoriality results established in previous sections.
\end{remark}

\subsubsection{Smooth restrictions of critical loci}\label{sectisomliss}

Recall that smooth restriction of stacky critical loci $\phi:(\mU',f')\to(\mU,f)$ is the data the data of a smooth morphism $\phi:\mU'\to \mU$ such that $f'=f\circ\phi$. Because $\phi$ is smooth, we have a Cartesian square:
\[\begin{tikzcd}
\mR'\arrow[r,"i'"]\arrow[d,"\tilde{\phi}"] & \mU'\arrow[d,"\phi"]\\
\mR\arrow[r,"i"]& \mU
\end{tikzcd}\]
We have then a natural sequence of isomorphisms:
\begin{align}\label{natsmooth}
    \tilde{\phi}^\ast\{d_\phi/2\} P_{\mU,f}&:=\tilde{\phi}^\ast\{d_\phi/2\}i^\ast\{d_\mU/2\}\phi_f^{mon,tot}\un_\mU\nn\\
    &\simeq (i')^\ast\{d_{\mU'}/2\}\phi^\ast\phi_f^{mon,tot}\un_\mU\nn\\
    &\simeq (i')^\ast\{d_{\mU'}/2\}\phi_{f'}^{mon,tot}\phi^\ast\un_\mU\nn\\
    &\simeq (i')^\ast\{d_{\mU'}/2\}\phi_{f'}^{mon,tot}\un_{\mU'}\nn\\
    &=:P_{\mU',f'}
\end{align}
where the first and third isomorphism are obtained by functoriality of pullbacks, and the second one by functoriality of the specialization system $\phi^{mon,tot}$ with respect to smooth pullbacks. By functoriality of the pullback, and by compatibility of specialization system with respect to composition of smooth pullbacks, this isomorphism is compatible with composition of smooth restriction of equivariant critical loci. It means that for $\phi':(\mU'',f'')\to(\mU',f')$ a smooth restriction of critical loci, the following square  of isomorphisms is commutative:
\[\begin{tikzcd}
    (\tilde{\phi}')^\ast\{d_{\phi'}/2\}\tilde{\phi}^\ast \{d_\phi/2\}P_{\mU,f}\arrow[r,"\simeq"]\arrow[d,"\simeq"] & (\tilde{\phi}')^\ast \{d_{\phi'}/2\}P_{\mU',f'}\arrow[d,"\simeq"]\\
(\tilde{\phi\circ\phi'})^\ast\{d_{\phi\circ\phi'}/2\}P_{\mU,f}\arrow[r,"\simeq"] & P_{\mU'',f''}
\end{tikzcd}\]
\medskip

\subsubsection{Product of critical loci}\label{sectproductcritloc}

Given two critical loci $(\mU_1,f_1)$ and $(\mU_2,f_2)$, one consider their product $(\mU_1\times \mU_2,f_1\boxplus f_2)$. One has then a natural isomorphism:
\begin{align}\label{thomsebcritchart}
    P_{\mU_1\times \mU_2,f_1\boxplus f_2}:=&(i_1\times i_2)^\ast\{d_{\mU_1\times \mU_2}/2\}\phi^{mon,tot}_{f_1\boxplus f_2}\un_{\mU_1\times \mU_2}\nn\\
    \simeq &(i_1\times i_2)^\ast\{d_{\mU_1\times \mU_2}/2\}\phi^{mon,tot}_{f_1\boxplus f_2}(\un_{\mU_1}\boxtimes\un_{\mU_2})\nn\\
    \simeq &(i_1\times i_2)^\ast\{d_{\mU_1\times \mU_2}/2\}(\phi^{mon,tot}_{f_1}\un_{\mU_1})\boxtimes(\phi^{mon,tot}_{f_1}\un_{\mU_2})\nn\\
    \simeq &((i_1)^\ast\{d_{\mU_1}/2\}\phi_{f_1}^{mon,tot}\un_{\mU_1})\boxtimes((i_2)^\ast\{d_{\mU_2}/2\}\phi_{f_2}^{mon,tot}\un_{\mU_2})\nn\\
    =:&P_{\mU_1,f_1}\boxtimes P_{\mU_2,f_2}
\end{align}
where the first and third isomorphisms follows from functoriality of pullbacks with respect to exterior products, and the second isomorphism is the Thom-Sebastiani isomorphism of Theorem \ref{theothomseb}. From the functoriality of pullbacks with respect to exterior products, and the functoriality of the Thom-Sebastiani isomorphism of Theorem \ref{theothomseb} with respect to smooth pullbacks (coming from the fact that it is an isomorphism of specialization systems), we obtain that for any smooth restriction $\phi_i:(\mU'_i,f'_i)\to (\mU_i,f_i)$, the following square of isomorphisms commutes:
\[\begin{tikzcd}
    (\tilde{\phi_1\times\phi_2})^\ast\{d_{\phi_1\times\phi_2}/2\} P_{\mU_1\times \mU_2,f_1\boxplus f_2}\arrow[r,"\simeq"]\arrow[d,"\simeq"] & P_{\mU'_1\times \mU'_2,f'_1\boxplus f'_2}\arrow[d]\\
    ((\tilde{\phi}_1)^\ast\{d_{\phi_1}/2\} P_{\mU_1,f_1})\boxtimes((\tilde{\phi}_2)^\ast \{d_{\phi_2}/2\}P_{\mU_2,f_2})\arrow[r,"\simeq"]& P_{\mU'_1,f'_1}\boxtimes P_{\mU'_2,f'_2}
\end{tikzcd}\]
Moreover, because pullbacks are symmetric monoidal functors, and the Thom-Sebastiani isomorphism of Theorem \ref{theothomseb} satisfies commutativity and associativity, we obtain that the above morphism satisfies commutativity and associativity.

\subsubsection{Stabilization by quadratic bundles}\label{sectstabquad}

As said above, given a critical chart $(\mR,\mU,f,i)$, and a quadratic bundle $(\mE,q)$ on $\mU$, one consider the critical chart $(\mR,\bV_\mU(\mE),f\circ \pi+q,s\circ i)$ obtained by stabilization by $(\mE,q)$. We obtain the following sequence of isomorphism:
\begin{align}
    P_{\bV_\mU(\mE),f\circ p+q}&:=(s\circ i)^\ast \{d_{\bV_\mU(\mU)}/2\}\phi^{mon,tot}_{f\circ p+q}\un_{\bV_\mU(\mE)}\nn\\
    &\simeq 
    i^\ast\{d_\mU/2+d_\mE/2\} s^\ast\phi^{mon,tot}_{f\circ p+q}\un_{\bV_\mU(\mE)}\nn\\
    &\simeq i^\ast\{d_\mU/2+d_\mE/2\}(\phi^{mon,tot}_f\un_\mU\otimes_{\bZ/2\bZ}P_{(\mE,q)}\{-d_\mE/2\})\nn\\
    &\simeq i^\ast\{d_\mU/2\}\phi^{mon,tot}_f\un_\mU\otimes_{\bZ/2\bZ}P_{(\mE,q)})|_{\mR^{red}}\nn\\
    &=: P_{\mU,f}\otimes_{\bZ/2\bZ}P_{(\mE,q)}|_{\mR^{red}}
\end{align}
where the third line comes from Lemma \ref{lemactquadbun}. From the compatibility with smooth pullbacks and exterior tensor products obtained in Lemma \ref{lemactquadbun}, we obtain then that the above isomorphism commutes also with smooth restrictions and exterior products. Namely, given $\phi:(\mR',\mU',f',i')\to (\mR,\mU,f,i)$, denoting $(\mE',q'):=\phi^\ast(\mE,q)$, the following square of isomorphisms commutes:
\[\begin{tikzcd}
    \tilde{\phi}^\ast\{d_\phi/2\}P_{\bV_\mU(\mE),f\circ p+q}\arrow[d,"\simeq"]\arrow[r,"\simeq"] & \tilde{\phi}^\ast\{d_\phi/2\}(P_{\mU,f}\otimes_ {\bZ/2\bZ}P_{(\mE,q)}|_{\mR^{red}})\arrow[d,"\simeq"]\\
    P_{\bV_{\mU'(\mE')},f'\circ p'+q'}\arrow[r,"\simeq"] & P_{\mU',f'}\otimes_ {\bZ/2\bZ}P_{(\mE',q')|_{\mR'^{red}}}
\end{tikzcd}\]
given two critical charts $(\mR_i,\mU_i,f_i,i_i)$, $i=1,2$, and quadratic bundles $(\mE_i,q_i)$ on $\mU_i$, denoting the following square of isomorphisms is commutative:
\[\begin{tikzcd}
    P_{\bV_{\mU_1\times\mU_2}(\mE_1\boxplus\mE_2),(f_1\boxplus f_2)\circ (p_1\times p_2)+(q_1\boxplus q_2)}\arrow[r,"\simeq"]\arrow[d,"\simeq"] & P_{\mU_1\times\mU_2,f_1\boxplus f_2}\otimes_{\bZ/2\bZ}P_{(\mE_1\boxplus\mE_2,q_1\boxplus q_2)}|_{(\mR_1\times\mR_2)^{red}}\arrow[d,"\simeq"]\\
    P_{\bV_{\mU_1}(\mE_1),f_1\circ p_1+q_1}\boxtimes P_{\bV_{\mU_2}(\mE_2),f_2\circ p_2+q_2}\arrow[r,"\simeq"] & (P_{\mU_1,f_1}\otimes_{\bZ/2\bZ}P_{(\mE_1,q_1)}|_{\mR_1^{red}})\boxtimes (P_{\mU_2,f_2}\otimes_{\bZ/2\bZ}P_{(\mE_2,q_2)}|_{\mR_2^{red}})
\end{tikzcd}\]
The compatibility with direct sum of quadratic bundles gives that, given two quadratic bundles $(\mE_i,q_i)$ on the same stack $\mU$, denoting $(\mE,q):=(\mE_1\oplus\mE_2,q_1\oplus q_2)$, the following square of isomorphisms is commutative:
\[\begin{tikzcd}
    P_{\bV_\mU(\mE),f\circ \pi+q}\arrow[r,"\simeq"]\arrow[d,"\simeq"] & P_{\mU,f}\otimes_{\bZ/2\bZ}P_{(\mE,q)}|_{\mR^{red}}\arrow[d,"\simeq"]\\
    P_{\bV_\mU(\mE_1),f\circ\pi_1+q_1}\otimes_{\bZ/2\bZ}P_{(\mE_2,q_2)}|_{\mR^{red}}\arrow[r,"\simeq"] & P_{\mU,f}\otimes_{\bZ/2\bZ}P_{(\mE_1,q_1)}\otimes_{\bZ/2\bZ}P_{(\mE_2,q_2)}|_{\mR^{red}}
\end{tikzcd}\]
\ie this isomorphism is compatible with composition of stabilizations by quadratic form.

Moreover, given a trivialization $(\mE,q)\simeq (\mathcal{O}_\mX^n,\sum_{i=1}^n x_i^2)$, this isomorphism is from Lemma \ref{lemactquadbun} simply the isomorphism considered in \cite{Joycesymstab}, obtained from Thom-Sebastiani:
\begin{align}
    P_{\mU\times\bA^n_k,f\boxplus \sum_{i=1}^n x_i^2}\simeq P_\mU\boxtimes_{mon}P_{\bA^n_k,\sum_{i=1}^n x_i^2}\simeq P_\mU
\end{align}

\subsection{Gluing of the DT sheaf}

The next proposition (which is a formal consequence of \cite[Theorem 4.8]{darbstack}) ensure that one can work with stacky critical charts, and arbitrary stabilization by quadratic bundle in cohomological DT theory:

\begin{proposition}\label{proplocdefperv}
    Consider an oriented d-critical stack $(\mX,s,K_{\mX,s}^{1/2})$, the object $P_{\mX,s,K_{\mX,s}^{1/2}}$ of $\mA_{mon,c}(\mX)$ built in \cite[Theorem 4.8]{darbstack} (building from \cite[theorem 6.9]{Joycesymstab}) is uniquely characterized by the following properties:
    \begin{enumerate}
        \item[$i)$] For any critical chart $(\mR,\mU,f,i)$, there is a natural isomorphism:
        \begin{align}
            P_{\mX,s,K_{\mX,s}^{1/2}}|_{\mR}\{d_{\mR/\mU}/2\}\simeq P_{\mU,f}\otimes_{\bZ/2\bZ} Q_{(\mR,\mU,f,i)}
        \end{align}.
        \item[$ii)$] Given a smooth restriction of critical charts $\phi:(\mR',\mU',f',i')\to (\mR,\mU,f,i)$, the following square of isomorphisms is commutative:
        \[\begin{tikzcd}
            \tilde{\phi}^\ast\{d_\phi/2\} P_{\mX,s,K_{\mX,s}^{1/2}}|_{\mR}\{d_{\mR/\mU}/2\}\arrow[r,"\simeq"]\arrow[d,"\simeq"] & \tilde{\phi}^\ast\{d_\phi/2\} (P_{\mU,f}\otimes_{\bZ/2\bZ} Q_{(\mR,\mU,f,i)})\arrow[d,"\simeq"]\\
            P_{\mX,s,K_{\mX,s}^{1/2}}|_{\mR'}\{d_{\mR'/\mU}/2\}\arrow[r,"\simeq"] & P_{\mU',f}\otimes_{\bZ/2\bZ} Q_{(\mR',\mU',f',i')}
        \end{tikzcd}\]
        where the horizontal arrows are the isomorphisms from $i)$, and the right vertical arrow is the tensor product of the isomorphism of Section \ref{sectisomliss} and Lemma \eqref{defcanbun} $ii)$.
        \item[$iii)$]  Given a critical chart $(\mR,\mU,f,i)$, and a quadratic bundle $(\mE,q)$, the following diagram of isomorphism is commutative:
        \[\begin{tikzcd}
    P_{\mX,s,K_{\mX,s}^{1/2}}|_{\mR}\{d_{\mR/\mX}/2\}\arrow[r,"\simeq"]\arrow[d,"\simeq"] & P_{\mU,f}\otimes_{\bZ/2\bZ} Q_{(\mR,\mU,f,i)}\arrow[d,"\simeq"]\\
            P_{\bV_\mU(\mE),f\circ\pi+q}\otimes_{\bZ/2\bZ}Q_{(\mR,\bV_\mU(\mE),f\circ p+q,s\circ i)}\arrow[r,"\simeq"] & P_{\mU,f}\otimes_{\bZ/2\bZ}P_{\mE,q}\otimes_{\bZ/2\bZ} Q_{(\mR,\bV_\mU(\mE),f\circ p+q,s\circ i)}
        \end{tikzcd}\]
        where the upper horizontal and left vertical arrows arrows are the isomorphism of $i)$ for $(\mR,\mU,f,i)$ and its stabilization $(\mR,\bV_\mU(\mE),f\circ p+q,s\circ i)$, the lower horizontal arrow is the isomorphism from Section \ref{sectstabquad}, and the right vertical arrow the isomorphism of Lemma \ref{defcanbun} $iii)$.
    \end{enumerate}
\end{proposition}

\begin{proof}
In \cite[theorem 6.9]{Joycesymstab} (the scheme case), the authors works with perverse sheaves and mixed Hodge modules, whence in \cite[Theorem 4.8]{darbstack} (the stack case), the authors works with perverse sheaves. Thanks to the development of the formalism of mixed Hodge modules and perverse Nori motives on Artin stacks in \cite{tubachMHM}, and the developement of the formalism of monodromic perverse Nori motives and Thom-Sebastiani in Section \ref{sectmonothomseb}, this construction extends automatically to the level of monodromic mixed Hodge modules and monodromic perverse Nori motives on stacks (more precisely, perverse Nori motives satisfies all the properties required in \cite[Section 2.5]{Joycesymstab} for the construction to work).
    \begin{enumerate}
        \item[$i)$] consider a critical chart $(\mR,\mU,f,i)$ of $(\mX,s,K_{\mX,s}^{1/2})$, and a smooth cover $\phi:U'\to\mU$ by a scheme. Consider the smooth restriction $\phi:(R',U',f',i')\to(\mR,\mU,f,i)$. From \cite[Theorem 4.8 a)]{darbstack} and \cite[Theorem 6.9 i)]{Joycesymstab}, there is a natural isomorphism:
        \begin{align}
            P_{\mX,s,K_{\mX,s}^{1/2}}|_{R'}\{d_{R'/\mX}/2\}\simeq P_{U',f'}\otimes_{\bZ/2\bZ} Q_{(R',U',f',i')}
        \end{align}
        Consider any smooth cover $U''\to U'\times_\mU U'$: from \cite[Theorem 4.8 b), Proposition 4.5 a)]{darbstack}, the two pullbacks along $R''\to R'$ of the last isomorphism agree with the last isomorphism for $(R'',U'',f'',i'')$, hence are equal, which means that the last isomorphism descend to the isomorphism of $i)$. Given an other choice of smooth cover $\phi:\tilde{U}'\to\mU$, using a smooth cover $U''\to U'\times_\mU\tilde{U}'$, a similar argument shows that this isomorphism is independent of the choice of $U'\to\mU$.\medskip

        \item[$ii)$] Consider a smooth cover $U''\to\mU'$ by a scheme, and the induced smooth morphism $U''\to\mU$. By the definition of the isomorphism of $i)$, the pullback of the square of $ii)$ along $R''\to\mR'$ is the identity, hence by conservativity of smooth pullbacks the square of $ii)$ is commutative.\medskip
        
        \item[$iii)$] Consider a smooth morphism $\phi:U'\to\mU$ from a scheme and a trivialization of $(\mE',q'):=\phi^\ast(\mE,q)\simeq (\mathcal{O}_\mU^n,\sum_{i=1}^n x_i^2)$. The square of $iii)$ for $(R',U',f',i')$ and $(\mE',q')$ is then:
        \[\begin{tikzcd}
    P_{\mX,s,K_{\mX,s}^{1/2}}|_{\mR'}\{d_{\mR'/\mX}/2\}\arrow[r,"\simeq"]\arrow[d,"\simeq"] & P_{U',f'}\otimes_{\bZ/2\bZ} Q_{(R',U',f',i')}\arrow[d,"\simeq"]\\
            P_{U\times\bA^n_k,f\boxplus \sum_{i=1}^n x_i^2}\otimes_{\bZ/2\bZ}Q_{(R',U'\times\bA^n_k,f\boxplus \sum_{i=1}^n x_i^2,i'\times 0)}\arrow[r,"\simeq"] & P_{U'\times\bA^n_k,f'\boxplus \sum_{i=1}^n x_i^2}\otimes_{\bZ/2\bZ} Q_{(R',U',f',i')}
        \end{tikzcd}\]
        It is an isomorphism from \cite[Theorem 6.9 ii), Theorem 5.4 a)]{Joycesymstab}. From $ii)$, the left vertical and upper horizontal arrows are obtained from pullback along $R'\to\mR$ from the corresponding arrows of the square of $iii)$ for $(\mR,\mU,f,i)$ and $(\mE,q)$. From Section \eqref{sectstabquad}, the isomorphism of stabilization by quadratic forms commutes with smooth pullbacks, and, using the fact that the trivialization of $(\mE',q')$ gives a section of $P_{\mE',q'}$, we obtain that the lower horizontal and right vertical arrows are obtained by pullback from the corresponding arrows of the square of $iii)$. By covering $\mU$ by such trivialization, using faithfullness of smooth pullbacks, one obtains that the square of $iii)$ is commutative.
    \end{enumerate}
\end{proof}

From the above characterization, one obtains directly:

\begin{corollary}\label{corfunctjoyce}
    The construction of $P_\mX$ from \cite[Theorem 4.8]{darbstack} enhance naturally to a symmetric monoidal functor from the 2-category of oriented d-critical stacks with smooth morphisms. Namely, given a smooth morphism of oriented $d$-critical charts $\phi:(\mX,s,K_{\mX,s}^{1/2})\to (\mY,t,K_{\mY,t}^1/2)$, there is a natural isomorphism in $\mA_{mon,c}(\mX)$:
    \begin{align}
        \phi^\ast\{d_\phi/2\}P_{\mY,t,K_{\mY,t}^1/2}\simeq P_{\mX,s,K_{\mX,s}^{1/2}}
    \end{align}
    compatible with composition, and given oriented d-critical stacks $(\mX_i,s_i,K_{\mX_i,s_i}^{1/2})$, there is a natural isomorphism in $\mA_{mon,c}(\mX_1\times\mX_2)$:
    \begin{align}\label{isomoprodjoyce}
        P_{\mX_1\times\mX_2,s_1\times s_2,K_{\mX_1,s_1}^{1/2}\otimes K_{\mX_2,s_2}^{1/2}}\simeq P_{\mX_1,s_1,K_{\mX_1,s_1}^{1/2}}\boxtimes_{mon}P_{\mX_2,s_2,K_{\mX_2,s_2}^{1/2}}
    \end{align}
    satisfying the obvious commutativity, associativity, and compatibility with smooth pullbacks.
\end{corollary}

\begin{proof}
    The functoriality with smooth morphisms follows directly from the unique characterization given above, as $\phi^\ast\{d_\phi/2\}P_{\mY,t,K_{\mY,t}^1/2}$ satisfies it with an obvious choice of isomorphism in $i)$, which is obviously compatible with composition. As recalled in the proof of Lemma \ref{lemsymmon}, $(\mX_1\times\mX_2,s_1\times s_2)$ can be covered by critical charts being products, and any two such critical charts can be compared by products of smooth restrictions and products of stabilization by quadratic bundles. Hence the restriction of $i),ii),iii)$ to such product characterize $P_{\mX_1\times\mX_2,s_1\times s_2,K_{\mX_1,s_1}^{1/2}\otimes K_{\mX_2,s_2}^{1/2}}$ uniquely. Using the isomorphism from Section \ref{sectproductcritloc}, one obtains that $P_{\mX_1,s_1,K_{\mX_1,s_1}^{1/2}}\boxtimes_{mon}P_{\mX_2,s_2,K_{\mX_2,s_2}^{1/2}}$ satisfies $i)$ with a natural choice of isomorphism, which is compatible with smooth restriction, hence satisfies $ii)$, from Section \ref{sectproductcritloc}, and is compatible with stabilization by quadratic forms, hence satisfies $iii)$, from Section \ref{sectstabquad}. It gives a natural choice of isomorphism \eqref{isomoprodjoyce}. From Section \ref{sectproductcritloc}, the isomorphisms of $i)$ are satisfies commutativity and associativity and are compatible with smooth pullbacks, hence \eqref{isomoprodjoyce} does it too.   
    \end{proof} 

Because the authors of \cite{Joycesymstab} and \cite{darbstack} have to build an object in a $1$-category, the proof contains three main steps, of increasing complexity:
\begin{itemize}
    \item Defining locally the perverse sheaf of vanishing cycles on a critical chart (with a careful treatment of orientation). This corresponds to point $i)$ of the above proposition
    \item Defining comparison isomorphisms locally on intersection of critical charts. This is done by comparing critical charts by étale (hence, smooth) restrictions and stabilization by quadratic forms (hence, stabilization by trivialized quadratic bundles), and then defining locally on intersections of critical charts gluing isomorphisms from the equations of point $ii)$ and $iii)$.
    \item Showing that these isomorphisms glue into global isomorphism on intersections of two critical charts, and that they satisfies cocycle condition on intersection of three critical charts.
\end{itemize}
The third step is certainly the harder and deeper part of the proof, using deep results on $\bA^1$-deformation of isomorphisms of critical charts, $\bA^1$-homotopy invariance, and where orientation data becomes crucial. Fortunately, once this is done, any construction or computation on cohomological DT theory can virtually be expressed as the construction of an (iso)-morphism between such objects, hence (as long as one manage to stay in the Abelian category $\mA_{mon,c}$) one has to consider only the two first steps, namely:
\begin{itemize}
    \item Defining locally the (iso)morphism, working on a critical chart.
    \item Showing that these (iso)morphisms agrees locally (hence globally, because they are morphisms in a $1$-category) on intersection of critical charts. it suffices to prove that they are compatible with smooth restrictions of critical charts, and stabilization by quadratic form.
\end{itemize}

Notice that the authors of \cite{Joycesymstab} and \cite{darbstack} works only with schematic critical charts. However, in order to study the $\Theta$ correspondence, this cannot works, as the $\Theta$-correspondence is trivial for schemes, hence the $\Grad$ of a d-critical stack cannot be covered by $\Grad$ of schematic critical charts. This is why we had to extend slightly the formalism of \cite{Joyce2013ACM} to consider stacky critical charts, and stabilization by quadratic bundles: see Proposition \ref{propcomparchart}. We could have consider only critical charts of the form $([R/\bG_m],[U/\bG_m],f,i)$, but we felt that adopting a more flexible formalism should be more natural and useful (for example, to use also étale critical charts of the quotient form, using étale-local structure results on stacks from \cite{Alper2015AL}, \cite{Alper2019TheL}). Notice that stabilization by quadratic bundles was also used in \cite{Bussi2019OnMV} to glue motives of vanishing cycles: for this, one has to work in the Nisnevich topology, and then one cannot trivialize quadratic bundles as in the Zariski topology. Stabilization by quadratic bundle was also used crucially in the homotopy-coherent formalism of \cite{Hennion2024GluingIO}.

\section{Hyperbolic localization in DT theory}\label{sectproof}

In this section, we work over an algebraically closed field $k$ of characteristic $0$. All our stacks are assumed to be quasi-separated Artin $1$-stacks locally of finite type over $k$, with affine stabilizers. Given an oriented d-critical stack $(\mX,s,K_{\mX,s}^{1/2})$, we will often abuse the notation by denoting $P_\mX$ for $P_{\mX,s,K_{\mX,s}^{1/2}}$, when the d-critical structure and the orientation are clear from the context.

\subsection{Hyperbolic localization on critical loci}

\subsubsection{Local definition of the isomorphism}

Consider a critical locus $(\mR,\mU,f,i)$. As recalled from Halpern-Leistner in Lemma \ref{lemcotanbun}, $\Grad(\mU)$ (resp. $\Filt(\mU)$) is also smooth, with cotangent complex $(\iota^\ast\bL_\mU)^0$ (resp. $(\eta^\ast\bL_\mU)^{\leq 0}$. We have the following result:

\begin{lemma}\label{lemcritgrad}
    $\Grad(i):\Grad(\mR)\to\Grad(\mU)$ is the closed immersion of the critical locus of $\Grad(f)$.
\end{lemma}

\begin{proof}
    Denote by $T^\ast\mU:=\bV_\mU(\bL_\mU^\vee)$ the cotangent stack of $\mU$. We have that $\Grad(T^\ast\mU)\simeq T^\ast \Grad(\mU)$, and the section $0,d\Grad(f):\Grad(\mU)\to T^\ast \Grad(\mU)$ identifies with $\Grad(0),\Grad(df):\Grad(\mU)\to \Grad(T^\ast\mU)$. Then, as $\Crit(f):=\mU\times_{0,T^\ast \mU,df}\mU$, we obtain directly, from the compatibility of the mapping construction with fiber product:
    \begin{align}
        \Grad(\Crit(f))=\Crit(\Grad(f))
    \end{align}
    Notice that the same reasoning, applied to the derived critical locus $\ud{\Crit}(f):=\mU\overset{h}{\times}_{0,T^\ast \mU,df}\mU$, gives also at the derived level:
    \begin{align}
        \ud{\Grad}(\ud{\Crit}(f))=\ud{\Crit}(\ud{\Grad}(f))
    \end{align}
\end{proof}

In particular, $(\Grad(\mR),\Grad(\mU),\Grad(f),\Grad(i))$ is a critical locus. We consider the locally constant function $\Ind_{\mU,f}:\Grad(\mR)\to \bZ$ defined by:
\begin{align}\label{inddplus}
    \Ind_{\mU,f}:=d^+_\mU|_{\Grad(\mR)}-d^-_\mU|_{\Grad(\mR)}
\end{align}
where $d^+_\mU$ (resp. $d^-_\mU$) denotes the virtual dimension of $\bL^{<0}_\mU$ (resp. $\bL^{>0}_\mU$) (see Section \ref{sectsmoothhyploc}). We choose this sign convention such that $d^+_\mU$ counts the virtual dimension of the positive weight part of the tangent stack. Notice that, when we consider the derived critical locus $\ud{\mR}$, $\Ind_{\mU,f}$ is the virtual dimension of $\bL_{\ud{\mR}}^{<0}$.\medskip

Given a smooth restriction of critical loci $\phi:(\mR',\mU',f',i')\to (\mR,\mU,f,i)$, we define similarly the locally constant function $\Ind_\phi:\Grad(\mR')\to \bZ$:
\begin{align}\label{inddplusphi}
    \Ind_\phi:=d^+_\phi|_{\Grad(\mR')}-d^-_\phi|_{\Grad(\mR')}
\end{align}
where $d^+_\phi$ (resp. $d^-_\phi$) denotes the virtual dimension of $\bL^{<0}_\phi$ (resp. $\bL^{>0}_\phi$). one obtains directly $d^\pm_{\mU'}=d^\pm_{\mU}\circ\tilde{\phi}+d^\pm_{\phi}$, which gives:
\begin{align}\label{Indsm}
    \Ind_{\mU',f'}=\Ind_{\mU,f}\circ\tilde{\phi}+\Ind_\phi
\end{align}
Given $(\mR_i,\mU_i,f_i,i_i)$, $i=1,2$, we have also:
\begin{align}\label{Indprod}
     \Ind_{\mU_1\times \mU_2,f_1\boxplus f_1}= \Ind_{\mU_1,f_1}\boxplus \Ind_{\mU_2,f_2}
\end{align}

\begin{proposition}\label{prophyploccrit}
    For any critical locus $(\mR,\mU,f,i)$, there is a canonical isomorphism in $\mA_{mon,c}(\mR)$:
    \begin{align}
     p_!\eta^\ast P_{\mU,f}\simeq P_{\Grad(\mU),\Grad(f)}\{-\Ind_{\mU,f}/2\}
    \end{align}
    It is compatible with products of critical loci, namely for two critical chart $(\mR_i,\mU_i,f_i,i_i)$, $i=1,2$, the following square of isomorphisms is commutative:
    \[\begin{tikzcd}[column sep=tiny]
     p_!\eta^\ast P_{\mU_1\times \mU_2,f_1\boxplus f_2}\arrow[r,"\simeq"]\arrow[d,"\simeq"] &P_{\Grad(\mU_1\times\mU_2),\Grad(f_1\boxplus f_2)}\{- \Ind_{\mU_1\times \mU_2,f_1\boxplus f_2}/2\}\arrow[d,"\simeq"]\\
((p_1)!(\eta_1)^\ast P_{\mU_1,f_1})\boxtimes((p_2)_!(\eta_2)^\ast P_{\mU_2,f_2})\arrow[r,"\simeq"] &(P_{\Grad(\mU_1),\Grad(f_1)}\{- \Ind_{\mU_1,f_1}/2\})\boxtimes (P_{\Grad(\mU_2),\Grad(f_2)}\{- \Ind_{\mU_2,f_2}/2\})
    \end{tikzcd}\]
    where the vertical arrows come from \eqref{thomsebcritchart}.
\end{proposition}

\begin{proof}
We consider the following sequence of isomorphisms:
   \begin{align}
    p_!\eta^\ast P_{\mU,f}:=&(p_\mR)_!(\eta_\mR)^\ast i^\ast\{d_\mU/2\}\phi_f^{mon,tot}\un_\mU\nn\\
        \simeq &\Grad(i)^\ast\{d_\mU/2\}(p_\mU)_!(\eta_\mU)^\ast \phi_f^{mon,tot}\un_\mU\nn\\
        \simeq &\Grad(i)^\ast\{(d_{\Grad(\mU)}+d^+_\mU+d^-_\mU)/2\}\phi_{\Grad(f)}^{mon,tot}(p_\mU)_!(\eta_\mU)^\ast\un_\mU\nn\\
        \simeq &\Grad(i)^\ast\{(d_{\Grad(\mU)}-d^+_\mU+d^-_\mU)/2\}\phi_{\Grad(f)}^{mon,tot}\un_{\Grad(\mU)}\nn\\
    =:&P_{\Grad(\mU),\Grad(f)}\{-\Ind_{\mU,f}/2\}
    \end{align}
    where the second line follows from Lemma \ref{lemhyploclosed} applied to the closed immersion $i:\mR\to\mU$, using the fact that $\phi_f^{mon,tot}\un_\mU$ is supported on $\mR$, the third line from Theorem \ref{commutspechyploc} applied to the specialization system $\phi^{mon,tot}$, the fourth line from Proposition \ref{propsmoothBB}, and the last line from the definition, and the formula \eqref{inddplus}.\medskip

    Consider two critical loci $(\mR_i,\mU_i,f_i,i_i)$ for $i=1,2$. From the Lemma \ref{lemhyploclosed}, the commutation between the restriction to the critical locus and hyperbolic localization is compatible with products. From Theorem \ref{theothomseb}, the Thom-Sebastiani isomorphism is an isomorphism of specialization system, hence commutes with the morphism of commutation between monodromic vanishing cycles and hyperbolic localization, obtained from the specialization system exchange morphisms. From Proposition \ref{propsmoothBB}, the Białynicki-Birula isomorphism for smooth stacks is compatible with products. One obtains then that the above square is commutative. 
\end{proof}

\subsubsection{Hyperbolic localization and smooth restrictions}

Consider a smooth restriction of critical loci $\phi:(\mR',\mU',f',i')\to(\mR,\mU,f,i)$. From Proposition \ref{propsmoothBB}, there is a natural isomorphism $(p_{\mR'})_!(\eta_{\mR'})^\ast(\tilde{\phi})^\ast\simeq \Grad(\tilde{\phi})^\ast\{-2d^+_\phi\}(p_\mR')_!(\eta_\mR')^\ast$, which can be rewritten from definition \eqref{inddplus}:
\begin{align}\label{natsmoooth2}
    (p_{\mR'})_!(\eta_{\mR'})^\ast\tilde{\phi}^\ast\{d_\phi/2\}\simeq \Grad(\tilde{\phi})^\ast\{d^0_\phi/2\}\{-\Ind_\phi/2\}(p_\mR)_!(\eta_\mR)^\ast
\end{align}
We have the following compatibility lemma:

\begin{lemma}\label{lemsmoothresthyploc}
    The hyperbolic localization isomorphism of Proposition \ref{prophyploccrit} is compatible with smooth restriction of critical loci, namely for $\phi:(\mR',\mU',f',i')\to(\mR,\mU,f,i)$ a smooth restriction, the following diagram of isomorphisms is commutative:
    \[\begin{tikzcd}
 \Grad(\tilde{\phi})^\ast\{d^0_\phi/2\}\{-\Ind_\phi/2\}p_!\eta^\ast P_{\mU,f} \arrow[r,"\simeq"] &\Grad(\tilde{\phi})^\ast\{d^0_\phi/2\}\{-\Ind_\phi/2\} P_{\Grad(\mU),\Grad(f)}\{- \Ind_{\mU,f}/2\}\arrow[d,"\simeq"]\\
(p')_!(\eta')^\ast \tilde{\phi}^\ast \{d_\phi/2\}P_{\mU,f}\arrow[u,swap,"\simeq"]\arrow[d,"\simeq"] & P_{\Grad(\mU'),\Grad(f')}\{-\Ind_{\mU,f}/2\}\{-\Ind_\phi/2\}\\
(p')_!(\eta')^\ast P_{\mU',f'}\arrow[r,"\simeq"] & P_{\Grad(\mU'),\Grad(f')}\{-\Ind_{\mU',f'}/2\}\arrow[u,swap,"\simeq"]
    \end{tikzcd}\]
    where the horizontal arrows come from Proposition \ref{prophyploccrit}, and the vertical arrows from \eqref{natsmooth} and \eqref{natsmoooth2}.
\end{lemma}

\begin{proof}
    From Proposition \ref{propsmoothBB}, the isomorphism of commutation between hyperbolic localization and smooth pullback is compatible with composition, hence the following diagram is commutative:
    \[\begin{tikzcd}[sep=small]
 \Grad(\phi)^\ast\{-2 d^+_\phi\} p_!\eta^\ast \un_\mU \arrow[r,"\simeq"] &\Grad(\phi)^\ast\{-2 d^+_\phi\}\un_{\Grad(\mU)}\{-2d^+_\mU\}\arrow[d,"\simeq"]\\
(p')_!(\eta')^\ast \phi^\ast\un_\mU\arrow[u,swap,"\simeq"]\arrow[d,"\simeq"] & \un_{\Grad(\mU')}\{-2d^+_\mU\}\{-2d^+_\phi\}\\
(p')_!(\eta')^\ast \un_{\mU'}\arrow[r,"\simeq"] & \un_{\Grad(\mU')}\{-2d^+_{\mU'}\}\arrow[u,swap,"\simeq"]
    \end{tikzcd}\]
    As $\phi^{mon,tot}$ is a specialization system, the morphisms of commutation between $\phi^{mon,tot}$, hyperbolic localization and smooth pullbacks are compatible from Lemma \ref{lemspecsmoothhyploc}. From Lemma \ref{lemhyploclosed}, the isomorphisms of commutation between smooth pullbacks, restriction to the critical locus, and hyperbolic localization, are compatible. Hence the result follows. 
\end{proof}

\subsubsection{Hyperbolic localization and quadratic bundles}

Consider a critical locus $(\mR,\mU,f,i)$, a quadratic bundle $(\mE,q)$ on $\mU$, and the critical locus $(\mR,\bV_\mU(\mE),f\circ \pi+q,s\circ i)$ obtained by stabilization by $(\mE,q)$. We have that $\iota^\ast\mE$ (resp $\eta^\ast\mE$) is naturally $\bZ$-graded (resp. $\bZ$-filtered), as recalled from \cite[Section 1]{HalpernLeistner2020DerivedA} in Section \ref{sectsmoothhyploc}. From the non-degenerate quadratic form $q$, one obtains a canonical splitting of the $\bZ$-filtration of $\eta^\ast\mE$, and then an isomorphism $\eta^\ast(\mE,q)^I\simeq p^\ast \iota^\ast(\mE,q)^I$. We denote:
\begin{align}
    (\Grad(\mE)^I,\Grad(q)^I)&:=\iota^\ast(\mE,q)^I\nn\\
    (\Filt(\mE)^I,\Filt(q)^I)&:=\eta^\ast(\mE,q)^I\simeq p^\ast (\Grad(\mE)^I,\Grad(q)^I)
\end{align}
In particular $\Grad(q)^0,\Filt(q)^0$ and $\Filt(q)$ are nondegenerate, \ie $(\Grad(\mE)^0,\Grad(q)^0),(\Filt(\mE)^0,\Filt(q)^0)$ and $(\Filt(\mE)^0,\Filt(q)^0)$ are quadratic bundles. A simple computation gives:
\begin{align}
    \Grad(\bV_\mX(\mE))&\simeq\bV_{\Grad(\mX)}(\Grad(\mE)^0)\nn\\
     \Filt(\bV_\mX(\mE))&\simeq\bV_{\Grad(\mX)}(\Filt(\mE)^{\leq 0})
\end{align}
We obtain that the $\Grad$ of $(\mR,\bV_\mU(\mE),f\circ \pi+q,s\circ i)$ is the critical chart:
\begin{align}
    (\Grad(\mR),\bV_{\Grad(\mU)}(\Grad(\mE)^0),\Grad(f)\circ\Grad(\pi)^0+\Grad(q)^0,\Grad(s)^0\circ\Grad(i))
\end{align} 
obtained by stabilization from $(\Grad(\mR),\Grad(\mU),\Grad(f),\Grad(i))$ by the quadratic bundle $(\Grad(\mE)^0,\Grad(q)^0)$. Moreover, $\Grad(q)$ (resp. $\Filt(q))$) induces a perfect pairing between $\Grad(\mE)^{<0}$ and $\Grad(\mE)^{>0}$ (resp. $\Filt(\mE)^{<0}$ and $\Filt(\mE)^{>0}$), in particular $d^-_\mE=d^+_\mE$, hence:
\begin{align}\label{indstab}
    \Ind_{\bV_\mU(\mE)}=\Ind_\mU+\Ind_\mE=\Ind_\mU
\end{align}
And
\begin{align}
    \det(\Grad(\mE))\simeq \det(\Grad(\mE)^0)\otimes \det(\Grad(\mE)^{<0})\otimes \det(\Grad(\mE)^{>0})\simeq \det(\Grad(\mE)^0)
\end{align}
where the last isomorphism is induced by the perfect pairing. It gives a canonical identification between orientations of $(\Grad(\mE),\Grad(q))$ and $(\Grad(\mE)^0,\Grad(q)^0))$, \ie an isomorphism:
\begin{align}\label{isomquadform}
   P_{(\Grad(\mE),\Grad(q))}\simeq P_{(\Grad(\mE)^0,\Grad(q)^0)}
\end{align}
and furthermore an isomorphism:
\begin{align}\label{isomhyplocorient}
    p_!\eta^\ast (-\otimes_{\bZ/2\bZ}P_{\mE,q})\simeq p_!(\eta^\ast\otimes_{\bZ/2\bZ}P_{\Filt(\mE),\Filt(q)}\simeq (p_!\eta^\ast)\otimes_{\bZ/2\bZ}P_{(\Grad(\mE),\Grad(q))}\simeq (p_!\eta^\ast)\otimes_{\bZ/2\bZ}P_{(\Grad(\mE)^0,\Grad(q)^0)}
\end{align}

The compatibility between hyperbolic localization and stabilization by quadratic forms is obtained as follows:

\begin{lemma}\label{lemcompathyplocstab}
    Given a critical locus $(\mR,\mU,f,i)$ and a quadratic bundle $(\mE,q)$ on $\mU$, the following square of isomorphisms is commutative:
    \[\begin{tikzcd}
        \tilde{p}_!\tilde{\eta}^\ast P_{\bV_\mU(\mE),f\circ\pi+q}\arrow[d,"\simeq"]\arrow[r,"\simeq"] & P_{\Grad(\bV_\mU(\mE)),\Grad(f\circ\pi+q)}\{-\Ind_{\bV_\mU(\mE)}/2\}\arrow[d,"\simeq"]\\
        \tilde{p}_!\tilde{\eta}^\ast (P_{\mU,f}\otimes_{\bZ/2\bZ}P_{\mE,q}|_{\mR^{red}})\arrow[r,"\simeq"] & P_{\Grad(\mU),\Grad(f)}\{-\Ind_\mU/2\}\otimes_{\bZ/2\bZ}P_{(\Grad(\mE)^0,\Grad(q)^0)}|_{\Grad(\mR)^{red}}
    \end{tikzcd}\]
    where the vertical arrows are obtained by the isomorphisms of Section \ref{sectstabquad}, and the horizontal arrows by the isomorphisms of Proposition \ref{prophyploccrit} and the isomorphism \eqref{isomhyplocorient}.
\end{lemma}

\begin{proof}
    Using Halpern-Leistner's Lemma \ref{covergrad}, one can cover $\mU$ by smooth maps $\phi:[U/\bG_m]\to\mU$ by quotients of affine scheme with $\bG_m$-action, such that $\Grad([U/\bG_m])\to\Grad(\mU)$ (and furthermore, $\Grad([R/\bG_m])\to\Grad(\mR)$) are jointly surjective. Then using compatibility of hyperbolic localization and the action of quadratic bundle with smooth pullbacks, one is reduced to the case where $\mU=[U/\bG_m]$, $U=\Spec(A)$, where $A$ is a $\bZ$-graded $k$-algebra $A$, and we have to work in the neighborhood of a point $u$ in $U$ which is $\bG_m$-fixed. Up to Zariski-local restriction, one can assume that $\mE$ is a free module generated by $\bZ$-graded elements. Now, $q\in \mE\otimes\mE$ induces a non-degenerate pairing between the positive and negative part, one can take a basis $e_i$ of the positive part, and the dual basis $f_j$ of the negative part, such that $q=\sum e_i\otimes f_i+q^0$, with $q^0$ of degree $0$. Up to taking an étale extension of $A^0$, one can diagonalize $q^0$, such that $q=\sum e_i\otimes f_i+\sum g_j\otimes g_j$.\medskip
    
    Then, we can assume that $\mU=[U/\bG_m]$ and that,  considering the smooth map $\phi':U\to[U/\bG_m]$:
    \begin{align}
        (\phi')^\ast(\mE,q)\simeq (\mathcal{O}_U^{n+2m},\sum_{i=1}^nx_i^2+\sum_{j=1}^my_jz_j)
    \end{align}
    where the $x_i$ (resp. $y_j$, resp. $z_j$) are homogeneous of weight $0$ (resp. $>0$, resp. $<0$). We have that $\Grad([U/\bG_m])=\bigsqcup_{\lambda\in\bZ}[U^0_\lambda/\bG_m]$. Up to replacing the $\bG_m$-action by some power, we just have to check that the pullback of the square of the Lemma along $R^0\to [R^0/\bG_m]$ commutes. Consider the following commutative diagram:
    \[\begin{tikzcd}[column sep=huge]
        U^0\times\bA^n\arrow[d] & U^+\times\bA^{n+m}\arrow[l,swap,"p:=p_U\times p_{\bA^{n+2m}}"]\arrow[r,"\eta:=\eta_U\times \eta_{\bA^{n+2m}}"]\arrow[d] & U\times\bA^{n+2m}\arrow[d]\nn\\
        U^0 & U^+\arrow[l,swap,"p_U"]\arrow[r,"\eta_U"] & U
    \end{tikzcd}\]
    Then, denoting $f':U\to\bA^1$, it follows from the last statement of Lemma \ref{lemactquadbun} that the pullback of the square of the Lemma to $R^0$ is the square of isomorphisms:
    \[\begin{tikzcd}[column sep=small]
        \tilde{p}_!\tilde{\eta}^\ast P_{U\times\bA^{n+2m},f'\boxplus \sum_ix_i^2+\sum_j y_jz_j}\arrow[r,"\simeq"]\arrow[d,"\simeq"] & P_{U^0\times\bA^{n},f'\boxplus \sum_i x_i^2}\{-\Ind_{U\times \bA^{n+2m}}/2\}\arrow[d,"\simeq"]\\
        ((\tilde{p}_U)_!(\tilde{\eta}_U)^\ast P_{U,f'})\boxtimes((\tilde{p}_{\bA^{n+2m}})_!(\tilde{\eta}_{\bA^{n+2m}})^\ast P_{\bA^{n+2m},\sum_ix_i^2+\sum_jy_jz_j})\arrow[r,"\simeq"]\arrow[d,"\simeq"] & P_{U^0,(f')^0}\{-\Ind_U/2\}\boxtimes P_{\bA^n,\sum_ix_i^2}\arrow[d,"\simeq"]\\
        ((\tilde{p_U})_!(\tilde{\eta_U})^\ast P_{U,f'}\arrow[r,"\simeq"] & P_{U^0,(f')^0}\{-\Ind_U/2\}
    \end{tikzcd}\]
    where the vertical arrows are the Thom-Sebastiani isomorphisms of subsection \ref{sectproductcritloc}, and the horizontal isomorphisms are obtained as those of Proposition \ref{prophyploccrit}. Then, from Proposition \ref{prophyploccrit}, the upper square commutes, and it suffices to prove that the following square commutes:
    \[\begin{tikzcd}
        (\tilde{p}_{\bA^{n+2m}})_!(\tilde{\eta}_{\bA^{n+2m}})^\ast P_{\bA^{n+2m},\sum_ix_i^2+\sum_jy_jz_j}\arrow[r,"\simeq"]\arrow[d,"\simeq"] &  P_{\bA^n,\sum_{i=1}^nx_i^2}\arrow[d,"\simeq"]\\
        \un_R\arrow[r,"="] & \un
    \end{tikzcd}\]
    where the upper horizontal arrow is obtained by hyperbolic localization, and the vertical arrows by Lemma \ref{lemsqroot}.
    Using Thom-Sebastiani again, it suffices to prove that the hyperbolic localization isomorphism $P_{\bA^2,xy}\simeq \un$ is the isomorphism \ref{isomsquareroot}. Consider the commutative diagram with Cartesian square:
    \[\begin{tikzcd}
        & \{(0,0)\}\arrow[r,"="]\arrow[dl,swap,"="]\arrow[d,"\hat{i}"] & \{(0,0)\}\arrow[d,swap,"i"]\\
        \{(0,0)\}\arrow[d,"="] & \bA^1\times\{0\}\arrow[d,"="]\arrow[l,swap,"p_0"]\arrow[r,"\eta_0"] & (xy)^{-1}(0)\arrow[d]\arrow[u,shift left=-2,swap, "\rho"]\\
        \{(0,0)\}\arrow[dr,"0"] & \bA^1\times\{0\}\arrow[d,"0"]\arrow[l,swap,"p"]\arrow[r,"\eta"] & \bA^2\arrow[dl,"xy"]\\
        & \bA^1 &
    \end{tikzcd}\]
    By definition, the isomorphism of commutation with hyperbolic localization of Proposition \ref{prophyploccrit}
    is given by (noticing that $0$ is the only singular value of $xy$):
    \begin{align}\label{isomhyploca2}
        i^\ast\phi_{xy}^{mon}\un_{\bA^2}\simeq (p_0)_!\hat{i}_!\hat{i}^!(\eta_0)^\ast\phi_{xy}^{mon}\un_{\bA^2}\overset{\simeq}{\leftarrow} (p_0)_!(\eta_0)^\ast\phi_{xy}^{mon}\un_{\bA^2}\overset{\simeq}{\to}\phi_0^{mon}p_!\eta^\ast\un_{\bA^2}\simeq \un[-2](-1)
    \end{align}
    Consider the following diagram:
    \[\begin{tikzcd}[column sep=small]
        i^\ast\phi_{xy}^{mon}\arrow[r,"\simeq"]\arrow[dr,"\simeq"]& \rho_!(\eta_0)_!\hat{i}_!\hat{i}^!(\eta_0)^\ast\phi_{xy}^{mon}\un_{\bA^2}\arrow[d,"\simeq"] & \rho_!(\eta_0)_!(\eta_0)^\ast\phi_{xy}^{mon}\un_{\bA^2}\arrow[l,swap,"\simeq"]\arrow[r] & \rho_!(\eta_0)_!\phi_{0}^{mon}\eta^\ast\un_{\bA^2}\arrow[r,"\simeq"]\arrow[d,"\simeq"]\arrow[dr,"\simeq"] &  \rho_!\phi_{0}^{mon}\eta_!\eta^\ast\un_{\bA^2}\arrow[d,"\simeq"]\\
        &\rho_!i_!i^\ast\phi_{xy}^{mon}\un_{\bA^2} & \rho_!\phi_{xy}^{mon}\un_{\bA^2}\arrow[l,swap,"\simeq"]\arrow[u,"\simeq"]\arrow[r] & \rho_!\phi_0^{mon}\eta_!\eta^\ast\un_{\bA^2} & \un[-2](-1)
    \end{tikzcd}\]
    Using $p_0=\rho\circ\eta_0$, the upper path corresponds to the hyperbolic localization isomorphism \eqref{isomhyploca2}; whence the bottom path corresponds to \ref{isomsquareroot}. Using $i=\hat{i}\circ\eta_0$, the left triangle commutes by functoriality and the left square commutes by compatibility of adjunctions morphisms with composition; the right triengle commutes by the compatiblity of the morphisms of specialization system with adjunctions, and the right square commutes because $\phi_0^{mon}=Id$. We obtained then the desired equality of isomorphisms. Notice that we could have instead defined directly \ref{isomsquareroot} using hyperbolic localization, but the definition using dimensional reduction is the usual one, closer to what is done in \cite[Exposé XV]{SGAVII}.
\end{proof}

\subsection{Existence and comparison of critical charts}

\subsubsection{D-critical version}\label{dcritchart}

A $d$-critical structure on a stack $\mX$ gives by definition a smooth covering of $\mX$ by critical charts $(R,U,f,i)$, where $R,U$ are schemes, and the results of \cite[Section 2.3]{Joyce2013ACM} gives a way to compare such charts, using stabilization by quadratic forms, still using critical charts given by schemes. However, the hyperbolic localization diagram for a scheme is trivial, so such 'd-critical atlas' is not suitable to study hyperbolic localization. One must be able to cover $\mX$ by critical charts $(\mR,\mU,f,i)$ such that the $\Grad(\mR)\to\Grad(\mX)$ are jointly surjective, and to be able to compare them in a similar way. Thanks to Halpern-Leistner's Lemma \ref{covergrad}, it suffices to build $\bG_m$-equivariant critical charts, and to compare them in a $\bG_m$-equivariant way. Torus-equivariant analogues of \cite[Section 2.3]{Joyce2013ACM} were partially given in \cite[Section 2.6]{Joyce2013ACM}, we just complete them here by following closely Joyce's arguments:

\begin{proposition}\label{propcompareqchart}(Joyce, \cite[Section 2.3, Section 2.6]{Joyce2013ACM})
    Consider an algebraic space $X$ with a torus action $\mu:T\times X\to X$, with a $T$-equivariant $d$-critical structure $s$ of weight $\chi:T\to\bG_m$ (\ie such that $\mu^\star s=\chi\cdot(p_2)^\star s\in H^0(\mathcal{S}^0_{T\times X})$).
    \begin{enumerate}
        \item[$i)$] $X$ is covered étale-locally by $T$-equivariant affine critical charts (\ie critical charts $(R,U,f,i)$ such that $R,U$ are affine with $T$-action, $R\to X$ is a $T$-equivariant open immersion, $i:R\to U$ is a $T$-equivariant closed immersion, and $f:U\to\bA^1$ is $T$-equivariant of weight $\chi$).
        \item[$ii)$] Given $(R_i,U_i,f_i,i_i)$, $i=1,2$ two étale affine $T$-equivariant critical charts and a $T$-fixed point $(x_1,x_2)\in R_1\times_X R_2$, there exist a diagram of $T$-equivariant critical charts:
        \[\begin{tikzcd}[column sep=tiny]
      &  & (R,U,f,i)\arrow[dl]\arrow[dr] &  &   \\
     (R'_1,U'_1,f'_1,i'_1)\arrow[d]\arrow[dotted,r] & (R'_1,U'_1\times E_1,f'_1\boxplus q_1,i'_1) &  &(R'_2,U'_2\times E_2,f'_2\boxplus q_2,i'_2\times0)&  (R'_2,U'_2,f'_2,i'_2)\arrow[d]\arrow[dotted,l] \\
     (R_1,U_1,f_1,i_1)&&&&(R_2,U_2,f_2,i_2)
\end{tikzcd}\]
with a $T$-fixed point $x\in R$ over $(x_1,x_2)$, where the dotted arrows corresponds to stabilization by $T$-equivariant quadratic forms (\ie, $E_i$ is an affine space with linear $T$-action, and $q_i$ is a $T$-equivariant non-degenerate quadratic form of weight $\chi$), and the plain arrows are $T$-equivariant étale morphisms of critical charts
    \end{enumerate}   
\end{proposition}

\begin{proof}
    \begin{enumerate}
        \item[$i)$] Consider $x\in X(k)$: by \cite[Theorem 4.1]{Alper2015AL} (notice that the conditions are trivially satisfied, as $X$ is an algebraic space), there is a $T$-equivariant pointed étale map $\phi:(X',x')\to (X,x)$, where $X'$ is affine, in particular its $T$-action is automatically 'good' in the sense of \cite[Definition 2.41]{Joyce2013ACM}. Considering $s':=\phi^\star(s)$, one can applies \cite[Proposition 2.44]{Joyce2013ACM} to $(X',s')$ near $X'$ to obtain a Zariski $T$-equivariant critical chart for $(X',s')$ near $x'$, which gives by compositions an étale $T$-equivariant critical chart for $(X,s)$ over $x$.\medskip
        
        \item[$ii)$] This is a torus-equivariant analogue of the comparison given in \cite[Section 2.3]{Joyce2013ACM}. We first applies \cite[Theorem 4.1]{Alper2015AL} to $(R_1\times_XR_2,(x_1,x_2))$ to obtain a $T$-equivariant pointed étale map $(\tilde{R},\tilde{x}\to (R_1\times_XR_2,(x_1,x_2))$ with $\tilde{R}$ affine, which is stabilizer preserving at $\tilde{x}$, and consider the $T$-equivariant $d$-critical structure $\tilde{s}$ of weight $\chi$ on $\tilde{R}$ obtained by pullback, and the $T$-fixed point $\tilde{x}$.\medskip
        
        According to \cite[Tag 04B1, 04D1]{stacks-project}, given a smooth (resp. étale) morphism of affine scheme $\tilde{R}=\Spec(\tilde{A})\to R_i=\Spec(A_i)$, and a closed immersion $R_i\to U_i=\Spec(B_i)$, and $\tilde{x}\in \tilde{R}$ there exists a distinguished open neighborhood $\check{R}$ of $\tilde{x}$ (resp. one can take $\check{R}=\tilde{R}$) in $\tilde{R}$, an affine scheme $\check{U}_i=\Spec(\check{B}_i)$ with a smooth (resp. étale) morphism $\check{U}_i\to U_i$, and a closed immersion $\check{R}\to \check{U}$. One can check easily that a $T$-equivariant version of this Lemma holds. Indeed, $\tilde{A},A_i,B_i$ have a $\Gamma:=\Hom(T,\bG_m)$-grading corresponding to the $T$-action, and it suffices to give a version of the proof using only homogeneous elements. The proof of \cite[Tag 04B1, 04D1]{stacks-project} (every smooth affine map is locally standard smooth, resp. every étale map is standard smooth) can easily be adapted to the $\Gamma$-graded setting. It gives that $\tilde{R}$ is covered by open subsets $\check{R}=D(g)$, for $g\in \tilde{A}$ homogeneous (such that $D(g)$ is $T$-invariant), such that $\tilde{A}_g$ can be written as $A_i[x_1,...,x_n]/(\bar{f}_1,...,\bar{f}_c)$, with $x_1,...,x_n,\bar{f}_1,...,\bar{f}_c$ homogeneous (resp. with $c=n$), such that $\det((\partial x_i/\partial \bar{f}_i)_{1\leq i\leq c})$ is invertible in $\tilde{A}_g$. Now, one can choose homogeneous lifts $f_1,...,f_c\in B_i[x_1,...,x_n]$ of $\bar{f}_1,...,\bar{f}_c\in A_i[x_1,...,x_n]$. As the $f_i$ are homogeneous, $\Delta:=\det((\partial x_i/\partial f_i)_{1\leq i\leq c})$ is homogeneous: consider $x_{n+1}$, of weight opposite to $\Delta$, and the $\bZ$-graded ring $\check{B}_i:=B_i[x_1,...,x_{n+1}]/(f_1,...,f_c,x_{n+1}\Delta-1)$. Then, $\check{U}_i:=\Spec(\check{A}_i)$ is $T$-equivariant, $\phi_i:\check{U}_i\to U_i$ is smooth (resp. étale) and $T$-equivariant, and $\check{i}_i:\check{R}\to \check{U}_i$ is a $T$-equivariant closed immersion.\medskip

        We applies the étale version of the last paragraph, and define $\check{R}_i:=\check{R}=\tilde{R}$ and $\check{f}_i:=f_i\circ\phi_i$, and we obtain two Zariski critical charts $(\check{R}_i,\check{U}_i,\check{f}_i,\check{i}_i)$ for $(\tilde{R},\tilde{s})$, and a $T$-fixed point $\tilde{x}\in \check{R}_1\cap \check{R}_2$. Using \cite[Proposition 2.44]{Joyce2013ACM}, up to restricting $\check{U}_i$ near $\tilde{x}$, and then $\check{R}_i$, to a $T$-invariant open subset $U'_i$ (that can be chosen to be distinguished, hence affine), one has $T$-equivariant closed immersions of $T$-equivariant critical charts $\Phi:(R'_i,U'_i,f'_i,i'_i)\to (\hat{R},\hat{U},\hat{f},\hat{i})$ (\ie, $\Phi:U'_i\to \hat{U}$ is a $T$-equivariant closed immersion, such that $f'_i=\hat{f}\circ\Phi$, and such that $\Crit(\hat{f})=\Phi(\Crit(f'_i))$). As $\tilde{x}$ is $T$-fixed, one applies then the following Lemma \ref{lemcompareqquadform} twice to obtain $T$ equivariant étale maps $(\hat{R}_i,\hat{U}_i,\hat{f}_i,\hat{i}_i)\to (\hat{R},\hat{U},\hat{f},\hat{i})$ and $(\hat{R}_i,\hat{U}_i,\hat{f}_i,\hat{i}_i)\to (R'_i,U'_i\times E_i,f'_i\boxplus q_i,i'_i)$, with a $T$-fixed point $\hat{x}_i$ over $x_i$. We take then:
\begin{align}
    (R,U,f,i):=(\hat{R}_1\times_{\hat{R}}\hat{R}_2,\hat{U}_1\times_{\hat{U}}\hat{U}_2,\hat{f}_1\times_{\hat{f}}\hat{f}_2,\hat{i}_1\times_{\hat{i}}\hat{i}_2)
\end{align}
and the $T$-fixed point $x=(\hat{x}_1,\hat{x}_2)\in R$, which is over $(x_1,x_2)$. Schematically, this gives:
\[\begin{tikzcd}
   (U'_1\times E_1,f'_1\boxplus q_1)  &  & (U,f)\arrow[dl]\arrow[dr]&  & (U'_2\times E_2,f'_2\boxplus q_2)\\
    (U'_1,f'_1)\arrow[drr,dashed]\arrow[d]\arrow[u,dotted] & (\hat{U}_1,\hat{f}_1)\arrow[dr]\arrow[ul] & & (\hat{U}_2,\hat{f}_2)\arrow[dl]\arrow[ur] & (U'_2,f'_2)\arrow[dll,dashed]\arrow[d]\arrow[u,dotted]\\
    (U_1,f_1) & & (\hat{U},\hat{f}) & & (U_2,f_2)
\end{tikzcd}\]
where dashed arrows denotes $T$-equivariant closed embeddings, dotted arrows stabilization by $T$-equivariant quadratic forms, and plain arrows $T$-equivariant étale restriction of critical charts.
    \end{enumerate}
\end{proof}

\begin{lemma}(Equivariant analogue of \cite[Proposition 2.23]{Joyce2013ACM})\label{lemcompareqquadform}
    Consider a $T$-equivariant closed immersion $\Phi:(R_1,U_1,f_1,i_1)\to (R_2,U_2,f_2,i_2)$ of $T$-equivariant affine étale critical charts on a $T$-equivariant $d$-critical algebraic space $(X,s)$. Given a $T$-fixed point $x\in R_1=R_2$, there is a commutative diagram of $T$-equivariant critical charts:
    \[\begin{tikzcd}
       (U'_1,f'_1)\arrow[dd]\arrow[rr,dashed] & & (U'_2,f'_2)\arrow[dd]\arrow[dl]\\
         & (U_1\times E,f_1\boxplus q) &\\
        (U_1,f_1)\arrow[rr,dashed]\arrow[ur,dotted] && (U_2,f_2)
        \end{tikzcd}\]
    with a $T$-fixed point $x'$ in $R'_1=R'_2$ over $x$, where dashed arrows denotes $T$-equivariant closed embeddings, dotted arrows stabilization by $T$-equivariant quadratic forms, and plain arrows $T$-equivariant étale restriction of critical charts.
\end{lemma}

\begin{proof}
We begin by adapting the arguments of the proof of \cite[Proposition 2.23]{Joyce2013ACM} to the $T$-equivariant setting. Because $T$ is a torus, as proven in the proof of \cite[Proposition 2.44]{Joyce2013ACM}, one can, up to restricting $U_2$ to an affine $T$-invariant distinguished open neighborhood $\tilde{U}_2$ of $\Phi(x)$ (and denote $\tilde{U}_1=\Phi^{-1}(\tilde{U}_2)$), find a $T$-equivariant étale map $\tilde{\gamma}\times\tilde{\beta}:\tilde{U}_2\to E'\times E$ where $E',E$ are affine spaces with linear $T$-action, such that $\tilde{U}_1=\tilde{\beta}^{-1}(0)$, and then $\tilde{\gamma}\circ\Phi:\tilde{U}_1\to E'$ is étale, and there is a $T$-fixed point over $x$. Then we define $\check{j}:\check{U}_2\to U_2$ to be the $T$-equivariant étale neighborhood of $i_2(x)$ in $\tilde{U}_2$ given by $(\tilde{U}_1\times E')\times_{E\times E'}\tilde{U}_2$, which is then still affine. We denote by $\check{\alpha}:\check{U}_2\to U_1$ the composition:
\[\begin{tikzcd}
    \check{U}_2\arrow[r] & \tilde{U}_1\times E'\arrow[r,"p_1"] & \tilde{U}_1\arrow[r] & U_1
\end{tikzcd}
\]
and $\check{\beta}:\check{U}_2\to E\times E'\to E$: then $\check{\alpha}\times\check{\beta}:\check{U}_2\to U_1\times E$ is étale. We denote $\check{f}_1:=f_1\circ\alpha:\check{U}_2\to\bA^1_k$, $\check{f}_2:=f_2\circ\check{j}:\check{U}_2\to\bA^1_k$, and  $\check{h}:=\check{f}_2-\check{f}_1:\check{U}_2\to\bA^1_k$, which is $T$-equivariant of weight $\chi$ by construction. Denote by $(\check{z}_1,...,\check{z}_n)$ the étale coordinates coming from $\check{\beta}:\check{U}_2\to E=\Spec(V)$: as $T$ is a torus, we can assume that each $\check{z}_b$ corresponds to a one dimensional representation of $T$ of character $\chi_b$. As shown in the proof of \cite[Proposition 2.23]{Joyce2013ACM}, $\check{h}$ lies in the ideal $(\check{z}_1,...,\check{z}_n)^2$. Then, we can write $\check{h}=\sum_{b,c}Q^{b,c}\check{z}_b \check{z}_c$, where $Q^{b,c}=Q^{c,b}:\check{V}\to\bA^1_k$ has weight $\chi-\chi_b-\chi_c$. Ax $x$ is $T$-fixed, $q^{b,c}=Q^{b,c}(x)$ defines a $T$-equivariant quadratic form on $V$ of weight $\chi$. We have then a $T$-equivariant quadratic form $q$ of $V$ of weight $\chi$, which is non-degenerate by the argument below \cite[equation 5.10]{Joyce2013ACM}: we restrict to the open subset where $Q$ is still non-degenerate, which is automatically affine and $T$-invariant.\medskip

We cannot adapt directly the end of the proof of \cite[Proposition 2.23]{Joyce2013ACM}, because the 'Gram Schmidt algorithm' used here to diagonalize the quadratic form breaks the $T$-equivariance. Denote by $V=\bigoplus_\psi V_\psi$ the weight decomposition of $V$, where the sum is over the characters $\psi:T\to\bG_m$. Notice that $q$ induces a non-degenerate pairing between $V_\psi$ and $V_{\chi-\psi}$: for each couple $\{\psi,\chi-\psi\}$, with $\psi\neq\chi/2$, we choose a representative, giving a set of cocharacter $\Lambda$, such that $V=\bigoplus_{\psi\in\Lambda}(V_\psi\oplus V_{\chi-\psi})\oplus V_{\chi/2}$ where the last term can be empty (in particular, if $\chi$ is not the square of a character). For each $\psi\in \Lambda$, we consider the étale coordinates $x_b:=\check{z}_b$ of weight $\psi$, and replace the étale coordinates $\check{z}_b$ of weight $\chi-\psi$ by the étale coordinates $y_b:=\sum_c (Q^{bc})^{-1}x_b$, and keep unchanged the étale coordinates $\check{z}_c$ of weight $\chi/2$. We obtain then $h=\sum_b x_by_b+\sum_{c,d} Q^{cd} \check{z}_c\check{z}_d$, where the last terms contains only coordinates of weight $\chi/2$ (notice that, if $\chi$ is not divisible by $2$, there are no such coordinates), and the $Q^{cd}$ are $T$-invariant.\medskip

We can then applies the procedure of \cite[Proposition 2.23, 2.24]{Joyce2013ACM} to the étale coordinates $\check{z}_c$ of weight $\chi/2$, because the formulas involves elements of the same weight. This allows, up to taking an étale $T$-equivariant cover $j':U'_2\to\check{U}_2$ (in order to take square roots of the $T$-invariant functions $Q^{cc}$), to obtain new étale coordinates $z_c$ of weight $-\chi/2$ replacing the $\check{z}_c$, such that $h\circ j'=q\circ\beta=\sum_b x_by_b+\sum_c (z_c)^2$, and that there is a $T$-fixed point over $x$. Considering $\alpha,\beta,f'_2,f'_1$ obtained by composing $\check{\alpha},\check{\beta},\check{f}_2,\check{f}_1$ with $j'$, and $U'_1:=\beta^{-1}(0)$, with an $T$-equivariant étale map to $U_1$, one obtains the $T$-equivariant diagram of the Lemma.
\end{proof}

If $s$ is $T$-invariant, \ie $\chi=0$, one obtains by definition of $\mathcal{S}_{[X/T]}$ that $s$ descends to a d-critical structure $\bar{s}$ on $[X/T]$: in particular, we obtain an isomorphism:
\begin{align}
    K_{X,s}\simeq K_{[X/T],\bar{s}}|_{X^{red}}\otimes K_{X/ [X/T]}|_{X^{red}}^{\otimes 2}
\end{align}
which gives a $T$-equivariant structure on $K_{X,s}$. For future reference, it will be useful to notice that one can define also such an equivariant structure for $\chi\neq 0$:

\begin{corollary}\label{corcanbunTeq}
    Consider an algebraic space $X$ with a torus action $\mu:T\times X\to X$, with a $T$-equivariant $d$-critical structure $s$ of weight $\chi:T\to\bG_m$. Then $K_{X,s}$ enhance naturally to a $T$-equivariant line bundle, such that, on a $T$-equivariant étale affine critical chart $(R,U,f,i)$, one has a canonical isomorphism $K_{X,s}|_{R^{red}}\simeq i^\ast(K_{U}^{\otimes 2})|_{R^{red}}\otimes\chi^{-\dim(U)}$ upgrading the one of Lemma \ref{defcanbun} $i)$.
\end{corollary}

\begin{proof}
    From Proposition \ref{propcompareqchart} $i)$, one obtains local definition of the equivariant structure on $K_{X,s}$ on an étale cover of $X^{red}$. As $K_{X,s}$ is a line bundle on a reduced algebraic space, it suffices to show that the equivariant structure at the stalk of each point is independent of the choice of $T$-equivariant critical chart. Given $x\in R(k)$, we have the isomorphism:
    \begin{align}
        K_{\mX,s}|_x\simeq i^\ast(K_U^{\otimes 2})|_x\simeq K_R^{\otimes 2}|_x\simeq K_X^{\otimes 2}|_x
    \end{align}
    where the second isomorphism comes from $\det(hess_f)$: as $f$, and then $hess_f$, is $T$-equivariant of weight $\chi$, $\det(hess_f)$ is $T$-equivariant of weight $\chi^{\dim(U)-\dim(R)}$, such that we have a $T$-equivariant isomorphism:
    \begin{align}
        K_{\mX,s}|_{R^{red}}|_x\simeq i^\ast(K_U^{\otimes 2})|_x\otimes \chi^{-\dim(U)}\simeq K_R^{\otimes 2}|_x\otimes \chi^{-\dim(R)}\simeq K_X^{\otimes 2}|_x\otimes \chi^{-\dim(X)}
    \end{align}
    where the right hand side is independent of the critical chart, hence the $T$-equivariant structures defined above glue.
\end{proof}

\begin{proposition}\label{propcomparchart}
    Consider a quasi-separated  d-critical Artin stack $(\mX,s)$, with affine stabilizer, locally of finite type over $k$.
    \begin{enumerate}
        \item[$i)$] $\mX$ is covered by critical charts of the form $(\mR,\mU,f,i)$, such that the smooth maps $\Grad(\mR)\to\Grad(\mX)$ are jointly surjective.
        \item[$ii)$] Given $(\mR_i,\mU_i,f_i,i_i)$, $i=1,2$ two critical charts, there is a family of diagrams of critical charts:
        \[\begin{tikzcd}
      &  & (\mU,f)\arrow[dl]\arrow[dr] &  &   \\
     (\mU'_1,f'_1)\arrow[d]\arrow[dotted,r] & (\bV_{\mU'_1}(\mE_1),f'_1\circ \pi_1+q_1) &  &(\bV_{\mU'_2}(\mE_2),f'_2\circ \pi_2+q_2) &  (\mU'_2,f'_2)\arrow[d]\arrow[dotted,l] \\
     (\mU_1,f_1)&&&&(\mU_2,f_2)
\end{tikzcd}\]
where the dotted arrows corresponds to stabilization by quadratic bundles, and the plain arrows are smooth morphisms of critical charts, such that the $\Grad(\mR)\to\Grad(\mR_1)\times_{\Grad(\mX)}\Grad(\mR_2)$ are jointly surjective.
    \end{enumerate}   
\end{proposition}

\begin{proof}
    \begin{enumerate}
        \item[$i)$] Using Halpern-Leistner's Lemma \ref{covergrad}, one has a smooth cover of $\mX$ by quotients $[X/\bG_m]$ of affine scheme with $\bG_m$-action such that the $\Grad([X/\bG_m])$ cover $\Grad(\mX)$. Then $X$ inherits a $\bG_m$-equivariant $d$-critical structure (\ie, with weight $\chi=0$). Using \cite[Proposition 2.43]{Joyce2013ACM} (Proposition \ref{propcompareqchart} i) here), one can find $\bG_m$-equivariant critical charts $(R,U,f,i)$ covering $X$ for the Zariski topology (with $f$ $\bG_m$-invariant), which gives critical charts $([R/\bG_m],[U/\bG_m],f,i)$ covering $[X/\bG_m]$ for the Zariski topology, such that; from \cite[Corollary 1.1.7]{HalpernLeistner2014OnTS}, the $(\Grad([R/\bG_m])$ are covering $(\Grad([X/\bG_m])$, \ie also $\Grad(\mX)$, hence we are done.\medskip

        \item[$ii)$] Using Halpern-Leistner's Lemma \ref{covergrad} applied to $\mU_1,\mU_2$, one can find smooth restrictions of critical charts $([R_i/\bG_m],[U_i/\bG_m],f_i,i_i)\to (\mR_i,\mU_i,f_i,s_i)$ with $U_i$ affine covering $\mU_i$ such that the $(\Grad([R_i/\bG_m])\to \mR_i$ are jointly surjective, hence we can suppose that the two charts are of the form $([R_i/\bG_m],[U_i/\bG_m],f_i,i_i)$ with $U_i$ affine. Apply now Halpern-Leistner's Lemma \ref{covergrad} to $[R_1/\bG_m]\times_\mX[R_2/\bG_m]$ to cover it by some $[\tilde{R}/\bG_m]$ such that the $\Grad([\tilde{R}/\bG_m])$ cover $\Grad([R_1/\bG_m])\times_{\Grad(\mX)}\Grad([R_2/\bG_m])$.\medskip
        
        Consider $(y_1,y_2)\in \Grad([R_1/\bG_m])\times_{\Grad(\mX)}\Grad([R_2/\bG_m])(k)$, and take a point $\tilde{y}\in \Grad([\tilde{R}/\bG_m])(k)$ over $(y_1,y_2)$. One can replace the $\bG_m$-action on $R_i,\tilde{R}$ by some power of it, such that $y_i,\tilde{y}$ corresponds to the cocharacter $1$ (this leads us to replace the critical charts by an étale cover). It means in particular that the points $x_i:=\iota(y_i)\in[R_i/\bG_m](k),\tilde{x}:=\iota(\tilde{y})\in[\tilde{R}/\bG_m](k)$ underlying $y_i,\tilde{y}$ have stabilizer $\bG_{m,k}$, and that the smooth morphism $[\tilde{R}/\bG_m]\to[R_i/\bG_m]$ induces an isomorphism of stabilizers at $\tilde{x}\to x_i$. Notice that these morphisms do not have to come from a $\bG_{m,k}$-equivariant morphism $\tilde{R}\to R_i$, but we can argue as in the proof of \cite[Theorem 1.13]{Alper2024StructureRF}. Namely, the fact that $[\tilde{R}/\bG_m]\to [R_i/\bG_m]$ comes from a $\bG_m$-equivariant morphism $\tilde{R}\to R_i$ is equivalent to say that the following square is commutative:
\[\begin{tikzcd}
    {[\tilde{R}/\bG_m]}\arrow[r]\arrow[d] & B\bG_m\arrow[d,"="]\\
    {[R_i/\bG_m]}\arrow[r] & B\bG_m
\end{tikzcd}\]
denote by $p,q:[\tilde{R}/\bG_m]\to B\bG_m$ the two compositions morphisms, which corresponds to two $\bG_m$-torsors $\mPP,\mathcal{G}$ on $[\tilde{R}/\bG_m]$. Notice that, because $[\tilde{R}/\bG_m]\to [R_i/\bG_m]$ induces an isomorphism of stabilizer at $\tilde{x}$, there is an isomorphism $p\simeq q$ over the gerbe of $\tilde{x}$, which gives a section at the gerbe of $\tilde{y}$ of the stack of isomorphisms between the torsors $\mPP$ and $\mathcal{Q}$. But the latter is a smooth and affine stack (because $\bG_m$ is itself smooth and affine) over $[\tilde{R}/G]$: From \cite[Theorem 7.18]{Alper2019TheL}, there is then an extension of this section étale locally near $\tilde{x}$. In other term, up to precomposing by an étale map, we can assume that $[\tilde{R}/\bG_m]\to [R_i/G]$ comes from a $\bG_m$-equivariant map $\tilde{R}\to R_i$. Applying this procedure for $i=1$ and further for $i=2$, and shrinking $\tilde{R}$ étale locally near $\tilde{x}$, we have obtained an affine $\bG_m$-equivariant scheme $\tilde{R}$ with $\bG_m$-equivariant smooth maps $\tilde{R}\to R_i$ such that, considering the smooth morphisms $\Grad([\tilde{R}/\bG_m])\to \Grad([R_1/\bG_m])\times_{\Grad(\mX)}\Grad([R_2/\bG_m])$, there is a point $\tilde{y}$ over $(y_1,y_2)$, with image $\tilde{x}\in R_1$. We will use repeatedly below the fact that the lift of a $\bG_m$-fixed point along a $\bG_m$-equivariant map $Y\to Z$ is automatically $\bG_m$-fixed, and then gives a lift of the corresponding point along $\Grad(Y)\to \Grad(Z)$.\medskip

We can now use the smooth version of the torus-equivariant analog of \cite[Tag 04B1]{stacks-project} discussed in the proof of Proposition \ref{propcompareqchart} above, giving smooth $\bG_m$-equivariant maps $\check{R}:=\check{R}_1\cap \check{R}_2\to R_i$ with a $\bG_m$-fixed point $\check{x}$ lying over the $x_i$, such that the two composed maps $[\check{R}/\bG_m]\to [R_i/\bG_m]\to\mX$ are isomorphic, and $\bG_m$-equivariant closed embeddings into smooth schemes $\check{R}\to \check{U}_i$ over $R_i\to U_i$. Now, consider the critical charts $([\check{R}/\bG_m],[\check{U}_i/\bG_m],\check{f}_i,\check{i}_i)$ of $(\mX,s)$ induced by smooth restriction along $[\check{U}_i/\bG_m]\to [U_i/\bG_m]$: these are two critical charts for the $d$-critical structure induced on $[\check{R}/\bG_m]$ by pullback from those of $\mX$. In other terms, $(\check{R},\check{U}_i,\check{f}_i,\check{i}_i)$ are two $\bG_m$-equivariant critical charts for the $\bG_m$-invariant (\ie, with weight $0$) $d$-critical structure on $\check{R}$. Applying Proposition \ref{propcompareqchart} $ii)$, one finds the following diagram of critical charts of $\mX$:

 \adjustbox{scale=0.8,center}{\begin{tikzcd}[column sep=tiny]
      &  & ([U/\bG_m],f)\arrow[dl]\arrow[dr] &  &   \\
     ([U'_1/\bG_m],f'_1)\arrow[d]\arrow[dotted,r] & ([U'_1\times E_1/\bG_m],f'_1\circ\pi_1+ q_1) &  &([U'_2\times E_2/\bG_m],f'_2\circ\pi_2+ q_2)&  ([U'_2/\bG_m],f'_2)\arrow[d]\arrow[dotted,l] \\
     ([\check{U}_1/\bG_m],\check{f}_1)\arrow[d]&&&&([\check{U}_2/\bG_m],\check{f}_2)\arrow[d]\\
     ([U_1/\bG_m],f_1) &&&& ([U_2/\bG_m],f_2)
\end{tikzcd}}

with a point $x\in [R/\bG_m]$ lying over $\check{x}$, hence over $\tilde{x}$, where the vertical arrows from the first to the second line, and from the second line to the third line, are étale and representable. Hence, from the above remark there is a point $y\in\Grad([R/\bG_m])$ over $\tilde{y}\in\Grad([\tilde{R}/\bG_m])$, hence over $(y_1,y_2)\in\Grad(\mR_1)\times_{\Grad(\mX)}\Grad(\mR_2)$. Notice that the dotted arrows give a stabilization by quadratic bundle.
\end{enumerate}

\end{proof}

\subsubsection{Shifted symplectic version}\label{shifsympchart}

We begin by discussing $T$-equivariant $-1$-shifted symplectic structures. Consider a derived algebraic space $\ud{X}$ with the action of a torus $T$. We follow the discussion of \cite[Remark 1.4, Section 2.1.1]{Calaque2016ShiftedCS}. From \cite[Remark 2.4.8]{Calaque2015ShiftedPS}, there is a relative version of the De Rahm complex, such that for any morphism of derived stack $\ud{\mX}\to \ud{\mY}$, $\mathbf{DR}(\ud{\mX}/\ud{\mY})$ forms global section of a sheaf of mixed graded complexes $\mathcal{DR}(\ud{\mX}/\ud{\mY})$ on $\ud{\mY}$. In particular, $\mathcal{DR}([\ud{X}/T]/BT)$ is a sheaf of mixed graded complex on $BT$, whose pullback along $\Spec(k)\to BT$ is $\mathbf{DR}(\ud{X})$: it gives then an extra $\Gamma:=\Hom(T,\bG_m)$-grading on the mixed graded complex $\mathcal{DR}(\ud{X})$ of forms: we will then talk about forms of weight $\chi\in \mathcal{DR}(\ud{X})$. Applying the functor $NC^w$ of \cite[Section 1.1]{shifsymp}, we obtain a $\Gamma$-grading on the complex $NC^w(\ud{X})$ of closed forms. From the functoriality of the relative De Rahm complex, we obtain directly that pullbacks along $T$-equivariant morphisms respect this grading.\medskip

We define then a $-1$-shifted symplectic structure on $\ud{X}$ of weight $\chi$ to be a nondegenerate closed $2$-form $\omega$ of degree $-1$ in the weight $\chi$ part of $NC^w(\ud{X})$. Given a $-1$-shifted symplectic structure on $\ud{X}$ $\omega$, putting an equivariant structure of weight $\chi$ on it amount to give an isomorphism $\omega\sim \omega^\chi$, with $\omega^\chi$ with grading $\chi$. Given two such choice, the isomorphism $\omega^\chi\sim \omega (\omega^\chi)'$ does not have to have grading $\chi$, hence this is really a structure, and not a property. To avoid confusion, we denote by $\sim_\chi$ a homotopy in the $\chi$-graded part. Putting an equivariant structure of weight $0$ on $\omega$ amount to descend it to a closed $2$-form of degree $-1$ on $[\ud{X}/T]$ (which cannot be nondegenerate from the inspection of the cotangent complex). We quickly check that the Darboux theorem \cite[Theorem 5.18]{darbscheme} admits a $T$-equivariant version:

\begin{lemma}\label{lemdarbgmeq}
    Consider a $T$-equivariant $-1$-shifted symplectic algebraic space $(\ud{X},\omega)$ of weight $\chi\in \Hom(T,\bG_m)$ and a point $x\in X(k)=\ud{X}(k)$. There is a smooth affine scheme $U$ with $T$-action, a $T$-equivariant function of weight $\chi$ $f:U\to\bA^1_k$ and a $T$-equivariant étale map $\phi:\ud{R}:=\ud{\Crit(f)}\to \ud{X}$ over $x$ such that $f|_{R^{red}}=0$ and $\phi^\ast(\omega)\sim_\chi \omega_{\ud{\Crit(f)}}$. In particular, its classical truncation $(X,s)$ is $T$-equivariant of weight $\chi$, $(R,U,f,i)$ is a $T$-equivariant étale critical chart, and the isomorphism $\det(\bL_{\ud{X}})|_{X^{red}}\simeq K_{X,s}$ respects the $T$-equivariant structure from Corollary \ref{corcanbunTeq}.
\end{lemma}

\begin{proof}
    From \cite[Theorem 4.1]{Alper2015AL}, we have at the level of the classical truncation a $T$-equivariant étale pointed map $(X',x')\to (X,x)$, which gives an étale map $[X'/T]\to [X/T]$ over $BT$. By descent from \cite[Lemma 2.1.5]{Gaitsgory_Rozenblyum_2017}, there is a unique étale map $\ud{\mX}'\to[\ud{X}/T]$ over $BT$ extending this map. By denoting $\ud{X}'$ the pullback of $\ud{\mX}'$ along $\Spec(k)\to BT$, one obtains that $\ud{X}'$ is affine, as its classical truncation $X'$ is affine, and it has a $T$-action such that $\ud{X}'\to\ud{X}$ is $T$-equivariant. Replacing $(\ud{X},\omega)$ by $(\ud{X}',\omega|_{\ud{X}'})$, we can then assume that $\ud{X}=\ud{\Spec}(\ud{R})$ is an affine derived scheme with $T$-action, \ie $\ud{R}$ is a $\Gamma$-graded cdga of finite presentation over $k$.\medskip
    
    We can now do exactly the same operations than in \cite[Theorem 4.1]{darbscheme}, but in a way consistent with the $\bZ$-grading of $\ud{R}$: then, up to replacing $\ud{\Spec}(\ud{R})$ by a $T$-invariant Zariski open neighborhood of $x$, we can assume that $\ud{R}$ is a $\bZ$-graded standard form cdga, giving a minimal standard neighborhood of $x$. Then $\ud{R}(0)$ is smooth of dimension $m$, and $\ud{R}$ is freely generated over $\ud{R}(0)$ by $m$ homogeneous generators in degree $-1$. Moreover, as pullbacks by $T$-equivariant map preserves the grading, $\ud{\phi}^\ast(\omega)$ enhance to a $T$-equivariant closed $2$-form of degree $-1$ and weight $\chi$ on $\ud{\Spec}(\ud{R})$. The periodic cyclic complex and cyclic complex of \cite[Definition 5.5]{darbscheme}, and then their cohomology, admits an obvious $\Gamma$-graded enhancement, such that the exact sequence \cite[Proposition 5.6 a)]{darbscheme} respects this $\Gamma$-grading. The isomorphism \cite[Proposition 5.6 b)]{darbscheme} admits an obvious $\Gamma$-graded enhancement, such that $HP^{-4}(\ud{R})(2)$ has obviously weight $0$. We can then argue as in \cite[Proposition 5.7 a)]{darbscheme}, taking a lift of $\omega$ to $HC^{-2}(\ud{R})(1)$ of weight $\chi$, giving $\Phi\in \ud{R}^0=\ud{R}(0)$ and $\phi\in (\Omega^1_{\ud{R}})^{-1}$ of weight $\chi$, such that $d\boldsymbol{\Phi}=0$, $d_{dR}\boldsymbol{\Phi}+d\Phi=0$ and $\omega\sim_\chi (d_{dR}\Phi,0,0,...)$. We can adapt the arguments of \cite[Proposition 5.7 b)]{darbscheme}: if $\chi\neq 0$, obviously $\Phi_{X^{red}}=0$ as it is constant on $X^{red}$, and, if $\chi=0$, one can adjust the value of $\Phi_{X^{red}}$ to be $0$ by adding an element of $HP^{-4}(\ud{R})(2)$, which automatically has weight $0$. We can now argue as in \cite[Section 5.6, Step 1]{darbstack}: in \cite[Diagram 5.30]{darbstack}, the morphism $(V^{-1})^\ast\to V^0$ induced by $d_{dR}\Phi$ has $\Gamma$-weight $\chi$, hence the open subset where it is an isomorphism is $T$-invariant, hence we can localize around $x$.\medskip
    
    We can now argue as in \cite[Section 5.6, Step 2]{darbstack}: we localize at a $T$-invariant open neighborhood of $x$, and take homogeneous étale coordinates $x_1,...,x_m$ for $\ud{R}(0)$. We can then choose generators $y_1,...,y_m$ in degree $-1$, homogeneous of $\Gamma$-weights opposite to those of the $x_i$, such that the $d_{dR}y_i$ forms the dual basis of the $d_{dR}x_i$ under $(V^{-1})^\ast\simeq V^0$. As $\boldsymbol{\Phi}$ and $\Phi$ have $\Gamma$-weight $\chi$, one checks directly that all the operations of \cite[Section 5.6, Step 2]{darbstack} respects the $\Gamma$-grading (notice that the discussion on the master equation is irrelevant for $k=-1$). We obtain then a version with homogeneous coordinates of \cite[Theorem 5.18 a)]{darbstack}. Then, consider the smooth scheme with $T$-action $U:=\Spec(\ud{R}(0))$, the $T$-invariant function $f:U\to\bA^1_k$ corresponding to the element $\boldsymbol{\Phi}\in \ud{B}^0$ of $\Gamma$-weight $\chi$. We obtain that $\ud{\Spec}(\ud{R})=\ud{\Crit}(f)$ as a derived scheme with $T$-action, and that $\omega_{\ud{\Crit(f)}}\sim_\chi \omega|_{\ud{R}}$.\medskip

    From \cite[Theorem 6.6]{darbscheme}, for such a data, $(R,U,f,i)$ forms a critical chart for $(X,s)$. Then $(R\times T,U\times T,f\circ\mu,i)$ (resp. $(R\times T,U\times T,f\circ pr_2,i)$) give a critical chart for $(X\times T,\mu^\star(s))$ (resp. $(X\times T,(pr_2)^\star(s))$, and, as $f\circ\mu=\chi\cdot f\circ pr_2$, we have that $(\mu^\star(s))|_{R\times T}=\chi\cdot(pr_2)^\star(s)|_{R\times T}$. As such critical charts cover $R\times T$, we have $(\mu^\star(s))=\chi\cdot(pr_2)^\star(s)$, and then $s$ is $T$-equivariant of weight $\chi$. On such a critical chart, the isomorphism $\det(\bL_{\ud{X}})|_{R^{red}}\simeq K_{X,s}|_{R^{red}}$ is given in \cite[Theorem 6.6]{darbscheme} by the isomorphism:
    \begin{align}
        \det(\bL_{\ud{R}})|_{R^{red}}\simeq i^\ast (K_U^{\otimes 2})|_{R^{red}}
    \end{align}
    obtained by writing $\bL_{\ud{R}}$ as $T_U|_{R}\overset{hess_f}{\to} T^\ast U$. $f$, and then $hess_f$, have weight $\chi$, so, as a $T$-equivariant complex, $\bL_{\ud{R}}$ can be written as $T_U\otimes \chi|_{R}\overset{hess_f}{\to} T^\ast U$, which gives an isomorphism:
    \begin{align}
        \det(\bL_{\ud{R}})|_{R^{red}}\simeq i^\ast (K_U^{\otimes 2})|_{R^{red}}\otimes\chi^{-n}
    \end{align}
    giving that $\det(\bL_{\ud{X}})|_{R^{red}}\simeq K_{X,s}|_{R^{red}}$ is $T$-equivariant, hence $\det(\bL_{\ud{X}})|_{X^{red}}\simeq K_{X,s}$ is $T$-equivariant.
\end{proof}

The following will be the key point to compare the $\Grad$ construction at the $-1$-shifted symplectic level and at the $d$-critical level:

\begin{lemma}\label{lemdarbgmst}
    Consider a $-1$-shifted symplectic stack $(\ud{\mX},\omega)$ and a point $x\in \Grad(\mX)(k)$. There is a smooth map $\bar{\ud{\phi}}:([\ud{R}/\bG_{m,k}],y)\to (\ud{\mX},\iota(x))$, with $\ud{R}$ a derived scheme with $\bG_{m,k}$-action restricting to the map $B\bG_{m,k}\to \ud{\mX}$ classified by $x$, a smooth scheme $U$ with $\bG_m$-action and a $\bG_m$-invariant functions $f:U\to\bA^1$ and a map $\ud{j}:\ud{R}\to \ud{\Crit}(f)$ inducing an isomorphism on the classical truncation, such that $f|_{R^{red}}=0$ and, denoting $\ud{\phi}:\ud{R}\to\ud{\mX}$, $\ud{j}^\ast(\omega_{\ud{\Crit(f)}})\sim \ud{\phi}^\ast(\omega)$.
\end{lemma}

\begin{proof}
    We will adapt the arguments of \cite[Theorem 2.10, Theorem 3.18]{darbstack}. Consider a point $x\in \Grad(\mX)(k)$, corresponding to a point $\iota(x)\in\mX(k)$, and a cocharacter $\bG_{m,k}\to G_{\iota(x)}$. Applying \cite[Theorem 1.13]{Alper2022ArtinAF} to the smooth map $B\bG_{m,k}\to BG_{\iota(x)}$, where $BG_{\iota(x)}$ is considered as a closed substack of $\ud{\mX}$, we obtain a smooth pointed map $\bar{\ud{\phi}}:([\ud{\Spec}(\ud{R})/\bG_{m,k}],y)\to (\ud{\mX},\iota(x))$, where $\ud{R}$ is a $\bZ$-graded cdga, and $\phi^{-1}(BG_{\iota(x)})=B\bG_{m,k}$. Then $\ud{\phi}:\ud{\Spec}(\ud{R})\to \ud{\mX}$ is smooth of dimension $\dim(G_{\iota(x)}):=n$, hence $\dim (T_{\ud{\Spec}(\ud{R}),y})=\dim(T_{\ud{\mX},\iota(x)}):=m$. We consider $\ud{\phi}^\ast(\omega)$, which is obtained by pullback from $\bar{\ud{\phi}}^\ast(\omega)$, and give then a closed $2$-form of degree $-1$ and weight $0$ on $\ud{\Spec}(\ud{R})$.\medskip
    
    As in the proof of \cite[Theorem 2.10 a)]{darbstack}, we follow closely the proof of Lemma \ref{lemdarbgmeq}, by first replacing $\ud{R}$ by a standard form $\bZ$-graded cdga, but where we have $n$ homogeneous generators $w_i$ in degree $-2$ that we fix. We follow then the operations of Lemma \ref{lemdarbgmeq}, with which the generators $w_i$ of degree $-2$ does not interfere. We obtain then a version with homogeneous coordinates of \cite[Theorem 2.10 a)]{darbstack}. As in \cite[Theorem 5.18 b)]{darbstack}, we denote by $\ud{B}\hookrightarrow\ud{R}$ the $\bZ$-graded sub-cdga generated by the $x_i$ and the $y_i$, and consider its dual $\ud{\Spec}(\ud{R})\to \ud{\Spec}(\ud{B})$, which is $\bG_m$-equivariant. Then, consider the smooth scheme with $\bG_m$-action $U:=\Spec(\ud{B}(0))$, the $\bG_m$-invariant function $f:U\to\bA^1_k$ corresponding to the element $\boldsymbol{\Phi}\in \ud{B}^0$ of $\bZ$-weight $0$. We obtain that $\ud{\Spec}(\ud{B})=\ud{\Crit}(f)$ as a derived scheme with $\bG_m$-action, and that $\ud{j}^\ast(\omega_{\ud{\Crit(f)}})\sim \ud{\phi}^\ast(\omega)$. By the construction of $\Phi$, we have also $f_{R^{red}}=0$.
\end{proof}

\subsection{Hyperbolic localization on d-critical stacks}

\subsubsection{D-critical stacks and graded points}

Consider a stack $\mX$. For a point $x\in\Grad(\mX)(k)$, corresponding to a map $g:B\bG_{m,k}\to\mX$, consider the Zariski tangent space $T_{\mX,\iota(x)}$, and the Lie algebra $\mathfrak{Iso}_{\iota(x)}(\mX)$ of the stabilizer group at $\iota(x)$. Both are obtained from the first cohomologies of $g^\ast\bT_\mX$, which provides them a $\bZ$-grading. Denote:
\begin{align}\label{defindmx}
    \Ind_\mX(x):=-\dim(\mathfrak{Iso}_{\iota(x)}(\mX)^{>0})+\dim((T_{\mX,\iota(x)})^{>0})-\dim((T_{\mX,\iota(x)})^{<0})+\dim(\mathfrak{Iso}_{\iota(x)}(\mX)^{<0})
\end{align}
When $\mX$ possesses an enhancement to a $-1$-shifted symplectic stack $(\ud{\mX},\omega)$, $\Ind_\mX(x)$ gives the signed count of negative weights in the tangent complex $(\ud{\iota}^\ast\bT_{\ud{\mX}})^{>0}$, see Lemma \ref{lemshifsymp} $ii)$. $\Ind_\mX$ can be thought as the virtual dimension of the fibers of the strata of $\Filt(\mX)$ flowing to a connected component of $\Grad(\mX)$. We obtain directly that, for two stacks $\mX,\mY$, one has $Ind_{\mX\times\mY}=Ind_\mX+Ind_\mY$. For a smooth morphism $\phi:\mX\to\mY$, the exact triangle of cotangent complexes gives an exact sequence of $\bZ$-graded vector spaces:
\begin{align}\label{exactseq}
    0\to \mathfrak{Iso}_{\iota(x)}(\mX/\mY)\to \mathfrak{Iso}_{\iota(x)}(\mX)\to \mathfrak{Iso}_{\phi(\iota(x))}(\mY)\to T_{\mX/\mY,\iota(x)}\to T_{\mX,\iota(x)}\to T_{\mY,\phi(\iota(x))}\to 0
\end{align}
which gives the formula $\Ind_\mX=\Ind_{\mY}\circ\phi +\Ind_\phi$.

\begin{proposition}\label{cordcritgraded}
    Consider $(\mX,s)$ a d-critical stack which is quasi-separated, locally of finite type, with reductive stabilizers and separated diagonal.
    \begin{enumerate}
        \item[$i)$] Denoting:
    \begin{align}
        \Grad(s):=\iota^\star(s)\in H^0(\Grad(\mX),\mS^0_{\Grad(\mX)})
    \end{align}
    (recalling that $\iota:\Grad(\mX)\to\mX$ is the natural morphism which 'forgets the gradation'), $(\Grad(\mX),\Grad(s))$ is a d-critical stack. A natural system of critical charts covering $(\Grad(\mX),\Grad(s))$ is given by the critical charts:
    \begin{align}
        \Grad(\mR,\mU,f,i):=(\Grad(\mR),\Grad(\mU),\Grad(f),\Grad(i))
    \end{align}
    for any critical chart $(\mR,\mU,f,i)$ of $(\mX,s)$. Then $\Grad$ enhance to a symmetric monoidal endofunctor of the 2-category of d-critical stacks (with smooth morphisms).
    \item[$ii)$] 
    $Ind_\mX:\Grad(\mX)\to\bZ$ is a locally constant function such that, for each critical chart $(\mR,\mU,f,i)$ of $(\mX,s)$, we have $Ind_\mX|_{\Grad(\mR)}=\Ind_{\mU,f}-\Ind_{\mR/\mX}$.
    \item[$iii)$] Given a d-critical stack $(\mX,s)$, there is a canonical square root $G_{\mX,s}$ of $\iota^\ast K_{\mX,s}\otimes K_{\Grad(\mX),\Grad(s)}^{\otimes-1}$, which allows to define an orientation $K_{\Grad(\mX),\Grad(s)}^{1/2}$ of $(\Grad(\mX),\Grad(s))$ from an orientation $K_{\mX,s}^{1/2}$ of ($\mX,s)$ by the formula:
    \begin{align}
        K_{\Grad(\mX),\Grad(s)}^{1/2}:=\iota^\ast K_{\mX,s}\otimes G_{\mX,s}^{\otimes -1}
    \end{align}
    Then $\Grad$ enhance to a symmetric monoidal endofunctor of the 2-category of oriented d-critical stacks (with smooth morphisms). With this choice of orientation, for each critical chart $(\mR,\mU,f,i)$, there is a natural isomorphism:
    \begin{align}\label{hyplocQ}
        \iota^\ast Q_{(\mR,\mU,f,i)}\simeq Q_{\Grad(\mR,\mU,f,i)}
    \end{align}
    which is compatible with smooth restriction and exterior product of critical charts. Given a stabilization by a quadratic form, the following square of isomorphisms is commutative:
    \[\begin{tikzcd}
        \iota^\ast( Q_{(\mR,\bV_\mU(\mE),f\circ\pi+q,s\circ i)}\otimes_{\bZ/2\bZ}P_{(\mE,q)}|_{\mR^{red}})\arrow[r,"\simeq"]\arrow[d,"\simeq"] &  \begin{tabular}{c}$Q_{(\Grad(\mR,\bV_{\mU}( \mE),f\circ\pi+q,s\circ i))}$\\$\otimes_{\bZ/2\bZ}P_{(\Grad(\mE)^0,\Grad(q)^0)}|_{\Grad(\mR)^{red}}$\end{tabular}\arrow[d,"\simeq"]\\
        \iota^\ast Q_{(\mR,\mU,f,i)}\arrow[r,"\simeq"] & Q_{\Grad(\mR,\mU,f,i))}
    \end{tikzcd}\]
    where the horizontal isomorphisms comes from \eqref{isomQstab} and \eqref{isomquadform}.
\end{enumerate}
\end{proposition}

\begin{proof}
\begin{enumerate}
    \item[$i)$] Consider a critical chart $(\mR,\mU,f,i)$ of $(\mX,s)$, and the diagram:
    \[\begin{tikzcd}
        \Grad(\mX)\arrow[d] & \Grad(\mR)\arrow[l]\arrow[r]\arrow[d] & \Grad(\mU)\arrow[d]\arrow[dr,"\Grad(f)"] &\\
        \mX & \mR\arrow[l]\arrow[r] & \mU\arrow[r,"f"] & \bA^1_k
    \end{tikzcd}\]
    From the defining property of $\iota^\star$ given in Definition \ref{defSx}, $\Grad(s)|_{\Grad(\mR)}=\Grad(f)+I_{\Grad(\mR),\Grad(\mU)}^2$, and, from Lemma \ref{lemcritgrad}, $\Crit(\Grad(f))=\Grad(R)$, hence $(\Grad(\mR),\Grad(\mU),\Grad(f),\Grad(i))$ is a critical chart for $(\Grad(\mX),\Grad(s))$. Using Proposition \ref{propcomparchart} $i)$, such charts cover $\Grad(\mX)$, hence $(\Grad(\mX),\Grad(s))$ is a d-critical stack. Notice that the definition of $\Grad(s)$ is obviously functorial with smooth morphisms and symmetric monoidal, as this is the case for $f\mapsto f^\star$.\medskip
    
    \item[$ii)$] Consider a critical chart of $(\Grad(\mX),\Grad(s))$ of the form $(\Grad(\mR),\Grad(\mU),\Grad(f),\Grad(i))$. As $\mR\to\mX$ is smooth, we have proven above that $\Ind_\mX|_{\Grad(\mR)}=\Ind_\mR-\Ind_{\mR/\mX}$. But, given $x\in\Grad(\mR)(k)$, and the underlying point $\iota(x)\in\mR(k)$, the Hessian of $f$ gives a $\bG_m$-invariant quadratic form on $T_{\mU,\iota(x)}$, whose kernel is $T_{\mR,x'}$. We obtain then:
    \begin{align}
        \dim((T_{\mR,\iota(x)})^{>0})-\dim((T_{\mR,\iota(x)})^{<0})=\dim((T_{\mU,\iota(x)})^{>0})-\dim((T_{\mU,\iota(x)})^{<0})
    \end{align}
    As moreover $i$ is a closed immersion, we have $\mathfrak{Iso}_{\iota(x)}(\mU)=\mathfrak{Iso}_{\iota(x)}(\mR)$, which gives $\Ind_\mR(x)=\Ind_{\mU,f}(x)$, inducing the result. In particular, $\Ind_\mX|_{\Grad(\mR)}$ is locally constant. As $\Grad(\mX)$ is covered by critical charts of $(\Grad(\mX),\Grad(s))$ of the form $(\Grad(\mR,\mU,f,i))$, $\Ind_\mX$ is then locally constant on $\Grad(\mX)$.\medskip
    
    \item[$iii)$] For a smooth map $\mX\to\mY$ and a subset $I$ of $\bZ$, we consider the line bundle on $\Grad(\mX)^{red}$:
    \begin{align}\label{defKI}
        K_{\mX/\mY}^I:=\det((\iota^\ast \bL_{\mX/\mY})^I)|_{\Grad(\mX)^{red}}
    \end{align}
    Considering a critical chart $(\mR,\mU,f,i)$, we have:
    \begin{align}\label{eqsqrootorientloc}
        ((\iota_{\mX^{red}})^\ast K_{\mX,s})|_{\Grad(\mR)^{red}}\simeq& (\iota_{\mR^{red}})^\ast (K_{\mX,s}|_{\mR^{red}})\nn\\
        \simeq& (\iota_\mR)^\ast(i^\ast(K_\mU^{\otimes 2})\otimes K_{\mR/\mX}^{\otimes -2})\nn\\
        \simeq& \Grad(i)^\ast (\det((\iota_\mU)^\ast\bL_\mU)^{\otimes 2})\otimes \det((\iota_\mR)^\ast\bL_{\mR/\mX})^{\otimes -2})\nn\\
        \simeq& \Grad(i)^\ast (\det(((\iota_\mU)^\ast\bL_\mU)^0)^{\otimes 2})\otimes  \det(((\iota_\mR)^\ast\bL_{\mR/\mX})^0)^{\otimes -2})\nn\\
        & \otimes \Bigl( \Grad(i)^\ast \det((\iota_\mU)^\ast\bL_\mU)^{\neq 0}\otimes  \det(((\iota_\mR)^\ast\bL_{\mR/\mX})^{\neq 0})^{\otimes -1}\Bigr)^{\otimes 2}\nn\\
        \simeq& K_{\Grad(\mX),\Grad(s)}|_{\Grad(\mR)^{red}}\otimes \Bigl(\Grad(i)^\ast K_\mU^{\neq 0}\otimes (K_{\mR/\mX}^{\neq 0})^{\otimes -1}\Bigr)^{\otimes 2} 
    \end{align}
    (where, by a slight abuse of notation, we have omitted the restrictions to the reduced locus except from the first and last line for readability).
    We define then the line bundle on $\Grad(\mR)^{red}$:
    \begin{align}
        G_{(\mR,\mU,f,i)}:=\Grad(i)^\ast K_\mU^{\neq 0}\otimes (K_{\mR/\mX}^{\neq 0})^{\otimes -1}
    \end{align}
    and we will glue these into a line bundle $G_{\mX,s}$ on $\Grad(\mX)^{red}$, and glue the above isomorphisms into an isomorphism:
    \begin{align}\label{eqsqrootorient}
         \iota^\ast K_{\mX,s}\simeq K_{\Grad(\mX),\Grad(s)}\otimes G_{\mX,s}^{\otimes 2}
    \end{align}
    \medskip
    
    Given a smooth restriction of critical charts $(\mR',\mU',f',i')\to(\mR,\mU,f,i)$, we build an isomorphism as in \eqref{isomsmoothcanbun}:
    \begin{align}\label{isomsmoothrescanbunhyploc}
        G_{(\mR,\mU,f,i)}|_{\Grad(\mR')^{red}}&:=\Grad(i)^\ast K_\mU^{\neq 0}|_{\Grad(\mR')^{red}}\otimes K_{\mR/\mX}^{\neq 0}|^{\otimes -1}_{\Grad(\mR')^{red}}\nn\\
        &\simeq \Grad(i')^\ast(K_{\mU'}^{\neq 0}\otimes (K_{\mU'/\mU}^{\neq 0})^{\otimes -1})\otimes K_{\mR/\mX}^{\neq 0}|^{\otimes -1}_{\Grad(\mR')^{red}}\nn\\
        &\simeq \Grad(i')^\ast K_{\mU'}^{\neq 0}\otimes (K_{\mR'/\mR}^{\neq 0}\otimes K_{\mR/\mX}^{\neq 0}|_{\Grad(\mR')^{red}})^{\otimes -1}\nn\\
        &\simeq \Grad(i')^\ast K_{\mU'}^{\neq 0}\otimes (K_{\mR'/\mX}^{\neq 0})^{\otimes -1}\nn\\
        &=: G_{(\mR',\mU',f',i')}
    \end{align}
    By construction, this isomorphism is compatible with \eqref{eqsqrootorientloc}.\medskip
    
    Consider a critical chart $(\mR,\mU,f,i)$ and its stabilization by quadratic bundle stack $(\mR,\bV_\mX(\mE),f\circ\pi+q,s\circ i)$. We consider the isomorphism:
    \begin{align}
        K_{\bV_\mU(\mE)/\mU}^{\neq 0}\simeq K_{\bV_\mU(\mE)/\mU}^{< 0}\otimes K_{\bV_\mU(\mE)}^{> 0}\simeq \mathcal{O}_{\Grad(\mX)^{red}}
    \end{align}
    where the last isomorphism comes from the prefect pairing induced by $\det(q)$. We obtain from it the isomorphism:
    \begin{align}\label{isomstabquadcanbunhyploc}
        G_{(\mR,\bV_\mX(\mE),f\circ\pi+q,s\circ i)}&:= \Grad(s\circ i)^\ast K_{\bV_\mU(\mE)}^{\neq 0}\otimes (K_{\mR/\mX}^{\neq 0})^{\otimes -1}\nn\\
        &\simeq \Grad(s\circ i)^\ast (K_{\bV_\mU(\mE)/\mU})^{\neq 0}\otimes \Grad(i)^\ast K_\mU^{\neq 0}\otimes (K_{\mR/\mX}^{\neq 0})^{\otimes -1}\nn\\
        &\simeq G_{(\mR,\mU,f,i)}
    \end{align}
    By Lemma \ref{defcanbun} $iii)$, this isomorphism is compatible with \eqref{eqsqrootorientloc}.\medskip

    According to Proposition \ref{propcomparchart}, we can use these isomorphisms to glue the $G_{(\mR,\mU,f,i)}$ into an unique line bundle $G_{\mX,s}$, such that the isomorphisms \eqref{eqsqrootorientloc} glue into a unique isomorphism \eqref{eqsqrootorient} if we have checked some compatibility conditions. Namely, we need to check that, given two comparisons as in Proposition \ref{propcomparchart} $ii)$, the two induced isomorphisms agree locally, and, given three critical charts and a cycle of comparisons as in Proposition \ref{propcomparchart} $ii)$, the three isomorphisms satisfies the cocycle condition. We can use the same trick that in \cite[Section 6]{Joyce2013ACM}: namely, we want to compare isomorphisms between line bundles on reduced stacks (here, the reduced assumption is very important), hence we just have to check that they agree on the stalk at each point. For a stack $\mX$ (not necessarily smooth) and $x\in\Grad(\mX)(k)$, denote:
    \begin{align}
        K_{\mX,x}^I:=\det((T^\ast_{\mX,\iota(x)})^I)\otimes\det(\mathfrak{Iso}_{\iota(x)}(\mX)^I)
    \end{align}
    such that the notation is consistent with \ref{defKI} when $\mX$ is smooth, and, using \eqref{exactseq}, if $\phi:\mX\to\mY$ is smooth, one has:
    \begin{align}
        K_{\mX,x}^I\simeq K_{\mY,\phi(x)}^I\otimes K_{\mX/\mY}^I|_x
    \end{align}
    Consider a critical chart $(\mR,\mU,f,i)$: for $x\in\Grad(\mR)$, the Hessian of $f$ gives a $\bG_m$-invariant quadratic form on $T_{\mU,\iota\circ i(x)}$, whose kernel is $T_{\mR,\iota(x)}$, such that its determinant gives an isomorphism:
    \begin{align}
        K_{\mU,i(x)}^{\neq 0}\simeq K_{\mR,x}^{\neq 0}
    \end{align}
    giving a natural isomorphism:
    \begin{align}\label{eqstalkG}
        G_{\mR,\mU,f,i}|_x:=K_\mU^{\neq 0}|_{i(x)}\otimes (K_{\mR/\mX}^{\neq 0})^{\otimes -1}|_x\simeq K_{\mR,x}^{\neq 0}\otimes (K_{\mR/\mX}^{\neq 0})^{\otimes -1}|_x\simeq K_{\mX,x}^{\neq 0}
    \end{align}
    This isomorphism is obviously compatible with \eqref{isomsmoothrescanbunhyploc}. Because the Hessian of $f\circ \pi+q$ is identified with the direct sum of the Hessian of $f$ and $q$, this isomorphism is compatible with \eqref{isomstabquadcanbunhyploc}. Then the comparison isomorphisms between the $G_{\mR,\mU,f,i}$ agree on their overlap and satisfies cocycle condition,  which finishes the construction of $G_{\mX,s}$. As \eqref{eqsqrootorientloc} is compatible with \eqref{isomsmoothrescanbunhyploc} and \eqref{isomstabquadcanbunhyploc}, it glues into the isomorphism \eqref{eqsqrootorient}.\medskip

    By construction, the definition of $G_{\mX,s}$ is compatible with smooth morphisms of d-critical loci, in the sense that, given a smooth morphism $(\mX,s)\to(\mY,t)$, there is a functorial isomorphism:
    \begin{align}\label{isomsmoothG}
        G_{\mX,s}\simeq G_{\mY,t}|_{\mX^{red}}\otimes \det((\iota^\ast\bL_{\mX/\mY})^{\neq 0})|_{\mX^{red}}
    \end{align}
    compatible with \eqref{eqsqrootorient}. Moreover, for an exterior product of critical charts, one has a natural isomorphism:
    \begin{align}
        G_{(\mR_1\times \mR_2,\mU_1\times\mU_2,f_1\boxplus f_2,i_1\times i_2)}\simeq G_{(\mR_1,\mU_1,f_1,i_1)}\boxtimes G_{(\mR_2,\mU_2,f_2,i_2)}
    \end{align}
    compatible with the isomorphisms isomorphisms \eqref{isomsmoothrescanbunhyploc} and \eqref{isomstabquadcanbunhyploc}, which give monoidal and functorial isomorphisms:
    \begin{align}\label{isomprodG}
        G_{\mX_1,\times\mX_2,s_1\oplus s_2}\simeq G_{\mX_1,s_1}\boxtimes G_{\mX_2,s_2}
    \end{align}
    \medskip
    
    Given an orientation $K_{\mX,s}^{1/2}$ consider:
    \begin{align}
        \Grad(K_{\mX,s}^{1/2}):=\iota^\ast K_{\mX,s}^{1/2}\otimes G_{\mX,s}^{\otimes -1}
    \end{align}
    with the isomorphism $(\Grad(K_{\mX,s}^{1/2}))^{\otimes 2}\simeq K_{\Grad(\mX),\Grad(s)}$ obtained from \eqref{eqsqrootorient}. Then \eqref{isomsmoothG} and \eqref{isomprodG} allows to upgrade:
    \begin{align}
        (\mX,s,K_{\mX,s}^{1/2})\to (\Grad(\mX),\Grad(s),\Grad(K_{\mX,s}^{1/2}))
    \end{align}
    to a symmetric monoidal endofunctor of the 2-category of oriented d-critical Artin stacks with smooth morphisms.\medskip

    Given an isomorphism $K_{\mX,s}^{1/2}|_{\mR^{red}}\simeq i^\ast K_\mU\otimes K_{\mR/\mX}^{\otimes -1}$, one builds from \eqref{eqsqrootorientloc} an isomorphism
    \begin{align}
        \Grad(K_{\mX,s}^{1/2})|_{\Grad(\mR^{red})}\simeq \Grad(i)^\ast K_{\Grad(\mU)}\otimes K_{\Grad(\mR)/\Grad(\mX)}^{\otimes -1}
    \end{align}
    which provides the natural isomorphism \eqref{hyplocQ}. Its compatibility with smooth restriction and exterior products of critical charts is obvious from the definitions. Considering a stabilization by a quadratic form, the isomorphism \eqref{isomstabquadcanbunhyploc} used to glue $G_{\mX,s}$, hence giving the difference between the isomorphisms \eqref{hyplocQ} for $(\mR,\mU,f,i)$ and $(\mR,\bV_\mU(\mE),f\circ\pi+q,s\circ i)$, is built using the perfect pairing between $K_{\bV_\mU(\mE)}^{<0}$ and $K_{\bV_\mU(\mE)}^{>0}$ induced by $\det(q)$, exactly as \eqref{isomquadform}, hence the square of the proposition commutes as claimed.
\end{enumerate}
    
\end{proof}

\subsubsection{Shifted symplectic version}

\begin{lemma}\label{lemorientgm}\begin{enumerate}
    \item[$i)$]Given a $-1$ shifted symplectic stack $(\ud{\mX},\omega_{\ud{\mX}})$, the closed $2$ form of degree $-1$ $\Grad(\omega_{\ud{\mX}}):=\ud{\iota}^\ast(\omega_{\ud{\mX}})$ is nondegenerate, with pairing
    \begin{align}
\bL_{\ud{\Grad}(\ud{\mX})}^\vee\simeq \bL_{\ud{\Grad}(\ud{\mX})}[-1]
    \end{align}
    obtained by applying $(\ud{\iota}^\ast -)^0$ to the nondegenerate pairing of $2$-form of degree $1$ underlying $\omega_{\ud{\mX}}$:
    \begin{align}
        \bL_{\ud{\mX}}^\vee\simeq \bL_{\ud{\mX}}[-1]
    \end{align}
    \ie it gives a natural $-1$-shifted symplectic structure on $\ud{\Grad}(\ud{\mX})$. Equivalently, $\Grad(\omega_{\ud{\mX}})$ can be obtained by applying the mapping construction \cite[Theorem 2.5]{shifsymp} using the natural $\mathcal{O}$-orientation of degree $0$ (in the sense of \cite[Definition 2.4]{shifsymp}) on $B\bG_{m,k}$ given by $\Gamma(B\bG_{m,k},\mathcal{O}_{B\bG_{m,k}})\simeq k$.
    \item[$ii)$] Given a $\bG_m$-invariant $-1$ shifted symplectic algebraic space $(\ud{X},\omega_{\ud{X}})$, the closed $2$ form of degree $-1$ $\omega_{\ud{X}}^0:=\ud{\iota}^\ast(\omega_{\ud{\mX}})$ is nondegenerate, with pairing
    \begin{align}
\bL_{\ud{X}^0}^\vee\simeq \bL_{\ud{\mX}^0}[-1]
    \end{align}
    obtained by applying $(\ud{\iota}^\ast -)^0$ to the nondegenerate pairing of $2$-form of degree $1$ underlying $\omega_{\ud{X}}$, \ie it gives a natural $-1$-shifted symplectic structure on $\ud{X}^0$.
\end{enumerate}
\end{lemma}

\begin{proof}
    \begin{enumerate}
    \item[$i)$] We first define the  $\mathcal{O}$-orientation of degree $0$ on $B\bG_{m,k}$. For any affine derived scheme $\ud{\Spec}(\ud{R})$, consider $B\bG_{m,\ud{R}}:=B\bG_m\times \ud{\Spec}(\ud{R})$. Then $\QCoh(B\bG_{m,\ud{R}})$ is the stable category of $\bZ$-graded $\ud{R}$-complexes, hence it is obviously strictly $\mathcal{O}$-compact over $\ud{R}$ following \cite[Definition 2.1]{shifsymp}, hence $B\bG_{m,k}$ is $\mathcal{O}$-compact from the same definition. Under this identification, $\mathcal{O}_{B\bG_{m,k}}$ is just $k$ with $\bZ$-grading $0$, and we have an obvious isomorphism $\eta:\Gamma(B\bG_{m,k},\mathcal{O}_{B\bG_{m,k}})\simeq k$. For an affine derived scheme $\ud{\Spec}(\ud{R})$ and a perfect complex $\mathcal{E}$ on $B\bG_{m,\ud{R}}$, \ie a $\bZ$-graded perfect $\ud{R}$-complex, we have $\Gamma(B\bG_{m,\ud{R}},E)=E^0$. Then the pairing:
    \begin{align}
        \Gamma(B\bG_{m,\ud{R}},E)\otimes_A \Gamma(B\bG_{m,\ud{R}},E^\vee)\to  \Gamma(B\bG_{m,\ud{R}},\mathcal{O}_{B\bG_{m,\ud{R}}})=\ud{R}
    \end{align}
    induced by $\eta$ is the pairing $E^0\otimes_{\ud{R}}(E^0)^\vee\to \ud{R}$, which is obviously nondegenerate. Then $\eta$ is an $\mathcal{O}$-orientation of degree $0$ (in the sense of \cite[Definition 2.4]{shifsymp}) on $B\bG_{m,k}$.\medskip
    
    Consider an $\ud{R}$-point $x$ of $\ud{Map}(B\bG_{m,k},\ud{\mX})$, corresponding to a map $f:B\bG_{m,\ud{R}}\to\ud{\mX}$. Using the identification:
    \begin{align}
        \bL_{\ud{\Grad}(\ud{\mX})}|_x\simeq \Gamma(B\bG_{m,\ud{R}},f^\ast\bL_{\ud{\mX}})=(f^\ast\bL_{\ud{\mX}})^0
    \end{align}
    from the description given at the end of the proof of \cite[Theorem 2.5]{shifsymp}, the pairing obtained from $\ud{\Grad}(\omega_{\ud{\mX}})$ is clearly the one described in the Lemma.\medskip

    We will now explain why $Map(B\bG_{m,k},\omega_{\ud{\mX}})\sim\ud{\iota}^\ast(\omega_{\ud{\mX}})$. The former is defined in \cite[Section 2.1]{shifsymp} from the sequence of morphisms (where $NC^w$ denotes the graded complex of closed forms):
    \begin{align}
        &NC^w(\ud{\mX})\overset{(\ud{ev})^\ast}{\to }NC^w(\ud{\Grad}(\ud{\mX})\times B\bG_{m,k})\nn\\
        \overset{\kappa_{B\bG_{m,k}}}{\to}& NC^w(\ud{\Grad}(\ud{\mX})\times\Gamma(B\bG_{m,k},\mathcal{O}_{B\bG_{m,k}})\overset{Id\times\eta}{\to}NC^w(\ud{\Grad}(\ud{\mX})
    \end{align}
    it suffices then to show that $(Id\times\eta)\circ \kappa_{B\bG_{m,k}}$ is isomorphic with the pullback along $\ud{\Grad}(\ud{\mX})\to \ud{\Grad}(\ud{\mX})\times B\bG_{m,k}$. By left Kan extension, it suffices to build a functorial isomorphism for $\ud{\mX}=\ud{\Spec}(\ud{R})$. The map $(Id\times\eta)\circ\kappa_{B\bG_m}$ is build in \cite[Page 30]{shifsymp} (where we denote global section by $\Gamma$) by applying $NC^-$ to:
    \begin{align}
        &\mathbf{DR}(B\bG_{m,\ud{R}})\simeq \Gamma(B\bG_{m,k},\mathbf{DR}(\ud{R})\otimes_k\mathbf{DR}(\mathcal{O}_{B\bG_{m,k}}))\nn\\
        \to &\Gamma(B\bG_{m,k},\mathbf{DR}(\ud{R})\otimes_k\mathcal{O}_{B\bG_{m,k}})\simeq \mathbf{DR}(\ud{R})\otimes_k \Gamma(B\bG_{m,k},\mathcal{O}_{B\bG_{m,k}})\simeq \mathbf{DR}(\ud{R})
    \end{align}
    where the arrow comes from the natural projection to the component of weight $0$. We interpret now objects on $B\bG_{m,k}$ as objects with an extra $\bZ$-grading, and recall that $\Gamma$ takes the part of homogeneous grading $0$, and pullbacks along $\Spec(k)\to B\bG_{m,k}$ forgets the homogeneous grading. Then, following the discussion in \cite[Page 26]{shifsymp}, as $k\simeq k^\vee$ is the Lie algebra of $\bG_{m,k}$, $\bL_{B\bG_{m,k}}\simeq k[-1]$ with homogeneous grading $0$, and so $\mathbf{DR}(\mathcal{O}_{B\bG_{m,k}})$ is simply $\Sym_k^\ast k[0]$ with homogeneous grading $0$. Then the above map is simply:
    \begin{align}
        \mathbf{DR}(B\bG_{m,\ud{R}})\simeq \mathbf{DR}(\ud{R})\otimes_k\mathbf{DR}(B\bG_{m,k})\to \mathbf{DR}(\ud{R})\otimes_k\mathbf{DR}(\Spec(k))\simeq \mathbf{DR}(\ud{R})
    \end{align}
    where the arrow comes from the pullback along $\Spec(k)\to B\bG_{m,k}$, which proves the result.\medskip

    \item[$ii)$] The data of the $\bG_m$-invariant closed $2$-form $\omega_{\ud{X}}$ of degree $-1$ is equivalent to the data of a closed $2$-form $\bar{\omega}_{\ud{X}}$ of degree $-1$ on $[\ud{X}/\bG_{m,k}]$, which is degenerate (this explains why $ii$ is not just a formal consequence of $i$). Notice that we can apply the mapping construction of \cite[Theorem 2.5]{shifsymp} also to degenerate forms, only the proof of nondegeneracy will not carry on. We obtain then a degenerate $2$-form $Map(\bG_{m,k},\bar{\omega}_{\ud{X}})$ of degree $-1$ on $\ud{\Grad}([\ud{X}/\bG_{m,k}])$, and the arguments above (which do not use nondegeneracy) show that it is equivalent to $\ud{\iota}^\ast(\bar{\omega}_{\ud{X}})$: it gives then by pullback a $\bG_m$-invariant closed $2$-form $\ud{\iota}^\ast(\omega_{\ud{X}})$ on $\ud{X}^0$. The computation of the map:
    \begin{align}
\bL_{\ud{\Grad}(\ud{\mX})}^\vee\to \bL_{\ud{\Grad}(\ud{\mX})}[-1]
    \end{align}
    induced by $\ud{\iota}^\ast(\bar{\omega}_{\ud{X}})$ in the proof of \cite[Theorem 2.5]{shifsymp} carries over in the degenerate case, giving that it is obtain by applying $(\ud{\iota}^\ast -)^0$ to the map induced by $\bar{\omega}_{\ud{X}}$: in particular, pulling back to $\ud{X}^0$, we obtain that the map induced by $\ud{\iota}^\ast(\omega_{\ud{X}})$ is the one described in the Lemma, in particular it is an isomorphism, \ie $\ud{\iota}^\ast(\omega_{\ud{X}})$ is a $-1$-shifted symplectic form on $\ud{X}^0$.
    \end{enumerate}
\end{proof}

In particular, by functoriality and symmetric monoidality of the pullback of closed forms, we obtain that $(\ud{\mX},\omega_{\ud{\mX}})\mapsto (\ud{\Grad}(\ud{\mX}),\Grad(\omega_{\ud{\mX}}))$ enhances naturally to a symmetric monoidal endofunctor of the $(\infty,1)$-category of $-1$-shifted symplectic stacks, with étale morphisms. Notice that $-1$-shifted symplectic stacks behaves well only with étale morphisms, when d-critical stacks behaves well with smooth morphisms.\medskip

If $(\mX,s)$ is the classical truncation of a $-1$-shifted symplectic stack $(\ud{\mX},\omega_{\ud{\mX}})$, we reinterpret the constructions of Proposition \ref{cordcritgraded}:

\begin{lemma}\label{lemshifsymp}
 Consider a $-1$ shifted symplectic stack $(\ud{\mX},\omega_{\ud{\mX}})$, with its underlying $d$-critical stack $(\mX,s)$ built from \cite[Theorem 3.18]{darbstack}
 \begin{enumerate}
    \item[$i)$] The $d$-critical stack $(\Grad(\mX),\Grad(s))$ from Proposition \ref{cordcritgraded} $i)$ is the d-critical stack obtained from the $-1$-shifted symplectic stack $(\ud{\Grad}(\ud{\mX}),\ud{\Grad}(\omega_{\ud{\mX}}))$ using the Darboux theorem \cite[Theorem 3.18]{darbstack} (and the similar statement holds for $\bG_m$-invariant $-1$-shifted symplectic algebraic spaces).
     \item[$ii)$] The locally constant function $\Ind_{\mX,s}:\mX^{red}\to\bZ$ built in Proposition \ref{cordcritgraded} $ii)$ is the signed dimension of $(\ud{\iota}^\ast\bT_{\ud{\mX}})^{>0}|_{\mX^{red}}$ (and the similar statement holds for $\bG_m$-invariant $-1$-shifted symplectic algebraic spaces).
     \item[$iii)$] Using the nondegenerate pairing induced by $\omega_{\ud{\mX}}$, we obtain:
    \begin{align}\label{eqsqrootshif}
        \iota^\ast\det(\bL_{\ud{\mX}})|_{\mX^{red}}&\simeq \det(\bL_{\ud{\Grad}(\ud{\mX})})|_{\Grad(\mX)^{red}}\otimes \det((\ud{\iota}^\ast\bL_{\ud{\mX}})^{\neq 0})|_{\Grad(\mX)^{red}}\nn\\
        &\simeq \det(\bL_{\ud{\Grad}(\ud{\mX})})|_{\Grad(\mX)^{red}}\otimes \det((\ud{\iota}^\ast\bL_{\ud{\mX}})^{<0})|^{\otimes 2}_{\Grad(\mX)^{red}}
    \end{align}
there is a natural isomorphism between $\det((\ud{\iota}^\ast\bL_{\ud{\mX}})^{<0})|_{\Grad(\mX)^{red}}$ and the square root $G_{\mX,s}$ built in Proposition \ref{cordcritgraded} $iii)$, compatible with étale maps and products of $-1$-shifted symplectic stacks, such that the orientation on $(\Grad(\mX),\Grad(s))$ is given by:
\begin{align}
\Grad(K_{\mX,s}^{1/2}):=\iota^\ast(K_{\mX,s}^{1/2})\otimes\det((\ud{\iota}^\ast\bL_{\ud{\mX}})^{<0})|_{\Grad(\mX)^{red}}^{\otimes -1}
\end{align}
 \end{enumerate}
 (and the similar statement holds for $\bG_m$-invariant $-1$-shifted symplectic algebraic spaces)
\end{lemma}

\begin{proof}
\begin{enumerate}
    \item[$i)$] We have seen in the proof of Lemma \ref{lemcritgrad} that, given a function $\bar{f}:\mU\to\bA^1_k$ on a smooth stack, $\ud{\Grad}(\ud{\Crit}(f))=\ud{\Crit}(\ud{\Grad}(f))$: in particular, given a $\bG_{m,k}$-invariant function $f:U\to\bA^1_k$ on a smooth scheme with $\bG_{m,k}$-action $U$, we have $(\ud{\Crit}(f))^0=\ud{\Crit}(f^0)$. We will now show that the same formula holds for the $-1$-shifted symplectic structures. Given a function $f$ on a smooth space, one obtains two Lagrangian $0,df:U\to T^\ast U$, and by definition $\ud{\Crit}(f):=U\times^h_{0,T^\ast U,df}U$. According to \cite[Page 25]{shifsymp}, the space of closed forms of degree $0$ on $U$ is discrete, and is just the classical space of $p$-form. So the usual symplectic structure $\omega_{T^\ast U}$ on $T^\ast U$ is a $0$-shifted symplectic structure. One have then a unique choice of isotropic structures $h_1:0^\ast(\omega_{T^\ast U})\sim 0$ and $h_2:(df)^\ast(\omega_{T^\ast U})\sim 0$, such that $0,df$ enhance uniquely to $0$-shifted Lagrangian in the sense of \cite[Definition 2.8]{shifsymp}. Then $\omega_{\ud{\Crit}(f)}$ is then define in \cite[Corollary 2.11]{shifsymp} by using the construction of \cite[Corollary 2.10]{shifsymp}. Namely, $\omega_{\ud{\Crit}(f)}$ is the closed $2$-form of degree $-1$ given by:
    \begin{align}
        0\overset{(\ud{pr_1})^\ast (h_1)}{\sim} (\ud{pr_1})^\ast0^\ast(\omega_{T^\ast U})\sim(\ud{pr_2})^\ast(df)^\ast(\omega_{T^\ast U})\overset{(\ud{pr_2})^\ast (h_2)}{\sim}0
    \end{align}
    Consider now $f$ which is $\bG_m$-invariant as above: one obtains then $\iota^\ast(\omega_{T^\ast U})=\omega_{T^\ast U^0}$, and, as they is a single choice for $h_i$, $\iota^\ast(0^\ast(\omega_{T^\ast U})\sim 0)=(0^\ast(\omega_{T^\ast U^0})\sim 0)$ and $\iota^\ast((df)^\ast(\omega_{T^\ast U})\sim 0)=((df)^\ast(\omega_{T^\ast U^0})\sim 0)$. Then, by pulling back the above definition along $\ud{\iota}^\ast$, one obtains:
    \begin{align}
        \omega_{\ud{\Crit}(f)}^0:=\ud{\iota}^\ast\omega_{\ud{\Crit}(f)}\sim \omega_{\ud{\Crit}(f^0)}
    \end{align}
    \medskip
    
    Consider a point $x\in \Grad(\mX)(k)=\ud{\Grad}(\ud{\mX})(k)$, and a data as in Lemma \ref{lemdarbgmst}. In particular, from \cite[Theorem 3.18 a)]{darbstack}, $(R,U,f,i)$ is a critical chart for $(\mX,s)$, and then $([R/\bG_{m,k}],[U/\bG_{m,k}],\bar{f},\bar{i})$ is a critical chart for $(\mX,s)$. As $\omega\mapsto\Grad(\omega)$ and $\omega\mapsto \omega^0$ are defined by pullback along $\ud{\iota}$, they are compatible with pullbacks. We obtain then, denoting $\ud{\phi}^0:\ud{R}^0\to[\ud{R}^0/\bG_{m,k}]\to  \ud{\Grad}(\ud{\mX})$:
    \begin{align}
        (\ud{\phi}^0)^\ast \ud{\Grad}(\omega_{\ud{\mX}})&\sim (\ud{\phi}^\ast\omega_{\ud{\mX}}))^0\nn\\
        &\sim (\ud{j}^\ast\omega_{\ud{\Crit}(f)})^0\nn\\
        &\sim (\ud{j}^0)^\ast(\omega_{\ud{\Crit}(f)})^0\nn\\
        &\sim (\ud{j}^0)^\ast(\omega_{\ud{\Crit}(f^0)}) 
    \end{align}
    We have then a smooth map $\ud{\phi}^0:\ud{R}^0\to  \ud{\Grad}(\ud{\mX})$, a smooth scheme $U^0$, a function $f^0:U^0\to\bA^1_k$ and a map $\ud{j}^0:\ud{R}^0\to \ud{\Crit}(f^0)$ such that $f^0_{(R^0)^{red}}=0$  and $(\ud{\phi}^0)^\ast\ud{\Grad}(\omega_{\ud{\mX}})\sim(\ud{j}^0)^\ast(\omega_{\Crit(f^0)})$. Then, from \cite[Theorem 3.18 a)]{darbstack}, $(R^0,U^0,f^0,i^0)$ gives a critical chart for the $d$-critical stack obtained from the classical truncation of $(\ud{\Grad}(\ud{\mX}),\ud{\Grad}(\omega_{\ud{\mX}}))$. Using this for each cocharacter of $\bG_m$, we obtain that $(\Grad([R/\bG_{m,k}]),\Grad([U/\bG_{m,k}]),\Grad(\bar{f}),\Grad(\bar{i}))$ gives a critical chart for this $d$-critical structure. Then, by Proposition \ref{cordcritgraded} $i)$, this d-critical structure coincide with the d-critical structure $\Grad(s):=\iota^\star(s)$. Notice that d-critical structures on a stack forms a set, hence there is no need to check any compatibility with product or étale maps of d-critical stacks.\medskip
    
    Notice that, using Lemma \ref{lemdarbgmeq}, we obtain similarly that, for $(\ud{X},\omega)$ a $\bG_m$-equivariant shifted symplectic algebraic space, with underlying $d$-critical algebraic space $(X,s)$, the $d$-critical structure $s$ is $\bG_m$-invariant, and the $d$-critical algebraic space underlying $(\ud{X}^0,\omega^0)$ is $(X^0,s^0:=\iota^\ast(s))$, as expected. This does not follows directly from the result of the Lemma, as $([\ud{X}/\bG_m],\bar{\omega})$ is not $-1$-shifted symplectic.\medskip

    \item[$ii)$] For $x\in\Grad(\mX)(k)=\ud{\Grad}(\ud{\mX})(k)$, with underlying point $\iota(x)\in\mX(k)$, we have (recalling that we use the chomological conventions, as usual):
    \begin{align}
        &H^{-1}((\ud{\iota}^\ast\bT_{\ud{\mX}})|_x)=H^{-1}((\iota^\ast\bT_{\mX})|_x)=\mathfrak{Iso}_{\iota(x)}(\mX)\nn\\
&H^0((\ud{\iota}^\ast\bT_{\ud{\mX}})|_x)=H^0((\iota^\ast\bT_{\mX})|_x)=T_{\mX,\iota(x)}
    \end{align}
    with their $\bZ$-grading. The $-1$ shifted symplectic form induces moreover isomorphisms:
    \begin{align}
        H^2((\ud{\iota}^\ast\bT_{\ud{\mX}})|^{>0}_x)&\simeq H^{-1}((\ud{\iota}^\ast\bT_{\ud{\mX}})|^{<0}_x)^\vee\nn\\
        H^1((\ud{\iota}^\ast\bT_{\ud{\mX}})|^{>0}_x)&\simeq H^0((\ud{\iota}^\ast\bT_{\ud{\mX}})|^{<0}_x)^\vee
    \end{align}
    and the other $H^i$'s vanishes by duality. Then, one can rewrite from the definition \eqref{defindmx}: 
    \begin{align}
    \Ind_\mX(x)&:=-\dim(\mathfrak{Iso}_{\iota(x)}(\mX)^{>0})+\dim((T_{\mX,\iota(x)})^{>0})-\dim((T_{\mX,\iota(x)})^{<0})+\dim(\mathfrak{Iso}_{\iota(x)}(\mX)^{<0})\nn\\
    &=-\dim(H^{-1}((\ud{\iota}^\ast\bT_{\ud{\mX}})|^{>0}_x))+\dim(H^0((\ud{\iota}^\ast\bT_{\ud{\mX}})|^{>0}_x))-\dim(H^1((\ud{\iota}^\ast\bT_{\ud{\mX}})|^{>0}_x))+\dim(H^2((\ud{\iota}^\ast\bT_{\ud{\mX}})|^{>0}_x))
\end{align}
which gives the result.\medskip

We obtain similarly that, for $\bG_m$-invariant $-1$-shifted symplectic algebraic spaces $(\ud{X},\omega)$ with classical truncation the $\bG_m$-invariant $d$-critical algebraic space $(X,s)$, $\Ind_X$ is the signed dimension of $\bT_X|_{X^0}^{>0}$ (notice that it does not follows directly from this Lemma as said above).\medskip

    \item[$iii)$] Consider again a smooth map $\ud{\phi}:\ud{R}\to\ud{\mX}$ from a derived scheme, a smooth scheme $U$, a function $f:U\to\bA^1_k$ and a map $\ud{j}:\ud{R}\to\ud{\Crit}(f)$ inducing an isomorphism of classical schemes such that $\ud{\phi}^\ast\omega\sim \ud{j}^\ast\omega_{\ud{\Crit}(f)}$, such that $(R,U,f,i)$ is a critical chart for $(\mX,s)$. Consider the two exact triangles:
    \begin{align}\label{exactri}
        &\ud{\phi}^\ast \bL_{\ud{\mX}}\to\bL_{\ud{R}}\to\bL_{\ud{\phi}}\nn\\
        &\ud{j}^\ast \bL_{\ud{\Crit}(f)}\to\bL_{\ud{R}}\to\bL_{\ud{j}}
    \end{align}
    We obtain an isomorphism:
    \begin{align}
        \det(\bL_{\ud{\mX}})|_{R^{red}}&\simeq \det(\bL_{\ud{\Crit}(f)})|_{R^{red}}\otimes \det(\bL_{\ud{j}})|_{R^{red}}\otimes \det(\bL_{\ud{\phi}})|^{\otimes -1}_{R^{red}}\nn\\
        &\simeq \det(\bL_U)|_{R^{red}}^{\otimes 2}\otimes\det(\bL_{\phi})|^{\otimes -2}_{R^{red}}\nn\\
        &\simeq K_{\mX,s}|_{R^{red}}
    \end{align}
    where the second line comes from $\bL_{\ud{j}}\simeq \bL_{\ud{\phi}}^\vee[2]$ (\cite[Equation 2.14]{darbstack}), and from the fact that $\bL_{\ud{\Crit}(f)}$ is quasi-isomorphic to $hess^\vee:\bL_U^\vee\to\bL_U$. These isomorphisms are glued in the proof of \cite[Theorem 3.18 b)]{darbstack} into a global isomorphism $\det(\bL_{\ud{\mX}})|_{\mX^{red}}\simeq K_{\mX,s}$. 
    
    Consider a data as in  Lemma \ref{lemdarbgmeq}. Applying the functor $\det((\iota^\ast-)^{<0})$ to the two exact triangles \eqref{exactri}, using the notations in the proof of Proposition \ref{cordcritgraded} $iii)$, we find moreover an isomorphism:
    \begin{align}\label{eqcomparG}
\det((\iota^\ast\bL_{\ud{\mX}})^{<0})|_{(R^0)^{red}}&\simeq\det((\ud{\iota}^\ast\bL_{\ud{\Crit}(f)})^{<0})|_{(R^0)^{red}}\otimes \det((\ud{\iota}^\ast\bL_{\ud{j}})^{<0})|_{(R^0)^{red}}\otimes \det((\ud{\iota}^\ast\bL_{\ud{\phi}}^{<0})|^{\otimes -1}_{(R^0)^{red}}\nn\\
        &\simeq K_U^{\neq 0}|_{(R^0)^{red}}\otimes K_{R/X}^{\neq 0}|_{(R^0)^{red}}\nn\\
        &\simeq G_{\mX,s}|_{(R^0)^{red}}
    \end{align}
    where we have used in the second line the fact that taking the dual reverses the $\bZ$-weight, hence one obtains the weights $>0$. This isomorphism is by construction compatible with \eqref{eqsqrootorientloc} and \eqref{eqsqrootshif}. As those are isomorphisms of line bundles on reduced stacks, to show that they glue into an isomorphism:
    \begin{align}
        \det((\iota^\ast\bL_{\ud{\mX}})^{<0})|_{\Grad(\mX)^{red}}\simeq G_{\mX,s}
    \end{align}
    of square roots, it suffices to show that their restriction to the stalk of an intersection point of critical charts agree. Consider the composed isomorphism:
    \begin{align}\label{eqstalkGshif}
\det((\iota^\ast\bL_{\ud{\mX}})^{<0})|_x\simeq G_{\mX,s}|_x\simeq \det((\iota^\ast\tau^{\leq 0}\bL_\mX)^{\neq 0})|_x
    \end{align}
    where the first isomorphism is the stalk of \eqref{eqcomparG} and the second one is \eqref{eqstalkG}. Unwrapping the definition, this isomorphism is the natural one coming from the quasi-isomorphism $\tau^{\leq 0}\bL_{\ud{\mX}})^\vee\simeq (\tau^{>1}\bL_{\ud{\mX}})[1]$ coming from $\omega_{\ud{\mX}}$, hence is independent from the critical chart, which completes the proof.\medskip

    We obtain a similar formula for $\bG_m$-invariant $-1$-shifted symplectic algebraic spaces $(\ud{X},\omega)$ by using the same arguments (notice that it does not follows directly from this Lemma as said above).
\end{enumerate}

\end{proof}

\subsubsection{Hyperbolic localization of the DT sheaf}

\begin{theorem}\label{theohyploc}
    Consider $(\mX,s,K_{\mX,s}^{1/2})$ an oriented d-critical stack which is quasi-separated, locally of finite type, with reductive stabilizers and separated diagonal. Consider the oriented d-critical stack:
    \begin{align}
        (\Grad(\mX),\Grad(s), K_{\Grad(\mX),\Grad(s)}^{1/2})
    \end{align}
    from Proposition \ref{cordcritgraded}. There is a natural isomorphism in $\mA_{mon,c}(\mX)$:
    \begin{align}\label{isomhyplocjoyce}
        p_!\eta^\ast P_{\mX,s,K_{\mX,s}^{1/2}}\simeq P_{\Grad(\mX),\Grad(s),\Grad(K_{\mX,s}^{1/2})}\{-\Ind_{\mX}/2\}
    \end{align}
    This isomorphism is compatible with the isomorphism of smooth pullbacks and exterior product of d-critical stacks from Corollary \ref{corfunctjoyce}.
\end{theorem}

\begin{proof}
    Consider a critical chart $(\mR,\mU,f,i)$ and the smooth morphism $\phi:\mR\to\mX$. We consider the isomorphism:
    \begin{align}\label{isomhyplocjoyceloc}
         \Grad(\phi)^\ast\{d_{\Grad(\phi)}/2\}(p_\mX)_!(\eta_\mX)^\ast P_{\mX,s,K_{\mX,s}^{1/2}}&\simeq (p_\mR)_!(\eta_\mR)^\ast \phi^\ast\{d_\phi/2\}P_\mX\{\Ind_\phi/2\}\nn\\
         &:=(p_\mR)_!(\eta_\mR)^\ast (P_{\mU,f}\otimes_{\bZ/2\bZ} Q_{(\mR,\mU,f,i)})\{\Ind_\phi/2\}\nn\\
         &\simeq (p_\mR)_!(\eta_\mR)^\ast P_{\mU,f}\otimes_{\bZ/2\bZ} Q_{\Grad(\mR,\mU,f,i))}\{\Ind_\phi/2\}\nn\\
         &\simeq P_{\Grad(\mU),\Grad(f)}\{-\Ind_{\mU,f}/2\}\otimes_{\bZ/2\bZ} Q_{(\Grad(\mR,\mU,f,i))}\{\Ind_\phi/2\}\nn\\
         &=: \Grad(\phi)^\ast\{d_{\Grad(\phi)}/2\}P_{\Grad(\mX),\Grad(s),\Grad(K_{\mX,s}^{1/2})}\{-\Ind_\mX/2\}
    \end{align}
    were the first isomorphism comes from Proposition \ref{propsmoothBB}, the second by Proposition \ref{proplocdefperv} $i)$, the third from \eqref{hyplocQ} and Lemma \ref{lemtwisthyploc}, the fourth by Proposition \ref{prophyploccrit}, and the last by Proposition \ref{proplocdefperv} $i)$ and Proposition \ref{cordcritgraded} $ii)$.\medskip

    Consider a smooth restriction of critical charts $\phi':(\mR',\mU',f',i')\to(\mR,\mU,f,i)$, and the square of isomorphisms (were we have removed the dependency in $s$ and $K_{\mX,s}^{1/2}$ for readability):
    
    \begin{tikzcd}[column sep=tiny]
        \begin{tabular}{c}$\Grad(\tilde{\phi}')^\ast\{d_{\Grad(\tilde{\phi}')}/2\}$\\$\Grad(\phi)^\ast\{d_{\Grad(\phi)}/2\}(p_\mX)_!(\eta_\mX)^\ast P_{\mX}$\end{tabular}\arrow[r,"\simeq"]\arrow[d,"\simeq"] & \begin{tabular}{c}$\Grad(\tilde{\phi}')^\ast\{d_{\Grad(\tilde{\phi}')}/2\}$\\$\Grad(\phi)^\ast\{d_{\Grad(\phi)}/2\}P_{\Grad(\mX)}\{-\Ind_\mX/2\}$\end{tabular}\arrow[d,"\simeq"]\\
        \Grad(\phi\circ\tilde{\phi}')^\ast\{d_{\Grad(\phi\circ\tilde{\phi}')}/2\}P_{\mX}\arrow[r,"\simeq"] & \Grad(\phi\circ\tilde{\phi}')^\ast\{d_{\Grad(\phi\circ\tilde{\phi}')}/2\}P_{\Grad(\mX)}\{-\Ind_\mX/2\}
    \end{tikzcd}
    
    where the horizontal arrows comes from \eqref{isomhyplocjoyceloc} respectively for $(\mR,\mU,f,i)$ and $(\mR',\mU',f',i')$. To show that this square commutes, it suffices to decompose it into five little square, according to the decomposition of the isomorphism \eqref{isomhyplocjoyceloc}. The first square commutes from the compatibility with composition stated in Proposition \ref{propsmoothBB}, the second one by Proposition \ref{proplocdefperv} $ii)$, the third one by the compatibility with smooth pullbacks stated below \eqref{hyplocQ} and in Lemma \ref{lemtwisthyploc}, the fourth by Lemma \ref{lemsmoothresthyploc}, and the last one by Proposition \ref{proplocdefperv} $ii)$ again.\medskip

    Consider a critical chart $(\mR,\mU,f,i)$ and its stabilization $(\mR,\bV_\mU(\mE),f\circ\pi+q,s\circ i)$ by the quadratic bundle $(\mE,q)$ on $\mU$. Consider the square of isomorphisms:
    \[\begin{tikzcd}[column sep=tiny]
        \Grad(\phi)^\ast\{d_{\Grad(\phi)}/2\}(p_\mX)_!(\eta_\mX)^\ast P_{\mX}\arrow[r,"\simeq"]\arrow[d,"\simeq"] & \Grad(\phi)^\ast\{d_{\Grad(\phi)}/2\}P_{\Grad(\mX)}\{-\Ind_\mX/2\}\arrow[d,"\simeq"]\\
        \Grad(\phi)^\ast\{d_{\Grad(\phi)}/2\}(p_\mX)P_{\mX}\arrow[r,"\simeq"] & \Grad(\phi)^\ast\{d_{\Grad(\phi)}/2\}P_{\Grad(\mX)}\{-\Ind_\mX/2\}
    \end{tikzcd}\]
    where the horizontal arrows comes from the \eqref{isomhyplocjoyceloc} respectively for $(\mR,\bV_\mU(\mE),f\circ\pi+q,s\circ i)$ and $(\mR,\mU,f,i)$. To show that this square commutes, it suffices to decompose it into five little square, according to the decomposition of the isomorphism \eqref{isomhyplocjoyceloc}. The first square trivially commutes, the second one by Proposition \ref{proplocdefperv} $iii)$, the third is the commutative square giving the compatibility with stabilization stated below \eqref{hyplocQ}, the fourth one by Lemma \ref{lemcompathyplocstab}, and the last one by Proposition \ref{proplocdefperv} $iii)$ again.\medskip

    By Proposition \ref{propcomparchart} $i)$, $\Grad(\mX)$ is covered by critical charts of the form $(\Grad(\mR),\Grad(\mU),\Grad(f),\Grad(i))$, hence \eqref{isomhyplocjoyceloc} gives a smooth local definition of the isomorphism \eqref{isomhyplocjoyce}. By the above, these isomorphisms are compatible with the $\Grad$ of smooth restrictions and stabilization of critical charts of $\mX$, hence they agree smooth locally from \ref{propcomparchart} $ii)$, \ie by smooth descent they glue into a unique isomorphism \eqref{isomhyplocjoyce}.\medskip

    The compatibility with the isomorphisms of smooth pullbacks from Corollary \ref{corfunctjoyce} is direct from the definition, as, given a smooth morphism $\phi:(\mX,s)\to(\mY,t)$, the critical charts of $(\mX,s)$ forms a subset of the critical charts of $(\mY,t)$, hence the isomorphism \eqref{isomhyplocjoyce} for $(\mX,s,K_{\mX,s}^{1/2})$ is defined locally by applying $\phi^\ast\{d_\phi/2\}$ to the isomorphisms giving the local definition of isomorphism \eqref{isomhyplocjoyce} for $(\mY,t,K_{\mY,t}^{1/2})$.\medskip

    We want to show the compatibility with the isomorphism of exterior product from \ref{corfunctjoyce}, namely that, given two oriented d-critical stacks $(\mX_i,s_i,K_{\mX_i,s_i}^{1/2})$ for $i=1,2$, the following square of isomorphisms is commutative (were we have removed the dependency in $s$ and $K_{\mX,s}^{1/2}$ for readability):
    \[\begin{tikzcd}
        (p_1\times p_2)_!(\eta_1\times\eta_2)^\ast P_{\mX_1\times\mX_2}\arrow[r,"\simeq"]\arrow[d,"\simeq"] & P_{\Grad(\mX_1\times\mX_2)}\{-\Ind_{\mX_1\times\mX_2}/2\}\arrow[d,"\simeq"]\\
        (p_1)_!(\eta_1)^\ast P_{\mX_1}\boxtimes (p_2)_!(\eta_2)^\ast P_{\mX_2}\arrow[r,"\simeq"] & P_{\Grad(\mX_1)}\{-\Ind_{\mX_1}/2\}\boxtimes P_{\Grad(\mX_2)}\{-\Ind_{\mX_2}/2\}
    \end{tikzcd}\]
    it suffices to check that smooth locally, \ie that, for critical charts $(\mR_i,\mU_i,f_i,i_i)$ of $(\mX_i,s_i)$ with $\phi_i:\mR_i\to\mX_i$, the square obtained by applying $\Grad(\phi_1\times\phi_2)^\ast\{d_{\phi_1\times\phi_2}/2\}$ to the latter one, with horizontal arrows given by \eqref{isomhyplocjoyceloc}, commutes. To show that this square commutes, it suffices to decompose it into five little square, according to the decomposition of the isomorphism \eqref{isomhyplocjoyceloc}. The first square commutes from the compatibility with exterior product in Proposition \ref{propsmoothBB}, the second one by the definition of the isomorphism in Corollary \ref{corfunctjoyce}, the third one by the compatibility with exterior products stated below \eqref{hyplocQ} and in Lemma \ref{lemtwisthyploc}, the fourth one by the compatibility with exterior products in Proposition \ref{prophyploccrit}, and the last one by the definition of the isomorphism in Corollary \ref{corfunctjoyce} again.
\end{proof}
    
\subsubsection{Composition of hyperbolic localizations}\label{secttau}

Consider an Artin $k$-stack $\mX$ which is quasi-separated, locally of finite type, with affine diagonal. We have:
\begin{align}
    \Grad(\Grad(\mX)):=Map(B\bG_{m,k},Map(B\bG_{m,k},\mX))\simeq Map(B\bG_{m,k}\times B\bG_{m,k},\mX)
\end{align}
There is then a natural involution $\tau$ of $\Grad(\Grad(\mX))$ obtained by swapping the two $B\bG_{m,k}$.

\begin{lemma}\label{lemtauorientdcrit}
\begin{itemize}
    \item[$i)$] Given an oriented $d$-critical stack $(\mX,s,(K_{\mX,s}^{1/2}))$, $\tau$ enhances to an involution of the (oriented) d-critical stack $(\Grad(\Grad(\mX)),\Grad(\Grad(s)),(\Grad(\Grad(K_{\mX,s}^{1/2})))$. This enhancement is compatible with smooth morphisms and products of (oriented) d-critical stacks.
    \item[$ii)$] Consider a $-1$-shifted symplectic stack $(\ud{\mX},\omega)$ with classical truncation $(\mX,s)$. Using the isomorphism $K_{\mX,s}\simeq \det(\bL_{\mX})|_{\mX^{red}}$ and Lemma \ref{lemshifsymp} $iii)$, consider the isomorphism:
    \begin{align}
\Grad(\Grad(K_{\mX,s}^{1/2})):=&\Grad(\iota^\ast(K_{\mX,s}^{1/2})\otimes\det((\ud{\iota}^\ast\bL_{\ud{\mX}})^{<0})|^{\otimes -1}_{\Grad(\Grad(\mX))^{red}})\nn\\
        :=&\Grad(\iota)^\ast\iota^\ast(K_{\mX,s}^{1/2})\otimes \det((\ud{\Grad}(\ud{\iota})^\ast\bL_{\ud{\Grad}(\ud{\mX})})^{<0})|^{\otimes -1}_{\Grad(\Grad(\mX))^{red}}\nn\\
        &\otimes \det(\ud{\Grad}(\ud{\iota})^\ast(\ud{\iota}^\ast\bL_{\ud{\mX}})^{<0})|^{\otimes -1}_{\Grad(\Grad(\mX))^{red}}\nn\\
        \simeq& \Grad(\iota)^\ast\iota^\ast(K_{\mX,s}^{1/2})\otimes  (\ud{\Grad}(\ud{\iota})^\ast\ud{\iota}^\ast\bL_{\ud{\mX}})^{H_1}|^{\otimes -1}_{\Grad(\Grad(\mX))^{red}}
    \end{align}
    where $H_1$ is the subset of $\bZ^2$ of $(w_1,w_2)$ such that $w_1<0$ or $w_2=0$ and $w_1<0$, and denote by $H_2$ its image by swapping the two factors, \ie the subset of $\bZ^2$ of $(w_1,w_2)$ such that $w_2<0$ or $w_1=0$ and $w_2<0$. The isomorphism of $i)$ is identified with:
    \begin{align}\label{isomtaushifsymp}
        \tau^\ast\Grad(\Grad(K_{\mX,s}^{1/2}))&\simeq\tau^\ast\Grad(\iota)^\ast\iota^\ast(K_{\mX,s}^{1/2})\otimes  \ud{\tau}^\ast((\ud{\Grad}(\ud{\iota})^\ast\ud{\iota}^\ast\bL_{\ud{\mX}})^{H_1})|^{\otimes -1}_{\Grad(\Grad(\mX))^{red}}\nn\\
        &\simeq \Grad(\iota)^\ast\iota^\ast(K_{\mX,s}^{1/2})\otimes  (\ud{\Grad}(\ud{\iota})^\ast\ud{\iota}^\ast\bL_{\ud{\mX}})^{H_2}|^{\otimes -1}_{\Grad(\Grad(\mX))^{red}}\nn\\
        &\simeq \Grad(\iota)^\ast\iota^\ast(K_{\mX,s}^{1/2})\otimes  (\ud{\Grad}(\ud{\iota})^\ast\ud{\iota}^\ast\bL_{\ud{\mX}})^{H_1}|^{\otimes -1}_{\Grad(\Grad(\mX))^{red}}\nn\\
        &\simeq \Grad(\Grad(K_{\mX,s}^{1/2}))
    \end{align}
    where the isomorphism of the second line comes by applying the isomorphism:
    \begin{align}\label{eqtildeLweight}
        \ud{\Grad}(\ud{\iota})^\ast\ud{\iota}^\ast(\omega):(\ud{\Grad}(\ud{\iota})^\ast\ud{\iota}^\ast\bL_{\ud{\mX}})^{w_1,w_2}\simeq (\ud{\Grad}(\ud{\iota})^\ast\ud{\iota}^\ast\bL_{\ud{\mX}})^{-w_1,-w_2}
    \end{align}
    to the weights $(w_1\geq 0,w_2<0)$.
\end{itemize}
\end{lemma}

\begin{proof}
\begin{enumerate}
    \item[$i)$] 
    Given a d-critical stack $(\mX,s)$, using the natural 2-isomorphism:
    \begin{align}\label{iotatau}
        \iota_\mX\circ\iota_{\Grad(\mX)}\simeq \iota_\mX\circ\iota_{\Grad(\mX)}\circ \tau
    \end{align}
    one finds that $\tau^\star(\Grad(\Grad(s)))=\Grad(\Grad(s))$, hence $\tau$ gives an automorphism of the $d$-critical stack $(\Grad(\Grad(\mX)),\Grad(\Grad(s)))$.\medskip
    
    Consider an orientation $K_{\mX,s}^{1/2}$ on $(\mX,s)$, we have to build a natural isomorphism:
    \begin{align}\label{eqgradgradcanbun}
        \tau^\ast\Grad(\Grad(K_{\mX,s}^{1/2}))\simeq \Grad(\Grad(K_{\mX,s}^{1/2}))
    \end{align}
    which is a square root of $\tau^\ast\Grad(\Grad(K_{\mX,s}))\simeq \Grad(\Grad(K_{\mX,s}))$. We will define it smooth-locally on critical charts, and check that their stalks at each $k$-point is independent from the choice of critical chart.\medskip

    The two actions of $B\bG_{m,k}$ on $\Grad(\Grad(\mX))$ give two $\bZ$-grading on quasi-coherent complexes on $\Grad(\Grad(\mX))$, with $\tau^\ast$ swapping the two factors. Denote by $\mE^{Z}$ the component of $\mE$ whose weights are in $Z\subset\bZ^2$. We have then, for a smooth map $\mX\to\mY$:
    \begin{align}
        (\Grad(\iota)^\ast \bL_{\Grad(\mX)/\Grad(\mY)})^{\neq 0}&\simeq ((\iota\circ \Grad(\iota))^\ast \bL_{\mX/\mY})^{\{0\}\times\bZ^\ast}\nn\\
        \Grad(\iota)^\ast (\iota^\ast\bL_{\mX/\mY})^{\neq 0}&\simeq ((\iota\circ \Grad(\iota))^\ast \bL_{\mX/\mY})^{\bZ^\ast\times\bZ}
    \end{align}
    Recall the notation of \eqref{defKI}, and, for $Z\subset\bZ^2$, define similarly:
    \begin{align}
        K_{\mX/\mY}^Z:=\det((\iota\circ\Grad(\iota))^\ast \bL_{\mX/\mY})^Z)|_{\Grad(\Grad(\mX))^{red}}
    \end{align}
    we obtain:
    \begin{align}
        K_{\Grad(\mX)/\Grad(\mY)}^{\neq 0}
        \otimes\Grad(\iota)^\ast (K_{\mX/\mY}^{\neq 0})\simeq K_{\mX/\mY}^{\bZ^2-\{(0,0)\}}
    \end{align}
    and we obtain then, using $\iota_\mX\circ\iota_{\Grad(\mX)}\simeq \iota_\mX\circ\iota_{\Grad(\mX)}\circ \tau$ and the fact that $\bZ^2-\{(0,0)\}$ is invariant under exchange of the two factors, a natural isomorphism:
    \begin{align}\label{eqtauKsmooth}
        \tau^\ast (K_{\Grad(\mX)/\Grad(\mY)}^{\neq 0}
        \otimes\Grad(\iota)^\ast K_{\mX/\mY}^{\neq 0})\simeq  K_{\Grad(\mX)/\Grad(\mY)}^{\neq 0}
        \otimes\Grad(\iota)^\ast K_{\mX/\mY}^{\neq 0}
    \end{align}
    which is compatible with exterior product, and composition of smooth morphisms. Given a critical chart $(\mR,\mU,f,i)$ of $(\mX,s)$, recall that:
    \begin{align}
        \Grad(K_{\mX,s}^{1/2})|_{\Grad(\mR)^{red}}:=\iota^\ast K_{\mX,s}^{1/2}|_{\Grad(\mR)}\otimes(\Grad(i)^\ast K_\mU^{\neq 0})^{\otimes -1}\otimes K_{\mR/\mX}^{\neq 0}
    \end{align}
    Using \eqref{eqtauKsmooth} for $\mR\to\mX$ and $\mU\to\Spec(k)$, we obtain then a natural isomorphism: \begin{align}\label{eqgradgradcanbunloc}
\tau^\ast\Grad(\Grad(K_{\mX,s}^{1/2}))|_{\Grad(\Grad(\mR))^{red}}\simeq \Grad(\Grad(K_{\mX,s}^{1/2}))|_{\Grad(\Grad(\mR))^{red}}
    \end{align}
    compatible with exterior products of critical charts.\medskip

    Consider a point $x\in\Grad(\Grad(k))$. Using the notation in the proof of Proposition \ref{cordcritgraded} $iii)$, under the isomorphism \eqref{eqstalkG}, using the determinant of the Hessian of the $\bG_m$-invariant function $f$, one find directly that the stalk of \eqref{eqgradgradcanbunloc} comes from:
    \begin{align}\label{eqstalktau}
        \Grad(\Grad(K_{\mX,s}^{1/2}))|_x&\simeq \Grad(K_{\mX,s}^{1/2})|_{\Grad(\iota)(x)}\otimes (K_{\Grad(\mX)}|_x^{\neq 0})^{\otimes -1}\nn\\
        &\simeq K_{\mX,s}^{1/2}|_{\iota\circ\Grad(\iota)(x)}\otimes (K_\mX|_{\Grad(\iota)(x)}^{\neq 0})^{\otimes -1}\otimes (K_{\Grad(\mX)}|_x^{\neq 0})^{\otimes -1}\nn\\
        &\simeq K_{\mX,s}^{1/2}|_{\iota\circ\Grad(\iota)(x)}\otimes (K_\mX|_x^{\bZ^2-\{(0,0)\}})^{\otimes -1}
    \end{align}
    then the stalk of \eqref{eqgradgradcanbunloc} is independent from the critical chart.\medskip

     We can now conclude as in Proposition \ref{cordcritgraded}. Indeed, \cite[Lemma 4.4.6]{HalpernLeistner2014OnTS} applies also to $\Grad\circ\Grad$, hence applying Proposition \eqref{propcompareqchart} $i)$ to the two dimensional torus $\bG_{m,k}\times\bG_{m,k}$ one obtains an analogue of Proposition \ref{propcomparchart} $i)$ saying that $\Grad(\Grad(\mX))$ is covered by $\Grad\circ\Grad$ of critical charts of $(\mX,s)$. As we want to define an isomorphism of line bundles on a reduced stack, and the stalk of \eqref{eqgradgradcanbunloc} at each $k$-point is independent of the critical chart, one  obtain directly that they glue into a global isomorphism \eqref{eqgradgradcanbun}. The compatibility with smooth morphisms come for free, and the compatibility with products from the fact that \eqref{eqgradgradcanbunloc} is compatible with products.\medskip

     Consider a critical chart $(\mR,\mU,f,i)$, and the square of isomorphisms:
    \[\begin{tikzcd}
        \tau^\ast\Grad(\iota)^\ast\iota^\ast Q_{(\mR,\mU,f,i)}\arrow[r,"\simeq"]\arrow[d,"\simeq"] & \Grad(\iota)^\ast\iota^\ast Q_{(\mR,\mU,f,i)}\arrow[d,"\simeq"]\\
\tau^\ast\Grad(\iota)^\ast Q_{\Grad((\mR,\mU,f,i))}\arrow[d,"\simeq"] & \Grad(\iota)^\ast Q_{\Grad((\mR,\mU,f,i))}\arrow[d,"\simeq"]\\
\tau^\ast Q_{\Grad(\Grad((\mR,\mU,f,i)))}\arrow[r,"\simeq"] & Q_{\Grad(\Grad((\mR,\mU,f,i)))}
    \end{tikzcd}\]
    where the vertical isomorphisms are \eqref{hyplocQ}, the upper horizontal is \eqref{iotatau}, and the lower left comes from the smooth restriction $\tau:\Grad(\Grad((\mR,\mU,f,i)))\to \Grad(\Grad((\mR,\mU,f,i)))$. Notice that we have:
    \begin{align}\label{eqtau}
       &\Grad(\Grad(K_{\mX,s}^{1/2}))|_{\Grad(\Grad(\mR))}\nn\\
       \simeq& (\iota\circ\Grad(\iota))^\ast K_{\mX,s}^{1/2}|_{\Grad(\Grad(\mR))}\otimes(\Grad(\Grad(i))^\ast K_\mU^{\bZ^2-\{(0,0)\}})^{\otimes -1}\otimes (K_{\mR/\mX}^{\bZ^2-\{(0,0)\}})
    \end{align}
    And the vertical isomorphisms associate, to an isomorphism:
    \begin{align}
    K_{\mX,s}^{1/2}|_{\mR^{red}}\simeq i^\ast(K_\mU)|_{\mR^{red}}\otimes (K_{\mR/\mX})|^{\otimes -1}_{\mR^{red}}
\end{align}
the isomorphism:
\begin{align}
&\Grad(\Grad(K_{\mX,s}^{1/2})|_{\Grad(\Grad(\mR))^{red}}\nn\\
\simeq& (\Grad(\Grad(i)))^\ast K_{\mU}^{\{(0,0)\}}|_{\Grad(\Grad(\mR))^{red}}\otimes (K_{\mR/\mX}^{\{(0,0)\}})^{\otimes -1}|_{\Grad(\Grad(\mR))^{red}}\nn\\
=:& (\Grad(\Grad(i)))^\ast(K_{\Grad(\Grad(\mU)))}|_{\Grad(\Grad(\mR))^{red}}\otimes (K_{\Grad(\Grad(\mR))/\Grad(\Grad(\mX))})|^{\otimes -1}_{\Grad(\Grad(\mR))^{red}}
\end{align}
defined in the obvious way from \eqref{eqtau}, and this is compatible with the action of $\tau$ swapping the two factors of $\bZ$, hence the above diagram of isomorphisms commutes.\medskip

\item[$ii)$] We will show that these isomorphisms agree at the stalk of each $k$ point $x$ of $\Grad(\Grad(\mX))$. We can argue as in Lemma \ref{lemdarbgmst}, adapting Lemma \ref{lemdarbgmeq} to find a $\bG_m\times\bG_m$-equivariant data as in Lemma \ref{lemdarbgmeq} over $x$, and follow then Lemma \ref{lemshifsymp} $iii)$. Under \eqref{eqstalkGshif}, which is induced by the pairing $\tilde{\omega}$, one obtains directly that the stalk of \eqref{isomtaushifsymp} identifies with the stalk of \eqref{eqgradgradcanbun}, coming from \eqref{eqstalktau}. Then \eqref{isomtaushifsymp} and \eqref{eqgradgradcanbun} are two isomorphisms of line bundle on reduced stacks agreeing on the stalk of each $k$-point, hence they are equal.
\end{enumerate}

\end{proof}

Consider the commutative diagram:
\[\begin{tikzcd}
    B\bG_{m,k}\times B\bG_{m,k}\arrow[r] & \Theta_k\times B\bG_{m,k} & \Spec(k)\times B\bG_{m,k}\arrow[l]\\
    B\bG_{m,k}\times B\bG_{m,k}\arrow[d]\arrow[r]\arrow[u] & B\bG_{m,k}\times\Theta_k\arrow[d]\arrow[u] & B\bG_{m,k}\times \Spec(k)\arrow[d]\arrow[l]\arrow[u] \\
    \Theta_k\times B\bG_{m,k}\arrow[r] & \Theta_k\times\Theta_k & \Theta_k\times \Spec(k)\arrow[l]\\
    \Spec(k)\times B\bG_{m,k}\arrow[u]\arrow[r] & \Spec(k)\times\Theta_k\arrow[u] & \Spec(k)\times \Spec(k)\arrow[u]\arrow[l]
\end{tikzcd}\]
where the arrows from the second to the first line are obtained by swapping the factors, and the two upper squares the right central square, and the left lower square are Cocartesian. Using mapping stacks, we obtains the following commutative diagram:

\begin{equation}\label{diagcomphyploc}\begin{tikzcd}
\Grad(\Grad(\mX))\arrow[d,"\tau"] & \Filt(\Grad(\mX))\arrow[r,"\eta_{\Grad(\mX)}"]\arrow[l,swap,"p_{\Grad(\mX)}"]\arrow[d,"\tau'"] & \Grad(\mX)\arrow[d,"Id"]\\
    \Grad(\Grad(\mX)) & \Grad(\Filt(\mX))\arrow[r,"\Grad(\eta_\mX)"]\arrow[l,swap,"\Grad(p_\mX)"] & \Grad(\mX)\\
    \Filt(\Grad(\mX))\arrow[u,swap,"p_{\Grad(\mX)}"]\arrow[d,"\eta_{\Grad(\mX)}"] & \Filt(\Filt(\mX))\arrow[r,"\Filt(\eta_\mX)"]\arrow[l,swap,"\Filt(p_\mX)"]\arrow[u,swap,"p_{\Filt(\mX)}"]\arrow[d,"p_{\Filt(\mX)}"] & \Filt(\mX)\arrow[u,swap,"p_\mX"]\arrow[d,"\eta_\mX"]\\
    \Grad(\mX) & \Filt(\mX)\arrow[r,"\eta_\mX"]\arrow[l,swap,"p_\mX"] & \mX
\end{tikzcd}\end{equation}

where the two upper squares the right central square, and the left lower square are Cartesian. Using functoriality and base change in this diagram, we obtain a natural isomorphism:
\begin{align}\label{isomtau}
\tau^\ast(p_{\Grad(\mX)})_!(\eta_{\Grad(\mX)})^\ast(p_\mX)_!(\eta_\mX)^\ast\simeq (p_{\Grad(\mX)})_!(\eta_{\Grad(\mX)})^\ast(p_\mX)_!(\eta_\mX)^\ast
\end{align}

The following proposition is the key to prove the associativity of the CoHA: it gives some kind of operadic properties on the hyperbolic localization isomorphism.

\begin{lemma}\label{lemcomposHyploc}
    Given an oriented d-critical Artin stack $(\mX,s,K_{\mX,s}^{1/2})$, the following diagram of isomorphisms is commutative:
    \[\begin{tikzcd}
        \tau^\ast(p_{\Grad(\mX)})_!(\eta_{\Grad(\mX)})^\ast(p_\mX)_!(\eta_\mX)^\ast P_\mX\arrow[r,"\simeq"]\arrow[d,"\simeq"] & (p_{\Grad(\mX)})_!(\eta_{\Grad(\mX)})^\ast(p_\mX)_!(\eta_\mX)^\ast P_\mX\arrow[d,"\simeq"]\\
\tau^\ast(p_{\Grad(\mX)})_!(\eta_{\Grad(\mX)})^\ast P_{\Grad(\mX)}\{-\Ind_\mX/2\}\arrow[d,"\simeq"] & (p_{\Grad(\mX)})_!(\eta_{\Grad(\mX)})^\ast P_{\Grad(\mX)}\{-\Ind_\mX/2\}\arrow[d,"\simeq"]\\
\tau^\ast P_{\Grad(\Grad(\mX))}\{-\Ind_\mX/2-\Ind_{\Grad(\mX)}/2\}\arrow[r,"\simeq"] & P_{\Grad(\Grad(\mX))}\{-\Ind_\mX/2-\Ind_{\Grad(\mX)}/2\}
    \end{tikzcd}\]
    where the upper horizontal arrow comes from \eqref{isomtau}, and the lower horizontal arrow by the functoriality of Corollary \ref{corfunctjoyce} with respect to the isomorphism of oriented d-critical stack $\tau$.
\end{lemma}

\begin{proof}
    The isomorphism of commutation between specialization systems (like $\phi^{mon,tot}$) and hyperbolic localization from Theorem \ref{commutspechyploc} is built from the natural transformations of the specialization system which are compatible with base change. By base change in the commutative diagram with Cartesian square \eqref{diagcomphyploc}, we obtain that it is compatible isomorphism \eqref{isomtau} in the obvious sense.\medskip

    We will sketch now why \eqref{isomtau} is compatible with the isomorphisms of Proposition \ref{propsmoothBB} and Lemma \ref{lemhyploclosed}. Both are built using the same principle: one consider morphisms of correspondence:
    \[\begin{tikzcd}
        \mX_3'\arrow[d,"\phi_3"] & \mX_2'\arrow[d,"\phi_2"]\arrow[l,"p'"]\arrow[r,"\eta'"] & \mX_1'\arrow[d,"\phi_1"]\arrow[l]\\
        \mX_3 & \mX_2\arrow[l,"p"]\arrow[r,"\eta"] & \mX_1
    \end{tikzcd}\]
    with special properties on the morphism $\pi:\mX'_2\to\mX_2\times_{\mX_3}\mX'_3$. In Proposition \ref{propsmoothBB}, one ask that $\pi$ is smooth vector bundle stack modeled on $\mE$, which gives a natural adjunction isomorphism $\pi_!\pi^\ast\simeq\Sigma^{-\mE}$, and in Lemma \ref{lemhyploclosed}, one ask that $\pi$ is closed, which gives a natural adjunction morphism $Id\to \pi_!\pi^\ast$. The morphism is then obtained by base change. Consider now the composition of two such correspondences, and consider the following commutative diagram:
    
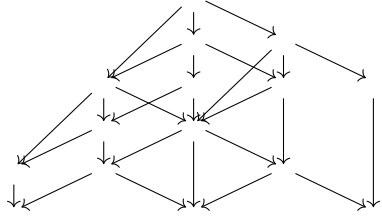
\[\begin{tikzcd}[row sep=small]
        && \arrow[ddl]\arrow[d]\arrow[dr] &&\\
        &&\arrow[dl]\arrow[d]\arrow[dr] & \arrow[ddl]\arrow[d]\arrow[dr]&\\
        & \arrow[ddl]\arrow[d]\arrow[dr]& \arrow[dl]\arrow[d] &\arrow[dl]\arrow[dd]& \arrow[ddd]\\
        & \arrow[dl]\arrow[d] & \arrow[dl]\arrow[dd]\arrow[dr] & & \\
        \arrow[d] & \arrow[dl]\arrow[dr] & & \arrow[dl]\arrow[dr] &\\
       \; && \; && \;
    \end{tikzcd}\]
    in which the squares parallel to the left plane are Cartesian. we obtain that the arrow $\pi$ for the composed correspondence is the composition of base change of the arrow $\pi$ for the two correspondence. As noticed in Proposition \ref{propsmoothBB} and Lemma \ref{lemhyploclosed}, the adjunction morphisms for $\pi$ are compatible with composition of $\pi$, and base change of $\pi$, hence  we obtain by base change in the above diagram that the composition of the morphisms for the two correspondences is the morphism for the composed correspondence. Noticing that the correspondence given by the diagonal of the lower big square of diagram \ref{diagcomphyploc} can be seen as the composition of correspondences in two ways, and that the isomorphisms of Proposition \ref{propsmoothBB} and Lemma \ref{lemhyploclosed} are compatible with smooth pullbacks, in particular pullbacks along the upper rectangle of diagram \ref{diagcomphyploc}, we obtain the claimed compatibility.\medskip

    it suffices to show that the diagram of the Lemma is commutative smooth locally, hence we can work on a critical chart $(\mR,\mU,f,i)$ and with the morphism \eqref{isomhyplocjoyceloc} (as $\Grad(\Grad(\mR,\mU,f,i))$ cover $\Grad(\Grad(\mX))$ as remarked in the proof of Lemma \ref{lemtauorientdcrit} $i)$). We reproduce this isomorphism here for convenience:
    \begin{align}
         \Grad(\phi)^\ast\{d_{\Grad(\phi)}/2\}(p_\mX)_!(\eta_\mX)^\ast P_{\mX,s,K_{\mX,s}^{1/2}}&\simeq (p_\mR)_!(\eta_\mR)^\ast \phi^\ast\{d_\phi/2\}P_\mX\{\Ind_\phi/2\}\nn\\
         &:=(p_\mR)_!(\eta_\mR)^\ast (P_{\mU,f}\otimes_{\bZ/2\bZ} Q_{(\mR,\mU,f,i)})\{\Ind_\phi/2\}\nn\\
         &\simeq (p_\mR)_!(\eta_\mR)^\ast P_{\mU,f}\otimes_{\bZ/2\bZ} Q_{\Grad(\mR,\mU,f,i))}\{\Ind_\phi/2\}\nn\\
         &\simeq P_{\Grad(\mU),\Grad(f)}\{-\Ind_{\mU,f}/2\}\otimes_{\bZ/2\bZ} Q_{(\Grad(\mR,\mU,f,i))}\{\Ind_\phi/2\}\nn\\
         &=: \Grad(\phi)^\ast\{d_{\Grad(\phi)}/2\}P_{\Grad(\mX),\Grad(s),\Grad(K_{\mX,s}^{1/2})}\{-\Ind_\mX/2\}
    \end{align}
    it suffices then to show that the five squares, obtained from the five isomorphisms composing this isomorphism, commutes. The commutativity of the first square follows from the compatibility with the isomorphism of Proposition \ref{propsmoothBB}. The commutativity of the second and last square is trivial from the definition. By base change in the diagram \ref{diagcomphyploc}, the isomorphism of Lemma \ref{lemtwisthyploc} is compatible with composition of hyperbolic localizations, and the commutativity of the diagram at the end of the proof of Lemma \ref{lemtauorientdcrit} give that the third square is commutative. The compatibility of \eqref{isomtau} with the isomorphism of Theorem \ref{commutspechyploc}, Proposition \ref{propsmoothBB} and Lemma \ref{lemhyploclosed} give directly the compatibility with the isomorphism of Proposition \ref{prophyploccrit}, \ie the commutativity of the fourth square.\end{proof}

\section{Applications}\label{sectapp}

\subsection{Bialynicki-Birula decomposition}

Consider now a $\bG_m$-invariant $d$-critical algebraic space $(X,s)$: recall from Section \ref{dcritchart} that we define this to be an algebraic space with $\bG_m$-action $X$, and a $d$-critical structure $s$ on $X$ that descends to a d-critical structure $\bar{s}$ on $[X/\bG_m]$, \ie such that $\mu^\star(s)=(pr_2)^\star (s)$ (as $s$ lives in a set, $\bG_m$-invariance is a property and not an extra structure). From Lemma \ref{lemdarbgmeq}, given a $\bG_m$-invariant $-1$-shifted symplectic algebraic space $(\ud{X},\omega)$ (\ie a derived algebraic space $\ud{X}$ with $\bG_m$-action, and a nondegenerate closed $2$-form of degree $-1$ $\bar{\omega}$ of weight $0$, or equivalently a closed $2$-form of degree $-1$ $\bar{\omega}$ on $[\ud{X}/\bG_m]$ whose pullback $\omega$ to $\ud{X}$ is nondegenerate), its classical truncation $(X,s)$ is a $\bG_m$-invariant d-critical algebraic space. $K_{X,s}$ has then a natural $\bG_m$-invariant structure coming from the formula:
\begin{align}\label{eqstructK}
    K_{X,s}\simeq K_{[X/\bG_m],\bar{s}}|_{X^{red}}\otimes K_{X/[X/\bG_m]}|_{X^{red}}^{\otimes 2}
\end{align}
(an alternative description that works also for nonzero weights is given in Corollary \ref{cordcritgraded}). When $(X,s)$ enhance to a $\bG_m$-invariant $-1$-shifted symplectic algebraic space $(\ud{X},\omega)$, the isomorphism $K_{X,s}\simeq \det(\bL_{\ud{X}})|_{X^{red}}$ respects this $\bG_m$-equivariant structure. Given $\bG_m$-invariant $d$-critical algebraic space $(X,s)$, we define a $\bG_m$-equivariant orientation to be a $\bG_m$-equivariant square root of $K_{X,s}$ (\ie a $\bG_m$-equivariant line bundle $K_{X,s}^{1/2}$ on $X^{red}$ with a $\bG_m$-equivariant isomorphism $K_{X,s}^{1/2}\otimes K_{X,s}^{1/2}\simeq K_{X,s}$). Using \eqref{eqstructK}, it is equivalent to the data of an orientation on $([X/\bG_m],\bar{s})$.\medskip

Classically, one obtains a $\bG_m$-invariant oriented d-critical algebraic space if one consider the intersection of two $\bG_m$-invariant Lagrangian with $\bG_m$-invariant spin structure in a $\bG_m$-invariant algebraic symplectic variety, or the moduli space of stable objects for a Bridgeland stability conditions on a CY3 with an action of $\bG_m$ leaving the CY3 form invariant, or the moduli space of stable representations of a quiver with potential with an action of $\bG_m$ leaving the potential invariant.\medskip

Consider the Bialynicki-Birula correspondence:
\[\begin{tikzcd}
    X^0:=Map^{\bG_m}(\Spec(k),X)\arrow[rr,bend right=15,"\iota"] & X^+:=Map^{\bG_m}((\bA^1_k)^+,X)\arrow[l,swap,"p^+"]\arrow[r,"\eta^+"] & X
\end{tikzcd}\]
where $\bA^1$ is provided with its natural $\bG_m$-action. It is a smooth cover of a connected component of the $\Theta$-correspondence for $[X/\bG_m]$. In particular, the formula $s^0:=\iota^\star(s)$ gives a $d$-critical structure on $X^0$, pulled back from those of $[X^0/\bG_m]$ from Proposition \ref{cordcritgraded} $i)$. When $X$ enhance to a $\bG_m$-invariant $-1$-shifted symplectic stack, this is also the case for $(X^0,s^0)$ from Lemma \ref{lemorientgm} $ii)$ and Lemma \ref{lemshifsymp} $i)$. Similarly, Proposition \ref{cordcritgraded} $iii)$ allow to define a canonical orientation on $(X^0,s^0)$ from an orientation on $(X,s)$. When $X$ enhance to a $\bG_m$-invariant $-1$-shifted symplectic stack $(\ud{X},\omega)$, from \ref{lemshifsymp} $iii)$, this is simply given by the formula:
\begin{align}
    (K_{X,s}^{1/2})^0:=(\iota^\ast K_{X,s}^{1/2})\otimes  \det((\ud{\iota}^\ast\bL_{\ud{X}})^{<0})|^{\otimes -1}_{\Grad(\mX)^{red}}
\end{align}
Given a point $x\in X^0$, $T_{X,x}$ is provided with a $\bZ$-grading, and consider
\begin{align}
    \Ind_X(x):=\dim(T_X|_{X^0}^{>0})-\dim(T_X|_{X^0}^{<0})
\end{align}
Notice that, considering $x$ as a point of $[X^0/\bG_m]$, $\mathfrak{Iso}_{x'}([X/\bG_m])$ has weight $0$, which gives, using \eqref{defindmx}:
\begin{align}
    \Ind_X(x)=\Ind_{[X/\bG_m]}(x)
\end{align}
hence this definition is consistent with the definition for stacks, and $\Ind_X$ is locally constant on $X^0$ from Proposition \ref{cordcritgraded} $ii)$. From Lemma \ref{lemshifsymp} $ii)$, when $(X,s)$ enhance to a $\bG_m$-invariant $-1$-shifted symplectic algebraic space $(\ud{X},\omega)$, $\Ind_X(x)$ is the signed dimension of $(\ud{\iota}^\ast\bT_{\ud{X}})^{>0}$, and can be generally thought as the virtual dimension of the fibers of the strata of $X^+$ flowing to a connected component of $X^0$. Consider now $X^0=\bigsqcup_{\mathbf{c}\in \Gamma} X^0_{\mathbf{c}}$ the decomposition into connected components, and consider $X^+=\bigsqcup_{\mathbf{c}\in\Gamma} X^+_{\mathbf{c}}$ with $X^+_{\mathbf{c}}:=X^+\times_{X^0}X^0_{\mathbf{c}}$ the decomposition into strata flowing to them (notice that it is also a decomposition into connected components of $X^+$), and denote by $\Ind_{\mathbf{c}}$ the value of $\Ind_X$ on $X^0_{\mathbf{c}}$.

\begin{theorem}\label{theoBB}
    Consider a $\bG_m$-invariant oriented $d$-critical algebraic space $(X,s,K_{X,s}^{1/2})$ (as always, assumed to be locally of finite type over an algebraically closed field $k$ of characteristic $0$). Then, considering the oriented $d$-critical algebraic space $(X^0,s^0,(K_{X,s}^{1/2})^0)$, we have a canonical isomorphism:
    \begin{align}
        (p^+)_!(\eta^+)^\ast P_{X,s,K_{X,s}^{1/2}}\simeq P_{X^0,s^0,(K_{X,s}^{1/2})^0}\{-\Ind_X/2\}
    \end{align}
    Suppose moreover that $X$ is separated of finite type, then $\eta$ is a geometrically injective, and we have the following equality in the Grothendieck ring of monodromic Nori motives (and then monodromic mixed Hodge structures):
    \begin{align}
[\bH ^T_c(X,P_X)]=\sum_{\mathbf{c}\in \Gamma}\mathbb{L}^{\Ind_{\mathbf{c}}/2}[\bH _c(X^0_{\mathbf{c}},P_{X^0_{\mathbf{c}}})]+[\bH ^T_c(X-\eta(X^+),P_X|_{X-\eta(X^+)})]
    \end{align}
    In particular, if $X$ is proper, $\eta$ is geometrically bijective, and we obtain the simpler formula:
    \begin{align}
[\bH ^T_c(X,P_X)]=\sum_{\mathbf{c}\in\Gamma}\mathbb{L}^{\Ind_{\mathbf{c}}/2}[\bH _c(X^0_{\mathbf{c}},P_{X^0_{\mathbf{c}}})]
    \end{align}
\end{theorem}

\begin{proof}
    As $[X^0/\bG_m]$ is a component of $\Grad([X/\bG_m])$, by Theorem \ref{theohyploc}, we have a canonical isomorphism:
    \begin{align}
        p_!\eta^\ast P_{[X/\bG_m]}\simeq P_{[X^0/\bG_m]}
    \end{align}
    Considering the smooth morphisms $\phi:X\to [X/\bG_m]$ and $\phi^0:X^0\to [X^0/\bG_m]$, enhancing by definition to smooth morphisms of oriented d-critical stacks, we obtain:
    \begin{align}
        (p^+)_!(\eta^+)^\ast P_X&\simeq (p^+)_!(\eta^+)^\ast\phi^\ast\{1\} P_{[X/\bG_m]}\nn\\
        &\simeq (\phi^0)^\ast\{1\}p_!\eta^\ast P_{[X/\bG_m]}\nn\\
        &\simeq  (\phi^0)^\ast\{1\}P_{[X^0/\bG_m]}\{-\Ind_{[X/\bG_m]}/2\}\nn\\
        &\simeq P_{X^0}\{-\Ind_X/2\}
    \end{align}
    where the first and last line come from Corollary \ref{corfunctjoyce}, and the second line by base change, giving the first claim.\medskip

    Suppose now that $X$ is separated of finite type. In particular, $X^0$ is of finite type too, so the cohomology with compact support is constructible, and the above isomorphism gives:
    \begin{align}
        \bH_c(X^+,(\eta^+)^\ast P_X)\simeq \bH_c(X^0,(p^+)_!(\eta^+)^\ast P_X)\simeq \bigoplus_{\mathbf{c}\in\Gamma}\bH_c(X^0_{\mathbf{c}},P_{X^0_{\mathbf{c}}})\{-\Ind_{\mathbf{c}}/2\}
    \end{align}
    Given a $k$ point $x\in X(k)$, the orbit gives a $\bG_m$-equivariant map $\bG_{m,k}\to X$. Consider an extension of this map to $f:\bA^1_k\to X$, and denote by $x_0\in X(k)$ the image of $\{0\}$: by considering $t\cdot f:\bA^1_k\to X$ for any $t\in\bG_{m,k}$, one obtains by separation of $X$ that $t\cdot x_0=x_0$, \ie $x_0\in X^0$, and $f$ is $\bG_m$-equivariant. Then points of $X^+(k)$ above $x\in X(k)$ are in bijection with extensions $f:\bA^1_k\to X$ as above. Then, because $X$ is separated, $\eta$ is injective on $k$-point, and it is furthermore bijective on $k$-point when $X$ is proper. We obtain then the result using Lemma \ref{lemmotdec}.
\end{proof}

\begin{remark}
    Notice that, even in the smooth case treated in \cite{BiaynickiBirula1973SomeTO}, the $X^+_{\mathbf{c}}$ do not give a stratification of $X$ in the usual sense, see the counterexample of \cite{konarski}. However, in the smooth case, even without assuming the existence of an equivariant projectivization, the restriction of $\eta$ to each connected component is a locally closed immersion as proven in \cite{BiaynickiBirula1973SomeTO}. In the singular case, it is no longer true without this assumption: consider for example the nodal curve $X$ obtained by gluing $0$ and $\infty$ in $\bP^1$, with the quotient of the canonical $\bG_m$-action on $\bP^1$. one has $\bA^1=X^+\to X$ is not a locally closed immersion. See \cite[Remark 1.4.4, Remark 5.3.9]{HalpernLeistner2014OnTS}, \cite[Section 1.6]{Drinfeld2013ONAT}. One obtains a true stratification if one can embeds $X$ in a projective space with linear action.
\end{remark}

\subsection{Motivic formula for Theta-stratifications}\label{sectthetastrat}

We recall the definition of a $\Theta$-stratification from \cite[Definition 2.1.2]{HalpernLeistner2014OnTS}:

\begin{definition}
    Consider a quasi-separated, locally of finite type Artin $k$-stack $\mX$ with affine stabilizers. A $\Theta$-stratification of $\mX$ is the data of:
    \begin{itemize}
        \item a totally ordered set $\Gamma$ with a minimal element $0$ and a collection of open substacks $\mX_{\leq {\mathbf{c}}}$ for ${\mathbf{c}}\in\Gamma$ such that $\mX_{\leq {\mathbf{c}}}\subseteq \mX_{\leq {\mathbf{c}}'}$ for ${\mathbf{c}}<{\mathbf{c}}'$ and $\mX=\bigcup_{\mathbf{c}\in\Gamma} \mX_{\leq {\mathbf{c}}}$ (We call $\mX^{ss}:=\mX_{\leq 0}$ the semistable locus.);
        \item For each ${\mathbf{c}}\in\Gamma$, a $\Theta$-stratum $\Theta_{\mathbf{c}}$ of $\mX_{\leq {\mathbf{c}}}$ - \ie an union of open components of $\Filt(\mX_{\leq {\mathbf{c}}})$ such that $\eta|_{\Theta_{\mathbf{c}}}:\Theta_{\mathbf{c}}\to\mX_{\leq {\mathbf{c}}}$ is a closed immersion and:
        \begin{align}
            \mX_{\leq {\mathbf{c}}}-\eta(\Theta_{\mathbf{c}})=\mX_{<{\mathbf{c}}}:=\bigcup_{{\mathbf{c}}'<{\mathbf{c}}}\mX_{\leq {\mathbf{c}}'}
        \end{align}
        \item for every point $x\in\mX(k)$, the set $\{{\mathbf{c}}\in\Gamma|x\in\mX_{\leq {\mathbf{c}}}\}$ has a minimal element ${\mathbf{c}}(x)$. The unique $y\in\Theta_{{\mathbf{c}}(x)}$ such that $\eta(y)=x$ is called the Harder-Narasimhan filtration of $x$.
    \end{itemize}
    We define the center $\mathcal{Z}_{\mathbf{c}}:=\tilde{\iota}^{-1}(\Theta_{\mathbf{c}})$ of the $\Theta$-stratum $\Theta_{\mathbf{c}}$.
\end{definition}

Notice that from the definition,  $\Theta_{\mathbf{c}}$ is open in $\Filt(\mX_{\leq {\mathbf{c}}})$, which is open in $\Filt(\mX)$, hence $\mathcal{Z}_{\mathbf{c}}$ is open in $\Grad(\mX)$. We have moreover (it is noticed in \cite[Page 42]{HalpernLeistner2014OnTS}: 

\begin{lemma}\label{lemcartstrat}
    Given a $\Theta$-stratification of a quasi-separated, locally of finite type $k$ Artin stack $\mX$, we have for any stratum $\Theta_{\mathbf{c}}=p^{-1}(\mathcal{Z}_{\mathbf{c}})$.
\end{lemma}

\begin{proof}
    Because $\mX_{\leq\mathbf{c}}\to \mX$ is an open immersion we have from \cite[Corollary 1.1.7, Proposition 1.3.1 3)]{HalpernLeistner2014OnTS} that $\Grad(\mX_{\leq \mathbf{c}})=\tilde{\iota}^{-1}(\Filt(\mX_{\leq \mathbf{c}})$, and $\Filt(\mX_{\leq\mathbf{c}})=p^{-1}(\Grad(\mX_{\leq\mathbf{c}})$. But $\Theta_{\mathbf{c}}$ is a connected component of $\Filt(\mX_{\leq\mathbf{c}})$, by \cite[lemma 1.3.8]{HalpernLeistner2014OnTS} $(i_{\mX_{\leq\mathbf{c}}}\circ p_{\mX_{\leq\mathbf{c}}})^{-1}(\Theta_{\mathbf{c}})=\Theta_{\mathbf{c}}$, which means that $\Theta_{\mathbf{c}}=p^{-1}(\mathcal{Z}_{\mathbf{c}})$. 
\end{proof}

The following theorem allows to computes the cohomological DT invariants of $\mX$ in terms of those of the center:

\begin{theorem}\label{theothetastrat}
    Consider $(\mX,s,K_{\mX,s}^{1/2})$ an $d$-critical stack which is quasi-separated, of finite type with affine stabilizers over an algebraically closed field $k$ of characteristic $0$, with an isomorphism class of orientation. Consider a $\Theta$-stratification on $\mX$ (in particular, because $\mX$ is of finite type, it has a finite number of nonempty strata). The center $\mathcal{Z}_{\mathbf{c}}$ of the $\Theta$-strata have a natural oriented d-critical structure, and, up to refining the $\Theta$-stratification, we can assume that the locally constant function $\Ind_\mX$ have a constant value $\Ind_{\mathbf{c}}$ on $\mathcal{Z}_{\mathbf{c}}$. We have then the following equality in the completed Grothendieck ring of monodromic Nori motives or monodromic mixed Hodge structures:
    \begin{align}
        [\bH _c(\mX,P_\mX)]=\sum_{\mathbf{c}\in\Gamma}\mathbb{L}^{\Ind_{\mathbf{c}}/2}[\bH _c(\mathcal{Z}_{\mathbf{c}},P_{\mathcal{Z}_{\mathbf{c}}})]
    \end{align}
\end{theorem}

\begin{proof}

Notice that, as the Grothendieck ring form a set, and not a category, any choice of orientation in the isomorphism class will give the same object. Consider the following commutative diagram:
\[\begin{tikzcd}
    \Grad(\mX) & \Filt(\mX)\arrow[l,swap,"p"]\arrow[r,"\eta"] & \mX\\
    \mathcal{Z}_{\mathbf{c}}\arrow[u,"i_{\mathbf{c}}"] & \Theta_{\mathbf{c}}\arrow[l,swap,"p_{\mathbf{c}}"]\arrow[ur,"\eta_{\mathbf{c}}"]\arrow[u]& 
\end{tikzcd}\]
where the left square is Cartesian from Lemma \ref{lemcartstrat}, and the vertical arrows are open immersions. We have from Theorem \ref{theohyploc} a canonical isomorphism:
\begin{align}
    p_!\eta^\ast P_\mX\simeq P_{\Grad(\mX)}\{-\Ind_{\mX}\}
    \end{align}
Then, pulling back along the open immersion $i_{\mathbf{c}}:\mathcal{Z}_{\mathbf{c}}\to \Grad(\mX)$, we obtain:
\begin{align}\label{eqloctheta}
    (p_{\mathbf{c}})_!(\eta_{\mathbf{c}})^\ast P_\mX&\simeq(i_{\mathbf{c}})^\ast p_!\eta^\ast P_\mX\nn\\
    &\simeq (i_{\mathbf{c}})^\ast P_{\Grad(\mX)}\{-\Ind_{\mX}\}\nn\\
    &\simeq P_{\mathcal{Z}_{\mathbf{c}}}\{-\Ind_{\mathbf{c}}\}
\end{align}
where the first line comes from functoriality and base change, the second from Theorem \ref{theohyploc}, and the last from the compatibility of the DT construction with open immersions given in Corollary \ref{corfunctjoyce}. Using proper pushforward to a point we obtain:
\begin{align}
    [\bH _c(\Theta_{\mathbf{c}},(\eta_{\mathbf{c}})^\ast P_\mX)]\simeq \mathbb{L}^{\Ind_{\mathbf{c}}/2}\bH _c(\mathcal{Z}_{\mathbf{c}},P_{\mathcal{Z}_{\mathbf{c}}})
\end{align}
Notice now that, by assumption, $\bigsqcup_{\mathbf{c}\in\Gamma}\eta_{\mathbf{c}}:\Theta_{\mathbf{c}}\to \mX$ is a geometric bijection, and then Lemma \ref{lemmotdec} gives the result.
\end{proof}

\subsection{Recollection on moduli stacks of complexes}\label{secStackComp}

\subsubsection{Moduli of objects in dg categories}

We follow here \cite{Toen2005ModuliOO}, \cite{Brav2018RelativeCS}: in particular, we adopt the conventions of \cite{Brav2018RelativeCS}. Consider a presentable dg-category $\mC$ over $k$ (\ie a category enriched over $k$-complexes, with arbitrary colimits, generated by a set of compact objects) and denote by $\mC_c$ its (small) dg-subcategory of compact objects, \ie objects $E$ such that $\Hom(E,-)$ commutes with filtered colimits (such objects are also sometimes called objects of finite presentation). We always assume $\mC$ to be of finite type according to \cite[Definition 2.4]{Toen2005ModuliOO}, \ie it is equivalent to the category of module of a dg-algebra which is homotopically of finite presentation. It is in particular smooth, following \cite[Proposition 2.14]{Toen2005ModuliOO}. Recall that a continuous (\ie, colimit preserving) functor $F$ between two presentable dg-categories always has a right adjoint $F^r$, which is itself continous iff $F$ preserves compact objects.\medskip

A dg-category is said to be proper if, for any $c,c'$ compact, $\Hom(c,c')\in Vect_k$ is compact, \ie is a homotopically finite dimensional complex. A smooth and proper dg category is of finite type from \cite[Corollary 2.13]{Toen2005ModuliOO}. If $\mC$ is not necessarily proper, an object $c\in\mC$ is said to be left-proper if $\Hom(c,c')$ is homotopically of finite dimension for every $c'\in\mC_c$, and we denote by $\mC_{lp}$ the full subcategory of left-proper objects. When $\mC$ is smooth, left proper objects are in particular compact from \cite[Lemma 2.8 2)]{Toen2005ModuliOO}, and if $\mC$ is proper, then compact objects are proper by definition.\medskip

Typically, if $X$ is a separated scheme of finite type and $\mC=\QCoh(X)$, we obtain that $\mC_c=\Perf(X)$, and $\mC$ is smooth iff $X$ is smooth from \cite[Proposition 3.13]{Lunts2009CategoricalRO}, in which case $\mC_c=\Perf(X)=\Coh(X)$ and $\mC$ is of finite type from \cite[Theorem 3.1.1]{BonVan03}. Similarly, $\mC$ is proper iff $X$ is proper from \cite[Proposition 3.30]{Orlov2014SMOOTHAP}. Taking now any separated scheme of finite type $X$, $\mC=\IndCoh(X)$, $\mC_c=D^b\Coh(X)$, $\mC$ is of finite type from \cite{effimov}: in particular, it is smooth, as proven earlier in \cite[Theorem 6.3]{Lunts2009CategoricalRO} (even without any smoothness assumption on $X$!).\medskip

Given a $k$ dg-algebra $\ud{R}$, an object $E$ of $\mC\otimes_k \QCoh(\ud{R})$ is said to be left proper over $\ud{R}$ if $\Hom_{\ud{R}}(E,-):\mC\otimes_k \QCoh(\ud{R})\to \QCoh(\ud{R})$ maps compact objects, \ie objects of $\mC\otimes_k \Perf(\ud{R})$, to compact objects, \ie objects of $\Perf(\ud{R})$. As $\mC$ is smooth, from \cite[Lemma 2.8 2)]{Toen2005ModuliOO}, left-proper objects are compact, \ie lies in $\mC_c\otimes_k \Perf(\ud{R})$. As $\mC$ is smooth, from \cite[Proposition 2.4, Corollary 2.5]{Brav2018RelativeCS}, there is a natural continuous involution:
\begin{align}
Id_{\mC\otimes_k\QCoh(\ud{R})/\QCoh(\ud{R})}:\mC\otimes_k\QCoh(\ud{R})\simeq \mC\otimes_k\QCoh(\ud{R})
\end{align}
such that, for $E$-left proper, there is a natural isomorphism of functors:
\begin{align}
    \Hom_{\ud{R}}(E,-)\simeq \Hom_{\ud{R}}(-,(Id_{\mC\otimes_k \QCoh(\ud{R})/\QCoh(\ud{R})}^!)^{-1}(E))^\vee:\mC\otimes_k \QCoh(\ud{R})\to \QCoh(\ud{R})
\end{align}
called the relative Serre functor (it agrees with the classical Serre functor for $\mC=\IndCoh(X)$). \medskip

In \cite{Toen2005ModuliOO}, a higher derived stack $\ud{\mM}$ classifying left-proper objects of $\mC$ was defined, and was shown in \cite[Theorem 3.6]{Toen2005ModuliOO} to be locally Artin (\ie it is an union of Artin $n$-stacks for increasing $n$) and locally of finite presentation over $k$, when $\mC$ is of finite type (which we will always assume). In \cite[Example 3.7]{Brav2018RelativeCS}, it is described by the moduli functor:
\begin{align}\label{toenfunctpoin}
    \ud{\mM}(\Spec(\ud{R}))=\Map_{dgcat}(\mC_c,\Perf(\ud{R}))
\end{align}
sending an affine derived scheme $\Spec(\ud{R})$ to the space of exact functors from the small dg-category of compact objects $\mC_c$ to the small dg-category $\Perf(\ud{R})$ of perfect $\ud{R}$-modules (\ie, perfect complexes on $\Spec(\ud{R})$). The functoriality comes here from pullback of perfect modules. Those corresponds to continuous adjunction:
\begin{align}
    f:\mC\rightleftharpoons \QCoh(\ud{R}):f^r
\end{align}
and so, as $\mC$ is smooth, such a data corresponds from \cite[Corollary 2.6]{Brav2018RelativeCS} to the data of a left proper objects $E$ of $\mC\otimes \QCoh(\ud{R})$, such that:
\begin{align}
    f=\Hom_{\ud{R}}(E,-\boxtimes_k\ud{R}):\mC\to \QCoh(\ud{R})
\end{align}
In particular, every $k$-points $f$ is then corepresented by a left proper object $E$ of $\mC$, \ie $f\simeq \Hom(E,-)$, but the usual convention is to use the relative Serre duality functor and to identify the $k$-point with the right proper object $F$ of $\mC$ representing $f$, \ie such that $f\simeq \Hom(-,F)^\vee$.\medskip

By construction, one have a universal functor $F_\mC:C_c\to \Perf(\ud{\mM}))$, which corresponds from \cite[Corollary 2.6]{Brav2018RelativeCS} to a left proper object $\mathcal{U}$ of $\mC_c\times \Perf(\ud{\mM})$, called the universal complex. $\ud{\mM}$ has a cotangent complex $\bL_{\ud{\mM}}$, and a tangent complex $\bT_{\ud{\mM}}:=\bL_{\ud{\mM}}^\vee\in \Perf(\ud{\mM})$, given from \cite[Proposition 3.3]{Brav2018RelativeCS} by the formula:
\begin{align}\label{eqText}
    \bT_{\ud{\mM}}\simeq \Hom_{\ud{\mM}}(\mathcal{U},\mathcal{U})[1]
\end{align}
(notice that we consider $\bT_{\ud{\mM}}=\bL_{\ud{\mM}}^\vee$ as living in $\Perf({\ud{\mM}})$). In particular, as proven in \cite[Corollary 3.17]{Toen2005ModuliOO}, given an compact object $E$ of $\mC$, seen as a $k$-point of $\ud{\mM}$, one have a natural isomorphism:
\begin{align}
    \bT_{\ud{\mM},E}\simeq \mC(E,E)[1]
\end{align}
in other terms, $\Ext^1(E,E)$ gives the Zariski tangent space, $Ext^{\geq 2}(E,E)$ the obstructions, $\Ext^0(E,E)$ the Lie algebra of the isotropy group, and $\Ext^{<0}(E,E)$ the higher stacky part (which explains why one has to consider Artin $n$-stacks for $n$ arbitrarily large).\medskip

Given a $n$-dimensional Calabi-Yau structure on $\mC$ (also called a left Calabi-Yau structure in \cite[Definition 3.5]{Brav2016RelativeCS}), \cite[Theorem 5.5 1)]{Brav2018RelativeCS} gives a natural $n-2$-shifted symplectic structure $\omega_\mC$ on $\ud{\mM}$. Informally a $n$-dimensional Calabi-Yau structure gives a functorial nondegenerate paring:
\begin{align}
    \Hom(E,F)\simeq \Hom(F,E)^\vee[n]
\end{align}
giving a version of Serre duality for $\mC$, inducing a nondegenerate pairing:
\begin{align}
    \bT_{\ud{\mM},E}\simeq \bL_{\ud{\mM},E}[n-2]
\end{align}
that glue into the nondegenerate pairing induced by the $2$ form of degree $n-2$ underlying $\omega$. More precisely, from \cite[Corollary 2.5, Proposition 5.3]{Brav2018RelativeCS}, the class of the Hochschild homology $HH(\mC)$ of degree $n$ underlying the Calabi-Yau structure induces an isomorphism of functors:
\begin{align}\label{pairhochs}
    \Hom_{\ud{\mM}}(\mathcal{U},-)\simeq \Hom_{\ud{\mM}}(-,\mathcal{U})^\vee[n]
\end{align}
from which one obtains the nondegenerate pairing:
\begin{align}
\bT_{\ud{\mM}}\simeq \bL_{\ud{\mM}}[n-2]
\end{align}
from \eqref{eqText}. Of course, one need also an extra structure, that will define the closed $2$ form of degree $n-2$ $\omega$, we refer to \cite[Section 3.2]{Brav2016RelativeCS}, \cite[Section 5]{Brav2018RelativeCS} or the proof of Lemma \ref{lemgradprod} below for more details. We specialize here to the case $n=3$, where one obtains a $-1$ shifted symplectic structure. From \cite[Proposition 5.11]{Brav2016RelativeCS}, for $X$ a Gorenstein scheme of dimension $n$, the data of a Calabi-Yau structure of dimension $n$ on $\mC_c=D^b\Coh(X)$ is equivalent to the data of a trivialisation of the cotangent bundle $\omega_X\simeq\mathcal{O}_X$.\medskip

\begin{remark}
    We will now follow and adapt the discussion of \cite[Section 3]{JOYCEorient} on orientations on CY3 categories. Remark that here that $\mC_c$ is not assumed to be $D^b\Coh(X)$ for a CY threefold $X$, so we have to adapt a bit the notations. We consider the CY3 case as in \cite[Section 4.1]{JOYCEorient}, so $m=3$ and the spin structure $J$ is trivial. The classical truncation $\mM$ of $\ud{\mM}$ was denoted by $\bar{\mM}$ in \cite[Section 3]{JOYCEorient}, and the authors denote there by $\mM$ the classical open $1$-substack of objects in the Abelian heart $Coh(X)$: we will below denote the analogue of this by $\mM_\mA$, where $\mA$ will be the Abelian heart of a $t$-structure on $\mC$, restricting to a $t$-structure on $\mC_c$. The perfect complexes $\mC^\bullet$ and $\mathcal{D}^\bullet$ of \cite[Section 3]{JOYCEorient} will be denoted below by $\bT_{\ud{\mM}}|_{\mM}[-1]$ and $\tilde{\bT}_{\ud{\mM}}|_{\mM\times \mM}[-1]$, and our $K_{\mM},L_{\mM}$ will be their $K_{\bar{\mM}},L_{\bar{\mM}}$. Our $\oplus_2$ will be their $\Phi$, and we denote the projections by $pr_i$ instead of $\pi_i$, as in the rest of this article. For clarity, given a morphism $f:\ud{\mX}\to\ud{\mY}$, we will denote the functor $(Id_{\mC_c}\otimes_k f^\ast):\mC_c\otimes_k\Perf(\ud{\mY})\to \mC_c\otimes_k\Perf(\ud{\mY})$ simply by $f^\ast$, when no confusion can happen.
\end{remark}

Direct sum gives a natural structure of commutative monoid $\ud{\oplus_n}:\ud{\mM}^n\to \ud{\mM}$, $0:\Spec(k)\to \ud{\mM}$, given on $\ud{\Spec}(\ud{R})$ points by:
\begin{align}
    &(f_i:\mC_c\to\Perf(\ud{R}))_{1\leq i\leq n}\mapsto (\bigoplus f_i:\mC_c\to\Perf(\ud{R}))\nn\\
    &(\ud{\Spec}(\ud{R})\to \ud{\Spec}(k))\mapsto (0:\mC_c\to\Perf(\ud{R}))
\end{align}
As the support of a perfect complex of $\ud{R}$-modules is open, $0:\Spec(k)\to \ud{\mM_{Vect_k}}$ is an open immersion, hence, choosing an affine generator homotopically of finite type of $\mC$, an considering the map $\ud{\mM}\to \ud{\mM_{Vect_k}}$ of \cite[Section 3.2]{Toen2005ModuliOO}, we find that $0:\Spec(k)\to \ud{\mM}$ is an open immersion. One obtains from the definition a coherent system of isomorphisms:
\begin{align}
(\ud{\oplus_n})^\ast(\mathcal{U})\simeq\bigoplus_{1\leq i\leq n}(\ud{pr_i})^\ast(\mathcal{U})
\end{align}
Denote by:
\begin{align}\label{isom}   \tilde{\bT}_{\ud{\mM}}:=\Hom_{\ud{\mM}\times \ud{\mM}}((pr_1)^\ast(\mathcal{U}),(pr_2)^\ast(\mathcal{U}))[1]\in \Perf(\ud{\mM}\times \ud{\mM})
\end{align}
and $\tilde{\bL}_{\ud{\mM}}:=\tilde{\bT}_{\ud{\mM}}^\vee$, such that $\bL_{\ud{\mM}}\simeq(\Delta_{\ud{\mM}})^\ast(\tilde{\bL}_{\ud{\mM}})$. Notice that, considering the commutativity isomorphism $\Sigma_{\ud{\mX}}:\ud{\mX}\times \ud{\mX}\to\ud{\mX}\times \ud{\mX}$, one obtains from \eqref{pairhochs} the nondegenerate pairing:
\begin{align}\label{dualtildeL}
\Sigma_{\ud{\mM}}^\ast(\tilde{\bT}_{\ud{\mM}})\simeq \tilde{\bL}_{\ud{\mM}}[n-2]
\end{align} 
Using \eqref{eqText}, we find a coherent system of isomorphisms:
\begin{align}\label{isomiotaL}  (\ud{\oplus_n})^\ast(\bL_{\ud{\mM}})&\simeq \Hom_{\ud{\mM}^n}((\ud{\oplus_n})^\ast(\mathcal{U}),(\ud{\oplus_n})^\ast(\mathcal{U}))^\vee[-1]\nn\\
&\simeq \Hom_{\ud{\mM}^n}(\bigoplus_{1\leq i\leq n}(\ud{pr_i})^\ast(\mathcal{U}),\bigoplus_{1\leq j\leq n}(\ud{pr_j})^\ast(\mathcal{U}))^\vee[-1]\nn\\
&\simeq \bigoplus_{1\leq i\leq n}(\ud{pr_i})^\ast(\bL_{\ud{\mM}})\oplus\bigoplus_{1\leq i\neq j\leq n}(\ud{pr}_{ij})^\ast(\tilde{\bL}_{\ud{\mM}})
\end{align}
We denote now:
\begin{align}
K_{\mM}&:=\det(\bL_{\ud{\mM}})|_{\mM}\nn\\
L_{\mM}&:=\det(\tilde{\bL}_{\ud{\mM}})|_{\mM\times\mM}
\end{align}
Such that we obtain $K_{\mM}\simeq (\Delta_{L_{\mM}})^\ast L_{\mM}$, and we obtain from \eqref{dualtildeL} an isomorphism:
\begin{align}\label{eqsigmaL}
(\Sigma_{\mM})^\ast(L_{\mM})\simeq L_{\mM}
\end{align}
whose restriction along the diagonal is $Id:K_{\mM}\simeq K_{\mM}$. Combining this with the determinant of \eqref{isomiotaL}, we obtain as in \cite[(3.16)]{JOYCEorient} an isomorphism:
\begin{align}\label{isomdirectsum}
    (\oplus_2)^\ast (K_{\mM})\simeq (pr_1)^\ast (K_{\mM})\oplus (pr_2)^\ast (K_{\mM})\oplus L_{\mM}^{\otimes 2}
\end{align}
There are moreover from \cite[(3.17),(3.18)]{JOYCEorient} natural isomorphisms, derived with the same kinds of computation:
\begin{align}\label{isomdirectsumL}
    (\oplus_2\times Id_{\mM})^\ast L_{\mM}&\simeq (pr_{13})^\ast (L_{\mM})\oplus (pr_{23})^\ast (L_{\mM})\nn\\
    (Id_{\mM}\times \oplus_2)^\ast L_{\mM}&\simeq (pr_{12})^\ast (L_{\mM})\oplus (pr_{13})^\ast (L_{\mM})
\end{align}

Considering the square giving the associativity of the monoid structure $\oplus_n$:
\[\begin{tikzcd}
    \mM\times \mM\times \mM\arrow[r,"\oplus_2\times Id_{\mM}"]\arrow[d,"Id_{\mM}\times\oplus_2"] & \mM\times \mM\arrow[d,"\oplus_2"]\nn\\
    \mM\times \mM\arrow[r,"\oplus_2"] & \mM
\end{tikzcd}\]
There is from \cite[(3.19)]{JOYCEorient} a commutative square of isomorphisms:

\begin{equation}\label{diagjoyceorient}\begin{tikzcd}[column sep=tiny]
    \begin{tabular}{c}$(pr_1)^\ast (K_{\mM})\otimes (pr_2)^\ast (K_{\mM})\otimes (pr_3)^\ast (K_{\mM})$\\$\otimes (pr_{12})^\ast (L_{\mM}^{\otimes 2})\otimes (pr_{13})^\ast (L_{\mM}^{\otimes 2})\otimes (pr_{23})^\ast (L_{\mM}^{\otimes 2})$\end{tabular}\arrow[r,"\simeq"]\arrow[d,"\simeq"] & (\oplus_2\times Id_{\mM})^\ast((pr_1)^\ast (K_{\mM})\otimes (pr_2)^\ast (K_{\mM})\otimes  L_{\mM}^{\otimes 2})\arrow[d,"\simeq"]\\
(Id_{\mM}\times\oplus_2)^\ast((pr_1)^\ast (K_{\mM})\otimes (pr_2)^\ast (K_{\mM})\otimes L_{\mM}^{\otimes 2})\arrow[r,"\simeq"] & \begin{tabular}{c}$(\oplus_2\times Id_{\mM})^\ast(\oplus_2)^\ast (K_{\mM})$\\$\simeq(Id_{\mM}\times\oplus_2)^\ast(\oplus_2)^\ast (K_{\mM})$\end{tabular}
\end{tikzcd}\end{equation}
where the arrows come from \eqref{isomdirectsum} and \eqref{isomdirectsumL}.

We give the definition of orientation data (resp. strong orientation data) from \cite[Definition 4.2]{JOYCEorient}, extended straightforwardly from $D^b\Coh(X)$ to any CY3 category $\mC_c$. More precisely, we give the definition of spin structure compatible with direct sum (resp. strong spin structure compatible with direct sum) from \cite[Definition 3.4]{JOYCEorient} (resp. \cite[Definition 3.7]{JOYCEorient}), which give are equivalent data from \cite[Proposition 4.3]{JOYCEorient} (the strong case follows from the same $\bA^1$-contractibility arguments). Notice that we will work with strong spin structure because our strategy to prove Kontsevich-Soibelman wall crossing formula and build the CoHA is different from the one sketched in \cite[Section 4.3]{JOYCEorient}, and do not use Joyce's conjecture from \cite[Conjecture 1.1]{joyce2019lagrangian}.

\begin{definition}(\cite[Definition 3.4, Definition 3.7]{JOYCEorient})\label{defOrientdg}
    Consider a presentable CY3 category $\mC$ of finite type:
    \begin{enumerate}
        \item[$i)$] An orientation data on $\mC_c$ is the data of an isomorphism class $[K_{\mM}^{1/2}]$ of squares root of $K_{\mM}$, such that, for one (and then, any) representative $K_{\mM}^{1/2}$, there is an isomorphism:
        \begin{align}
    (\oplus_2)^\ast (K_{\mM}^{1/2})\simeq (pr_1)^\ast (K_{\mM}^{1/2})\otimes (pr_2)^\ast (K_{\mM}^{1/2})\otimes L_{\mM}
        \end{align}
        which is a square root of \eqref{isomdirectsum}.
        \item[$ii)$] A strong orientation data on $\mC_c$ is the data of a square root $K_{\mM}^{1/2}$ of $K_{\mM}$, with the data of a square root of \eqref{isomdirectsum} as above, such that the diagram:
        
        \begin{equation}\label{diagjoyceorientsq}
        \begin{tikzcd}[column sep=tiny]
    \begin{tabular}{c}$(pr_1)^\ast (K_{\mM}^{1/2})\otimes  (pr_2)^\ast (K_{\mM}^{1/2})\otimes  (pr_3)^\ast (K_{\mM}^{1/2})$\\$\otimes (pr_{12})^\ast (L_{\mM})\otimes (pr_{13})^\ast (L_{\mM})\otimes (pr_{23})^\ast (L_{\mM})$\end{tabular}\arrow[r,"\simeq"]\arrow[d,"\simeq"] & (\oplus_2\times Id_{\mM})^\ast((pr_1)^\ast (K_{\mM}^{1/2})\otimes  (pr_2)^\ast (K_{\mM}^{1/2})\otimes  L_{\mM})\arrow[d,"\simeq"]\\
(Id_{\mM}\times\oplus_2)^\ast((pr_1)^\ast (K_{\mM}^{1/2})\otimes  (pr_2)^\ast (K_{\mM}^{1/2})\otimes  L_{\mM}^{\otimes 2})\arrow[r,"\simeq"] & \begin{tabular}{c}$(\oplus_2\times Id_{\mM})^\ast(\oplus_2)^\ast (K_{\mM}^{1/2})$\\$\simeq(Id_{\mM}\times\oplus_2)^\ast(\oplus_2)^\ast (K_{\mM}^{1/2})$\end{tabular}
\end{tikzcd}\end{equation}
obtained from it and \eqref{isomdirectsumL}, which is a square root of \eqref{diagjoyceorient}, commutes.
    \end{enumerate}
\end{definition}
Notice that this definition depends only on the nondegenerate pairing induced by the CY3 structure.\medskip

\begin{remark}
    Canonical orientation data for $\mC_c=D^b\Coh(X)$, for $X$ a smooth projective Calabi-Yau threefold, were build in \cite[Theorem 4.4]{JOYCEorient}, and for $X$ a smooth quasiprojective Calabi-Yau threefold with the data of a spin compactification in \cite[Theorem 4.8]{JOYCEorient} (but it is, to our knowledge, not known to be independent of this compactification). In \cite[Theorem 3.10]{JOYCEorient}, the authors describes the topological obstruction to lift this orientation data to a strong orientation data. From the arguments of \cite[Section 5]{JOYCEorient} (Proposition below), giving a (strong) orientation data on $\mC$ is equivalent to giving a (strong) orientation data on the Abelian heart of a $t$-structure. Also, one can define a strong orientation data when $\mC_c$ is equivalent to the dg-category of a quiver with potential (indeed, in this case, $\ud{\mM}$ is a global critical chart, hence one has a trivial orientation), but, to our knowledge, one does not known if it is independent of the choice of quiver with potential, nor if the underlying orientation data coincide with the one of \cite[Theorem 4.8]{JOYCEorient} in the presence of a spin compactification. 
\end{remark}

Consider $E,F\in \mC_{lp}$, and consider:
\begin{align}
    \langle E,F\rangle:=\sum_{i\in\bZ}(-1)^{i+1}\dim(\Ext^i(E,F))
\end{align}
which is finite, as $E$ is left proper and $F$ is in particular compact. The CY3 pairing gives $\langle E,F\rangle=-\langle F,E\rangle$.\medskip

Denote by $(\ud{\mM}^{(\bZ)},\omega_\mC^{(\bZ)})$ the $n-2$-shifted symplectic stack obtained by taking the colimit of the products $\prod_{i\in\{-n,n\}}(\ud{\mM},\omega_\mC)$, where the maps are defined by:
\begin{align}
    0\times Id\times 0:\prod_{i\in\{-n,n\}}(\ud{\mM},\omega_\mC)\hookrightarrow\prod_{i\in\{-(n+1),n+1\}}(\ud{\mM},\omega_\mC)
\end{align}
which are open and closed immersions, as $0$ is isolated. It is the stack of $\bZ$-graded left-proper objects of $\ud{\mM}$: notice that it is a locally Artin stack, locally of finite presentation. Given a line bundle $L$ on $\mM$, denote by $L^{(\bZ)}$ the line bundle on $\mM^{(\bZ)}$ obtained by the colimit of $\boxtimes_{-n\leq i\leq n}L$.\medskip

The point $i)$ of the following is standard for classical stacks without shifted symplectic structure (see \cite[Section 6.3]{HalpernLeistner2014OnTS}, \cite[Section 7.2]{Alper2018ExistenceOM}, and it will be the only place where we will have to unwrap the construction of \cite[Theorem 5.5]{Brav2018RelativeCS}. The discussion about $\tau$, following the one of Section \ref{secttau}, will be important for the proof of the associativity of the CoHA.

\begin{lemma}\label{lemgradprod}
    Consider a dg-category of finite type $\mC$ with a Calabi-Yau structure of dimension $3$.
    \begin{enumerate}
        \item[$i)$] There is a natural isomorphism of $-1$-shifted symplectic stacks:
    \begin{align}
(\ud{\Grad}(\ud{\mM}),\Grad(\omega_\mC))\simeq (\ud{\mM}^{(\bZ)},\omega_\mC^{(\bZ)})
    \end{align}
    (in particular, the left hand side is locally Artin, and locally of finite presentation), such that $\ud{\iota}$ is identified with the colimit of the $\ud{\oplus_n}$. Under this identification, the involution:
    \begin{align}
        \ud{\tau}:\ud{\Grad}(\ud{\Grad}(\ud{\mM}))\simeq \ud{\Grad}(\ud{\Grad}(\ud{\mM}))
    \end{align}
    obtained by swapping the two $B\bG_{m,k}$ factors
 is identified with the involution $(\ud{\mM}^{(\bZ)})^{(\bZ)}\simeq (\ud{\mM}^{(\bZ)})^{(\bZ)}$ obtained by swapping the indices. 
 \item[$ii)$] Under the identification of $i)$, the isomorphism \eqref{isomiotaL} gives the weight decomposition:
\begin{align}\label{computiotaL}
    \ud{\iota}^\ast \bL_{\ud{\mM}}=&(\ud{\iota}^\ast \bL_{\ud{\mM}})^0\oplus(\ud{\iota}^\ast \bL_{\ud{\mM}})^{<0}\oplus(\ud{\iota}^\ast \bL_{\ud{\mM}})^{>0}\simeq (\bL_{\ud{\mM}})^0\oplus((\ud{\iota}^\ast \bL_{\ud{\mM}})^{<0})^{\otimes 2} \nn\\
    \simeq &\bigoplus_{i\in\bZ}(pr_i)^\ast\bL_{\ud{\mM}}\oplus\bigoplus_{i<j\in\bZ}(pr_{ij})^\ast\tilde{\bL}_{\ud{\mM}}\oplus\bigoplus_{i>j\in\bZ}(pr_{ij})^\ast\tilde{\bL}_{\ud{\mM}}\simeq \bigoplus_{i\in\bZ}(pr_i)^\ast\bL_{\ud{\mM}}\oplus\bigoplus_{i<j\in\bZ}((pr_{ij})^\ast\tilde{\bL}_{\ud{\mM}})^{\otimes 2}
\end{align}
where the term corresponding to $i,j$ has weight $i-j$, and the last isomorphism comes from \eqref{eqsigmaL}. In particular, the signed dimension of $(\ud{\iota}^\ast \bL_{\ud{\mM}})^{>0}$ at the point $(E_i)_{i\in\bZ}$ is given by the classical formula $\sum_{i<j}\langle E_i,E_j\rangle$.
\item[$iii)$] Suppose now that $\mC_c$ has an orientation data $[K_{\mM}^{1/2}]$ (resp. a strong orientation data $K_{\mM}^{1/2}$). Extending straightforwardly the definitions of Lemma \ref{lemshifsymp} $ii)$ and Lemma \ref{lemtauorientdcrit} $iii)$ and using the identification of $i)$, there is a (resp. canonical) isomorphism:
\begin{align}
    \Grad(K_{\mM}^{1/2})\simeq (K_{\mM}^{1/2})^{(\bZ)}
\end{align}
(resp. such that the isomorphism:
\begin{align}
    \tau^\ast\Grad(\Grad(K_{\mX,s}^{1/2}))\simeq \Grad(\Grad(K_{\mX,s}^{1/2}))
\end{align}
and the isomorphism:
\begin{align}
    \tau^\ast ((K_{\mM}^{1/2})^{(\bZ)})^{(\bZ)}\simeq ((K_{\mM}^{1/2})^{(\bZ)})^{(\bZ)}
\end{align}
obtained by swapping the indices agree when restricted on the open substack:
\begin{align}
(\mM)_{(0,0)}\times(\mM)_{(1,0)}\times (\mM)_{(1,1)}
\end{align})

\end{enumerate}
    \end{lemma}

 \begin{proof}
     \begin{enumerate}
         \item[$i)$] By definition, $\ud{\Grad}(\ud{\mM})$ is defined as the prestack:
     \begin{align}
         \ud{\Spec}(\ud{R})\mapsto \ud{\mM}(B\bG_{m,\ud{R}})
     \end{align}
     by descent, the latter is identified with the space of objects of:
     \begin{align}
    \mC\otimes_k\QCoh(B\bG_{m,\ud{R}})\simeq \mC\otimes_k\QCoh(\ud{R})\otimes_k\QCoh(B\bG_{m,k})
     \end{align}
     whose pullback to $\mC\otimes_k\QCoh(\ud{R})$ are left-proper. From \cite[Section 7.2]{Alper2018ExistenceOM}, $\QCoh(B\bG_{m,k})$ is the dg-category of $\bZ$-graded complexes, hence the latter is identified as the dg-category of $\bZ$-graded objects $(E_n)_{n\in\bZ}$ of $\mC\otimes_k\QCoh(\ud{R})$, where the pullbacks takes the direct sum. If a $\bZ$-graded object $\bigoplus_{n\in\bZ}E_n$ is left proper, it is in particular compact, hence a finite number of $E_n$ must be nonzero. But a finite sum of objects is left proper iff each object is left proper. It gives then the equivalence of spaces:
     \begin{align}
         \ud{\Grad}(\ud{\mM})(\ud{\Spec}(\ud{R}))\simeq (\colim_{n\to\infty}\prod_{-n\leq i\leq n}\ud{\mM}(\ud{\Spec}(\ud{R})))
     \end{align}
     which is obviously functorial in $\ud{R}$, and gives then a natural isomorphism:
     \begin{align}
         \ud{\Grad}(\ud{\mM})\simeq (\ud{\mM}^{(\bZ)})
     \end{align}
     such that $\ud{\iota}$ is the colimit of the $\ud{\oplus_n}$. Notice that the map:
     \begin{align}
         i_n:\prod_{-n\leq i\leq n}\ud{\mM}\to \ud{\Grad}(\ud{\mM})
     \end{align}
     is simply identified with the functor that sends $2n+1$ continuous adjunctions:
     \begin{align}
         f_i:\mC\rightleftharpoons \QCoh(\ud{R}):(f_i)^r
     \end{align}
     to the continuous adjunction:
     \begin{align}
    \mC\rightleftharpoons \QCoh(B\bG_{m,\ud{R}})\simeq \QCoh(\ud{R})^{\bZ}
     \end{align}
     obteined by taking the $i$-th component to be $(f_i,(f_i)^r)$ for $-n\leq i\leq n$, and the other one to be $(0,0)$.\medskip
     
     Using the above description, and the fact that $\QCoh((B\bG_m\times B\bG_m)_{\ud{R}})$ is the dg-category of $\bZ\times\bZ$-graded objects of $\QCoh(\ud{R})$, we find the claim about $\ud{\tau}$.\medskip

     We identify now the shifted symplectic structures. Fortunately, the definition of $\omega_\mC$ in \cite[Proposition 5.2]{Brav2018RelativeCS} is pretty transparent, the hard part of \cite[Theorem 5.5]{Brav2018RelativeCS} being the fact that it is nondegenerate. In \cite[Section 4.2]{Brav2018RelativeCS}, the authors define, for any presentable dg-category $\mC'$, a graded mixed complex $HH(\mC')$ (a $S^1$-equivariant version of the Hochschild homology), which is functorial with respect to continuous adjunctions, \ie given a continuous adjunction $f:\mC'\rightleftharpoons \mC''$, there is a functorial morphism $HH(\mC')\to HH(\mC'')$. It is moreover symmetric monoidal functor, as it is defined in \cite[Section 4.2]{Brav2018RelativeCS} by applying the symmetric monoidal functor of $S^1$ equivariant traces from \cite[Theorem 2.14]{HOYOIS201797} to the symmetric monoidal functor $\mC\to (\mC,Id_\mC)$. Applying the negative cycle complex functor $HC^-$ from graded mixed complexes to complexes described in \cite[Section 5.1]{Brav2018RelativeCS} (which is also symmetric monoidal, as can be seen using its definition as a pushforward), one obtains a graded complex, denoted by the shorthand notation $HC^-(\mC')$, functorial with continuous adjunctions and symmetric monoidal.\medskip
     
     Using the naturality of the symmetric monoidal functor of \cite[Theorem 2.14]{HOYOIS201797} applied to the symmetric monoidal functor of $2$-categories:
     \begin{align}
         QCoh:Corr(Aff)\to (DGCat^2_{con})^{2-op}
     \end{align}
     from \cite[5.5.3]{Gaitsgory_Rozenblyum_2017}, the authors obtain in \cite[Theorem 4.6]{Brav2018RelativeCS} an isomorphism:
     \begin{align}
         \Gamma(L\ud{\Spec}(\ud{R}),\mathcal{O}_{L\ud{\Spec}(\ud{R})})\simeq HH(\QCoh(\ud{R}))
     \end{align}
     which is functorial and symmetric monoidal in $\ud{R}$. The space of closed $p$-form of degree $n$ on an affine scheme is defined by:
     \begin{align}
         \mA^{p,cl}(\ud{\Spec}(\ud{R}),n):=|HC_w^-(\Gamma(L\ud{\Spec}(\ud{R}),\mathcal{O}_{L\ud{\Spec}(\ud{R})}))(p)[n-p]|
     \end{align}
     Using descent from the affine case, the authors define in \cite[Section 5.2]{Brav2018RelativeCS}, for any stack $\ud{\mX}$ and any $p\in\bZ$, a morphism:
     \begin{align}\label{eqHKR}
             \chi_{\ud{\mX}}:|HC^-(\Ind(\Perf(\ud{\mX})))[-n]|\to \mA^{p,cl}(\ud{\mX},p-n)
     \end{align}
     which is functorial, and symmetric monoidal in $\ud{\mX}$ (this is a derived version of the Hochschild-Kostant-Rosenberg theorem).\medskip
     
     A Calabi-Yau structure of dimension $n$ is defined to be a map $k[n]\to HC^-(\mC)$, \ie the space of Calabi-Yau structure of dimension $n$ is given by $|HC^-(\mC)[-n]|$. Using the universal continuous adjunction:
     \begin{align}
        \mathcal{F}_\mC:\mC\rightleftharpoons \Ind(\Perf(\ud{\mM})):\mathcal{F}_\mC^r
     \end{align}
     the map from the space of Calabi-Yau structure of dimension $n$ to the space of closed $2$ form of degree $n-2$ $\omega$ is then defined in \cite[Proposition 5.2]{Brav2018RelativeCS} as the composition:
     \begin{align}
         |HC^-(\mC)[-n]|\overset{HC^-(\mathcal{F}_\mC)}{\to} |HC^-(\Ind(\Perf(\ud{\mM})))[-n]|\overset{\chi_{\ud{\mM}}}{\to}\mA^{2,cl}_k(\ud{\mM},2-n)
     \end{align}
     and its image is shown in \cite[Theorem 5.5]{Brav2018RelativeCS}. to lie in the subspace of nondegenerate closed $2$ forms of degree $2-n$, \ie $2-n$-shifted symplectic structures.\medskip

    Consider the continuous adjunction of smooth dg-categories:
     \begin{align}
         \pi_n:\mC\rightleftharpoons\prod_{-n\leq i\leq n}\mC:(\pi_n)^r
     \end{align}
    where $\pi_n$ sends an objects to $2n+1$ copies of it, and $(\pi_n)^r$ sends $2n+1$ objects of $\mC$ to their direct sum. For a functor $G$ on the category of stacks to a category of objects with an additive structure $\oplus$ and a finite set $I$, we denote:
    \begin{align}
        \boxplus:=((g_i)_{i\in I}\mapsto \bigoplus_{i\in I}(pr_i)^\ast g_i):\prod_{i\in I}G(\ud{\mX}_i)\to G(\prod_{i\in I}\ud{\mX}_i)
    \end{align}
    We consider now the following diagram of spaces:
     \[\begin{tikzcd}[column sep=small]
         {|HC^-(\mC)[-n]|}\arrow[r,"HC^-(F_\mC)"]\arrow[dd,"HC^-(\pi_n)"] & {|HC^-(\Ind(\Perf(\ud{\mM})))[-n]|}\arrow[r,"\chi_{\ud{\mM}}"]\arrow[d,"HC^-(\ud{\iota}^\ast)"] & \mA^{2,cl}_k(\ud{\mM},2-n)\arrow[d,"\ud{\iota}^\ast"]\\
         & {|HC^-(\Ind(\Perf(\ud{\Grad}(\ud{\mM})))[-n]|}\arrow[d,"HC^-((i_n)^\ast)"]\arrow[r,"\chi_{\ud{\Grad}(\ud{\mM})}"{yshift=3pt}] & \mA^{2,cl}_k(\ud{\Grad}(\ud{\mM}),2-n)\arrow[d,"(i_n)^\ast"]\\
         {|HC^-(\prod_{-n\leq i\leq n}\mC)[-n]|}\arrow[r,"HC^-(\boxplus\circ\prod_i F_\mC)"{yshift=3pt}] & {|HC^-(\Ind(\Perf(\prod_{-n\leq i\leq n}\ud{\mM})))[-n]|}\arrow[r,"\chi_{\prod_i\ud{\mM}}"{yshift=3pt}] & \mA^{2,cl}_k(\prod_{-n\leq i\leq n}\ud{\mM},2-n)\\
         {\prod_{-n\leq i\leq n}|HC^-(\mC)[-n]|}\arrow[r,"\prod_iHC^-(F_\mC)"{yshift=3pt}]\arrow[u,swap,"\boxplus"]& {\prod_{-n\leq i\leq n}|HC^-(\Ind(\Perf(\ud{\mM})))[-n]|}\arrow[r,"\prod_i\chi_{\ud{\mM}}"{yshift=3pt}]\arrow[u,swap,"\boxplus"] & \prod_{-n\leq i\leq n}\mA^{2,cl}_k(\ud{\mM},2-n)]\arrow[u,swap,"\boxplus"]
     \end{tikzcd}\]
     We have to prove that it is commutative: it will build an equivalence $(\ud{i_n})^\ast\ud{\iota}^\ast\omega_{\ud{\mM}}\sim \prod_{-n\leq i\leq n}\omega_{\ud{\mM}}$ which is exactly what we want to prove. Notice that this equivalence will be obviously a compatible system of equivalence for the system $i_n$, but we will be at the end only interested in comparing the d-critical structures, which lives in a set, so we are not so interested in these questions. The four little squares are commutative from fact that $HC^-$ and \eqref{eqHKR} are functorial with continuous adjunction and symmetric monoidal. To show that the left rectangle commutes, it suffices to prove:
     \begin{align}\label{eq3}
         (i_n)^\ast\iota^\ast\circ F_\mC=\boxplus\circ(\prod_{-n\leq i\leq n}F_\mC)\circ \pi_n
     \end{align}
     we will do this by identifying the prestacks. Beware that we will work with the restriction of these functors for the small dg-categories of compact objects. Consider a point $f=(f_{-n},...,f_n):\ud{\Spec}(\ud{R})\to \prod_{-n\leq i\leq n}\ud{\mM}$, which gives $(2n+1)$ exact continuous maps $f_i:\mC_c\to \Perf(\ud{\Spec}(\ud{R}))$ with continuous right inverse. Then $i_n\circ f:\ud{\Spec}(\ud{R})\to\ud{\Grad}(\ud{\mM})$ is the map:
     \begin{align}
         (...,0,f_{-n},...,f_n,...,0):\mC_c\to \Perf(B\bG_{m,\ud{R}})
     \end{align}
     and so $\iota\circ i_n\circ f:\ud{\Spec}(\ud{R})\to\ud{\mM}$ is the map:
     \begin{align}
         \bigoplus_{-n\leq i\leq n}f_i:\mC_c\to \Perf(\ud{R})
     \end{align}
     such that by definition:
     \begin{align}\label{eq1}
         f^\ast ((i_n)^\ast \iota^\ast \circ F_c):=\bigoplus_{-n\leq i\leq n}f_i:\mC_c\to \Perf(\ud{R})
     \end{align}
     On the other hand, we have by definition:
     \begin{align}
         f^\ast\boxplus:=((g_{-n},...,g_n)\mapsto \bigoplus_{-n\leq i\leq n}(f_i)^\ast g_i:\prod_{-n\leq i\leq n}\Perf(\ud{\mM})\to \Perf(\ud{R})
     \end{align}
     such that, still by definition:
     \begin{align}
         f^\ast\boxplus\circ(\prod_{-n\leq i\leq n}F_\mC):= ((g_{-n},...,g_n)\mapsto \bigoplus_{-n\leq i\leq n}f_i(g_i)):\prod_{-n\leq i\leq n}\mC_c\to \Perf(\ud{R})
     \end{align}
     and finally, by definition of $\pi_n$:
     \begin{align}\label{eq2}
         f^\ast(\boxplus\circ(\prod_{-n\leq i\leq n}F_\mC)\circ\pi_n):= \bigoplus_{-n\leq i\leq n}f_i:\mC_c\to \Perf(\ud{R})
     \end{align}
     So \eqref{eq1} and \eqref{eq2} gives \eqref{eq3}, and then the compatibility of the big left rectangle of the above diagram, which finishes the proof.\medskip

     \item[$ii)$] To prove the claim about $\ud{\iota}^\ast \bL_{\ud{\mM}}$, we need to consider a point $f:B\bG_{m,\ud{R}}\to \ud{\mM}$, \ie a functor $f:\mC_c\to\Perf(B\bG_{m,\ud{R}})$ corepresented by $(g_i)_{i\in\bZ}\in \mC_c\otimes \Perf(B\bG_{m,\ud{R}})$, and describe the $\bZ$-grading on $f^\ast\bL_{\ud{\mM}}$. But, by \eqref{eqText}:
     \begin{align}
        f^\ast\bT_{\ud{\mM}}\simeq \Hom_{B\bG_{m,\ud{R}}}((g_i)_{i\in\bZ},(g_j)_{j\in\bZ})[1]=\bigoplus_{i,j}\Hom_{\ud{R}}(g_i,g_j)[1]
     \end{align}
     where the $i,j$ term has degree $j-i$: passing to the dual, the $i,j$ term of $f^\ast\bL_{\ud{\mM}}$ has weight $i-j$, which proves the claim. The last claim follows straightforwardly from the definition of $\langle-,-\rangle$.\medskip

     \item[$iii)$] Consider first the case of an orientation data $[K_{\mM}^{1/2}]$, and fix an orientation $K_{\mM}^{1/2}$ and a (noncanonical) isomorphism:
     \begin{align}\label{isomsqrootorient}
    \zeta_2:(\oplus_2)^\ast (K_{\mM}^{1/2})\simeq (pr_1)^\ast (K_{\mM}^{1/2})\oplus (pr_2)^\ast (K_{\mM}^{1/2})\oplus L_{\mM}
        \end{align}
        which is a square root of:
        \begin{align}
    \phi_2:(\oplus_2)^\ast (K_{\mM})\simeq (pr_1)^\ast (K_{\mM})\oplus (pr_2)^\ast (K_{\mM})\oplus L_{\mM}^{\otimes 2}
        \end{align}
    Denote by $\phi_n$ the analogue of $\phi_2$ for $\oplus_n$. Writing $\oplus_3\simeq \oplus_2\times Id_{\mM}$ and using the isomorphisms \eqref{isomsquareroot}, as in the upper and right arrow of the diagram \eqref{diagjoyceorientsq}, we obtain a square root $\zeta_3$ of $\phi_3$, which restricts to \eqref{isomsqrootorient} on $\mM\times\mM\times\{0\}$. Working iteratively, we write an inductive system $\sigma_n$ of square roots $\phi_n$, whose colimit gives an isomorphism:
    \begin{align}\label{eqiotaK}
\iota^\ast(K_{\mM}^{1/2})\simeq\bigotimes_{(i,i')\in\bZ^2}(pr_i)^\ast(K_{\mM}^{1/2})\otimes \bigotimes_{i<j\in\bZ}(pr_{ij})^\ast(L_{\mM}) 
    \end{align}
    which is a square root of the determinant of \eqref{computiotaL}. From the description of the weights in $ii)$ and the formula of Lemma \ref{lemshifsymp} $iii)$, we have:
    \begin{align}\label{eqK}
    \Grad(K_{\mM}^{1/2}):=\iota^\ast (K_{\mM}^{1/2})\otimes \det((\ud{\iota}^\ast\bL_{\ud{\mM}})^{<0})^{\otimes-1}\simeq \bigotimes(pr_i)^\ast(K_{\mM}^{1/2})=:(K_{\mM}^{1/2})^{(\bZ)}
    \end{align}
    Which proves the claim. Notice that this construction depends on the initial choice of $\zeta_2$, but also on the choices that we made to write $\oplus_n$ in terms of $\oplus_2$.\medskip

    Consider now a strong orientation data $K_{\mM}^{1/2}$, with a fixed choice of isomorphism \eqref{isomsqrootorient}. We do the same construction to define an isomorphism $\Grad(K_{\mM}^{1/2})\simeq(K_{\mM}^{1/2})^{(\bZ)}$, and we notice this time that, from the commutation of the diagram \eqref{diagjoyceorientsq} in the definition, this construction does not depends on the choice that we made to write $\oplus_n$ in terms of $\oplus_2$, \ie this isomorphism is canonical. Consider now the open substack $\mX:=
(\mM)_{(0,0)}\times(\mM)_{(1,0)}\times (\mM)_{(1,1)}$, one obtains then from the construction of \eqref{eqiotaK} a commutative diagram:

\adjustbox{scale=0.9,center}{\begin{tikzcd}[column sep=tiny]
    \begin{tabular}{c}$\tau^\ast\bigl((pr_{(0,0)})^\ast (K_{\mM}^{1/2})\otimes  (pr_{(0,1)})^\ast (K_{\mM}^{1/2})\otimes  (pr_{(1,1)})^\ast (K_{\mM}^{1/2})$\\$\otimes (pr_{(0,0),(0,1)})^\ast (L_{\mM})\otimes (pr_{(0,0),(1,1)})^\ast (L_{\mM})\otimes (pr_{(0,1),(1,1)})^\ast (L_{\mM})\bigr)$\end{tabular}\arrow[r,"\simeq"]\arrow[d,"\simeq"] & \begin{tabular}{c}$\Bigl(\tau^\ast\Grad(\iota)^\ast\bigl((pr_0)^\ast (K_{\mM}^{1/2})$\\$\otimes  (pr_1)^\ast (K_{\mM}^{1/2})\otimes  (pr_{0,1})^\ast L_{\mM})\Bigr)|_\mX$\end{tabular}\arrow[d,"\simeq"]\\
    \begin{tabular}{c}$(pr_{(0,0)})^\ast (K_{\mM}^{1/2})\otimes  (pr_{(1,0)})^\ast (K_{\mM}^{1/2})\otimes  (pr_{(1,1)})^\ast (K_{\mM}^{1/2})$\\$\otimes (pr_{(0,0),(1,0)})^\ast (L_{\mM})\otimes (pr_{(0,0),(1,1)})^\ast (L_{\mM})\otimes (pr_{(1,0),(1,1)})^\ast (L_{\mM})$\end{tabular}\arrow[d,"\simeq"] & \Bigl(\tau^\ast\Grad(\iota)^\ast\iota^\ast(K_{\mM}^{1/2})\Bigr)|_\mX\arrow[d,"\simeq"]\\
    \Bigl(\Grad(\iota)^\ast\bigl((pr_0)^\ast (K_{\mM}^{1/2})\otimes  (pr_1)^\ast (K_{\mM}^{1/2})\otimes  (pr_{0,1})^\ast L_{\mM})\Bigr)|_\mX\arrow[r,"\simeq"] & \Bigl(\Grad(\iota)^\ast\iota^\ast(K_{\mM}^{1/2})\Bigr)|_\mX
\end{tikzcd}}

From the definition of the isomorphism:
\begin{align}
    \tau^\ast ((K_{\mM}^{1/2})^{(\bZ)})^{(\bZ)}\simeq ((K_{\mM}^{1/2})^{(\bZ)})^{(\bZ)}
\end{align}
in Lemma \ref{lemtauorientdcrit} $ii)$, one finds then directly that, under the isomorphism \eqref{eqK}, its restriction to $\mX$ is simply the natural isomorphism:
\begin{align}
    &\tau^\ast\bigl((pr_{(0,0)})^\ast (K_{\mM}^{1/2})\otimes  (pr_{(0,1)})^\ast (K_{\mM}^{1/2})\otimes  (pr_{(1,1)})^\ast (K_{\mM}^{1/2})\nn\\
    \simeq &(pr_{(0,0)})^\ast (K_{\mM}^{1/2})\otimes  (pr_{(1,0)})^\ast (K_{\mM}^{1/2})\otimes  (pr_{(1,1)})^\ast (K_{\mM}^{1/2})
\end{align}
as claimed. Indeed, in this case, the $\bZ^2$-weights of the $L$ factors are respectively $(-1,0)$, $(-1,-1)$ and $(0,-1)$, which are both in $H_1\cap H_2$, hence there is no need to use \eqref{eqsigmaL} to do the step \eqref{eqtildeLweight} in Lemma \ref{lemtauorientdcrit} $ii)$.\medskip

By similar arguments, one can show that the result of $iii)$ holds without having to restrict to $\mX$ if one further assumes that $\zeta_2$ is symmetric under $\Sigma_{\mM}$, \ie that the following square of isomorphisms is commutative:
\[\begin{tikzcd}
    (\Sigma_{\mM})^\ast(\oplus_2)^\ast (K_{\mM}^{1/2})\arrow[r,"(\Sigma_{\mM})^\ast(\zeta_2)"{yshift=3pt}]\arrow[d,"\simeq"]& (\Sigma_{\mM})^\ast\bigl((pr_1)^\ast (K_{\mM}^{1/2})\oplus (pr_2)^\ast (K_{\mM}^{1/2})\oplus L_{\mM}\bigr)\arrow[d,"\simeq"]\\
    (\oplus_2)^\ast (K_{\mM}^{1/2})\arrow[r,"\zeta_2"]& (pr_1)^\ast (K_{\mM}^{1/2})\oplus (pr_2)^\ast (K_{\mM}^{1/2})\oplus L_{\mM}
\end{tikzcd}\]
where the left vertical arrow uses the commutativity isomorphism of $\oplus_2$, and the right vertical arrow  uses the isomorphism \eqref{eqsigmaL}. Notice that, as this is the square root of a commutative diagram, this is always true up to a locally constant sign on $\mM\times \mM$. As this will not be needed to prove the associativity of the CoHA, and this is not part of the initial definition of strong orientation data in \cite{JOYCEorient}, we didn't want to impose this condition.
\end{enumerate}     
 \end{proof}

 We give now a description of $\Filt(\ud{\mM})$, adapted from \cite[Corollary 7.13]{Alper2018ExistenceOM}:

\begin{lemma}\label{lemFiltdg}
    Given a dg-category of finite type $\mC$, $\Filt(\ud{\mM})$ is the stack classifying finite extensions of left-proper objects of $\mC$. More precisely, given a $k$ dg-algebra $\ud{R}$, $\Filt(\ud{\mM})(\ud{R})$ is the space of objects $(E_n)_{n\in\bZ}$ with $E_n\in (\mC\otimes_k \QCoh(\ud{R}))_{lp}$ with a map of degree $1$ $x:E_{n+1}\to E_n$, which stabilize to some left-proper object $E$ when $n\ll 0$, and to $0$ when $n\gg 0$. Then $\eta$ takes the total object $E$, and $p$ the $\bZ$-graded object $(\Cofib(E_{n+1}\overset{x}{\to} E_n))_{n\in\bZ}$.
\end{lemma}

 \begin{proof}
    By descent, objects of $\Filt(\ud{\mM})(\ud{R})$ are objects of $\mC\otimes_k \QCoh(\ud{R})\otimes_k \QCoh(\Theta_k)$ whose pullback along $\Spec(R[x])\to \Theta_{\ud{R}}$ is left proper. By Rees construction, $\QCoh(\Theta_k)$ is the dg-category of graded complexes with a map $x$ of degree $1$, \ie $\mC\otimes_k \QCoh(\ud{R})\otimes_k \QCoh(\Theta_k)$ is the dg-category of $\bZ$-graded objects $(E_n)$ of $\mC\otimes_k \QCoh(\ud{R})$ with a map $x$ of degree $1$. The same arguments that \cite[Corollary 7.13]{Alper2018ExistenceOM} gives that the restriction along $B\bG_{m,\ud{R}}\to\Theta_{\ud{R}}$ is $((E_n)_n,x)\mapsto (\Cofib(x:E_{n+1}\to E_n))_n$, and the restriction along $\Spec(\ud{R})\to\Theta_{\ud{R}}$ is $((E_n)_n,x)\mapsto \colim(\cdots E_{n+1}\overset{x}{\to}E_n\cdots)$.\medskip
     
     If $(E_n)_n,x)\in \Filt(\ud{\mM})(\ud{R})$, its restriction along $B\bG_{m,\ud{R}}\to\Theta_{\ud{R}}$ is in $\Grad(\mM)(\ud{R})$, \ie the $\Cofib(x:E_{n+1}\to E_n)$ are left-proper, and only a finite number of them is nonvanishing: in particular $E_n$ stabilize for $n\ll 0$ to $E=\eta(((E_n)_n,x))$, and it stabilize for $n\gg 0$ to some object $F$, that we claim to be zero. Indeed, consider a compact generator $K$ of $\mC\otimes_k \QCoh(\ud{R})$, and consider the sequence $K_r\in \mC\otimes_k \QCoh(\ud{R}[x])$ of objects: $\cdots K\overset{Id}{\to} K\overset{0}{\to}0\overset{0}{\to}0\cdots$ where the last $K$ is in degree $r$, such that $\colim_r K\simeq 0$. As $((E_n)_n,x)$ is compact, we have $\colim_r \Hom_{R[x]}(((E_n)_n,x),K_r)\simeq 0$, but, for $n\gg 0$, the former if $\Hom_{\ud{R}}(F,K)$, hence, as $K$ is a compact generator, $F\simeq 0$. In particular, each $E_n$ is left-proper, as it is a finite extension of left-proper objects. Hence each object of $\Filt(\ud{\mM})(\ud{R})$ is of the claimed form.\medskip
     
     On the other hand, an object as in the Lemma is obtained by a finite number of extensions from the left-proper objects $\Cofib(x:E_{n+1}\to E_n)\otimes_{\ud{R}} \ud{R}[x]\langle -n\rangle$ (where $\langle -n\rangle$ denote a shift of the grading), hence it is left proper, \ie is an object of $\Filt(\ud{\mM})(\ud{R})$.\medskip
 \end{proof}

\subsubsection{Moduli of object in the heart of a $t$-structure}

We consider now the case where $X$ is a smooth and projective scheme, and $\mC=\QCoh(X)$, $\mC_c=\Perf(X)=D^b\Coh(X)$. The results presented here could surely be obtained for smooth and proper dg-category $\mC$, by generalizing existing techniques, but we stick to this case to keep this section short. In the non-proper case, there are some difficulties linked with the fact that left-proper objects do not generates $\mC$, an the fact that one cannot extend a left-proper object along an open immersion (look at the diagonal family $\mathcal{O}_{\Delta}\in D^b\Coh(\bG_m)\otimes_k\Perf(\bG_m)$, that cannot be extended to a left-proper object of $D^b\Coh(\bG_m)\otimes_k\Perf(\bA^1)$). We leave in particular the case of a smooth quasi-projective variety $X$ for future works. As $\QCoh(X)$ is smooth and proper, the small dg-subcategory $D^b\Coh(X)$ of compact objects is also the subcategory of left-proper objects. Similarly,  from \cite[Lemma 2.8]{Toen2005ModuliOO}, for any dg-algebra $\ud{R}$, $\mM(\ud{R})$ classifies left-proper objects of $\QCoh(X_{\ud{R}})=\QCoh(X)\otimes_k\QCoh(\ud{R})$, which are exactly the compact objects, \ie objects of $\Perf(X_{\ud{R}})=D^b\Coh(X)\otimes_k\Perf(\ud{R})$.\medskip

Consider now a $t$-structure $(D^b\Coh(X)^{\leq 0},D^b\Coh(X)^{\geq 0})$ on $D^b\Coh(X)$, with Abelian heart denoted by $\mA_c$, which is assumed to be nondegenerate, \ie:
\begin{align}
    \bigcup_{n\in\bZ}D^b\Coh(X)^{\leq n}=\bigcup_{n\in\bZ}D^b\Coh(X)^{\geq-n}=D^b\Coh(X)
\end{align}
(by definition, this condition is always satisfied for any heart from a Bridgeland stability condition). We denote by $(\QCoh(X)^{\leq 0},\QCoh(X)^{\geq 0})$ its extension to a an accessible $t$-structure on $\QCoh(X)$ built from \cite[Prop. 1.4.11]{ha}, \ie $\QCoh(X)^{\leq 0}$ is the presentable subcategory of $\QCoh(X)$ generated under colimits and extensions by $D^b\Coh(X)^{\leq 0}$. From \cite[Lemma 2.1.1]{Pol07}, we have:
\begin{align}
    D^b\Coh(X)^{[a,b]}= \QCoh(X)^{[a,b]}\cap D^b\Coh(X)
\end{align}
hence the truncation functors preserves $D^b\Coh(X)$, and preserves filtered colimits by \cite[Prop. 1.4.13]{ha}, and $\mA_c=\mA\cap D^b\Coh(X)$. The objects of $\mA_c$ are compacts in $\mC$, and then furthermore compacts in $\mA$, they are stable by extensions and retracts in $\mA$, and, as the truncation functors preserve the compact objects, they generates $\mA$: it gives that $\mA_c$ is the Abelian subcategory of compact objects of $\mA$. In particular, we will use the results of \cite[Section 7.2]{Alper2018ExistenceOM} (where compact objects are called finitely presented).\medskip

A small Abelian category is said to be Noetherian if every object is Noetherian (\ie any sequence of subobjects stablilize). A presentable Abelian category $\mA$ is said to be locally Noetherian if it has a set of Noetherian generators: in this case, from \cite[Proposition B1.3]{Artin2001AbstractHS}, Noetherian and compact objects of $\mA$ coincide. In particular, $\mA$ is locally Noetherian iff $\mA_c$ is Noetherian. From \cite[Proposition 5.0.1]{AbramovichPolishchuk}, this is satisfies for any heart of a Bridgeland stability condition for which $Z(\Gamma)\in\bQ+i\bQ$, \ie for a dense subset in the space of stability conditions.

We will now explain how to define a substack $\mM_\mA$ of $\mM$ classifying objects of $\mA_c$, following and slightly adapting \cite[Section 6.2]{HalpernLeistner2014OnTS} and \cite[Section 7]{Alper2018ExistenceOM}. Namely, in \cite[Section 6]{HalpernLeistner2014OnTS}, Halpern-Leistner, build $\mM_\mA$ as a substack of the classical 1-stack of perfectly gluable complexes from \cite{Lieblich2005ModuliOC}: we will consider instead their stack as a substack of the stack $\mM$ from \cite{Toen2005ModuliOO} which will gives us a derived enhancement. In \cite[Section 7]{Alper2018ExistenceOM}, the authors works with locally Noetherian Abelian categories and build a functor classifying compact objects in them, which is proven in \cite[Proposition 6.2.7]{HalpernLeistner2014OnTS} to agree with $\mM_\mA$.\medskip

For $R$ a $k$-algebra, we have that $\QCoh(X_R)$ is naturally identified with the $R$-algebra objects of $\QCoh(X)$. In \cite[Definition 6.1.1]{HalpernLeistner2014OnTS}, Halpern-Leistner defines a $t$-structure on $\QCoh(X_R)$ whose heart corresponds from \cite[Proposition 6.1.7]{HalpernLeistner2014OnTS} to the Abelian subcategory $\mA_R$ of $R$-algebra objects of $\mA$. An object $E\in \QCoh(X_R)$ is said to be flat if $E\otimes_R-:\QCoh(X_R)\to \QCoh(R)$ is exact with respect to this $t$-structure.

\begin{definition}(\cite[Definition 6.2.4]{HalpernLeistner2014OnTS}, \cite[Prob. 3.5.1]{AbramovichPolishchuk})\label{defgenflat}
    Given a smooth and projective scheme $X$, a $t$-structure on $D^b\Coh(X)$ is said to satisfy generic flatness if, given a domain $R$ of finite type over $k$ with fraction field $K$ and an object $E\in D^b(X_R)$ such that $E|_K\in\mA_K$, there is an $f\in R$ such that $E|_{R_f}$ is flat.
\end{definition}

We define then $\mM_\mA(R)\subset\mM$ to be the subspace of flat objects of $\Perf(X_R)$. We will then use this result, using the results of \cite{AbramovichPolishchuk}, \cite{Pol07} in a fundamental way:

\begin{proposition}(\cite[Proposition 6.2.7]{HalpernLeistner2014OnTS})\label{propheart}
    Consider a smooth and projective scheme $X$, and a nondegenerate $t$-structure on $D^b\Coh(X)$, satisfying generic flatness, with Noetherian heart $\mA_c$. Then $\mM_\mA$ is an open substack of $\mM$: in particular, it is an Artin $1$-stack, locally of finite presentation, with affine diagonal. Moreover, $\mM_\mA$ coincide with the stack of \cite[Definition 7.8]{Alper2018ExistenceOM}, \ie the stack of compacts objects $E$ of $\mA_R$ such that $H^0(E\boxtimes-):R-Mod\to \mA_R$ is exact. When $R$ is of finite type, $\mM_\mA(R)$ coincide with the objects $E\in \mM$ such that $E|_x\in\mA_c$ for each closed point $x\in\Spec(R)$.
\end{proposition}

Indeed, \cite[Proposition 6.2.7]{HalpernLeistner2014OnTS} establish that $\mM_\mA$ is an open substack of the $1$-stack of universally gluable relative perfect complexes from \cite{Lieblich2005ModuliOC}, which from \cite[Remark 3.30]{Toen2005ModuliOO}, coincide with the open substack of $\mM$ of complexes $E\in\Perf(X_R)$ such that $Ext^i_R(E,E)=0$ for $i<0$. The statement about affineness of the diagonal follows from \cite[Lemma 7.20]{Alper2018ExistenceOM}. The condition that the $t$ structure is bounded with respect to the usual $t$-structure is automatic, thanks to the following lemma (which is maybe known to experts):

\begin{lemma}
    Given a smooth and projective scheme $X$, any nondegenerate $t$-structure on $D^b\Coh(X)$ satisfying generic flatness, with Noetherian heart, is bounded with respect to the usual $t$-structure.
\end{lemma}

\begin{proof}
    Consider the substack $\mM_{\QCoh(X)^{[a,b]}}$ of $\mM$ such that, for $R$ of finite type, $\mM_{[a,b]}$ is the subspace of $E\in\Perf(X_R)$ such that $E|_x\in \QCoh(X_R)^{[a,b]}$ for any closed point $x$ (one can define it, as $\mM$ is locally of finite presentation, hence it is determined by its value on $k$-agebras of finite type). Given a domain $R$ of finite type over $k$ with fraction field $K$ and an object $E\in D^b(X_R)$ such that $E|_K\in D^b(X_R)^{[a,b]}$, using generic flatness and the fact that restriction to an open subset is $t$-exact, there is an $f\in R$ such that $H^i(E)|_{R_f}=H^i(E|_{R_f})$ is flat for $a\leq i\leq b$. Then, for each closed point $x\in\Spec(R_f)$, $H^i(E|_{R_f})|_x\in \mA_c$: but $E|_{R_f}$ is generated by the $H^i(E|_{R_f})[i]$ under extensions, hence $E|_x$ is generated by the $H^i(E|_{R_f})|_x[i]$ for $a\leq i\leq b$, then $E|_x\in \QCoh(X_R)^{[a,b]}$. Using the fact that the truncation functors preserves $D^b(X_R)$ for $R$ of finite type from \cite[Theorem 3.3.6]{Pol07}, we find as in \cite[Corollary 6.2.3]{HalpernLeistner2014OnTS} that, for any $R$ of finite type and $E\in \mM(R)$, the set:
    \begin{align}
        U=\{\mathfrak{p}\in\Spec(R)|E|_{R_{\mathfrak{p}}\in }D^b(X_{R_{\mathfrak{p}}})^{[a,b]}\}
    \end{align}
    is open in $\Spec(R)$. Then, one argue by Noetherian induction as in \cite[Proposition 6.2.7]{HalpernLeistner2014OnTS}, giving that $\mM_{[a,b]}$ is an open substack of $\mM$.\medskip

    Consider the diagonal family $\mathcal{O}_\Delta\in \Perf(X\times X)$, seen as an object of $\mM(X)$. As $X$ is of finite type, from the above, the set of closed points $x\in X(k)$ such that $\mathcal{O}_x\in D^b(X)^{[a,b]}$ is open. As the $t$-structure is nondegenerate, and $X$ is quasi-compact, there is some $a,b$ such that $\mathcal{O}_x\in D^b(X)^{[a,b]}$ for any $x\in X(k)$. We use then a well-known argument from \cite[Lemma 10.1]{Bridgeland2003StabilityCO}. Consider $E\in \mA_c$ and $x\in X(k)$: from the orthogonality of the $t$-structure, for any $i<a$, $\Ext^i(E,\mathcal{O}_x)=0)$, and, for any $i>b+n$, $\Ext^i(E,\mathcal{O}_x)=\Ext^{n-i}(\mathcal{O}_x,E)^\vee=0$ by Serre duality. Then, from \cite[Proposition 5.4]{Bridgeland1999FourierMukaiTF}, $E$ has perfect amplitude in $[a,b+n]$ with respect to the usual $t$-structure, and this range does not depend on $E\in \mA_c$, \ie the $t$ structure $(D^b\Coh(X)^{\leq 0},D^b\Coh(X)^{\geq 0})$ is bounded with respect to the usual $t$-structure.
\end{proof}

We have moreover:

\begin{lemma}(\cite[Lemma 6.3.1]{HalpernLeistner2014OnTS}, \cite[Proposition 7.12, Corollary 7.13,Lemma 7.17]{Alper2018ExistenceOM})\label{lemmaFiltAb}
    Consider a smooth and projective scheme $X$, and a nondegenerate $t$-structure on $D^b\Coh(X)$, with Noetherian heart $\mA$, satisfying generic flatness. Then $\Grad(\mM_\mA)$ is the stack of $\bZ$-graded points of $\mA_c$ (\ie $\Grad(\mM_\mA)(R)$ is the stack of $\bZ$-graded objects $(E_n)_{n\in\bZ}$ of $\mA_R$, with each $E_n$ compact and $E_n=0$ unless for a finite number of $n$), and the map $\iota:\Grad(\mM_\mA)\to \mM_\mA$ takes the direct sum. $\Filt(\mM_\mA)$ is the stack of $\bZ$-filtered points of $\mA_c$ (\ie $\Filt(\mM_\mA)(R)$ is the stack of $\bZ$-filtered objects $(E_n)_{n\in\bZ}$ of $\mA_R$, with each $E_n$ compact and $E_n/E_{n+1}=0$ unless for a finite number of $n$), the map $\eta:\Filt(\mM_\mA)\to \mM_\mA$ takes the total object, and the map $p:\Filt(\mM_\mA)\to \Grad(\mM_\mA)$ takes the associated graded $(E_n/E_{n+1})_{n\in\bZ}$. Moreover, $\mM_\mA$ is $\Theta$-reductive, \ie $\eta:\Filt(\mM_\mA)\to\mM_\mA$ satisfies the valuative criterion for properness with respects to any DVR.
\end{lemma}

\begin{proof}
    Except for $\Theta$-reductivity, this is \cite[Lemma 6.3.1]{HalpernLeistner2014OnTS}, \cite[Proposition 7.12, Corollary 7.13]{Alper2018ExistenceOM}. In \cite[Lemma 7.17]{Alper2018ExistenceOM}, the author shows that $\eta$ satisfies the valuative criterion for properness with respects to any DVR which is essentially of finite type. As $\mM_\mA$ is locally of finite presentation, we will argue as in \cite[Proposition 3.18]{Alper2018ExistenceOM} to show that $\mM_\mS$ is $\Theta$-reductive. Consider any DVR $R$ with fraction field $K$, and a map $f:\Spec(R)\to\mM_\mA$, and $g:\Spec(K)\to\Filt(\mM_\mA)$, with $\eta\circ g\simeq f|_K$. As $\mM_\mA$ is locally of finite presentation, $f$ factorize through an affine scheme of finite type $\Spec(A)$: the base change of $\eta$ along $\Spec(A)\to\mM_\mS$ is then a morphism algebraic space as $\eta$ is representable, which implies from \cite[Lemma A.11]{Alper2018ExistenceOM} that it satisfies the valuative criterion with respect to any DVR, which implies that $f,g$ extends uniquely to a morphism $\Spec(R)\to \mM_\mS$, which proves the claim.
\end{proof}

\begin{remark}\label{remarkorient}
    From \cite[Proposition 2.1]{Schrg2011DerivedAG}, there is a unique open immersion $\ud{j}:\ud{\mM_\mA}\to\ud{\mM}$ with classical truncation $j$, giving a derived enhancement $\ud{\mM_\mA}$ of $\mM_\mA$. In particular, given a Calabi-Yau structure $\omega_X\simeq \mathcal{O}_X$ on $X$ of dimension $n$ on $X$, there is a canonical Calabi-You structure of dimension $n$ on $\QCoh(X)$ from \cite[Proposition 5.12]{Brav2016RelativeCS}, which gives from \cite[Theorem 5.5 1)]{Brav2018RelativeCS} a canonical $n-2$-shifted symplectic form on $\ud{\mM}$, and then on $\ud{\mM_\mA}$. When $n=3$, we obtain from \cite[Theorem 3.18]{darbstack} a canonical $d$-critical structure $s$ on $\mM_\mA$, with canonical bundle:
\begin{align}
    K_{\mM_\mA,s}\simeq \det(\bL_{\ud{\mM}})|_{(\mM_\mA)^{red}}=:K_{\mM}|_{(\mM_\mA)^{red}}
\end{align}
with the preceding notation. In particular, an the orientation data on $\mM$ from \cite{JOYCEorient} gives an isomorphism class of orientation on $(\mM_\mA,s)$ compatible with direct sum, and a strong orientation data on $\mM$ would give a canonical orientation on $(\mM_\mA,s)$, compatible with direct sum.
\end{remark}

For $E,F\in D^b\Coh(X)$, recall the pairing:
\begin{align}
    \langle E,F\rangle:=\sum_{i\in\bZ}(-1)^{i+1}\dim(\Ext^i(E,F))
\end{align}
which is antisymmetric when $X$ is CY3, by Serre duality. This pairing descends to an pairing on the Grothendieck group $K(X)=K(D^b\Coh(X))$. As in \cite{bridgeland_stability}, we denote by $K^{num}(X)$ the quotient of $K(X)$ by the kernel of this pairing: by Grothendieck-Hirzebruch-Riemann-Roch, this is a finite dimensional lattice. We argue now as in \cite[Lemma 6.4.1]{HalpernLeistner2014OnTS}. For $\gamma\in K^{num}(X)$, and $R$ is of finite type, we consider the subspace $\mM_{\QCoh(X),\gamma}(R)\subset \mM(R)$ of objects $E\in \Perf(X_R)$ such that $[E|_x]=\gamma$ for every closed point of $\Spec(R)$. Consider now $E\in \mM(R)$, and $F\in D^b\Coh(X)$: as $E$ is left-proper, $Hom_R(E,F\otimes_k R)\in\Perf(R)$, \ie $\langle E|_x,F\rangle $ is locally constant on $\Spec(R)$, \ie the class $[E|_x]\in K^{num}(X)$ is locally constant on $\Spec(R)$. As $\mM$ is locally of finite presentation, it defines an open and closed substack $\mM_{\QCoh(X),\gamma}$ of $\mM$, such that $\mM=\bigsqcup_{\gamma\in K^{num}(X)}\mM_{\QCoh(X),\gamma}$, and $\oplus_n$ restricts to:
\begin{align}
    \oplus_n:\prod_{1\leq i\leq n}\mM_{\QCoh(X),\gamma_i}\to \mM_{\QCoh(X),\sum_{1\leq i\leq n}\gamma_i}
\end{align}
Now, the intersection with the open substack $\mM_\mA$ gives open and closed substacks $\mM_{\mA,\gamma}$ with the same properties.

\subsection{Kontsevich-Soibelman wall crossing formula}\label{sectKSWCF}

We consider now the space of Bridgeland stability condition $\Stab(X)$ on a smooth and projective scheme $X$, from \cite{bridgeland_stability}. More precisely, we consider the space of numerical Bridgeland stability conditions on $D^b\Coh(X)$ satisfying the support property from \cite{ks}.

\begin{definition}\label{defstab}
A stability condition $\sigma=(Z,(\mPP(\phi)))$ in $\Stab(X)$ is the data of a group homomorphism $Z:K^{num}(X)\to\bC$ called a central charge, and full additive subcategories $\mPP(\phi)\subset D^b\Coh(X)$ for each $\phi\in\mathbb{R}$, whose objects (resp simple objects) are called semistable(resp stable) of phase $\phi$, satisfying:
\begin{enumerate}
    \item[$i)$] If $E\in\mPP(\phi)$, then $Z(E)=m(E)e^{i\pi\phi}$ for some $m(E)\in\mathbb{R}_{>0}$ (we write $Z(E):=Z([E])$).
    \item[$ii)$] For all $\phi\in\mathbb{R}$, $\mPP(\phi+1)=\mPP(\phi)[1]$.
    \item[$iii)$] If $\phi_1>\phi_2$ and $A_j\in\mPP(\phi_j)$ then $\Ext^0_\mathcal{D}(A_1,A_2)=0$.
    \item[$iv)$]  For each nonzero object $E\in D^b\Coh(X)$ there are a finite sequence of real numbers $\phi_1>\phi_2>...>\phi_n$ and a collection of triangles:
    \begin{equation}\begin{tikzcd}[column sep=tiny]
    0=E_{0}\arrow[rr]&&E_{1}\arrow[rr]\arrow[dl]&&E_{2}\arrow[r]\arrow[dl]&...\arrow[r]&E_{n-1}\arrow[rr]&&E_{n}=E\arrow[dl]\\
    &A_{1}\arrow[ul,dashed]&&A_{2}\arrow[ul,dashed]&&&&A_{n}\arrow[ul,dashed]
\end{tikzcd}\end{equation}
    with $A_j\in\mPP(\phi_j)$ for all $j$. This decomposition (which is unique from $iii)$) is called the Harder-Narasimhan decomposition.
    \item[$v)$] (support property) For $\|\cdot\|$ a norm on $K^{num}(X)$, there is a constant $C>0$ such that $\|\gamma\|\leq C|Z(\gamma)|$ for $\gamma\in K^{num}(X)$ such that there is a semistable object of classs $\gamma$ (this condition is independent from the norm, as $K^{num}(X)$ is finite dimensional). Equivalently, there is quadratic form $Q$ on $K^{num}(X)$ such that $Q|_{\ker(Z)}=0$, and, if there is an object $E\in\mPP(\phi)$ of numerical class $\gamma$, then $Q(\gamma)\geq 0$ (it suffices to take $Q(\gamma)=-\|\gamma\|^2+C^2|Z(\gamma)|^2$).
\end{enumerate}
\end{definition}

Recall that, from \cite{bridgeland_stability}, $\Stab(X)$ has a natural structure of metric space (where the metric is given by comparing the extremal slopes of Harder-Narasimhan filtrations), such that $Z:\Stab(X)\to Hom(K^{num}(X),\bC)$ is a local homeomorphism, giving furthermore a structure of complex analytic space on $\Stab(X)$.\medskip

For any interval $I$, we denote as usual by $\mPP(I)$ the full subcategory of objects $E\in D^b\Coh(X)$ whose direct factors have slope in $I$ (equivalently, it is the full subcategory generated by extensions from the $\mPP(\phi)$ for $\phi\in I$. In particular, for $\phi\in\bR$, $\mPP([\phi,\phi+1))$ and $\mPP((\phi,\phi+1])$ are the heart of nondegenerate $t$-structures on $D^b\Coh(X)$. We consider then $\mM_I^\sigma$ (resp. $\mM_{I,\gamma}^\sigma$), the substack of $\mM$ such that, for $R$ of finite type, $\mM_I^\sigma(R)$ (resp. $\mM_{I,\gamma}^\sigma(R)$) is the subspace of $E\in\mM(R) $ such that, for any closed point $x\in \Spec(R)$, $E|_x\in \mPP(I)$ (resp. and $[E|_x]=\gamma$). Notice that $\mM_{I,\gamma}^\sigma$ is an open and closed substack of $\mM_I^\sigma$. For an interval of $I$ length $<1$, we consider the associated strict sector $V:=\{me^{i\pi\phi}|m>0\;\rm{and}\;\}$ of $\bC$. Notice that, given a strict sector $V$, the corresponding interval $I$ is only well defined up to a shift of $2$: in particular, the corresponding stacks $\mM^\sigma_{I+2k}$ are isomorphic as stacks, but different as substacks of $\mM$ (the embedding differ by a shift of $[2k]$).  For $\phi\in\bR$, consider the half line $l_\phi:=\{me^{i\pi\phi}|m\geq 0\}$, such that $\mM_\phi^\sigma=\bigsqcup_{\gamma\in Z^{-1}(l_\phi)}\mM_{\phi,\gamma}^\sigma$, and, for $\gamma\in K^{num}(X)-\{0\}$, we denote also by $l_\gamma$ the half-line containing $Z(\gamma)$.\medskip

We recall the terminology of \cite{Toda2007ModuliSA}, \cite{PiyaratneToda+2019+175+219}:

\begin{definition}\label{defgood}
    Consider a smooth and projective scheme $X$, and a stability condition $\sigma\in\Stab(X)$.
    \begin{itemize}
        \item We say that $\sigma$ is algebraic if $Z(K^{num}(X))\in\bQ+i\bQ$. As $K^{num}(X)$ is finite dimensional, such stability conditions are dense in $\Stab(X)$, and, by \cite[Proposition 5.0.1]{AbramovichPolishchuk}, for an algebraic stability condition, $\mPP((0,1])$ is Noetherian.
        \item We say that $\sigma$ is bounded if, for any $\phi\in (0,1]$ and $\gamma\in K^{num}(X)$, $\mM_{\phi,\gamma}^\sigma$ is bounded, \ie admits a surjection from a scheme of finite type.
        \item We say that $\sigma$ satisfies generic flatness if the $t$-structure defining the heart $\mPP((0,1])$ satisfies it (see Definition \ref{defgenflat}).
        \end{itemize}
\end{definition}

The main result about this is the following:

\begin{proposition}(\cite[Proposition 4.12, Corollary 4.21]{PiyaratneToda+2019+175+219})\label{propPiyaTod}
    Consider a connected component $\Stab^\ast(X)$ of $\Stab(X)$. Then, if an algebraic stability condition of $\Stab^\ast(X)$ satisfies boundedness and generic flatness, then any algebraic stability condition of $\Stab^\ast(X)$ satisfies boundedness and generic flatness (we will call such a component a good component). When $X$ is a Calabi-Yau threefold, the connected components containing the stability conditions from \cite{Bayer2011BridgelandSC} (using strong Bogomolov-Gieseker inequalities) are good.
\end{proposition}

We follow now the discussion of \cite[Section 2]{ks}, \cite[Section 2.3]{kontsevich_wall}. Consider the ring $N_{mon}$ of monodromic Nori motives completed at $\bL^{-1/2}$ from Section \ref{SectKgroup}, and the motivic quantum torus:
\begin{align}\label{defQuantTor}
    G:=N_{mon}\langle (x^\gamma)_{\gamma\in K^{num}(X)}\rangle/((x^\gamma\cdot x^{\gamma'}-\bL^{\langle \gamma,\gamma'\rangle/2}x^{\gamma+\gamma'})_{\gamma,\gamma'\in K^{num}(X)},x^0-1)
\end{align}
We see $G$ as a multiplicative group, with graded Lie algebra $\mathfrak{g}:=\bigoplus_{\gamma\in\Gamma}$, with Lie bracket $[x^\gamma,x^{\gamma'}]=(\bL^{\langle \gamma,\gamma'\rangle/2}-\bL^{-\langle \gamma,\gamma'\rangle/2})x^{\gamma+\gamma'}$.\medskip

Consider a quadratic form $Q$ on $K^{num}(X)$ as in $v)$ of Definition \ref{defstab}. For a strict sector $V$ of $\bC$, consider then the convex cone of $K^{num}(X)$:
\begin{align}
    C(V,Z,Q)=\{\gamma\in K^{num}(X)|Z(\gamma)\in V,m>0,\phi\in I)\;\rm{and}\; Q(\gamma)\geq 0\}
\end{align}
it is a strict convex cone, \ie its closure does not contain any line through the origin. By the support property, we have:
\begin{align}
    \mM_I^\sigma=\bigsqcup_{\gamma\in C(V,Z,Q)}\mM_{I,\gamma}^\sigma
\end{align}
We consider then $\bigoplus_{\gamma\in C(V,Z,Q)} \mathfrak{g}_\gamma$, and its completion $\mathfrak{g}_{V,Z,Q}:=\prod_{\gamma\in C(V,Z,Q)} \mathfrak{g}_\gamma$, which is a pro-nilpotent Lie algebra, as $C(V,Z,Q)$ is a strict cone. One consider then the pro-nilpotent group $G_{V,Z,Q}:=\exp(\mathfrak{g}_{V,Z,Q})$. Consider triangles $\Delta$ cut out from the strict sector $V$ by a line, $\mathfrak{g}_\Delta:=\oplus_{\gamma\in Z^{-1}(\Delta)}\mathfrak{g}_\gamma$ (which contains a finite number of elements as $C(V,Z,Q)$ is strict). The pro-nilpotent topology is given explicitly by $\mathfrak{g}_{V,Z,Q}=\lim_{\Delta}\mathfrak{g}_\Delta$, and, denoting $G_\Delta:=\exp(\mathfrak{g}_\Delta)$, we have $G_{V,Z,Q}=\lim_{\Delta}G_\Delta$.\medskip

We prove then the Kontsevich-Soibelman wall crossing formula from \cite{ks}, applying Theorem \ref{theothetastrat} to the $\Theta$-stratification built by Halpern-Leistner in \cite[Theorem 5.3]{HalpernLeistner2014OnTS}:

\begin{theorem}\label{theoKSwcf}
    Consider a smooth and projective Calabi-Yau threefold $X$, and a good connected component $\Stab^\ast(X)$ of the space of stability condition.
    \begin{enumerate}
        \item[$i)$] Consider $\sigma\in \Stab^\ast(X)$ and an interval $I$ of length $<1$ with associated strict sector $V$. Then $\mM_{I,\gamma}$, $\mM_{\phi,\gamma}$ are oriented $d$-critical Artin $1$-stacks of finite type with affine diagonals, for any choice of quadratic form $Q$ in the support property, $\mM_{I,\gamma}=\emptyset$ if $\gamma\not\in C(V,Z,Q)$, and we have the following equality in $G(V,Z,Q)$ (the Kontsevich-Soibelman wall crossing formula):
    \begin{align}
        A^\sigma_{V,Q}:=\sum_{\gamma\in C(V,Z,Q)}[\bH_c(\mM_{I,\gamma},P_{\mM_{I,\gamma}})]x^\gamma=\prod_{\phi\in I}^\to\sum_{\gamma\in C(l_\phi,Z,Q)}[\bH_c(\mM_{\phi,\gamma},P_{\mM_{\phi,\gamma}})]x^\gamma=:\prod_{l\subset V}^\to A^\sigma_l
    \end{align}
    where the symbol $\prod_{\phi\in I}^\to$ (resp. $\prod_{l\in V}^\to$) denotes an oriented product on increasing $\phi\in I$ (resp on the half-lines $l\subset V$ in the trigonometric order).
    \item[$ii)$] The family of stability data $(Z,(\log(A^\sigma_{l_\gamma})_\gamma)_{\gamma\in K^{num}_Z}))$ defines a continuous family of stability data on $\Stab^\ast(X)$ in the sense of \cite[Definition 3]{ks} (in particular, it defines a wall crossing structure on $\Stab^\ast(X)$ in the sense of \cite[Definition 2.2.1]{kontsevich_wall}).
    \end{enumerate}
\end{theorem}

Stability structures from \cite{ks}, and their generalization, the wall-crossing structures from \cite{kontsevich_wall}, are very versatile structures, that allows in particular to build scattering diagrams on $\Stab^\ast(X)$ encoding the Donaldson-Thomas invariants. We refer to \cite{ks} and \cite{kontsevich_wall} for more details.

\begin{proof}
\begin{enumerate}
    \item[$i)$] Consider $\gamma\in K^{num}(X)$ and a closed point $E\in \mM^\sigma_{I,\gamma}(k)$: by definition, the Harder-Narasimhan factors are semistable objects of class $\gamma_i$ and phase $\phi_i\in I$ such that $\gamma=\sum_i\gamma_i$. From the support property, $Q(\gamma_i)\geq 0$, which implies that $\gamma_i\in C(I,Z,Q)$, and then, as $C(I,Z,Q)$ is a cone, $\gamma \in C(I,Z,Q)$, \ie $\mM_{I,\gamma}=\emptyset$ if $\gamma\not\in C(I,Z,Q)$.\medskip

    Consider any triangle $\Delta$ in $V$ cut out from a line, and the set $S:=Z^{-1}(\Delta)\cap C(V,Z,Q)$: as $C(V,Z,Q)$ is a strict cone, it is finite. Moreover, from \cite[Proposition 2.8]{Toda2007ModuliSA}, there is a locally finite real wall and chamber structure on $\Stab^\ast(X)$ such that the set of semistable object of class $\gamma\in S$ are constant on the interior of each component of the intersection of walls. Notice that those walls are necessarily pulled back from rational walls of $Hom(K^{num}(X),\bC)$, hence algebraic stability conditions are dense in any intersection of walls. Then we can find an algebraic stability condition $\sigma'$ in $\Stab^\ast(X)$ such that the $\gamma$-semistable objects of $\sigma$ and $\sigma'$ are the same for any $\gamma\in S$. Then, as any $\gamma_i$ appearing in the Harder-Narasimhan decomposition of a closed point of $\mM^\sigma_{I,\gamma}$ is in $\Delta$ if $\gamma\in \Delta$, we have that $\mM^\sigma_{I,\gamma}=\mM^{\sigma'}_{I,\gamma}$ for any $\gamma\in \Delta$. Then, by the definition of the profinite completion, we can suppose that $\sigma$ is an algebraic stability condition, in particular it is Noetherian. From Piyaratne-Toda's Proposition \ref{propPiyaTod}, $\sigma$ satisfies moreover boundedness and generic flatness. From \cite[Lemma 3.15]{Toda2007ModuliSA}, it implies that $\sigma$ boundedness of quotients as in \cite[Definition 6.5.2]{HalpernLeistner2014OnTS}.\medskip

    As $\Delta$ contains a finite number of class, one can find $m_1e^{i\pi\phi_1},m_2e^{i\pi\phi_2}\in \bQ+i\bQ$ with $\phi'-\phi\leq 1$ such that, for $\gamma\in \Delta$, $\mM^\sigma_{I,\gamma}=\mM^\sigma_{(\phi_1,\phi_2],\gamma}$. But, by definition, $\mM^\sigma_{(\phi_1,\phi_2],\gamma}=\mM^\sigma_{(\phi_1,\phi_1+1],\gamma}\times_{\mM}\mM^\sigma_{(\phi_2,\phi_2+1],\gamma}$. From Halpern-Leistner's Proposition, \ref{propheart} applied to stability conditions obtained from $\sigma$ by the action of $\tilde{Gl}_2(\bQ)$, $\mM^\sigma_{I,\gamma}$ is an open substack of $\mM$, which is an Artin 1-stack, locally of finite presentation, with affine diagonal, which inherits as in Remark \ref{remarkorient} a natural $d$-critical structure, and a natural isomorphism class of orientation, compatible with direct sum (notice that these does not depends of any choices made here). The same reasoning applies to $\mM^\sigma_{\phi,\gamma}$ for any $\gamma\in S$: moreover, those are bounded by assumption, hence $\mM^\sigma_{\phi,\gamma}$ is quasi-compact and locally of finite presentation, \ie it is of finite type. Then the set of $\sigma$-semistable objects of slope $\phi\in I$ and class $\gamma\in S$ is bounded, and there is some $N\in\bN$ such that any object of $\mPP(I)$ of class $\gamma\in S$ is obtained from an extension of at most $N$ such elements: then, applying $N-1$ times \cite[Lemma 3.16]{Toda2007ModuliSA}, we find that $\mM^\sigma_{I,\gamma}$ is also bounded, hence of finite type, for any $\gamma\in S$.\medskip

    We apply now \cite[Theorem 6.5.3]{HalpernLeistner2014OnTS} to a shift of $\sigma$ under the action of $\tilde{Gl}_2(\bQ)$, which is still algebraic. It gives a $\Theta$-stratification on $\mM^\sigma_{(\phi_1,\phi_1+1),\gamma}$, which gives on closed points the Harder-Narasimhan decomposition of objects of $\mPP(\phi_1,\phi_1+1)$ of class $\gamma$. More precisely, the totally ordered set $\Gamma$ is the set of vectors $(\gamma_1,\cdots,\gamma_r)$ such that $\gamma_1+\cdots+\gamma_r=\gamma$, with a total order extending the partial order $(\sum_{j=1}^{r'_1} \gamma_{1,j},\cdots,\sum_{j=1}^{r'_r} \gamma_{r,j})\leq (\gamma_{1,1},\cdots,\gamma_{r,r'_r})$. The center $\mathcal{Z}_{(\gamma_1,\cdots,\gamma_r)}$ of the corresponding strata is the open substack $\prod_{i=1}^n\mM^\sigma_{\phi_i,\gamma_i}$ of $\Grad(\mM_{(\phi_1,\phi_1+1),\gamma})$, and $\Theta_{(\gamma_1,\cdots,\gamma_n)}$ is the open substack of $\Filt(\mM^\sigma_{(\phi_1,\phi_1+1),\gamma})$ of filtrations $(E_i)_{i\in\bZ}$ such that $E_i/E_{i+1}$ is semistable of class $\gamma_i$ for $1\leq i\leq r$, and is $0$ otherwise. In particular, $\mM^\sigma_{I,\gamma}$ is by definition the union of the open strata $(\mM^\sigma_{(\phi_1,\phi_1+1),\gamma})_{\leq (\gamma_1,\cdots,\gamma_r)}$ for $\gamma_1,\cdots,\gamma_r\in C(V,Z,Q)$, then from \cite[Lemma 2.3.2]{HalpernLeistner2014OnTS}, there is an induced $\Theta$-stratification on it, with the same description than above.\medskip

    We apply now Theorem \ref{theothetastrat} to the $d$-critical Artin $1$-stack with affine diagonal, locally of finite type $\mM^\sigma_{I,\gamma}$, with its isomorphism class of orientation from \cite{JOYCEorient}, which gives:
    \begin{align}
[\bH_c(\mM_{I,\gamma},P_{\mM_{I,\gamma}})]=\sum_{(\gamma_1,...,\gamma_r)\in C(I,Z,Q)^r}\bL^{\Ind_{(\gamma_1,...,\gamma_r)}/2}[\bH_c(\mathcal{Z}_{(\gamma_1,...,\gamma_r)},P_{\mathcal{Z}_{(\gamma_1,...,\gamma_r)}})]
    \end{align}
    But, from Lemma \ref{lemgradprod} $i)$, $iii)$, given a choice of orientation in the isomorphism class, the oriented $d$-critical stack $\mathcal{Z}_{(\gamma_1,...,\gamma_r)}$ is isomorphic to $\prod_{i=1}^n\mM^\sigma_{\phi_i,\gamma_i}$, then Corollary \ref{corfunctjoyce} gives:
    \begin{align}
[\bH_c(\mathcal{Z}_{(\gamma_1,...,\gamma_r)},P_{\mathcal{Z}_{(\gamma_1,...,\gamma_r)}})]=\prod_{i=1}^r[\bH_c(\mM_{\phi_i,\gamma_i},P_{\mM_{\phi_i,\gamma_i}})]
    \end{align}
    and Lemma \ref{lemgradprod} $ii)$ gives:
    \begin{align}
        \Ind_{(\gamma_1,...,\gamma_r)}=\sum_{1\leq i<j\leq r}\langle \gamma_i,\gamma_j\rangle
    \end{align}
    hence by the definition of the quantum affine torus \eqref{defQuantTor} we have:
    \begin{align}
[\bH_c(\mM_{I,\gamma},P_{\mM_{I,\gamma}})]x^\gamma=\sum_{r\geq 1}\sum_{\begin{tabular}{c}$(\gamma_1,...,\gamma_r)\in C(V,Z,Q)^r$\\$\gamma_1+\cdots+\gamma_r=\gamma$\end{tabular}}\prod_{i=1}^r[\bH_c(\mM_{\phi_i,\gamma_i},P_{\mM_{\phi_i,\gamma_i}})]x^{\gamma_i}
    \end{align}
    
    By definition of the profinite completion, this finishes the proof of $i)$. Notice that $A^\sigma_V,A^\sigma_l$ are independent from the choice of $I,\phi$, as the stacks $\mM^\sigma_{I+2k,\gamma},\mM^\sigma_{\phi+2k,\gamma}$ are independent from $k$ as oriented $d$-critical stacks.

    \item[$ii)$] We check the conditions $a),b),c)$ of \cite[Definition 3]{ks}. By definition of the topology on $\Stab(X)$, $Z$ is continuous, hence $a)$ is satisfied. Consider $\sigma\in \Stab^\ast(X)$: literally, the condition of $b)$ is equivalent to the fact that, if a quadratic form $Q$ gives the support property for $\sigma$, then there is an open neighborhood $U$ of $\sigma$ such that $Q$ gives the support property for any $\sigma'\in U$. This condition is too strong, and has no chance to hold in general. We replace it by the slightly weaker condition that, for any $\sigma$, there is a neighborhood $U$ of $\sigma$ and a quadratic form $Q$ which gives the support property for any $\sigma'\in U$: one checks directly that all the arguments of \cite[Section 2]{ks} can be proven also in this setting. From \cite[Lemma 5.5.4]{bayermacri}, for a choice of norm $\|\cdot\|$ on $K^{num}(X)$, there is an open neighborhood $U$ of $\sigma$ and $C<\infty$ such that, for any $\sigma'\in U$, one can take the constant to be $C$ in the support property, hence the quadratic form $Q(\gamma)=-\|\gamma\|^2+C^2|Z(\gamma)|^2$ works for any $\sigma'\in U$.\medskip

    We check now $c)$. Consider a closed interval $I=[\phi_1,\phi_2]$ and $\sigma\in\Stab^\ast(U)$. For any $\gamma$, $\log(A^\sigma_{l_\gamma})_\gamma$ depends only on the restriction of $A^\sigma_{l_\gamma}$ to a finite subset of $C(l,Z,Q)$, hence is independent from the choice of $Q$. We have to show that, given $\sigma\in \Stab^\ast(X)$ and a choice of $Q$ as above, given $I=[\phi_1,\phi_2]$ such that, for $\sigma$, $\mPP(\phi_1),\mPP(\phi_2)$ are empty, then $A^{\sigma'}_{V,Q}$ is continuous at $\sigma$. Consider a triangle $\Delta\subset V$, and the associated bounded mass subset $S:=Z^{-1}(\Delta)\cap C(V,Z,Q)$. From \cite[Proposition 2.8]{Toda2007ModuliSA}, there is $\epsilon >0$ and an open neighborhood $U$ of $\sigma$ (which we takes to be connected, and on which $Q$ gives the support property) such that, for any $\sigma'\in U$, there are no $\sigma'$-semistable objects of phase $\phi\in [\phi_1-\epsilon,\phi_1+\epsilon]\cup [\phi_2-\epsilon,\phi_2+\epsilon]$ and class $\gamma\in S$. As $U$ is connected, by the definition of the topology of $\Stab(X)$, the Harder-Narasimhan factors of any $\sigma'$-semistable object of phase $\phi\in I$ and class $\gamma\in S$ for any other $\sigma''\in U$ have also their phase in $I$. It means then that, for any $\gamma\in S$, $\mM^{\sigma'}_{I,\gamma}$, and then $[\bH_c(\mM_{I,\gamma},P_{\mM_{I,\gamma}})]$ is independent from $\sigma'\in U$. By definition of the profinite completion, it means exactly that $A^{\sigma'}_{V,Q}$ is continuous at $\sigma$.
\end{enumerate}
    
\end{proof}

\subsection{Construction of the Cohomological Hall algebra}

Cohomological Hall algebras were first introduced in \cite{KonSol10}, with inspiration from the classical theory of Ringel-Hall algebras \cite{Ringel1990HallAA}. In \cite{KonSol10}, Kontsevich and Soibelman have defined a Cohomological Hall algebra for quiver with potentials (see also \cite{Dav13}, \cite{DavMein} for latter development in this case). They use crucially the fact that the moduli stacks are in this case global critical loci, hence there is no need to use the gluing formalism from \cite{Joycesymstab}, inspired from \cite{ks}. The case of compact Calabi-Yau has remained opened during a long time, due to the difficulties in gluing. The major approach to this problem was through the so-called Joyce conjecture from Joyce and Safronov \cite[Conjecture 1.1]{joyce2019lagrangian}, using the fact that the map from the stack of short exact sequence $0\to E\to F\to G\to$ to $\mM_\mA\times\mM_\mA\times\mM_\mA$ is a $-1$ Lagrangian in the $-1$-shifted symplectic stack $\mM_\mA$. However, proving the Joyce conjectures amount to glue morphisms in the derived category of constructible sheaves (or monodromic mixed Hodge modules), which is require to do homotopy coherent constructions. In fact, it is possible to do that more easily, using the hyperbolic localization isomorphism from Theorem \ref{theohyploc}, which is an isomorphism in an Abelian category, hence easier to build. A said in the introduction, this idea was also recently used by Kinjo, Park and Safronov in \cite{Kinjo2024CohomologicalHA}, in an independent work.\medskip

Consider a smooth and projective Calabi-Yau threefold $X$, and a nondegenerate $t$-structure with Noetherian Abelian heart $\mA_c$ on $D^b\Coh(X)$ satisfying generic flatness. Then, from Halpern-Leistner's Proposition \ref{propheart}, $\mM_\mA$ is an Artin $1$-stack locally of finite presentation, with affine diagonal. Consider the correspondence:
\[\begin{tikzcd}
\mM_{\mA,\gamma_1}\times\mM_{\mA,\gamma_1}\arrow[rr,bend right,"\oplus_{\gamma_1,\gamma_2}"]& \Filt_{\mA,\gamma_1,\gamma_2}(\mA)\arrow[l,swap,"p_{\gamma_1,\gamma_2}"]\arrow[r,"\eta_{\gamma_1,\gamma_2}"] & \mM_{\mA,\gamma_1+\gamma_2}
\end{tikzcd}\]
obtained by restricting the $\Theta$-correspondence to the open subset of $\Grad(\mM_\mA)$ of graded objects $(E_i)_{i\in\bZ}$ such that $E_0\in \mM_{\mA,\gamma_1}$, $E_1\in \mM_{\mA,\gamma_2}$ and $E_i=0$ for $i\neq 0,1$. According to Lemma \ref{lemmaFiltAb}, the $k$-points of $\Filt_2(\mA)$ are the short exact sequences $0\to E\to F\to G\to 0$ of objects of $\mA_c$ such that $[E]=\gamma_1$, $[G]=\gamma_2$, and then $[F]=\gamma_1+\gamma_2$. Similarly, consider the following commutative diagram:

\begin{equation}\label{diagAssocCoha}
\begin{tikzcd}[column sep=large]
    \mM_{\mA,\gamma_1}\times\mM_{\mA,\gamma_2}\times \mM_{\mA,\gamma_3} & \Filt_{\mA,\gamma_1,\gamma_2}\times\mM_{\mA,\gamma_3}\arrow[l,swap,"p_{\gamma_1,\gamma_2}\times Id"]\arrow[r,"\eta_{\gamma_1,\gamma_2}\times Id"] & \mM_{\mA,\gamma_1+\gamma_2}\times \mM_{\mA,\gamma_3}\\
\mM_{\mA,\gamma_1}\times\Filt_{\mA,\gamma_2,\gamma_3}\arrow[u,swap,"Id\times p_{\gamma_2,\gamma_3}"]\arrow[d,"Id\times \eta_{\gamma_2,\gamma_3}"] & \Filt_{\mA,\gamma_1,\gamma_2,\gamma_3}\arrow[ul,swap,"p_{\gamma_1,\gamma_2,\gamma_3}"]\arrow[dr,"\eta_{\gamma_1,\gamma_2,\gamma_3}"]\arrow[u]\arrow[d]\arrow[r]\arrow[l] & \Filt_{\mA,\gamma_1+\gamma_2,\gamma_3}\arrow[u,swap,"p_{\gamma_1+\gamma_2,\gamma_3}"]\arrow[d,"\eta_{\gamma_1+\gamma_2,\gamma_3}"]\\
    \mM_{\mA,\gamma_1}\times\mM_{\mA,\gamma_2+\gamma_3} & \Filt_{\mA,\gamma_1,\gamma_2+\gamma_3}\arrow[l,swap,"p_{\gamma_1,\gamma_2+\gamma_3}"]\arrow[r,"\eta_{\gamma_1,\gamma_2+\gamma_3}"] & \mM_{\mA,\gamma_1+\gamma_2+\gamma_3}
\end{tikzcd}
\end{equation}

where the upper right and lower left squares are Cartesian, which is obtained by restricting the diagram \ref{diagcomphyploc} to the open subset of $\Grad(\Grad(\mM_\mA))$ of $\bZ^2$-graded objects $(E_{(i,j)})_{i,j\in\bZ}$ such that $E_{(0,0)}\in \mM_{\mA,\gamma_1}$, $E_{(1,0)}\in \mM_{\mA,\gamma_2}$ and $E_{(1,1)}\in \mM_{\mA,\gamma_3}$ and $E_{(i,j)}=0$ for the remaining $(i,j)$. In particular, $\Filt_{\mA,\gamma_1+\gamma_2,\gamma_3}$ is the stacks of two-steps filtrations of objects of $\mA_c$, whose graded components have respectively class $\gamma_1,\gamma_2,\gamma_3$.\medskip

A Serre subcategory of $\mA_c$ is a full subcategory $\mS$ such that, given any short exact sequence $0\to E\to F\to F/E\to 0$ in $\mA_c$, $F\in\mS$ iff $E,F/E\in\mS$. We define then the substack $\mM_\mS$ of $\mM_\mA$ such that, when $R$ is of finite type, $\mM_\mS(R)$ is the full subspace of $E\in\mM_\mA(R)$ such that $E|_x\in\mS$ for any closed point $x\in\Spec(R)$. We say that $\mS$ is locally closed if $\mM_\mS\to \mM_\mA$ is a locally closed immersion: in particular, $\mM_\mS$ is then an Artin $1$-stack locally of finite presentation, with affine diagonal. Typical examples are given as follows:

\begin{lemma}\label{lemIntSerre}
    Consider $\mA_c:=\mPP((\phi,\phi+1])$ an Abelian heart from a Bridgeland stability condition, and consider a subinterval $I\subset(\phi,\phi+1]$. Then $\mS:=\mPP(I)$ is a Serre subcategory of $\mA_c$.
\end{lemma}

\begin{proof}
    Consider first the case where $(\phi,\phi+1]$ is the ordered union of two intervals $I_1,I_2$. From Definition \ref{defstab} $iii)$:
    \begin{align}
        Hom_{\mA_c}(\mPP(I_2),\mPP(I_1))=Ext^0_{D^b\Coh(X)}(\mPP(I_2),\mPP(I_1))=0
    \end{align}
    and, from Definition \ref{defstab} $v)$, for each $F\in\mA_c$, there is a (necessarily unique) exact sequence $0\to E\to F\to G\to 0$ with $E\in \mPP(I_2)$ and $E\in \mPP(I_1)$. Then $(\mPP(I_1),\mPP(I_2))$ forms a tilting pair of $\mA_c$, hence both $\mPP(I_1),\mPP(I_2)$ are Serre subcategory of $\mA_c$. Writing $(\phi,\phi+1]$ as an ordered union of the intervals $I_1,I,I_2$, we have $\mPP(I)=\mPP(I_1\cup I)\cap \mPP(I\cup I_2)$, hence it is a Serre subcategory of $\mA_c$.
\end{proof}

Notice that, for any $R$ of finite type, and $E\in \Filt(\mM_\mA)(R)=\mM_\mA(\Theta_R)$ (resp. $E\in \Grad(\mM_\mA)(R)=\mM_\mA(B\bG_{m,R})$, by definition, $E\in \Filt(\mM_\mS)(R)$ (resp $E\in \Grad(\mM_\mS)(R)$ iff its restriction to every closed point of $\bA^1_R\to \Theta_R$ (resp. $\Spec(R)\to B\bG_{m,R}$ is in $\mS$, \ie iff, for each closed point $x\in\Spec(R)$, the total object and the associated graded of $E|_x$ is in $\mS$ (resp. the total object of $E|_x$ is in $\mS$). As $\mS$ is a Serre subcategory, an object is in $\mS$ iff its associated graded is in $\mS$, which gives:
\begin{align}\label{eqGradS}
    \Grad(\mM_\mS)&=\Grad(\mM_\mA)\times_{\mM_\mA}\mM_\mS\nn\\
\Filt(\mM_\mS)&=\Filt(\mM_\mA)\times_{\mM_\mA}\mS=\Filt(\mM_\mA)\times_{\Grad(\mM_\mA)}\Grad(\mM_\mS)
\end{align}
For $\gamma_1,\gamma_2\in K^{num}(X)$, we obtain then a Cartesian diagram:
\[\begin{tikzcd}
    \mM_{\mS,\gamma_1+\gamma_2}\arrow[d]& \mM_{\mS,\gamma_1}\times \mM_{\mS,\gamma_2}\arrow[d]\arrow[l,swap,"\oplus_{\gamma_1,\gamma_2}"] & \Filt(\mM_\mS)_{\gamma_1,\gamma_2}\arrow[d]\arrow[l,swap,"p_{\gamma_1,\gamma_2}"]\arrow[r,"\eta_{\gamma_1,\gamma_2}"] & \mM_{\mS,\gamma_1+\gamma_2}\arrow[d]\\
    \mM_{\mA,\gamma_1+\gamma_2}& \mM_{\mA,\gamma_1}\times \mM_{\mA,\gamma_2}\arrow[l,swap,"\oplus_{\gamma_1,\gamma_2}"]  & \Filt(\mM_\mA)_{\gamma_1,\gamma_2}\arrow[l,swap,"p_{\gamma_1,\gamma_2}"]\arrow[r,"\eta_{\gamma_1,\gamma_2}"] & \mM_{\mA,\gamma_1+\gamma_2}
\end{tikzcd}\]
And similarly, we obtain a diagram similar to \ref{diagAssocCoha} above for $\mM_\mS$, defining $\Filt_{\mS,\gamma_1,\gamma_2,\gamma_3}$, Cartesian over the diagram for $\mM_\mA$.

\begin{proposition}\label{propModuli}
    Consider a smooth and projective Calabi-Yau threefold $X$, a nondegenerate $t$-structure with Noetherian Abelian heart $\mA_c$ on $D^b\Coh(X)$ satisfying generic flatness, and a locally closed Serre subcategory $\mS$ of $\mA_c$
    \begin{enumerate}
        \item[$i)$](\cite[Theorem 7.23]{Alper2018ExistenceOM}) If $\mM_{\mS,\gamma}$ is bounded, then it admits a separated good moduli space $JH_\gamma:\mM_{\mS,\gamma}\to M_{\mS,\gamma}$, whose $k$-points represents semisimple objects of $\mS$. Moreover, $M_{\mS,\gamma}$ is proper if $\mM_\mS\to\mM_\mA$ is closed.
        \item[$ii)$] If this is the case for any $\gamma\in K^{num}(X)$, there is a canonical $2$-isomorphism making the following diagram commutative:
    \[\begin{tikzcd}
        \mM_{\mS,\gamma_1}\times \mM_{\mS,\gamma_2}\arrow[d,"JH_{\gamma_1}\times JH_{\gamma_2}"] & \Filt(\mM_\mS)_{\gamma_1,\gamma_2}\arrow[l,swap,"p_{\gamma_1,\gamma_2}"]\arrow[r,"\eta_{\gamma_1,\gamma_2}"] & \mM_{\mS,\gamma_1+\gamma_2}\arrow[d,"JH_{\gamma_1+\gamma_2}"]\\
        M_{\mS,\gamma_1}\times M_{\mS,\gamma_1}\arrow[rr,"\oplus_{\gamma_1,\gamma_2}"] && M_{\mS,\gamma_1+\gamma_2}
    \end{tikzcd}\]
    where, on $k$-points, $\oplus_{\gamma_1,\gamma_2}$ takes the direct sum of semisimple objects, which extends $JH:\mM_\mS\to M_\mS$ to a morphism of $K^{num}(X)$-graded monoids in the $2$-category of stacks. This $2$-isomorphism is compatible with the diagram \ref{diagAssocCoha}.
    \end{enumerate}
\end{proposition}

\begin{proof}
    \begin{enumerate}
        \item[$i)$] If $\mM_\mS\to\mM_\mA$ is closed, this is exactly \cite[Theorem 7.23]{Alper2018ExistenceOM}. We slightly adapt the proof of \cite[Theorem 7.23]{Alper2018ExistenceOM}, accounting for the restriction to a Serre subcategory, which is not necessarily closed. By Halpern-Leistner's Proposition \ref{propheart}, $\mM_\mA$ is an Artin $1$-stack locally finite presentation with affine diagonal: as $\mM_\mS\to\mM_\mA$ is locally closed by assumption, then $\mM_{\mS,\gamma}$ is an Artin $1$-stack of finite type (as it is locally of finite presentation and bounded), with affine diagonal. By \eqref{eqGradS}, $\Filt(\mM_{\mS,\gamma})=\Filt(\mM_\mA)\times_{\mM_\mA}\mM_{\mS,\gamma}$, and $\mM_\mA$ is $\Theta$-reductive by Lemma \ref{lemmaFiltAb}, \ie one obtains by base change that $\mM_{\mS,\gamma}$ is $\Theta$-reductive.\medskip

    Consider a discrete valuation ring $R$ (DVR) with fraction field $K$ and residual field $\kappa$, and a choice of uniformizer $\pi$. As in \cite[Section 3.5]{Alper2018ExistenceOM}, consider $\overline{ST}_R:=[\Spec(R[s,t]/(st-\pi))/\bG_m]$, where $s,t$ have $\bG_m$-weights $1,-1$ (a different choice of uniformizer $\pi$ gives an isomorphic stack). Consider the closed immersion $0:\Spec(\kappa)\to \overline{ST}_R$ given by $s=t=0$, such that $\overline{ST}_R-0=\Spec(R)\cup_{\Spec(K)}\Spec(R)$, \ie a morphism $\overline{ST}_R-0\to \mX$ is the data of two morphisms $\zeta,\zeta':\Spec(R)\to\mX$, and an isomorphism $\zeta_{\Spec(K)}\simeq\zeta'_{\Spec(K)}$. In \cite[Definition 3.38]{Alper2018ExistenceOM}, a $1$-stack $\mX$ is said to be $S$-complete if, for any DVR $R$, any morphism $\overline{ST}_R-0\to \mX$ can be uniquely extended to a morphism $\overline{ST}_R\to \mX$. From \cite[Lemma 7.16]{Alper2018ExistenceOM}, $\mM_\mA$ satisfies this property with respect to any DVR essentially of finite type. Consider $R$ essentially of finite type, and $E\in \mM_\mS(\overline{ST}_R-0)$: in particular, one have $E\in \mM_\mA(\overline{ST}_R-0)$, hence there is a unique $F\in \mM_\mA(\overline{ST}_R)$ such that $F|_{\overline{ST}_R-0}\simeq E$: it suffices then to show that $F\in \mM_\mS(\overline{ST}_R)$. The restriction along $\Theta_\kappa\overset{s=0}{\to} \overline{ST}_R$ is an objects of $\mM_\mA(\Theta_\kappa)$, whose restriction to the open point is in $\mM_\mA(\kappa)$, hence, from \eqref{eqGradS}, we have also $F|_0\in\mM_\mS(\kappa)$. As $\mM_\mA$ is locally of finite presentation, there is by definition a subring of finite type $R'$ of $R$ containing $\pi$, and an object $F'\in \mM_\mA(\overline{ST}_{R'})$ such that $F'|_{\overline{ST}_R}\simeq F$, $F'|_{\overline{ST}_{R'}-0'}\in\mM_\mS(\overline{ST}_{R'}-0')$ and $F'|_{0'}\in\mM_\mS(\kappa)$ (where $0':\Spec(R'/\pi)\to \overline{ST}_{R'}$). Then, by definition, $F'\in\mM_\mS(\overline{ST}_{R'})$, and then $F\in \mM_\mS(\overline{ST}_R)$. We obtain then that $\mM_\mS$, and furthermore $\mM_{\mS,\gamma}$, are $S$-complete with respect to any DVR essentially of finite type, and $\mM_{\mS,\gamma}$ is of finite type, with affine diagonal: then by \cite[Proposition 3.42]{Alper2018ExistenceOM}, $\mM_{\mS,\gamma}$ is $S$-complete.\medskip

    Then $\mM_{\mS,\gamma}$, is $\Theta$-reductive and $S$-complete, and moreover of finite type and with affine diagonal. Then \cite[Theorem A]{Alper2018ExistenceOM} gives that it admits a separated good moduli space $JH_\gamma:\mM_{\mS,\gamma}\to M_{\mS,\gamma}$, and \cite[Lemma 7.19]{Alper2018ExistenceOM} gives that closed points of $\mM_\mS$ are the semisimple objects of $\mS$ (as $\mS$ is a Serre subcategory, an object of finite type of $\mA_c$ is in $\mS$ iff each of its Jordan-Hölder factors is in $\mS$). From \cite[Lemma 7.16]{Alper2018ExistenceOM}, $\mM_\mA$ satisfies the existence part of the valuative criteria for properness, hence if $\mM_\mS\to\mM_\mA$ is closed, $\mM_{\mS,\gamma}$ does it too, hence from \cite[Proposition 3.48]{Alper2018ExistenceOM} $M_\mS$ is proper.\medskip
    
    \item[$ii)$]
    By the universal property of good moduli space, the map to an algebraic space:
    \begin{align}
        JH_{\gamma_1+\gamma_2}\circ \oplus_{\gamma_1,\gamma_2}:\mM_{\mS,\gamma_1}\times \mM_{\mS,\gamma_2}\to M_{\mS,\gamma_1+\gamma_2}
    \end{align}
    factorize uniquely through a map:
    \begin{align}
\oplus_{\gamma_1,\gamma_2}:M_{\mS,\gamma_1}\times M_{\mS,\gamma_2}\to M_{\mS,\gamma_1+\gamma_2}
    \end{align}
    the same argument for $\oplus_{\gamma_1,...,\gamma_r}$ gives a monoidal structure on $M_\mS$, such that, by construction $JH:\mM_\mS\to M_\mS$ enhance to a morphism of monoids. On closed points, $\oplus_{\gamma_1,\gamma_2}:\mM_{\mS,\gamma_1}(k)\times \mM_{\mS,\gamma_2}(k)\to \mM_{\mS,\gamma_1+\gamma_2}(k)$ is obtained by taking the direct sum, hence it is also the case at the level of the good moduli space.\medskip
    
    At the level closed points, the diagram commutes because $JH$ takes the polystable object associated to an object of finite length, hence commutes with taking the associated graded. Consider the following commutative diagram:
    \[\begin{tikzcd}
    \mM_{\mS}\arrow[d,"JH"]& \Grad(\mM_\mS)\arrow[l,swap,"\iota"]\arrow[d,"\Grad(JH)"]  & \Filt(\mM_\mS)\arrow[l,swap,"p"]\arrow[r,"\eta"]\arrow[d,"\Filt(JH)"] & \mM_\mS\arrow[d,"JH"]\\
    M_\mS & \Grad(M_\mS)\arrow[l,swap,"\simeq"] & \Filt(M_\mS)\arrow[l,swap,"\simeq"]\arrow[r,"\simeq "] & M_\mS
\end{tikzcd}\]
where we have used the fact that $\mM_\mS$ is an algebraic space, hence the hyperbolic localization diagram for it is trivial, which gives a natural isomorphism $JH\circ \eta\simeq JH\circ \iota\circ p$. one obtains then by base change a natural isomorphism:
\begin{align}
    JH_{\gamma_1+\gamma_2}\circ \eta_{\gamma_1,\gamma_2}\simeq JH_{\gamma_1+\gamma_2}\circ \oplus_{\gamma_1,\gamma_2}\circ p_{\gamma_1,\gamma_2}\simeq \oplus_{\gamma_1,\gamma_2}\circ (JH_{\gamma_1}\times JH_{\gamma_2})\circ p_{\gamma_1,\gamma_2}
\end{align}
making that the diagram of the proposition commutes.\medskip

Consider the following $2$-commutative diagrams:

\adjustbox{scale=0.9,center}{\begin{tikzcd}[column sep=tiny]
    & & \Filt(\mM_\mS)_{\gamma_1,\gamma_2,\gamma_3}\arrow[dr]\arrow[dl] & &\\
    & \Filt(\mM_\mS)_{\gamma_1,\gamma_2}\times\mM_{\mS,\gamma_3}\arrow[dl,swap,"p_{\gamma_1,\gamma_2}\times Id"]\arrow[dr,"\eta_{\gamma_1,\gamma_2}\times Id"] && \Filt(\mM_\mS)_{\gamma_1+\gamma_2,\gamma_3}\arrow[dl,swap,"p_{\gamma_1+\gamma_2,\gamma_3}"]\arrow[dr,"\eta_{\gamma_1+\gamma_2,\gamma_3}"] &\\
    \mM_{\mS,\gamma_1}\times\mM_{\mS,\gamma_2}\times\mM_{\mS,\gamma_3}\arrow[d,"JH_{\gamma_1}\times JH_{\gamma_2}\times JH_{\gamma_3}"] && \mM_{\mS,\gamma_1+\gamma_2}\times\mM_{\mS,\gamma_3}\arrow[d,swap,"JH_{\gamma_1+\gamma_2}\times JH_{\gamma_3}"] && \mM_{\mS,\gamma_1+\gamma_2+\gamma_3}\arrow[d,swap,"JH_{\gamma_1+\gamma_2+\gamma_3}"]\\
    M_{\mS,\gamma_1}\times M_{\mS,\gamma_2}\times M_{\mS,\gamma_3}\arrow[rr,"\oplus_{\gamma_1,\gamma_2}\times Id"] && M_{\mS,\gamma_1+\gamma_2}\times M_{\mS,\gamma_3}\arrow[rr,"\oplus_{\gamma_1+\gamma_2,\gamma_3}"] && M_{\mS,\gamma_1+\gamma_2+\gamma_3}
\end{tikzcd}}

and:

\adjustbox{scale=0.9,center}{\begin{tikzcd}[column sep=tiny]
    & & \Filt(\mM_\mS)_{\gamma_1,\gamma_2,\gamma_3}\arrow[dr]\arrow[dl] & &\\
    & \mM_{\mS,\gamma_1}\times \Filt(\mM_\mS)_{\gamma_2,\gamma_3}\arrow[dl,swap,"Id\times p_{\gamma_2,\gamma_3}"]\arrow[dr,"Id\times \eta_{\gamma_2,\gamma_3}"] && \Filt(\mM_\mS)_{\gamma_1,\gamma_2+\gamma_3}\arrow[dl,swap,"p_{\gamma_1,\gamma_2+\gamma_3}"]\arrow[dr,"\eta_{\gamma_1,\gamma_2+\gamma_3}"] &\\
    \mM_{\mS,\gamma_1}\times\mM_{\mS,\gamma_2}\times\mM_{\mS,\gamma_3}\arrow[d,"JH_{\gamma_1}\times JH_{\gamma_2}\times JH_{\gamma_3}"] && \mM_{\mS,\gamma_1}\times\mM_{\mS,\gamma_2+\gamma_3}\arrow[d,swap,"JH_{\gamma_1}\times JH_{\gamma_2+\gamma_3}"] && \mM_{\mS,\gamma_1+\gamma_2+\gamma_3}\arrow[d,swap,"JH_{\gamma_1+\gamma_2+\gamma_3}"]\\
    M_{\mS,\gamma_1}\times M_{\mS,\gamma_2}\times M_{\mS,\gamma_3}\arrow[rr,"Id\times\oplus_{\gamma_2,\gamma_3}"] && M_{\mS,\gamma_1+\gamma_2}\times M_{\mS,\gamma_3}\arrow[rr,"\oplus_{\gamma_1+\gamma_2,\gamma_3}"] && M_{\mS,\gamma_1+\gamma_2+\gamma_3}
\end{tikzcd}}
Applying the commutative diagram \eqref{diagcomphyploc} to $JH$, and using the monoidality of $JH$, we obtain that the $2$-isomorphism making these two diagram commutes are the same, which is the claimed compatibility with diagram \ref{diagAssocCoha}.
    \end{enumerate}
\end{proof}

We follow here the discussion of \cite[Section 3.2]{DavMein}. For any $K^{num}(X)$-graded monoid object in the $2$-category of stacks $(\mX_\gamma)_{\gamma\in K^{num}(X)}$, we consider the monoidal structure on the triangulated category $\mD_{mon}(\mX)$ given by $\boxtimes^{tw}_{\oplus}=(\boxtimes^{tw}_{\oplus,\gamma_1,\gamma_2})_{\gamma_1,\gamma_2\in K^{num}(X)}$:
\begin{align}
    \boxtimes^{tw}_{\oplus,\gamma_1,\gamma_2}:=(\oplus_{\gamma_1,\gamma_2})_!(-\boxtimes-)\{-\langle \gamma_1,\gamma_2\rangle/2\}
\end{align}
formally, it is obtained by applying the monoidal functor $\mD_!$ to the monoidal object $\mX$, and then twisting by the formula:
\begin{align}
   \boxtimes^{tw}_{\oplus}(F_{\gamma_1},\cdots,F_{\gamma_n}):=\boxtimes_\oplus(F_{\gamma_1},\cdots,F_{\gamma_n})\{-\sum_{1\leq i<j\leq n}\langle \gamma_i,\gamma_j\rangle/2\}
\end{align}
We denote by $K^{num}(X)$ the scheme $\bigsqcup_{\gamma\in K^{num}(X)}\Spec(k)$, with the monoidal product $\oplus_n:(\gamma_1,\cdots,\gamma_n)\mapsto \gamma_1+\cdots+\gamma_n$, such that a $K^{num}(X)$-graded monoid stack is equivalent to a monoid stack with a monoidal morphism to $K^{num}(X)$. A monoidal morphism $f:\mX\to\mY$ of $K^{num}(X)$-graded monoid stack induces a monoidal morphism $f_!:(\mD_{mon}(\mX),\boxtimes^{tw}_{\oplus})\to(\mD_{mon}(\mX),\boxtimes^{tw}_{\oplus})$; if $f$ is moreover of finite type, it restricts to a monoidal morphism $f_!:(\mD^-_{mon,c,-}(\mX),\boxtimes^{tw}_{\oplus})\to(\mD^-_{mon,c,-}(\mX),\boxtimes^{tw}_{\oplus})$ (recalling the notations from \eqref{defD-}).\medskip

We give now the construction of the CoHA. As we work with cohomology with compact support, we obtain a coalgebra: in general, one consider the Borel-Moore homology, which is obtained by taking the dual of the cohomology with compact support, hence one obtains an algebra.

\begin{theorem}\label{theoCoHA}
    Consider a smooth and projective Calabi-Yau threefold $X$, with a strong orientation data on $D^b\Coh(X)$, and a $t$-structure with Noetherian Abelian heart $\mA_c$ on $D^b\Coh(X)$ satisfying generic flatness. In particular, $\mM_\mA$ is an oriented $d$-critical Artin $1$-stack, locally of finite presentation, with affine diagonal.  Consider a locally closed Serre subcategory $\mathcal{S}$ of $\mA_c$, and suppose that, for any $\gamma_1,\gamma_2\in K^{num}(X)$, the map $\eta_{\gamma_1,\gamma_2}:\Filt_{\mS,\gamma_1,\gamma_2}\to \mM_{\mS,\gamma_1+\gamma_2}$ is of finite type.
    \begin{enumerate}
        \item[$i)$] 
    Then there is a natural coassociative coproduct:
    \begin{align}
\bH_c(\mM_{\mS,\gamma_1+\gamma_2},P_\mA|_{\mM_{\mS,\gamma_1+\gamma_2}}) \to \bH_c(\mM_{\mS,\gamma_1},P_\mA|_{\mM_{\mS,\gamma_1}})\otimes_k\bH_c(\mM_{\mS,\gamma_2},P_\mA|_{\mM_{\mS,\gamma_2}})\{-\langle\gamma_1,\gamma_2\rangle/2\}
    \end{align}
    defined from the extension correspondence, giving to $\bH_c(\mM_{\mS,\gamma_1},P_{\mA,\gamma_1})$ a natural structure of comonoid in $(\mD(k),\boxtimes^{tw}_\oplus)$ (the absolute CoHA)
    \item[$ii)$] If moreover the $\mM_{\mS,\gamma}$ are of finite type, there is a natural coassociative coproduct: \begin{align}
(JH_{\gamma_1+\gamma_2})_!(P_\mA|_{\mM_{\mS,\gamma_1+\gamma_2}}) \to (\oplus_{\gamma_1,\gamma_2})_!((JH_{\gamma_1})_!(P_\mA|_{\mM_{\mS,\gamma_1}})\boxtimes (JH_{\gamma_2})_!(P_\mA|_{\mM_{\mS,\gamma_2}}))\{-\langle\gamma_1,\gamma_2\rangle/2\}
    \end{align}
    defined from the extension correspondence, giving to $(JH_\gamma)_!(P_\mA|_{\mM_{\mS,\gamma}})$ a natural structure of comonoid in $(\mD^-_-(M_\mS),\boxtimes^{tw}_\oplus)$ (the relative CoHA), whose hypercohomology with compact support is the absolute CoHA.
    \end{enumerate}
\end{theorem}

\begin{remark}\label{remAssumCoHA}
    Consider a smooth projective CY3 $X$ such that $D^b\Coh(X)$ admits a strong orientation data compatible with direct sum (the orientation data of \cite{JOYCEorient} is not known to upgrades to a strong orientation data, but it is plausible that it is the case). The construction of the the absolute CoHa works in particular in the following situation:
    \begin{itemize}
        \item When $\mA_c=Coh(X)$ is the heart of the classical $t$-structure on $D^b\Coh(X)$: indeed, the corresponding heart is known to be Noetherian, and satisfying generic flatness. The map $\Filt_{\gamma_1,\gamma_2}\to\mM_{\mA,\gamma_1+\gamma_2}$ is then bounded by the Grothendieck Quot construction, see \cite[Proposition 9.5]{joyce2003ConfigurationsIA}. One can then take various Serre subcategory, giving a locally closed substacks: torsion or torsion free sheaves, semistable sheaves for a Gieseker stability conditions, sheaves with support on a locally closed subvariety, sheaves with support of dimension $\leq i$,... In general, the $\mM_{\mA,\gamma}$ will not be bounded, so there will be no relative CoHA, but there is a relative CoHA for locally closed substacks of finite type (e.g., for Gieseker-semistable sheaves).
        \item When $\mA_c$ is the heart of a $t$-structure coming from an algebraic stability condition in a good connected component $\Stab^\ast(X)$ (e.g., a component containing a stability condition from \cite{Bayer2011BridgelandSC}), from \cite[Proposition 4.21]{PiyaratneToda+2019+175+219} (Proposition \ref{propPiyaTod} here), the heart $\mA_c=\mPP((0,1])$ is Noetherian and satisfy generic flatness, and by \cite[Lemma 3.15]{Toda2007ModuliSA}, this heart satisfies boundedness of quotients, \ie in particular the maps $\eta_{\gamma_1,\gamma_2}$ are of finite type. One obtains then an absolute CoHA for $\mM_\mA$, and a restriction of it for any subinterval $I\subset (0,1]$ (indeed, $\mPP(I)$ is a Serre subcategory from Lemma \ref{lemIntSerre}, and it forms an open substack by arguing as in the proof of Theorem \ref{theoKSwcf}). When $I$ is of length $<1$, the $\mM_{\mS,\gamma}$ are bounded from Theorem \ref{theoKSwcf}, and then there is moreover a relative CoHA.
    \end{itemize}
\end{remark}

\begin{proof}
From Halpern-Leistner's Proposition \ref{propheart} and Remark \ref{remarkorient}, $\mM_\mA$ is $d$-critical Artin $1$-stack, with an orientation compatible with direct sum, locally of finite presentation, with affine diagonal. From Theorem \ref{theohyploc}, we have a canonical isomorphism:
\begin{align}
    p_!\eta^\ast P_{\mM_\mA}\simeq P_{\Grad(\mM_\mA)}\{-\Ind_{\mM_\mA}/2\}
\end{align}
From Lemma \ref{lemgradprod} and Corollary \ref{corfunctjoyce}, the restriction of this isomorphism along the open and closed immersion $\mM_{\mA,\gamma_1}\times \mM_{\mA,\gamma_2}\to \Grad(\mM_\mA)$ gives a canonical isomorphism in $\mA_{mon,c}(\mM_{\mA,\gamma_1}\times \mM_{\mA,\gamma_2})$:
\begin{align}\label{isomCoha}
    (p_{\gamma_1,\gamma_2})_!(\eta_{\gamma_1,\gamma_2})^\ast P_{\mM_{\mA,\gamma_1+\gamma_2}}\simeq P_{\mM_{\mA,\gamma_1}\times \mM_{\mA,\gamma_2}}\{-\langle \gamma_1,\gamma_2\rangle/2\}\simeq P_{\mM_{\mA,\gamma_1}}\boxtimes P_{\mM_{\mA,\gamma_2}}\{-\langle \gamma_1,\gamma_2\rangle/2\}
\end{align}
Consider now $\gamma_1,\gamma_2,\gamma_3\in K^{num}(X)$, and look at $\mM_{\mA,\gamma_1}\times\mM_{\mA,\gamma_2}\times\mM_{\mA,\gamma_3}$ as an open and closed substack of $Map(B\bG_m\times B\bG_m,\mM_\mA)$, where the factors have $\bZ^2$ grading respect $(0,0),(1,0),(1,1)$. Then, by Lemma \ref{lemcomposHyploc}, \ref{lemgradprod}, and the fact that the isomorphism of Theorem \ref{theohyploc} is compatible with products, the following square of isomorphisms in $\mA^{mon,c}(\mM_{\mA,\gamma_1}\times \mM_{\mA,\gamma_2}\times\mM_{\mA,\gamma_3})$ is commutative (where we denote $P_\gamma:=P_{\mM_{\mA,\gamma}}$ for readability):
\begin{equation}\label{diagassoc1}
\begin{tikzcd}
    \begin{tabular}{c}$(p_{\gamma_1,\gamma_2}\times Id)_!(\eta_{\gamma_1,\gamma_2}\times Id)^\ast  $\\$(p_{\gamma_1+\gamma_2,\gamma_3})_!(\eta_{\gamma_1+\gamma_2,\gamma_3})^\ast P_{\gamma_1+\gamma_2+\gamma_3}$\end{tabular}\arrow[r,"\simeq"]\arrow[d,"\simeq"] & \begin{tabular}{c}$(Id\times p_{\gamma_2,\gamma_3})_!(Id\times\eta_{\gamma_2,\gamma_3})^\ast$\\$  (p_{\gamma_1,\gamma_2+\gamma_3})_!(\eta_{\gamma_1,\gamma_2+\gamma_3})^\ast P_{\gamma_1+\gamma_2+\gamma_3}$\end{tabular}\arrow[d,"\simeq"]\\
    \begin{tabular}{c}$(p_{\gamma_1,\gamma_2}\times Id)_!(\eta_{\gamma_1,\gamma_2}\times Id)^\ast$\\$ (P_{\gamma_1+\gamma_2}\boxtimes P_{\gamma_3})\{-\langle \gamma_1+\gamma_2,\gamma_3\rangle/2\}$\end{tabular}\arrow[d,"\simeq"] & \begin{tabular}{c}$(Id\times p_{\gamma_2,\gamma_3})_!(Id\times \eta_{\gamma_2,\gamma_3})^\ast$\\$ (P_{\gamma_1}\boxtimes P_{\gamma_2+\gamma_3})\{-\langle \gamma_1,\gamma_2+\gamma_3\rangle/2\}$\end{tabular}\arrow[d,"\simeq"] \\
    \begin{tabular}{c}$(P_{\gamma_1}\boxtimes P_{\gamma_2}\boxtimes P_{\gamma_3})$\\$\{-\langle \gamma_1+\gamma_2,\gamma_3\rangle/2 -\langle \gamma_1,\gamma_2\rangle/2\}$\end{tabular}\arrow[r,"\simeq"] & \begin{tabular}{c}$(P_{\gamma_1}\boxtimes P_{\gamma_2}\boxtimes P_{\gamma_3})$\\$\{-\langle \gamma_1,\gamma_2+\gamma_3\rangle/2 -\langle \gamma_2,\gamma_3\rangle/2\}$\end{tabular}
\end{tikzcd}    
\end{equation}
where the vertical arrows come from \eqref{isomCoha}, and the upper horizontal one by base change in the diagram \ref{diagAssocCoha}. Restricting along $\mM_{\mS,\gamma_i}\to\mM_{\mA,\gamma}$, and using the fact that the diagrams of filtrations for $\mM_\mS$ are Cartesian over those for $\mM_\mA$ as $\mS$ is a Serre subcategory, we obtain natural isomorphisms in $\mD^b_{mon,c}(\mM_{\mS,\gamma_1}\times\mM_{\mS,\gamma_2})$:
\begin{align}\label{isomCohaS}
    (p_{\gamma_1,\gamma_2})_!(\eta_{\gamma_1,\gamma_2})^\ast (P_{\mM_\mA}|_{\mM_{\mS,\gamma_1+\gamma_2}})\simeq (P_{\mM_\mA}|_{\mM_{\mS,\gamma_1}})\boxtimes (P_{\mM_\mA}|_{\mM_{\mS,\gamma_2}})\{-\langle \gamma_1,\gamma_2\rangle/2\}
\end{align}
which fits into a similar commutative diagram.\medskip

We describe now how to build the relative CoHA of $ii)$, using the morphism $JH:\mM_\mS\to M_\mS$ to the $K^{num}(X)$-graded monoid $M_\mS$: the absolute CoHA  is built exactly in the same way, using instead the morphism $\mM_\mS\to K^{num}(X)$ to the $K^{num}(X)$-graded monoid $K^{num}(X)$, and the fact that the hypercohomology with compact support of the relative CoHA is the abosolute CoHA is obtained by applying the monoidal map $(\mM_\mS\to K^{num}(X))_!:(\mD^-_-(\mM_\mS),\boxtimes^{tw}_{\oplus})\to(\mD^-_-(K^{num}(X)),\boxtimes^{tw}_{\oplus})$ to the relative CoHA will then be automatic.\medskip

From Lemma \ref{lemmaFiltAb}, $\mM_\mA$ is $\Theta$-reductive, and then from \eqref{eqGradS}, $\mM_\mS$ is $\Theta$-reductive too, \ie the maps $\eta_{\gamma_1,\gamma_2}:\mM_{\mS,\gamma_1}\times\mM_{\mS,\gamma_1}\to \mM_{\mS,\gamma_1+\gamma_2}$ satisfies the valuative criterion for properness with respect to any DVR. Then, as they are moreover assumed to be of finite type, they are then proper, and representable by algebraic space: in particular, the natural morphism $(\eta_{\gamma_1,\gamma_2})_!\to (\eta_{\gamma_1,\gamma_2})_\ast$ is an isomorphism, \ie we get a natural morphism:
\begin{align}\label{morpheta}
    Id\to (\eta_{\gamma_1,\gamma_2})_\ast(\eta_{\gamma_1,\gamma_2})^\ast\simeq (\eta_{\gamma_1,\gamma_2})_!(\eta_{\gamma_1,\gamma_2})^\ast
\end{align}
which will be the key point, as it is generally the case to build CoHA. Notice that the morphism $\eta_!\to\eta_\ast$ is compatible with base change and composition by Lemma \ref{lemmaspecproper} and \ref{lemmapurcomp}, and the adjunction morphisms are also compatible with base change and composition, so we obtain that the morphism $Id\to\eta_!\eta^\ast$, for a proper morphism $\eta$, is compatible with base change and composition. Using the commutation of the diagram of Proposition \ref{propModuli} $ii)$, we obtain a natural morphism:
\begin{align}
(JH_{\gamma_1+\gamma_2})_!&\to (JH_{\gamma_1+\gamma_2})_!(\eta_{\gamma_1,\gamma_2})_!(\eta_{\gamma_1,\gamma_2})^\ast \nn\\
&\simeq (\oplus_{\gamma_1,\gamma_2})_!(JH_{\gamma_1}\times JH_{\gamma_2})_!(p_{\gamma_1,\gamma_2})_!(\eta_{\gamma_1,\gamma_2})^\ast 
\end{align}
Using the fact that $Id\to\eta_!\eta^\ast$ is compatible with composition and base change in the commutative diagrams at the end of the proof of Proposition \ref{propModuli} $ii)$, we obtain the following commutative diagram:
\begin{equation}\label{diagassoc2}
\begin{tikzcd}
& (JH_{\gamma_1+\gamma_2+\gamma_3})_!\arrow[dl]\arrow[dd]\arrow[dr] & \\
\begin{tabular}{c}$(\oplus_{\gamma_1+\gamma_2,\gamma_3})_!(JH_{\gamma_1+\gamma_2}\times JH_{\gamma_3})_!$\\$(p_{\gamma_1+\gamma_2,\gamma_3})_!(\eta_{\gamma_1+\gamma_2,\gamma_3})^\ast$\end{tabular}\arrow[d] & & \begin{tabular}{c}$(\oplus_{\gamma_1,\gamma_2+\gamma_3})_!(JH_{\gamma_1}\times JH_{\gamma_2+\gamma_3})_!$\\$(p_{\gamma_1,\gamma_2+\gamma_3})_!(\eta_{\gamma_1,\gamma_2+\gamma_3})^\ast$\end{tabular}\arrow[d]\\
\begin{tabular}{c}$(\oplus_{\gamma_1+\gamma_2,\gamma_3})_!(\oplus_{\gamma_1,\gamma_2}\times Id)_!$\\$(JH_{\gamma_1}\times JH_{\gamma_2}\times JH_{\gamma_3})_!$\\$(p_{\gamma_1,\gamma_2}\times Id)_!(\eta_{\gamma_1,\gamma_2}\times Id)^\ast $\\$(p_{\gamma_1+\gamma_2,\gamma_3})_!(\eta_{\gamma_1+\gamma_2,\gamma_3})^\ast$\end{tabular} & \begin{tabular}{c}$(\oplus_{\gamma_1,\gamma_2,\gamma_3})_!$\\$(JH_{\gamma_1}\times JH_{\gamma_2}\times JH_{\gamma_3})_!$\\$(p_{\gamma_1,\gamma_2,\gamma_3})_!(\eta_{\gamma_1,\gamma_2,\gamma_3})^\ast$\end{tabular}\arrow[r,"\simeq"]\arrow[l,swap,"\simeq"] &  \begin{tabular}{c}$(\oplus_{\gamma_1,\gamma_2+\gamma_3})_!(Id\times \oplus_{\gamma_2,\gamma_3})_!$\\$(JH_{\gamma_1}\times JH_{\gamma_2}\times JH_{\gamma_3})_!$\\$(Id\times p_{\gamma_2,\gamma_3})_!(Id\times\eta_{\gamma_2,\gamma_3})^\ast  $\\$(p_{\gamma_1,\gamma_2+\gamma_3})_!(\eta_{\gamma_1,\gamma_2+\gamma_3})^\ast $\end{tabular}
\end{tikzcd}    
\end{equation}
\medskip

We obtain finally the comultiplication:
\begin{align}
    (JH_{\gamma_1+\gamma_2})_!(P_{\mM_\mA}|_{\mM_{\mS,\gamma_1+\gamma_2}})&\to 
    (\oplus_{\gamma_1,\gamma_2})_!(JH_{\gamma_1}\times JH_{\gamma_2})_!(p_{\gamma_1,\gamma_2})_!(\eta_{\gamma_1,\gamma_2})^\ast (P_{\mM_\mA}|_{\mM_{\mS,\gamma_1+\gamma_2}})\nn\\&\simeq (\oplus_{\gamma_1,\gamma_2})_!\Bigl((JH_{\gamma_1})_!(P_{\mM_\mA}|_{\mM_{\mS,\gamma_1}})\boxtimes (JH_{\gamma_1})_!(P_{\mM_\mA}|_{\mM_{\mS,\gamma_2}})\Bigr)\{-\langle \gamma_1,\gamma_2\rangle/2\}
\end{align}
And the two commutative diagrams \ref{diagassoc1}  and \ref{diagassoc2} give the associativity of the comultiplication.
\end{proof}

\begin{remark}
    We have built here the CoHA as a comonoidal object in a triangulated category. Using the formalism of $2$-segal spaces and the Waldhausen construction from \cite{Dyckerhoff2012HigherSS}, and the homotopy coherent version of six functor formalisms, one could probably write it as a comonoid in a stable $(\infty,1)$-category. Indeed, the only part in the above construction which uses the gluing technology from \cite{Joycesymstab} is the construction of the isomorphism \eqref{isomCoha}, and the check that the diagram \eqref{diagassoc1} is commutative, which is done in an Abelian category. The remaining part use the six operations, and do not need any gluing. As this paper is written mostly in the $2$-categorical language, we have not tried to do such a homotopy coherent construction. Moreover, there seems to be no substantial need to give a homotopy coherent version of this construction, as the CoHA is defined globally, and one do not have to glue it from local presentations using quiver with potential.
\end{remark}

\bibliography{ref}
\bibliographystyle{alpha}

\end{document}